\documentclass{amsart}
\usepackage{amsfonts,amssymb,amsmath,amsthm,amscd, geometry, lipsum}
\usepackage{fullpage}
\usepackage{float}
\usepackage{tikz, tikz-cd}
\usepackage{url}
\usepackage{enumerate}
\usepackage{fancyhdr, hyperref}
\usepackage[utf8]{inputenc}

\hypersetup{
    colorlinks=true,
    linkcolor=magenta,
    citecolor=magenta,
    urlcolor=magenta
}

\makeatletter
\AtBeginDocument{
  \hypersetup{linktoc=none} 
}
\makeatother

\newtheorem{theorem}{Theorem}[section]
\newtheorem{lemma}[theorem]{Lemma}
\newtheorem{proposition}[theorem]{Proposition}
\newtheorem{corollary}[theorem]{Corollary}
\newtheorem{conjecture}[theorem]{Conjecture}

\theoremstyle{definition}

\theoremstyle{definition}
\newtheorem{example}[theorem]{Example}

\theoremstyle{definition}
\newtheorem{remark}[theorem]{Remark}

\DeclareFontEncoding{OT2}{}{}
\newcommand{\textcyr}[1]{{\fontencoding{OT2}\fontfamily{wncyr}\fontseries{m}\fontshape{n}\selectfont #1}}
\newcommand{\Sha}{{\mbox{\textcyr{Sh}}}}
\newcommand{\Zha}{{\mbox{\textcyr{Zh}}}}

\begin{document} 

\title{Arithmetic Hirzebruch-Zagier divisors and central derivative values of Rankin-Selberg $L$-functions}

\author{Jeanine Van Order}

\begin{abstract} Let $E$ be an elliptic curve parametrized by a newform $\phi \in S_2(\Gamma_0(N))$. Let $k$ be a quadratic field of discriminant $d_k$ prime to $N$. 
Given a character $\chi$ of the ideal class group of $k$ with theta series $\theta(\chi)$, we relate the central derivative values $\Lambda'(E/k, \chi, 1) = \Lambda'(1/2, \phi \times \theta(\chi))$
of the $\chi$-twisted $L$-function of $E/k$ to arithmetic heights of Hirzebruch-Zagier divisors on $X_0(N)\times X_0(N)$ when $k$ is imaginary quadratic, 
and to sums of Green's functions of Hirzebruch-Zagier divisors along real geodesic cycles of $X_0(N) \times X_0(N)$ determined by ideal classes when $k$ is real quadratic. 
More generally, we refine the higher Gross-Zagier formulae for CM cycles on spin Shimura varieties of any dimension this way. 
This gives two arithmetic height formulae for $\Lambda'(E/k, \chi, 1)$ when $k$ is imaginary quadratic,
as well as a distinct proof of the Gross-Zagier formula, with implied relations between the arithmetic 
heights of Heegner divisors on $X_0(N)$ and Hirzebruch-Zagier divisors on $X_0(N) \times X_0(N)$. 
We also explain connections to the refined conjecture of Birch-Swinnerton-Dyer, and to Fourier coefficients of half-integral weight forms.  \end{abstract}

\maketitle

\section{Introduction}

The theorem of Gross-Zagier \cite[Theorem $\S$I (6.3)]{GZ} represents one of the most significant advances on the conjecture of Birch and Swinnerton-Dyer to date, 
and forms the foundation for all progress made on the case of Mordell-Weil rank one through the techniques of Kolyvagin, Euler systems, and the Iwasawa main conjectures. 
To recall it, let $E$ be an elliptic curve of conductor $N$ defined over ${\bf{Q}}$. 
Hence, $E$ is modular by fundamental work of Wiles \cite{Wi}, Taylor-Wiles \cite{TW}, 
and Breuil-Conrad-Diamond-Taylor \cite{BCDT}, and hence parametrized by a cuspidal newform 
\begin{align*} \phi(\tau) &= \sum\limits_{m \geq 1} c_{\phi}(m) e(m \tau) = \sum\limits_{m \geq 1} a_{\phi}(m) m^{\frac{1}{2}} e(m \tau) 
\in S^{\operatorname{new}}_2(\Gamma_0(N)), \quad \tau = u + iv \in \mathfrak{H}, e(z) = \exp(2 \pi i z). \end{align*} 
In particular, the Hasse-Weil $L$-function $L(E, s) = L(E/{\bf{Q}},s)$ of $E$ over ${\bf{Q}}$ has an analytic continuation 
$\Lambda(E, s) = L_{\infty}(E, s) L(E, s)$ to all $s \in {\bf{C}}$ given by a shift of the standard $L$-function, $\Lambda(s, \phi)= L_{\infty}(s, \phi)L (s, \phi)$
\begin{align*} \Lambda(E, s) = L_{\infty}(E, s) L(E, s) &= \Lambda(s-1/2, \phi) = L_{\infty}(s- 1/2, \phi) L(s-1/2, \phi). \end{align*}
Here, $L_{\infty}(s, \phi) = (2 \pi)^{-s} \Gamma(s)$ denotes the archimedean local Euler factor, 
with $L(s, \phi) = \prod_{p < \infty} L(s, \pi(\phi)_p)$ the finite Euler product whose Dirichlet series expansion for $\Re(s) >1$ is given by
\begin{align*} L(s, \phi) &= \sum\limits_{m \geq 1} a_{\phi}(m) m^{-s}  = \sum\limits_{m \geq 1} c_{\phi}(m) m^{-(s + 1/2)}. \end{align*}
Hence, $\Lambda(s, \phi)$ satisfies a symmetric functional equation $\Lambda(s, \phi) = \pm \Lambda(1-s, \phi)$
with odd sign or root number $\Lambda(s, \phi) = - \Lambda(1-s, \phi)$ if and only if $\phi$ is invariant under the Fricke involution $w_N \phi = \phi$.

Let $k$ be an imaginary quadratic field of discriminant $d_k$ prime to $N$ and character $\eta_k(\cdot) = (\frac{d_k}{\cdot})$.
Let $C(\mathcal{O}_k)$ denote its ideal class group, with $h_k = \# C(\mathcal{O}_k)$. Let $\chi \in C(\mathcal{O}_k)^{\vee}$ be any class group character, with
\begin{align*} \theta(\chi)(\tau) &= \sum\limits_{A \in C(\mathcal{O}_k)} \chi(A) \theta_A(\tau) \in M_1(\Gamma_0(\vert d_k \vert), \eta_k) \end{align*} 
the corresponding Hecke theta series of weight $1$, level $\Gamma_0(\vert d_k \vert)$, and character $\eta_k$. 
The Hasse-Weil $L$-function $L(E/k, \chi, s)$ of $E$ over $K$ twisted by $\chi$ has an analytic continuation 
$\Lambda(E/k, \chi, s) = L_{\infty}(E/k, \chi, s) L(E/k, \chi, s)$ given
by a shift of the completed Rankin-Selberg $L$-function $\Lambda(s, \phi \times \theta(\chi)) = L_{\infty}(s, \phi \times \theta(\chi)) L(s, \phi \times \theta(\chi))$,
\begin{align*} \Lambda(E/k, \chi, s) &= \Lambda(s-1/2, \phi \times \theta(\chi)) = L_{\infty}(s-1/2, \phi \times \theta(\chi)) L(s-1/2, \phi \times \theta(\chi)). \end{align*}
Here, the archimedean local factor $L_{\infty}(s, \phi \times \theta(\chi)) =  \Gamma_{\bf{C}}( s+\frac{1}{2} )^2$ for $\Gamma_{\bf{C}}(s) := 2(2\pi)^{-s} \Gamma(s)$
does not depend on the choice of class group character $\chi \in C(\mathcal{O}_k)$.
This $L$-function has a well-known analytic continuation to $s \in {\bf{C}}$, and satisfies the symmetric functional equation
\begin{align*} \Lambda(s, \phi \times \theta(\chi)) &= \eta_k(-N) \vert d_k N \vert^{1-2s} \Lambda(1-s, \phi \times \theta(\chi)). \end{align*}
In particular, if $\eta_k(-N) = -1$ so that the sign of this functional equation is odd, 
then the central value $\Lambda(1/2, \phi \times \theta(\chi)) = \Lambda(E/k, \chi, 1)$ is forced to vanish, 
and it makes sense to study central derivative values $\Lambda'(1/2, \phi \times \theta(\chi)) = \Lambda'(E/k, \chi, 1)$. 
This happens for instance if the level $N$ is squarefree, and the number of prime divisors $q \mid N$ which remain inert in $k$ %(so $(\frac{d_k}{q}) = -1$) 
is even, or more stringently if $N$ is squarefree and totally split, so that the `Heegner hypothesis'' of Gross-Zagier \cite{GZ} holds. 
In this latter setting, the compactified modular curve $X_0(N)$ comes equipped with a family of Heegner divisors $y$ of conductor $d_k$. 
In brief, there are $h_k$ many Heegner points $z: E \rightarrow E'$ of conductor 
$d_k$ on $X_0(N)(k[1])$, where $k[1]$ denotes the Hilbert class field of $k$. More precisely, the class group 
$C(\mathcal{O}_k) \cong \operatorname{Gal}(k[1]/k)$ acts simply transitively on the set of these points, and we denote this natural action by $z^A$.
We obtain from each Heegner point $z \in X_0(N)(k[1])$ the divisor $y = (z) - (\infty)$ in the corresponding Jacobian $J_0(N)(k[1])$. 
  
\begin{theorem}[Gross-Zagier]\label{Gross-Zagier} 

Let $E$ be an elliptic curve of conductor $N$, parametrized by a cuspidal newform $\phi \in S_2^{\operatorname{new}}(\Gamma_0(N))$.
Let $k$ be an imaginary quadratic field of discriminant $d_k$ prime to $2 N$ and odd Dirichlet character $\eta_k(\cdot) = (\frac{d_k}{\cdot})$. 
Assume that $N$ is squarefree and totally split, so that $\eta_k(-N)=-1$. 
Then for any character $\chi$ of the ideal class group $C(\mathcal{O}_k)$ of $k$, we have the central derivative value formula 
\begin{align*} L'(E/k, \chi, 1) = L'(1/2, \phi \times \theta(\chi)) &= \frac{8 \pi ^2 \vert \vert \phi \vert \vert^2}{h_k u_k^2} 
\cdot \left[ y_{\chi}^{\phi}, y_{\chi}^{\phi} \right]_{\operatorname{NT}}. \end{align*}
Here, $\vert \vert \phi \vert \vert^2 = \langle\phi, \phi \rangle$ denotes the Petersson inner product of $\phi$,
$h_k = \#C(\mathcal{O}_k)$ the class number, $u_k = w_k/2$ half the number of roots of unity in $k$,
and $\left[ y_{\chi}^{\phi}, y_{\chi}^{\phi} \right]_{\operatorname{NT}}$ 
the N\'eron-Tate height of the projection $y_{\chi}^{\phi}$
to the $\phi$-isotypical component of $J_0(N)(k[1]) \otimes {\bf{C}}$ of the twisted Heegner divisor
$y_{\chi} = \sum_{A \in C(\mathcal{O}_k)} \chi(A) y^A.$ \end{theorem}
 
Our main theorem (Theorem \ref{main}, Corollary \ref{ec}) expresses the same central derivative values 
$\Lambda'(E/k, \chi, 1)$ as the arithmetic intersections of Hirzebruch-Zagier divisors on $X_0(N) \times X_0(N)$, 
so that a comparison with the Gross-Zagier formula implies an identity between Heegner divisors on $X_0(N)$ and Hirzebruch-Zagier divisors on $X_0(N) \times X_0(N)$ (Corollary \ref{heightrelation}).
Theorem \ref{Gross-Zagier} has been generalized by Zhang \cite{SWZ, SWZ2, SWZ3, SWZ4} and Yuan-Zhang-Zhang \cite{YZZ} to Shimura curves via the Jacquet-Langlands 
correspondence \cite{Ja, JL}, in the spirit of Waldspurger's theorem \cite{Wa}. It has also been generalized via Borcherds regularized theta lifts to obtain `higher Gross-Zagier formulae'' 
for spin Shimura varieties by Bruinier-Yang \cite{BY} and Andreatta-Goren-Howard-Madapusi Pera \cite{AGHMP}. 
We refine these latter results for $X_0(N) \times X_0(N)$ via the underlying quadratic 
basechange equivalence for the $L$-function $\Lambda(s, \phi \times \theta(\chi)) = \Lambda(E/k, \chi, s-1/2)$. 
This appears to be a completely new perspective on the geometric side. We complete the following steps in this direction: \\

\begin{itemize}

\item Derive an equivalent arithmetic height formula for $\Lambda'(E/k, \chi, 1)$ in terms of twisted sums of Faltings heights of 
Hirzebruch-Zagier divisors on $X_0(N)\times X_0(N)$ (Theorem \ref{main}, Corollary \ref{ec}). \\

\item Compute sums of automorphic Green's functions for arithmetic divisors on spin Shimura varieties of any dimension $n$ over ${\bf{Q}}$
along geodesic cycles corresponding to ideal classes of a real quadratic field to obtain new 
integral presentations for central derivative values (Theorem \ref{sums} (ii), Corollary \ref{simplified} (ii)). \\

\item Derive novel expressions for the Birch-Swinnerton-Dyer constants of $E(k)$ (Corollary \ref{BSD}). \\  

\item Develop the approach of Bruinier-Yang \cite[Theorem 7.7]{BY} to reprove Gross-Zagier in terms of regularized theta lifts on $X_0(N)$ (Theorem \ref{GZF}), 
which implies a relation between heights of Heegner divisors on $X_0(N)$ and heights of Hirzebruch-Zagier divisors on $X_0(N) \times X_0(N)$ (Corollary \ref{heightrelation}). \\

\item Derive novel expressions $\Lambda'(E/k, \chi, 1)$ in terms of Fourier coefficients of certain half-integral weight Maass forms (Theorem \ref{metatrace}). \\ \end{itemize}

\subsection{Main results} 

In the setting of Theorem \ref{Gross-Zagier}, our main result gives an equivalent expression
\begin{equation*}\begin{aligned} \Lambda'(E/k, \chi, 1) = \Lambda'(1/2, \phi \times \theta(\chi)) 
&= -2 \sum\limits_{A \in C(\mathcal{O}_k)} \chi (A) \left[ \widehat{\mathcal{Z}}_A(f_{0,A}), \widehat{\mathcal{Z}}(V_{0,A}) \right]. \end{aligned}\end{equation*}
Here, each 
\begin{align*} \left[ \widehat{\mathcal{Z}}_A(f_{0,A}), \widehat{\mathcal{Z}}(V_{0,A}) \right] &= \left[ \widehat{\mathcal{Z}}_A(f_{0,A}), \widehat{\mathcal{Z}}(V_{0,A}) \right]_{\infty}
+ \left[ \widehat{\mathcal{Z}}_A(f_{0,A}), \widehat{\mathcal{Z}}(V_{0,A}) \right]_{\operatorname{fin}} \end{align*}
denotes the Faltings height of an arithmetic Hirzebruch-Zagier-like divisor $\widehat{\mathcal{Z}}_A(f_{0,A}) = (\mathcal{Z}(f_{0,A}), \Phi(f_{0,A}, \cdot))$ on 
the integral model $\mathcal{X}_0(N) \times \mathcal{X}_0(N)$ of $X_0(N) \times X_0(N)$, with coefficients given by the Fourier coefficients of a certain harmonic Maass form $f_{0,A}$ constructed from $\phi$,
along a zero cycle $\mathcal{Z}(V_{0,A}) \subset \mathcal{X}_0(N) \times \mathcal{X}_0(N)$ attached to the class $A \in C(\mathcal{O}_k)$ (see Theorem \ref{GZ-AHZ}).
We now describe this expression in more detail, together with the related steps outlined above.

\subsubsection{Setup with automorphic Green's functions on spin Shimura varieties}

Let $(V, Q)$ be a rational quadratic space of signature $(n,2)$ for any integer $n \geq 1$. 
Let $\operatorname{GSpin}(V)$ denote the corresponding general spin group, which fits into the short exact sequence 
\begin{align*} 1 &\longrightarrow {\bf{G}}_m \longrightarrow \operatorname{GSpin}(V) \longrightarrow \operatorname{SO}(V) \longrightarrow 1. \end{align*}
Let $D(V) = D^{\pm}(V)$ be a fixed connected component of the Grassmannian of oriented negative definite $2$-planes in $V({\bf{R}})$.
Given any integral lattice $L \subset V$ with corresponding compact open subgroup $K = K_L$ of $\operatorname{GSpin}(V)({\bf{A}}_f)$,
we consider the corresponding Shimura variety $X_K$ with complex points 
\begin{align*} X_K({\bf{C}}) &= \operatorname{GSpin}(V)({\bf{Q}}) \backslash D(V) \times \operatorname{GSpin}(V)({\bf{A}}_f)/K. \end{align*}
This determines a quasiprojective variety of dimension $n$ over ${\bf{Q}}$. 
Fixing a set of representatives $h$ for the finite set $\operatorname{GSpin}(V)({\bf{Q}}) \backslash \operatorname{GSpin}(V)({\bf{A}}_f)/K$ and writing 
$\Gamma_h = \operatorname{GSpin}(V)({\bf{Q}}) \cap h K h^{-1}$ for the arithmetic subgroup for each 
representative, we have the decomposition into geometrically connected components
\begin{align*} X_K({\bf{C}}) &= \coprod\limits_{h \in \operatorname{GSpin}(V)({\bf{Q}}) \backslash \operatorname{GSpin}(V)({\bf{A}}_f)/K} \Gamma_h \backslash D(V). \end{align*}
These Shimura varieties come equipped with algebraic $r$-cycles for each subspace of signature $(r, 2)$, where $0 \leq r \leq n-1$.
In particular, we have the following well-known construction of divisors. Let $L^{\vee}$ denote the dual lattice of $L$, and $L^{\vee}/L$ the discriminant group. 
Given any vector $x \in V$ with $Q(x) = m >0$, consider orthogonal complement $D(V)_x = \left\lbrace z \in D(V) : (z, x) =0 \right\rbrace$. 
Given any coset $\mu \in L^{\vee}/L$ and positive rational $m >0$, we have a special divisor $Z(\mu, m) \subset X_K$ defined by 
\begin{align*} Z(\mu, m) &= \coprod\limits_{ h \in \operatorname{GSpin}(v)({\bf{Q}}) \backslash \operatorname{GSpin}(v)({\bf{A}}_f) /K } \Gamma_h 
\Big\backslash \left(  \coprod\limits_{ x \in \mu_h + L_h \atop Q(x) = m } D(V)_x \right), \end{align*} 
where $L_h \subset V$ denotes the lattice determined by $\widehat{L}_h = h \cdot \widehat{L}$,
and $\mu_h = h \cdot \mu \in L_h^{\vee}/L_h$. Note that when $(V, Q)$ has signature $(1,2)$, these $Z(\mu, m)$ recover Heegner/CM divisors on Shimura curves. 
When $(V,Q)$ has signature $(2,2)$, they recover Hirzebruch-Zagier divisors on quaternionic Hilbert modular surfaces.

Important theorems of Borcherds \cite{Bo} and Bruinier \cite{Br} give explicit constructions
of the Arakelov theoretic automorphic Green's functions for these divisors $Z(\mu, m)$. To describe these, let 
\begin{align*} \omega_L: \operatorname{Mp}_2({\bf{Z}}) \rightarrow \mathfrak{S}_L \cong {\bf{C}}[L^{\vee}/L] \subset \operatorname{Aut} \left( \mathcal{S}(V({\bf{A}}_f) ) \right) \end{align*} 
denote the Weil representation associated with the chosen even lattice $L \subset V$. Let 
$\theta_L(\tau, z): \mathfrak{H} \times D(V) \rightarrow \mathfrak{S}_L^{\vee}$
denote the corresponding Siegel theta series defined in $(\ref{Stheta3})$ below. 
Let $f = f^+ + f^-\in H_{l}(\omega_L)$ be any harmonic weak Maass form of weight 
$l =1 - n/2$ and representation $\omega_L$. Here, we write 
\begin{align*} f^+(\tau) &= \sum\limits_{\mu \in  L^{\vee}/L} \sum\limits_{m \in {\bf{Q}} \atop m \gg - \infty} c_f^+(\mu, m) e(m \tau) {\bf{1}}_{\mu},
\quad {\bf{1}}_{\mu}:= \operatorname{char}(\mu + \widehat{L}) \in \mathfrak{S}_L \cong {\bf{C}}[L^{\vee}/L]\end{align*}
to denote the Fourier series expansion of the holomorphic part of $f$.
Write $M_l^!(\omega_L) \subset H_l(\omega_L)$ for the space of weakly holomorphic forms whose poles are supported at the cusps,
the subspace of holomorphic forms $M_l(\omega_L) \subset M_l^!(\omega_L)$, and the subspace of cusp forms
$S_l(\omega_L) \subset M_l (\omega_L)$. Here, we also have the antilinear differential
operator $\xi_l: H_l(\omega_L) \rightarrow S_{2-l}(\overline{\omega}_L)$ of Bruinier-Funke \cite{BF} defined in $(\ref{xi})$ below, 
which allows us to identify the weakly holomorphic forms as the kernel $\ker(\xi_l) = M_l^!(\omega_L)$. Given harmonic weak Maass forms 
\begin{align*} f(\tau) = \sum\limits_{\mu \in L^{\vee}/L} f_{\mu}(\tau) {\bf{1}}_{\mu} \in H_l(\omega_L) \quad\text{and}\quad
g(\tau) = \sum\limits_{\mu \in L^{\vee}/L} g_{\mu}(\tau) {\bf{1}}_{\mu} \in H_{-l}(\omega_L^{\vee}), \end{align*}
we consider the pairing\footnote{For $f, g \in H_l(\omega_L)$ of the same weight and representation, 
we define the inner product $\langle \langle f(\tau), \overline{g}(\tau) \rangle \rangle$ in the analogous way.} 
\begin{align*} \langle \langle f(\tau), g(\tau) \rangle \rangle &= \sum\limits_{\mu \in L^{\vee}/L} f_{\mu}(\tau) g_{\mu}(\tau). \end{align*}
Hence, $\langle \langle f, g \rangle \rangle$ determines a scalar-valued Maass form of weight $0$. We write 
\begin{align*} \operatorname{CT} \langle \langle f, g \rangle \rangle 
&= \sum\limits_{\mu \in L^{\vee}/L} \sum\limits_{m \in {\bf{Q}} } c_f(\mu, - m) c_g(\mu, m)  \end{align*} 
to denote the constant term in its Fourier series expansion. 
Let $\mathcal{F} = \left\lbrace \tau \in \mathfrak{H}: -1/2 \leq \Re(\tau) \leq 1/2, \tau \overline{\tau} \geq 1 \right\rbrace$ 
denote the standard fundamental domain for the action of $\operatorname{SL}_2({\bf{Z}})$ on $\mathfrak{H}$. 
Given any real number $T>0$, we also consider the truncated fundamental domain
$\mathcal{F}_T = \left\lbrace \tau \in \mathfrak{H}: -1/2 \leq \Re(\tau) \leq 1/2, \tau \overline{\tau} \geq 1, \Im(\tau) \leq T \right\rbrace$.
Let $\mu(\tau) = \frac{du dv}{v^2}$ denote the Poincar\'e measure on $\mathfrak{H}$. 
We define the regularized theta lift $\Phi(f, z, h)$ for $f \in H_{1-n/2}(\omega_L)$,
$z \in D(V)$ and $h \in \operatorname{GSpin}(V)({\bf{A}}_f)$ by the regularized integral 
\begin{align*} \Phi(f, z, h) &= \int_{\mathcal{F}}^{\star} \langle \langle f(\tau), \theta_L(\tau, z, h) \rangle \rangle d \mu(\tau)
= \operatorname{CT}_{s=0} \left( \lim_{T \rightarrow \infty} \int_{\mathcal{F}_T} 
 \langle \langle f(\tau), \theta_L(\tau, z, h) \rangle \rangle v^{-s} d \mu(\tau) \right) \end{align*}
given by the constant term in the Laurent series around $s=0$ of the function
\begin{align*}  \lim_{T \rightarrow \infty} \int_{\mathcal{F}_T} 
\langle \langle f(\tau), \theta_L(\tau, z, h) \rangle \rangle v^{-s} d \mu(\tau). \end{align*}
As we explain for Theorem \ref{Bruinier}, the main theorems of \cite{Bo} and \cite{Br} show that 
regularized theta lift $\Phi(f, \cdot)$ determines an automorphic Green's function in the sense of Arakelov theory for the divisor defined by 
\begin{align*} Z(f) &= \sum\limits_{\mu \in L^{\vee}/L} \sum\limits_{m \in {\bf{Q}} \atop m >0} c_f^+(\mu, -m)Z(\mu, m) \subset X_K. \end{align*}

\subsubsection{Sums over zero-cycles and real geodesic cycles}

We first calculate these Green's functions $\Phi(f, \cdot)$ along zero cycles, developing the theorems of Bruinier-Yang \cite[Theorem 4.7]{BY} and Schofer \cite{Sch}, 
as well as along real geodesic cycles in the style of the companion work \cite{VO}. 
Each rational quadratic subspace $(V_0, Q_0)$ of $(V,Q)$ of signature $(0, 2)$ gives rise to a zero cycle $Z(V_0) \subset X_K$ with complex points
\begin{align*} Z(V_0)({\bf{C}}) &= \operatorname{GSpin}(V_0)({\bf{Q}}) \backslash \lbrace z_0 \rbrace \times
\operatorname{GSpin}(V_0)({\bf{A}}_f) / (K \cap \operatorname{GSpin}(V_0)({\bf{A}}_f)), \quad z_0 = V_0({\bf{R}}) \in D(V) \end{align*} 
as defined in $(\ref{CMcycle})$ below. Such a zero-cycle is sometimes called a CM cycle, as it can be associated with an imaginary quadratic field $k = k(V_0)$. 
We associate to the negative-definite sublattice $L_0 = V_0 \cap L$ an incoherent Eisenstein series $E_{L_0}(\tau, s; 1) \in H_1(\omega_{L_0})$ of weight $1$, as defined in $(\ref{E-CM})$. 
We also consider the derivative Eisenstein series $E'_{L_0}(\tau, s; 1) \in H_1(\omega_{L_0})$, and write
$\mathcal{E}_{L_0}(\tau) = E^{\prime +}_{L_0}(\tau, 0; 1)$ to denote its holomorphic part at $s=0$.
On the other hand, any Lorentzian rational quadratic space $(W, Q_W) = (W, Q\vert_W)$ of signature $(1,1)$ of $(V,Q)$ gives rise to a finite ``geodesic set'' 
\begin{align*} \mathfrak{G}(W) &= \operatorname{GSpin}(W)({\bf{Q}}) \backslash \operatorname{GSpin}(W)({\bf{A}}_f)/ K_W, \quad K_W :=K \cap \operatorname{GSpin}(W)({\bf{A}}_f). \end{align*}
Let $D(W)$ denote the Grassmannian of oriented hyperbolic lines in $W({\bf{R}})$, with $D^{\pm}(W)$ a fixed connected component.
Given a class $[h] \in \mathfrak{G}(W)$ represented by 
$h \in \operatorname{GSpin}(W)({\bf{A}})$, we can consider the corresponding real one-dimensional locally symmetric space
\begin{align*} C_h &= \Gamma_{h} \backslash D^{\pm}(W), \quad \Gamma_h := h K_W h^{-1}, \end{align*} %definition of Gamma_h
as well as the corresponding real geodesic cycle 
\begin{align*} \mathcal{G}(W) &= \coprod\limits_{ [h] \in \mathfrak{G}(W) \atop h \in \operatorname{GSpin}(W)({\bf{A}}_f) } C_h 
= \coprod\limits_{ [h] \in \mathfrak{G}(W) \atop h \in \operatorname{GSpin}(W)({\bf{A}}_f) } \Gamma_h \backslash D^{\pm}(W). \end{align*}
This space $(W, Q_W)$ can be associated with a real quadratic field $k = k(W)$. 
We associate to the sublattice $L_W = W \cap L$ an incoherent Eisenstein series 
$E_{L_W}(\tau, s; 2) \in H_2(\omega_{L_W})$ of weight $2$, as defined in $(\ref{E-geo})$. 
We also consider the derivative Eisenstein series $E'_{L_W}(\tau, s; 2) = \frac{d}{ds} E_{L_W}(\tau, s; 2) \in H_2(\omega_{L_W})$, and write
$\mathcal{E}_{L_W}(\tau) = E^{\prime +}_{L_W}(\tau, 0; 2)$ to denote its holomorphic part at $s=0$. 
Using the functional equations and behaviour under Maass lowering operators of these Eisenstein series, we compute the weighted sum over the CM cycle %%automorphism groups
\begin{align*} \Phi(f, Z(V_0)) &= \sum\limits_{(z_0, h) \in Z(V_0)} \frac{\Phi(f, z_0, h)}{\# \operatorname{Aut}(z_0, h)} \end{align*}
of the Green's function $\Phi(f, \cdot)$ along the CM cycle $Z(V_0) \subset X_K$ corresponding to an 
imaginary quadratic field $k(V_0)$, by a minor variation of the arguments of \cite[Theorem 4.7]{BY} and \cite{Sch} 
(cf.~\cite[Theorem 5.7.1]{AGHMP}). We also use these properties to compute the weighted sum over the real geodesic cycle given by the finite disjoint union over classes 
\begin{align*} \Phi(f, \mathcal{G}(W)) 
&= \sum\limits_{ [h] \in \mathfrak{G}(W) \atop h \in \operatorname{GSpin}(W)({\bf{A}}_f)} \frac{1}{\# \operatorname{Aut}(h)} 
\int\limits_{C_h = \Gamma_h \backslash D^{\pm}(W)} 
\Phi(f, z, h) d \nu(z), \end{align*} 
of the Green's function $\Phi(f, \cdot)$ along the geodesic cycle $\mathcal{G}(W)$ corresponding to a real quadratic field $k(W)$.
Here, $d \nu$ denotes the $\operatorname{O}(1,1)$-invariant length measure. 
Let $U = V_0, W \subset V$ denote either of these subspaces of dimension $2$, with $k(U)$ the corresponding quadratic field. We consider the regularized integral 
\begin{equation*}\begin{aligned} I(s, f \times \xi_{\frac{n-2}{2}}(\theta_{L_W^{\perp}})) &=
\int_{\mathcal{F}}^{\star} \langle \langle f(\tau), \overline{\xi_{\frac{n-2}{2}}( \theta_{ L_W^{\perp} }  )(\tau)} \otimes E_{L_W}(\tau, s; 2) \rangle \rangle v^{3 - \frac{n}{2}} d \mu (\tau) \\
&= \lim_{T \rightarrow \infty} \int\limits_{ \mathcal{F}_T}  \langle \langle f(\tau), \overline{\xi_{\frac{n-2}{2}}(\theta_{L_W^{\perp} } )(\tau)} \otimes E_{L_W}(\tau, s; 2) \rangle \rangle v^{3- \frac{n}{2}} d \mu(\tau). 
\end{aligned}\end{equation*}
Here, we write $\xi_{\frac{n-2}{2}} \theta_{L_W^{\perp}}(\tau)$ to denote the holomorphic theta series of weight $3-n/2$
obtained by applying the Bruinier-Funke differential operator $\xi_{\frac{n-2}{2}}$ (see $(\ref{xi})$).
Note that this regularized integral has Rankin-Selberg type due to the appearance of the Eisenstein series $E_{L_{W}}(\tau, s; 2)$. 
In particular, $I(s, f \times \xi_{\frac{n-2}{2}}(\theta_{L_W^{\perp}}))$ inherits an analytic continuation $I^{\star}(s, f \times \xi_{\frac{n-2}{2}}(\theta_{L_W^{\perp}}))$ to a function of $s \in {\bf{C}}$,
and satisfies an odd symmetric functional equation $I^{\star}(s, f \times \xi_{\frac{n-2}{2}}(\theta_{L_W^{\perp}})) =  - I^{\star}(-s, f \times \xi_{\frac{n-2}{2}}(\theta_{L_W^{\perp}}))$ (see $(\ref{IFE})$). 

\begin{theorem}[Theorem \ref{BY4.7} and \ref{realquad}]\label{sums}

Let $(V, Q)$ be any rational quadratic space of signature $(n, 2)$. Fix an integral lattice $L$ with corresponding Weil representation $\omega_L$. 
Let $f \in H_l(\omega_L)$ be any harmonic weak Maass form of weight $l = 1-n/2$ and representation $\omega_L$, 
with $g = \xi_l(f) \in S_{2-l}(\overline{\omega}_L)$ its image under the Bruinier-Funke antilinear differential operator 
$\xi_l: H_l(\omega_L) \rightarrow S_{2-l}(\overline{\omega}_L)$. Hence, $g(\tau)$ is a holomorphic form.

\begin{itemize}

\item[(i)] Let $(V_0, Q_0)$ be a rational quadratic subspace of signature $(0, 2)$ with sublattices $L_0 = V_0 \cap L$,
$L_0^{\perp} \subset V$, and $L_0 \oplus L_0^{\perp} \subset L$. 
Write $k = k(V_0)$ for the imaginary quadratic field determined by the space.
Let \begin{align*} \theta_{L_0^{\perp}}(\tau) = \theta_{L_0^{\perp}}(\tau, 1, 1)\end{align*} 
denote the holomorphic Siegel theta series associated to the positive definite lattice $L_0^{\perp}$ of signature $(n,0)$,
defined via restriction of $\theta_L(\tau, z, h)$ as in Lemma \ref{lattice} and $(\ref{IPrelations})$. Let 
\begin{align*} L(s, g \times \theta_{L_0^{\perp}}) &= (4 \pi)^{- \left( \frac{s+2}{2} \right)} \Gamma \left(\frac{s+2}{2} \right)
\sum\limits_{ \mu \in (L_0^{\perp})^{\vee}/L_0^{\perp}  } \sum\limits_{m \geq 1} \overline{c_g(\mu, m)} r_{L_0^{\perp}}(\mu, m) m^{-\left( \frac{s+2}{2} \right)} \\
&= \langle g, \theta_{L_0^{\perp}}\otimes E_{L_0}(\cdot, s; 1) \rangle \end{align*}
denote the Rankin-Selberg $L$-function of $g(\tau)$ times $\theta_{L_0^{\perp}}(\tau)$. %(see $(\ref{RS-CM})$). 
This has an analytic continuation 
\begin{align*} L^{\star}(s, g \times \theta_{L_0^{\perp}}) = \Lambda(s+1, \eta_k) L(s, g \times \theta_{L_0^{\perp}}) = \langle g, \theta_{L_0^{\perp}}\otimes E^{\star}_{L_0}(\cdot, s; 1)   \rangle \end{align*}
to an entire function of $s \in {\bf{C}}$, and inherits from the incoherent Eisenstein series the odd symmetric functional equation 
$L^{\star}(s, g \times \theta_{L_0^{\perp}}) = - L^{\star}(-s, g \times \theta_{L_0^{\perp}})$. Let $\operatorname{vol}(K_0)$ denote the volume
of the compact open subgroup $K_0 = K \cap \operatorname{GSpin}(V_0)({{A}}_f)$. We have the summation formula 
\begin{align*} \Phi(f, Z(V_0)) &= - \frac{2}{\operatorname{vol}(K_0)} \left(
\operatorname{CT} \langle \langle f^+(\tau), \theta_{L_0^{\perp}}(\tau) \otimes \mathcal{E}_{L_0}(\tau) \rangle \rangle + L'(0, g \times \theta_{L_0^{\perp}}) \right) \\
&= - \frac{\operatorname{deg}(Z(V_0))}{2} \left( \operatorname{CT} \langle \langle f^+(\tau), \theta_{L_0^{\perp}}(\tau) \otimes \mathcal{E}_{L_0}(\tau) \rangle \rangle
+ L'(0, g \times \theta_{L_0^{\perp}}) \right). \\ \end{align*} 

\item[(ii)] Let $(W, Q_W)$ be a rational quadratic subspace of signature $(1,1)$ with sublattices $L_W = W \cap L$,
$L_W^{\perp} \subset V$, and $L_W \oplus L_W^{\perp} \subset L$. 
Write $k = k(W)$ for the real quadratic field determined by $W$.
Let \begin{align*} \theta_{L_W^{\perp}}(\tau) = \theta_{L_W^{\perp}}(\tau, 1, 1) \end{align*} 
denote the non-holomorphic Siegel theta series associated 
to the Lorentzian lattice $L_W^{\perp}$ of signature $(n-1, 1)$, 
defined via restriction of $\theta_L(\tau, z, h)$ as in Lemma \ref{lattice} and $(\ref{IPrelations})$. Let 
\begin{align*} L(s, g \times \theta_{L_W^{\perp}}) 
&= (4 \pi)^{- \left( \frac{s+2}{2} \right)} \Gamma \left(\frac{s+2}{2} \right)
\sum\limits_{ \mu \in (L_W^{\perp})^{\vee}/L_W^{\perp}  } \sum\limits_{m \geq 1} \overline{c_g(\mu, m)} r_{L_W^{\perp}}(\mu, m) m^{-\left( \frac{s+2}{2} \right)} \\
&= \langle g, \theta_{L_W^{\perp}} \otimes E_{L_W}(\cdot, s; 2) \rangle \end{align*}
denote the Rankin-Selberg $L$-series of $g(\tau)$ times $\theta_{L_W^{\perp}}(\tau)$. This has an analytic continuation
\begin{align*}  L^{\star}(s, g \times \theta_{L_W^{\perp}}) = \Lambda(s+1, \eta_k) L(s, g \times \theta_{L_W^{\perp}})
= \langle g, \theta_{L_W^{\perp}} \otimes E^{\star}_{L_W}(\cdot, s; 2) \rangle \end{align*}
to an entire function of $s \in {\bf{C}}$, and inherits from the incoherent Eisenstein series an odd symmetric functional equation 
$L^{\star}(s, g \times \theta_{L_W^{\perp}}) = - L^{\star}(-s, g \times \theta_{L_W^{\perp}})$.
Let $\operatorname{vol}(K_W)$ denote the volume of the compact open subgroup $K_W = K \cap \operatorname{GSpin}(W)({{A}}_f)$. 
If the space $(W^{\perp}, Q_{W^{\perp}})$ is anisotropic, then 
\begin{align*} &\Phi(f, \mathcal{G}(W)) \\ &= - \frac{2}{ \operatorname{vol}(K_W)} \left(
\operatorname{CT} \langle \langle f^+(\tau), {\bf{1}}_{L_W^{\perp} \oplus 0} \otimes \mathcal{E}_{L_W}(\tau) \rangle \rangle
+ L'(0, g \times \theta_{L_W^{\perp}})  + I'(0, f \times \xi_{\frac{n-2}{2}}(\theta_{L_U^{\perp}}) )  \right). \end{align*} 
Here, $ I'(0, f \times \xi_{\frac{n-2}{2}}(\theta_{L_U^{\perp}}) )$ denotes the 
the derivative $I'(s, f \times \xi_{\frac{n-2}{2}}(\theta_{L_U^{\perp}}) ) = \frac{d}{ds} I(s, f \times \xi_{\frac{n-2}{2}}(\theta_{L_U^{\perp}}) )$ at $s=0$ of the 
regularized integral $I(s, f \times \xi_{\frac{n-2}{2}}(\theta_{L_U^{\perp}}) )$ defined above. 

\end{itemize} \end{theorem}

While Theorem \ref{sums} (i) streamlines \cite[Theorem 4.7]{BY}, Theorem \ref{sums} (ii) is new, and generalizes \cite{VO}.

\subsubsection{Relations to quadratic fields and standard Rankin-Selberg $L$-functions}

Let $k$ be any quadratic field, real or imaginary, of discriminant $d_k$ prime to $N$.
We again write $\eta_k(\cdot) = (\frac{d_k}{\cdot})$ to denote the quadratic Dirichlet character, and $C(\mathcal{O}_k)$ to denote the ideal class group. 
We consider the following rational quadratic spaces $(V_A, Q_A)$ of signature $(2,2)$ attached to each class $A \in C(\mathcal{O}_k)$. 
Given $A \in C(\mathcal{O}_k)$, fix an integer ideal representative $\mathfrak{a} \subset \mathcal{O}_k$, and let 
$\mathfrak{a}_{\bf{Q}} := \mathfrak{a} \otimes_{\bf{Z}} {\bf{Q}}$ be the corresponding vector space. 
Let \begin{align*} Q_{\mathfrak{a}}(z) = {\bf{N}}(z)/{\bf{N}} \mathfrak{a} = {\bf{N}}_{k/{\bf{Q}}}(z)/ {\bf{N}} {\bf{a}} \end{align*} denote the corresponding norm form,
where ${\bf{N}}(z) = {\bf{N}}_{k/{\bf{Q}}}(z) = z z^{\tau}$ for $\tau \in \operatorname{Gal}(k/{\bf{Q}})$ the nontrivial automorphism denotes
the norm homomorphism. Hence, $Q_{\mathfrak{a}}$ has signature $(2, 0)$ if $k$ is imaginary quadratic, and signature $(1,1)$ if $k$ is real quadratic.
In either case, we consider the rational quadratic space $(V_A, Q_A)$ defined by $V_A = \mathfrak{a}_{\bf{Q}} \oplus \mathfrak{a}_{\bf{Q}}$ with
quadratic form $Q_A(z) = Q_A((z_1, z_2)) = Q_{\mathfrak{a}}(z_1) - Q_{\mathfrak{a}}(z_2)$. 
As we explain in Proposition \ref{Clifford2}, we have an exceptional isomorphism of algebraic groups 
\begin{align*} \operatorname{GSpin}(V_A) \cong  \operatorname{GL}_2 \times_{{\bf{G}}_m} \operatorname{GL}_2
= \left\lbrace (g_1, g_2) \in \operatorname{GL}_2 \times \operatorname{GL}_2: \det(g_1)=\det(g_2) \right\rbrace \end{align*} over ${\bf{Q}}$.
Moreover, we can choose a lattice $L_A = L_A(N) = N^{-1} \mathfrak{a} \oplus N^{-1} \mathfrak{a} \subset V_A$ whose 
adelization is fixed under the action of $\operatorname{GSpin}(V_A)({\bf{A}}_f) \cong  \operatorname{GL}_2({\bf{A}}_f) \times_{{\bf{G}}_m} \operatorname{GL}_2({\bf{A}}_f)$ 
by the level structure $K_A = K_{L_A(N)}$ corresponding to the congruence subgroup $K_0(N) \times K_0(N)$ (Corollary \ref{lattices}). In this way, we find that 
$X_A \cong Y_0(N) \times Y_0(N)$ is isomorphic to two copies of the noncompactified modular curve $Y_0(N)$,
with the special divisors $Z_A(\mu, m) \subset X_A \cong Y_0(N)^2$ corresponding to Hirzebruch-Zagier divisors. 
We can deduce from \cite[Theorem 5.2]{Str} or more directly via the Doi-Naganuma lifting (Theorem \ref{D-N}) that each 
$\phi \in S_l^{\operatorname{new}}(\Gamma_0(N))$ has a canonical lifting to a vector-valued form $g_{\phi, A} \in S_l(\omega_{L_A})$
(Corollary \ref{existence}). We use this to derive the following integral presentation of standard Rankin-Selberg $L$-functions.
Fix a character $\chi$ of $C(\mathcal{O}_k)$, and let $\theta(\chi)$ denote the corresponding Hecke theta series. Hence,  
$\theta(\chi)$ is a holomorphic modular form of weight $l(k) = 1$, level $\Gamma_0(\vert d_k \vert)$, and character $\eta_k$ if $k$ is imaginary quadratic;
$\theta(\chi)$ is a nonholomorphic Maass form of weight $l(k)=0$, level $\Gamma_0(\vert d_k \vert)$, and character $\eta_k$ if $k$ is real quadratic.
Consider the corresponding Rankin-Selberg $L$-function $\Lambda(s, \phi \times \theta(\chi)) = L_{\infty}(s, \phi \times \theta(\chi)) L(s, \phi \times \theta(\chi))$, and described in Proposition \ref{FE}. 

\begin{theorem}[Theorem \ref{RSIP} and Corollary \ref{simplified}]\label{Rankin-Selberg} 

Fix a holomorphic cuspidal newform $\phi \in S^{\operatorname{new}}_2(\Gamma_0(N))$ of weight $2$, level $\Gamma_0(N)$, and trivial character. 
Let $g_{\phi, A} \in S_2(\overline{\omega}_{L_A})$ denote the lifting of $\phi$ to a vector-valued cusp form of weight $2$ 
and conjugate Weil representation $\overline{\omega}_{L_A}$. Let $f_{0, A} \in H_{0}(\omega_{L_A})$ be any harmonic weak
Maass form of weight zero and representation $\omega_{L_A}$ whose image $\xi_0(f_{0, A})$ under the antilinear differential 
operator $\xi_0: H_0(\omega_{L_A}) \rightarrow S_2(\overline{\omega}_{L_A})$ equals $g_{\phi, A}$.
We have the following identifications of completed $L$-functions. \\

\begin{itemize}

\item[(i)] If $k$ is the imaginary quadratic field associated to the negative definite subspace $V_{A, 0} \subset V_A$
with $L_{A,0} = L_A \cap V_{A,0}$, then we have the identifications of completed Rankin-Selberg $L$-functions 
\begin{align*} L^{\star}(2s-2, g_{\phi, A} \times \theta_{L_{A, 0}^{\perp}}) 
= \Lambda(s-1/2, \phi \times \theta_A) \end{align*} 
for each class $A \in C(\mathcal{O}_k)$, and for each class group character $\chi \in C(\mathcal{O}_k)^{\vee}$ the identification 
\begin{align*} \sum\limits_{A \in C(\mathcal{O}_k)} \chi(A) L^{\star}(2s-2, g_{\phi, A} \times \theta_{L_{A, 0}^{\perp}}) 
= \Lambda(s-1/2, \phi \times \theta(\chi)). \end{align*} 
Hence, if $k$ is imaginary quadratic with $(d_k, N)=1$ and $\eta_k(-N) = -\eta_k(N) = -1$, we obtain 
\begin{align*} & \Lambda'(1/2, \phi \times \theta(\chi))  \\
&=- \frac{2 \pi h_k}{w_k} \sum\limits_{A \in C(\mathcal{O}_k)} \chi(A) 
\left[ \left(   \frac{\operatorname{vol}(K_{A,0})  }{ 2} \right) \Phi(f_{0, A}, Z(V_{A, 0})) 
+ \operatorname{CT} \langle \langle f_{0, A}^+(\tau), \theta_{L_{A, 0}^{\perp}}(\tau) \otimes \mathcal{E}_{L_{A, 0}}(\tau) \rangle \rangle \right]. \end{align*} \\

\item[(ii)] If $k$ is the real quadratic field associated to the Lorentzian subspace $W_A \subset V_A$ with 
$L_{A,W} = W_A \cap L_A$ and fundamental unit $\varepsilon_k$, then we have the identifications of completed Rankin-Selberg $L$-functions 
\begin{align*} L^{\star}(2s-2, g_{\phi, A} \times \theta_{L_{A, W}^{\perp}}) &= \Lambda(s-1/2, \phi \times \theta_A)\end{align*} 
for each class $A \in C(\mathcal{O}_k)$, and for each class group character $\chi \in C(\mathcal{O}_k)^{\vee}$ that 
\begin{align*} \sum\limits_{A \in C(\mathcal{O}_k)} \chi(A) L^{\star}(2s-2, g_{\phi, A} \times \theta_{L_{A, W}^{\perp}}) 
&= \Lambda(s-1/2, \phi \times \theta(\chi)).\end{align*}
Hence, if $k$ is real quadratic with $(d_k, N)=1$ and $\eta_k(-N) = \eta_k(N) = -1$, we obtain
\begin{align*} &\Lambda'(1/2, \phi \times \theta(\chi)) =-2 \log \varepsilon_k h_k \\ &\times \sum\limits_{A \in C(\mathcal{O}_k)} \chi(A) 
\left[ \left( \frac{\operatorname{vol}( K_{A,W} ) }{ 2} \right) \Phi(f_{0, A}, \mathcal{G}(W_A)) 
+ \operatorname{CT} \langle \langle f_{0, A}^+(\tau), {\bf{1}}_{L_{A, W}^{\perp} \oplus 0}(\tau) \otimes \mathcal{E}_{L_{A, W}}(\tau) \rangle \rangle
+ I'(0, f_{0,A} \times \xi_0(\theta_{L_{A, W}^{\perp}})) \right]. \end{align*}

\end{itemize} \end{theorem}

We remark that while these integral presentations can be viewed as an indirect consequence of the quadratic basechange equivalence 
of Rankin-Selberg $L$-functions we consider (described below), this does not seem to have been made explicit before. 

\subsubsection{Relations to arithmetic Hirzebruch-Zagier divisors}

Theorem \ref{Rankin-Selberg} (i) gives the following new integral presentation for $\Lambda'(E/k, \chi, 1)$ in terms of Faltings heights of Hirzebruch-Zagier divisors on $X_0(N) \times X_0(N)$.
In general, we know that $\Phi(f, \cdot)$ gives the automorphic Green's function for the divisor $Z(f) \subset X_K$. 
This gives a supply of arithmetic divisors $\widehat{Z}(f) = (Z(f), \Phi(f, \cdot))$ in the corresponding arithmetic Chow group of codimension one cycles on $X_K$. 
We can thus consider the arithmetic height $[ \widehat{Z}(f), Z(V_0)]$ of such a divisor $\widehat{Z}(f)$ along a CM cycle $Z(V_0)$.
Writing $Z_A(\mu, m) \in X_A \cong Y_0(N)^2$ for the arithmetic divisors for the spaces $(V_A, Q_A)$ of signature (2,2), 
we consider the corresponding extensions $Z_A^c(\mu, m)$ to the compactification $X_0(N) \times X_0(N)$. 
Let $\mathcal{Z}_A(\mu, m)$ denote the extension of each $Z_A(\mu, m)$ to the integral model $\mathcal{Y}_0(N) \times \mathcal{Y}_0(N)$,
and $\mathcal{Z}^c_A(\mu, m)$ the extension of $Z_A^c(\mu, m)$ to the integral model $\mathcal{X}_0(N) \times \mathcal{X}_0(N)$. 

\begin{theorem}[Theorem \ref{main}, Corollary \ref{ec}, Corollary \ref{heightrelation}]\label{GZ-AHZ} 

Retain the setup of Theorem \ref{Rankin-Selberg} (i).
For any class group character $\chi \in C(\mathcal{O}_k)^{\vee}$, we have the central derivative value formula 
\begin{align*} \Lambda'(1/2, \phi \times \theta(\chi)) &= - 2 \sum\limits_{A \in C(\mathcal{O}_k)} \chi(A) 
\left[ \widehat{\mathcal{Z}}_A(f_{0, A}): \mathcal{Z}(V_{A,0})\right],\end{align*}
where each term on the right-hand side denotes the Faltings height of the arithmetic special divisor 
\begin{align*} \widehat{\mathcal{Z}}_A(f_{0, A}) 
&= \sum\limits_{\mu \in L_A^{\vee}/L_A} \sum\limits_{m \in {\bf{Q}} \atop m > 0} c_{f_{0, A}}^+(\mu, -m) \mathcal{Z}_A(\mu, m) \end{align*} 
on $\mathcal{Y}_0(N) \times \mathcal{Y}_0(N)$ evaluated along the corresponding CM cycle $\mathcal{Z}(V_{A,0}) \subset \mathcal{Y}_0(N) \times \mathcal{Y}_0(N)$. 
More generally, we can extend to the boundary to derive the corresponding formula 
\begin{align*} \Lambda'(1/2, \phi \times \theta(\chi)) &= - 2 \sum\limits_{A \in C(\mathcal{O}_k)} \chi(A) 
\left[ \widehat{\mathcal{Z}}_A^c(f_{0, A}): \mathcal{Z}(V_{A,0})\right]. \end{align*}

In particular, if $\phi \in S_2^{\operatorname{new}}(\Gamma_0(N))$ parametrizes an elliptic curve $E/{\bf{Q}}$, we obtain the integral presentation
\begin{align*} \Lambda'(E/k, \chi, 1) &= - 2 \sum\limits_{A \in C(\mathcal{O}_k)} \chi(A) 
\left[ \widehat{\mathcal{Z}}_A(f_{0, A}): \mathcal{Z}(V_{A,0})\right]. \end{align*}
Extending to the compactification $\mathcal{X}_0(N) \times \mathcal{X}_0(N) \longrightarrow \operatorname{Spec}({\bf{Z}})$, we derive the central derivative value formula
\begin{align*} \Lambda'(E/k, \chi, 1) &= - 2 \sum\limits_{A \in C(\mathcal{O}_k)} \chi(A) 
\left[ \widehat{\mathcal{Z}}_A^c(f_{0, A}): \mathcal{Z}(V_{A,0})\right]. \end{align*} %NB there is a puzzling discrepancy of L(1, \eta) here and in Bruinier-Yang
By comparison with the Gross-Zagier formula (Theorem \ref{Gross-Zagier}) for $L'(E/k, \chi, 1)$, we derive the relation  
\begin{align*} 2 \cdot \left[ y_{\chi}^{\phi}, y_{\chi}^{\phi} \right]_{\operatorname{NT}} 
&= \frac{\vert d_k \vert^{\frac{1}{2}}}{ 8 \pi^2 \vert \vert \phi \vert \vert^2 } \cdot \Lambda'(1/2, \phi \times \theta(\chi))
= -  2 \sum\limits_{A \in C(\mathcal{O}_k)} \chi(A) \left[ \widehat{Z}^c(f_{0, A}) : \mathcal{Z}(V_{A,0}) \right] \end{align*}
between heights of Heegner divisors on the modular curve $X_0(N)$ and arithmetic heights of arithmetic Hirzebruch-Zagier
divisors on the the product of modular curves $X_0(N) \times X_0(N)$. \end{theorem}

\subsubsection{Relation to Birch-Swinnerton-Dyer constants and periods}

Theorem \ref{sums} implies results towards the refined conjecture of Birch and Swinnerton-Dyer 
via Euler characteristic calculations and known results on the Iwasawa main conjectures in the setting of Mordell-Weil rank one. 
Let $E$ be an elliptic curve of conductor $N$ defined over ${\bf{Q}}$, parametrized by $\phi \in S^{\operatorname{new}}_2(\Gamma_0(N))$ as above. 
Let $k$ be a quadratic field with discriminant $d_k$ and character $\eta_k(\cdot) = \left( \frac{d_k}{\cdot} \right)$. We consider the Mordell-Weil group 
$E(k) \cong {\bf{Z}}^{r_E(k)} \oplus E(k)_{\operatorname{tors}}$, along with that of the quadratic twist
$E^{(d_k)}({\bf{Q}}) \cong {\bf{Z}}^{r_{E^{(d_k)}}({\bf{Q}})} \oplus E^{(d_k)}({\bf{Q}})_{\operatorname{tors}}$
and $E({\bf{Q}}) \cong {\bf{Z}}^{r_E({\bf{Q}})} \oplus E({\bf{Q}})_{\operatorname{tors}}$ over ${\bf{Q}}$. 
Recall that the conjecture of Birch and Swinnerton-Dyer predicts that the completed $L$-function 
\begin{align*}\Lambda(E/k, s) = \Lambda(s-1/2, \phi \times \theta({\bf{1}})) = \Lambda(s-1/2, \phi) \Lambda(s-1/2, \phi \otimes \eta_k) = \Lambda(E, s) \Lambda(E^{(d_k)}, s) \end{align*}
has order of vanishing $\operatorname{ord}_{s=1} \Lambda(E/k, s) = r_E(k)$. Moreover, the leading term in the Taylor series expansion around $s=1$ of this function 
is expected to be given by the corresponding Birch-Swinnerton-Dyer constant $\kappa_E(k)$, which is defined more generally as follows. 
Let us write the Birch-Swinnerton-Dyer constant as 
\begin{equation}\begin{aligned}\label{B+S-Dconstant} \kappa_E(k) &:= 
\frac{\# \Sha(E/k) \cdot T(E/k) \cdot R(E/k) \cdot \Omega_{\infty}(E/k)}
{ \vert d_k \vert^{\frac{1}{2}} \vert E(k)_{\operatorname{tors}} \vert^2 }.\end{aligned}\end{equation}
Here, $\#\Sha(E/k)$ denotes the cardinality of the conjecturally finite Tate-Shafarevich group
\begin{align*} \Sha(E/k) &= \ker \left(H^1(k, E) \longrightarrow \prod\limits_w H^1(k_w, E) \right). \end{align*}
We write $R(E/k)$ to denote the regulator, defined for any basis $\lbrace e_j \rbrace_j$ of $E(k)/E(k)_{\operatorname{tors}}$ 
by the determinant of the corresponding height matrix $( \left[ e_i, e_j\right]_{\operatorname{NT}})_{i,j}$, so $R(E/k) = \det \left( \left[ e_i, e_j \right]_{\operatorname{NT}} \right)_{i,j}$.
We write $T(E/k)$ to denote the product over the local Tamagawa factors, 
\begin{align*} T(E/k) &= \prod\limits_{v < \infty \atop v \subset \mathcal{O}_k ~~~\text{prime}} 
\left[  E(k_v) : E_0(k_v) \right] \cdot \left\vert \frac{\omega}{\omega_v^*} \right\vert_v,\end{align*}
where $\omega$ denotes a fixed invariant differential for $E/k$, and each $\omega_v^*$ denotes 
the local N\'eron differential at $v$. We then define the corresponding archimedean local periods
\begin{align*} \Omega_{\infty}(E/k) 
&= \prod\limits_{  {v \mid \infty \atop v: k \hookrightarrow {\bf{R}} } \atop \text{real}} \int\limits_{E(k_v) \cong E({\bf{R}})} \vert \omega \vert 
\cdot \prod\limits_{ { v\mid \infty \atop v, \overline{v}: k \hookrightarrow {\bf{C}} } \atop \text{complex} } 2 \int\limits_{E(k_v) \cong E({\bf{C}})}
\omega \wedge \overline{\omega}. \end{align*}

\begin{theorem}\label{BSD} 

Let $E/{\bf{Q}}$ be an elliptic curve parametrized by a newform $\phi \in S_2^{\operatorname{new}}(\Gamma_0(N))$. 
Let $k$ be a quadratic field of discriminant $d_k$ prime to $N$ and character $\eta_k(\cdot) = (\frac{d_k}{\cdot})$. 
Assume $E$ has semistable reduction, hence that $N$ is squarefree. Assume that the completed $L$-function 
\begin{align*} \Lambda(E/k, s) = \Lambda(E, s) \Lambda(E^{(d_k)}, s) = \Lambda(s-1/2,\phi) \Lambda(s-1/2, \phi \otimes \eta_k) \end{align*}
has order of vanishing $\operatorname{ord}_{s=1}\Lambda(E/k, s) = 1$, so that exactly one of the central values $\Lambda(E, 1) = \Lambda(1/2, \phi)$ or 
$\Lambda(E^{(d_k)}, 1) = \Lambda(1/2, \phi \otimes \eta_k)$ vanishes. Write $\left[ e, e \right]$ to denote either the regulator $R(E/{\bf{Q}})$
or the regulator $R(E^{(d_k)}/{\bf{Q}})$ according to which factor vanishes. Let us also assume for each prime $p \geq 5$ that 

\begin{itemize}

\item The residual Galois representations $E[p]$ and $E^{(d_k)}[p]$ are irreducible. 

\item There exists a prime $l \mid N$ distinct from $p$ where $E[p]$ is ramified, and a prime $q \mid N$ distinct from $p$ where $E^{(d_k)}[p]$ is ramified. 

\end{itemize} Up to powers of $2$ and $3$, we have the following identifications for the Birch-Swinnerton-Dyer constant
\begin{align*} & \kappa_E({\bf{Q}}) \cdot \kappa_{E^{(d_k)}}({\bf{Q}}) \\
&= \frac{\# \Sha(E/{\bf{Q}}) \cdot \#\Sha(E^{(d_k)}/{\bf{Q}}) \cdot [e,e] \cdot T(E/{\bf{Q}}) \cdot T(E^{(d_k)}/{\bf{Q}}) 
\cdot \Omega_{\infty}(E/{\bf{Q}}) \cdot \Omega_{\infty}(E^{(d_k)}/{\bf{Q}})}
{ \#E({\bf{Q}})_{\operatorname{tors}}^2 \cdot \#E^{(d_k)}({\bf{Q}})_{\operatorname{tors}}^2 }. \end{align*}

\begin{itemize}

\item[(i)] If $k$ is imaginary quadratic with $\eta_k(-N) = -\eta_k(N) = -1$, then up to powers of $2$ and $3$,
\begin{align*} \kappa_E({\bf{Q}}) \cdot \kappa_{E^{(d_k)}}({\bf{Q}}) = \Lambda'(E/k, 1)
&%=- \frac{1 }{2} \sum\limits_{A \in C(\mathcal{O}_k)} \Phi(f_{0, A}, Z(V_{A, 0})) 
= -2 \sum\limits_{A \in C(\mathcal{O}_k)} \left[ \widehat{\mathcal{Z}}_A^c(f_{0, A}): \mathcal{Z}(V_{A,0})\right]. \end{align*}

\item[(ii)] If $k$ is real quadratic with $\eta_k(-N) = \eta_k(N) = -1$, then up to powers of $2$ and $3$,
\begin{align*} &\kappa_E({\bf{Q}}) \cdot \kappa_{E^{(d_k)}}({\bf{Q}}) = \Lambda'(E/k, 1) \\
&=-2 \log(\varepsilon_k) h_k \sum\limits_{A \in C(\mathcal{O}_k)} 
\left[ \left( \frac{ \operatorname{vol}(K_{A,W}) }{2 } \right) \Phi(f_{0, A}, \mathcal{G}(W_A)) 
+ \operatorname{CT} \langle \langle f_{0, A}^+(\tau), {\bf{1}}_{L_{A, W}^{\perp} \oplus 0} \otimes \mathcal{E}_{L_{A, W}}(\tau) \rangle \rangle
+ I'(0, f_{0,A} \times \xi_0(\theta_{L_{A,W}^{\perp}})) \right]. \end{align*}

\end{itemize} Moreover, the central derivative $\Lambda'(E/k, 1)$ lies in the ring of periods $\mathcal{P}$ described by Kontsevich-Zagier \cite{KZ}.

\end{theorem}

\subsubsection{Derivation of Gross-Zagier via regularized theta lifts on the modular curve}

We explain in Appendix A how to develop similar ideas via the rational quadratic space
\begin{align}\label{(1,2)} (V, Q) &= (\operatorname{Mat}_{2 \times 2}^{\operatorname{tr}=0}({\bf{Q}}), N \det(\cdot)) \end{align}
of signature $(1,1)$ from Bruinier-Yang \cite[Theorem 1.5, Theorem 7.7, $\S$ 7]{BY} to give another proof of the full/general
Gross-Zagier formula (Theorem \ref{Gross-Zagier}). We refer to Theorem \ref{GZF} for details. 

\subsubsection{Relation to Fourier coefficients of half-integral weight forms}

Finally, we derive links between the central derivative values $\Lambda'(E/k, \chi, 1)$ and Fourier coefficients of half-integral weight Maass forms. 
Here, we use the main theorem of Bruinier-Funke-Imamoglu \cite{BFI} for quadratic spaces of the form $(\ref{(1,2)})$ to express the sums
$\Phi(f, Z(V_0))$ and $\Phi(f, \mathcal{G}(W))$ of Theorem \ref{sums} (with $n=1$) in terms of Fourier coefficients of some harmonic weak Maass form of weight $1/2$. 
Let $k$ be a quadratic field of discriminant $d_k = D$ prime to $N$. Let $\mathcal{Q}_{d_k}$ denote the class group of binary quadratic forms
$q_{a,b,c}(x, y) = ax^2 + b xy + c y^2$ of discriminant $d_k = b^2 - 4ac$, with $[q_{a, b, c}] = [a, b, c] \in \mathcal{Q}_{d_k}$ its corresponding equivalence class. Hence,
\begin{align*} \varphi: \mathcal{Q}_{d_k} \cong C(\mathcal{O}_k), \quad [a, b, c] 
&\longmapsto [a, (-b + \sqrt{d_k})/2]. \end{align*}
For each $A \in C(\mathcal{O}_k)$, we fix a representative $\mathfrak{a} \subset \mathcal{O}_k$, so that $\varphi([Q_{\mathfrak{a}}]) = A \in C(\mathcal{O}_k)$. 
We then consider the lattice $\mathcal{L}_A = \mathcal{L}_A(N) \subset V$ of quadratic space $(\ref{(1,2)})$ given by 
\begin{align*} \mathcal{L}_A 
&= \left\lbrace \left( \begin{array}{cc} b & - a/N \\ c & - b \end{array} \right) : a, b, c \in {\bf{Z}}, \quad
N \det \left( \begin{array}{cc} b & - a/N \\ c & - b \end{array} \right)  \Bigg\vert_{L_{A,U}} \equiv - q_{a,b,c} \text{ for } \varphi([a, b, c]) = A \right\rbrace. \end{align*}
When $\phi \in S_2^{\operatorname{new}}(\Gamma_0(N))$ is invariant under the Fricke involution $w_N$,
there exist both a vector-valued lift $g_A = g_{\phi, A} \in S_{3/2}(\overline{\omega}_{\mathcal{L}_A})$ of the 
Shimura lift of $\phi$ and a harmonic Maass form $f_{1/2, A} \in H_{1/2}(\overline{\omega}_{\mathcal{L}_A})$ such that 

\begin{itemize}

\item We have the relation $\xi_{1/2} ( f_{1/2, A} ) = g_A / \vert \vert g_A \vert \vert^2$.

\item The Fourier coefficients $c_{ f_{1/2, A} }^+(\mu, m)$ lie in the Hecke field ${\bf{Q}}(\phi) = {\bf{Q}}$ of the newform $\phi$.

\item The constant Fourier coefficient $c_{ f_{1/2, A} }^+(0, 0)$ vanishes. 

\end{itemize}
The main argument of Bruinier-Funke-Imamoglu \cite{BFI} (Theorem \ref{BFI}) constructs a harmonic Maass form 
\begin{align*} I_{1/2}(\Phi(f_{1/2, A}), \tau) &= I^+_{1/2}(\Phi(f_{1/2, A}), \tau) + I^-_{1/2}(\Phi(f_{1/2, A}), \tau) \in H_{1/2}(\omega_{\mathcal{L}_A}) \end{align*}
of weight $1/2$ and representation $\omega_{\mathcal{L}_A}$, with holomorphic part 
\begin{align*} I_{1/2}^+(\Phi(f_{1/2, A}), \tau) 
&= - 2 \sqrt{v} \sum\limits_{\mu \in \mathcal{L}_A^{\vee} / \mathcal{L}_A} 
\operatorname{tr}_{\mu, 0} \left( \Phi(f_{1/2, A}) \right) {\bf{1}}_{\mu}
+ \sum\limits_{\mu \in \mathcal{L}_A^{\vee} / \mathcal{L}_A} \sum\limits_{m >0} \operatorname{tr}_{\mu, m}(\Phi(f_{1/2, A})) e(m \tau) {\bf{1}}_{\mu} \end{align*}
and nonholomorphic part 
\begin{align*} I^-_{1/2}(\Phi(f_{1/2, A}), \tau)
&= \sum\limits_{\mu \in \mathcal{L}_A^{\vee} / \mathcal{L}_A} 
\sum\limits_{m<0}  \operatorname{tr}_{\mu, m}(\Phi(f_{1/2, A}))  \frac{ \operatorname{erf} (2\sqrt{\pi \vert m \vert v})  }{  2 \sqrt{\vert m \vert} }  e(m \tau) {\bf{1}}_{\mu}. \end{align*}
We obtain the following relations of the singular moduli $ \operatorname{tr}_{\mu, m}(\Phi(f_{1/2, A})) $ appearing in the Fourier coefficients with the 
central derivative values of Theorem \ref{sums} for the quadratic spaces $(\ref{(1,2)})$ described above.

\begin{theorem}[Theorem \ref{metatrace}]

We have via Theorem \ref{sums} for the quadratic space $(V,Q) = (\operatorname{Mat}_{2 \times 2}^{\operatorname{tr}=0}({\bf{Q}}), N \det(\cdot))$ of signature $(1,2)$ described above the following 
identification of central derivative values of $L$-functions in terms of the Fourier coefficients of the half-integral weight Maass forms $I_{1/2}(\Phi(f_{1/2, A})) \in H_{1/2}(\omega_{\mathcal{L}_A})$.

\begin{itemize}

\item[(i)] If $k$ is imaginary quadratic, then for any ideal class character $\chi$ of $C(\mathcal{O}_k)$, we have the relation 
\begin{align*} \sum\limits_{A \in C(\mathcal{O}_k)} \chi(A) \cdot c_{g_A}(\mu, m) \cdot \operatorname{tr}_{\mu, m} \left( \Phi(f_{A, 1/2}) \right) 
&= - \frac{ \vert D \vert^{\frac{1}{2}}}{ 8 \pi^2 \vert \vert \phi \vert \vert^2  } \cdot L'(1/2, \phi \times \theta(\chi)).  \end{align*}
 
\item[(ii)] If $k$ is real quadratic, then we have for each class $A \in C(\mathcal{O}_k)$ the relation 
\begin{equation*}\begin{aligned} &\operatorname{tr}_{\mu, m}(\Phi(f_{1/2, A})) \\
&= - \frac{2}{\operatorname{vol}(K_{W_A})} 
\left(  \operatorname{CT} \langle \langle f_{1/2, A}^+(\tau), {\bf{1}}_{L_{W_A}^{\perp}} \otimes \mathcal{E}_{L_{W_A}}(\tau) \rangle \rangle
+ L'(0, g_A \times \theta_{L_{A, W}^{\perp}}) + I'(0, f_{1/2, A} \times \xi_{-\frac{1}{2}}(\theta_{L_{W_A}^{\perp}}) )\right).  \end{aligned}\end{equation*}

\end{itemize}

\end{theorem}

As we describe in Appendix B, it should be possible to extend these calculations in the real quadratic case (ii) to derive an affirmative answer to the question posed in 
Bruinier-Ono \cite[Remark 4]{BO}, relating Fourier coefficients of the holomorphic part of the vector-valued Shimura lift to the 
nonvanishing central derivative values of the real quadratic twisted $L$-function. We also speculate that such relations should appear more generally in the Hilbert case (Conjecture \ref{metatrace2}).

\subsubsection{Outline} We fill in the details outlined above, starting with background on spin groups in $\S2$, spin Shimura varieties in $\S3$, and regularized theta lifts in $\S4$, 
leading to computations along CM cycles and real geodesic cycles in $\S5$. 
We describe the vector-valued Rankin-Selberg products associated to the spaces $(V_A, Q_A)$ in terms of standard Rankin-Selberg $L$-functions in $\S6$, 
followed by arithmetic height formulae and implications for the CM case in $\S7$. 
For comparison, we develop the Bruinier-Yang proof of Gross-Zagier in Appendix A, then relate to traces of singular moduli in Appendix B.

\subsubsection*{Acknowledgements} I thank Jan Bruinier, \"Ozlem Imamoglu, and Don Zagier for helpful communications.

\tableofcontents

\section{Quadratic spaces and spin groups}

Let $k = {\bf{Q}}(\sqrt{d})$ be a quadratic field of discriminant 
\begin{align*} d_k &= \begin{cases} d &\text{ if $d \equiv 1 \bmod 4$} \\ 4d &\text{ if $d \equiv 2,3 \bmod 4$}. \end{cases}\end{align*}
We write $\mathcal{O}_k$ for its ring of integers, $C(\mathcal{O}_k)$ for its ideal class group, and $h_k = \# C(\mathcal{O}_k)$ its class number.
We also write $w_k = \# \mu(k)$ to denote the number of roots of unity in $k$. 

\subsection{Quadratic spaces associated to class groups of quadratic fields}
   
Fix an ideal class $A \in C(\mathcal{O}_k)$. Let $\mathfrak{a} \subset \mathcal{O}_k$ be an integral ideal representative.  
Consider the corresponding vector space $\mathfrak{a}_{\bf{Q}} = \mathfrak{a} \otimes_{\bf{Z}} {\bf{Q}}$, 
which when equipped with the norm form $Q_{\mathfrak{a}}(\lambda) {\bf{N}}_{k/{\bf{Q}}}(\lambda)/ {\bf{N}} \mathfrak{a}$ 
can be viewed as a rational quadratic space over ${\bf{Q}}$. 
That is, $(\mathfrak{a}_{\bf{Q}}, Q_{\mathfrak{a}})$ determines a rational quadratic space of signature 
\begin{align*} \begin{cases} (2, 0)&\text{ if $d<0$ so that $k = {\bf{Q}}(\sqrt{d})$ is imaginary quadratic} \\
(1,1) &\text{ if $d>0$ so that $k = {\bf{Q}}(\sqrt{d})$ is real quadratic}. \end{cases} \end{align*}
On the other hand, we can also consider the corresponding isomorphic quadratic space $(\mathfrak{a}_{\bf{Q}}, - Q_{\mathfrak{a}})$ of signature 
\begin{align*} \begin{cases} (0, 2)&\text{ if $d<0$ so that $k = {\bf{Q}}(\sqrt{d})$ is imaginary quadratic} \\
(1,1) &\text{ if $d>0$ so that $k = {\bf{Q}}(\sqrt{d})$ is real quadratic}. \end{cases} \end{align*}
Thus, we obtain for each class $A = [\mathfrak{a}] \in C(\mathcal{O}_k)$ a rational quadratic space of signature (2,2) defined by
\begin{align}\label{V_A} (V_A, Q_A), \quad V_A:= \mathfrak{a}_{\bf{Q}} \oplus \mathfrak{a}_{\bf{Q}}, 
\quad Q_A(z) = Q_A((z_1, z_2)):= Q_{\mathfrak{a}}(z_1) - Q_{\mathfrak{a}}(z_2). \end{align} 

\subsubsection{Anisotropic subspaces}

Henceforth, we consider the isotropic quadratic space $(V_A, Q_A)$ of signature $(2,2)$ 
for each class $A \in C(\mathcal{O}_k)$. We consider the corresponding anisotropic subspaces 
$(V_{A, 1}, Q_{A,1}) = (V_{\mathfrak{a}}, Q_{\mathfrak{a}})$ and $(V_{A,2}, Q_{A ,2}) = (\mathfrak{a}_{\bf{Q}}, - Q_{\mathfrak{a}})$ of respective
signatures $(2,0)$ and $(0, 2)$ when $k = {\bf{Q}}(\sqrt{d})$ is imaginary quadratic ($d<0$), and signature $(1,1)$ when $k = {\bf{Q}}(\sqrt{d})$ is real quadratic $(d>0)$.

\subsection{Spin groups}

Let $V = (V,Q)$ be a rational quadratic space of signature $(n, 2)$ for any integer $n \geq 0$. 
We consider the reductive group $\operatorname{GSpin}(V)$ over ${\bf{Q}}$, which fits into the short exact sequence 
\begin{align}\label{spinSES} 1 &\longrightarrow {\bf{G}}_m \longrightarrow \operatorname{GSpin}(V) \longrightarrow \operatorname{SO}(V) \longrightarrow 1. \end{align}

\subsubsection{General characterization}

Given $R$ a commutative ring with identity, we consider a quadratic module $(V, Q)$ over $R$. 
%Hence, $V$ is a projective $R$-module of finite rank equipped with a homogeneous function $Q: V \rightarrow R$ of degree two
%for which the corresponding symmetric pairing $(x, y) = \frac{1}{2} \left\lbrace Q(x+y) - Q(x) - Q(y) \right\rbrace$ is $R$-bilinear. 
%We call the space $(V, Q)$ {\it{self-dual}} if this pairing induces an isomorphism $V \cong \operatorname{Hom}(V, R)$.
%We call the space $(V, Q)$ {\it{non-degenerate}} if its orthogonal complement 
%$V^{\perp} = \left\lbrace x \in V: (x, y) = 0 ~~\forall y \in V \right\rbrace$ is $\lbrace 0 \rbrace$.
Let $C(V) = T(V) / I(V)$ denote its Clifford algebra, given by the quotient of the tensor algebra 
\begin{align*} T(V) &= \bigoplus_{m =0}^{\infty} V^{\otimes m} = R \oplus V \oplus (V \otimes_R V) \oplus \cdots \end{align*}
of $V$ by the two-sided ideal $I(V)$ generated by elements of the form $v \otimes v - Q(v)$ for $v \in V$. 
Note that we have canonical embeddings of $R$ and $V$ into $C(V)$. 
In this way, we see that $R$-module $C(V)$ is generated by the image of the natural injection $V \rightarrow C(V)$, 
and that the grading on $T(V)$ induces a ${\bf{Z}}/ 2{\bf{Z}}$ grading \begin{align*} C(V) = C^0(V) \oplus C^1(V). \end{align*} 
Concretely, $C^0(V)$ is the $R$-module of $C(V)$ generated by products of an even number of basis vectors, and $C^1(V)$ that generated by products of an odd number of basis vectors. 
We call $C^0(V)$ the {\it{even Clifford algebra of $V$}}. We write $v_1 \cdots v_m$ to denote the element of $C(V)$ represented by $v_1 \otimes \cdots \otimes v_m$ (for $v_1, \ldots, v_m \in V$).

\begin{proposition}\label{Clifford} 

Let $(V, Q)$ be a non-degenerate quadratic space over a field $F$ of characteristic $\operatorname{char}(F) \neq 2$.
Fix an orthogonal basis $v_1, \ldots, v_m$ of $V$, and let $\delta(V) := v_1 \cdots v_m \in C(V)$. Let $d(V)$ denote the discriminant of the
space $(V,Q)$, given by the determinant of the Gram matrix $((v_i, v_j))_{i, j}$ (for any basis $v_1, \cdots v_m$ of $V$).

\begin{itemize}

\item[(i)] We have that
\begin{align*} \delta(V)^2 &= \begin{cases} (-1)^{\frac{m}{2}} 2^{-m} d(V) \in F^{\times}/(F^{\times})^2 &\text{ if $m \equiv 0 \bmod 2$} \\
(-1)^{\frac{m-1}{2}} 2^{-m} d(V) \in F^{\times}/(F^{\times})^2 &\text{ if $m \equiv 1 \bmod 2$} \end{cases}. \end{align*}

\item[(ii)] The centre $Z(C(V))$ of $C(V)$ is given by 
\begin{align*} Z(C(V)) &= \begin{cases} F &\text{ if $m \equiv 0 \bmod 2$} \\ F + F \delta(V) &\text{ if $m \equiv 1 \bmod 2$}, \end{cases} \end{align*}
and the centre $Z(C^0(V))$ of $C^0(V)$ is given by 
\begin{align*} Z(C^0(V)) &= \begin{cases} F + \delta(V) F &\text{ if $m \equiv 0 \bmod 2$} \\ F &\text{ if $m \equiv 1 \bmod 2$}. \end{cases}, \end{align*}

\end{itemize} \end{proposition}

\begin{proof} For (i), see \cite[Remark 2.5]{Br-123}. For (ii), see \cite[Theorem 2.6]{Br-123}. \end{proof}

Let $J:C(V)\rightarrow C(V)$ denote the canonical automorphism, induced from multiplication by $-1$ on $V$.
Let $(x_1 \cdots  x_m)^t := x_m  \cdots  x_1$ denote the canonical involution on $C(V)$, and $N_{C(V)}(x) := {}^t x x$ the Clifford norm. 
On vectors $x \in V$, this reduces to $N_{C(V)}(x) = Q(x)$. Consider the Clifford group of $C(V)$, 
\begin{align*} G_{C(V)} &= \left\lbrace x \in C(V): ~~\text{$x$ is invertible and $x V J(x)^{-1} = V$} \right\rbrace. \end{align*}
The general spin group $\operatorname{GSpin}(V)$ can be defined as the intersection $\operatorname{GSpin}(V) = G_{C(V)} \cap C^0(V)$,
and the spin group as the subgroup $\operatorname{Spin}(V) = \left\lbrace  x \in \operatorname{GSpin}(V): N_{C(V)}(x) =1 \right\rbrace$ of elements of Clifford norm one. 

\begin{lemma}\label{SL} If $m = \dim_F(V) \leq 4$, then we have identifications 
\begin{align*} \operatorname{GSpin}(V) \cong \left\lbrace x \in C^0(V): N_{C(V)}(x) \in F^{\times} \right\rbrace 
~~~~~\text{and}~~~~
\operatorname{Spin}(V) \cong \left\lbrace x \in C^0(V): N_{C(V)}(x) =1 \right\rbrace. \end{align*} \end{lemma}
\begin{proof} See \cite[Lemma 2.14]{Br-123}. \end{proof}

\subsubsection{Exceptional isomorphisms}

Let $(V_A, Q_A) = (\mathfrak{a}_{\bf{Q}} \oplus \mathfrak{a}_{\bf{Q}}, Q_{\mathfrak{a}} - Q_{\mathfrak{a}})$ be any 
of the rational quadratic spaces of signature $(2, 2)$ considered above. Hence, $\dim_{{\bf{Q}}}(V_A)=4$,
$\dim_{\bf{Q}} C(V_A) = 2^4 = 16$, and $\dim_{\bf{Q}} C^0(V_A) = 8$.

\begin{proposition}\label{Clifford2}

Let $k$ be a quadratic field with class group $C(\mathcal{O}_k)$.
Fix any class $A \in C(\mathcal{O}_k)$, together with an integer ideal representative $\mathfrak{a} \subset \mathcal{O}_k$,
and write $Q_{\mathfrak{a}}(z) = {\bf{N}}_{k/{\bf{Q}}}(z)/{\bf{N}} \mathfrak{a}$ to denote the corresponding norm form. 
Consider the corresponding rational quadratic space 
$(V_A, Q_A) = (\mathfrak{a}_{\bf{Q}} \oplus \mathfrak{a}_{\bf{Q}}, Q_{\mathfrak{a}} - Q_{\mathfrak{a}})$ of signature $(2, 2)$,
with Clifford algebra $C(V_A)$ and even Clifford subalgebra $C^0(V_A) \subset C(V_A)$. We have the identification
$C^0(V_A) \cong M_2({\bf{Q}}) \oplus M_2({\bf{Q}})$, from which we derive exceptional isomorphisms 
$\operatorname{Spin}(V_A) \cong \operatorname{SL}_2^2$ 
and $\operatorname{GSpin}(V_A) \cong  \operatorname{GL}_2 \times_{{\bf{G}}_m} \operatorname{GL}_2$ of algebraic groups over ${\bf{Q}}$. 

\end{proposition} 

\begin{proof} Cf.~\cite[Proposition 3.3 (ii)]{VO}, which computes the discriminant $d(V_A)$ as the determinant of the Gram matrix for some explicit basis.
Since $(V_A, Q_A) \cong (\mathfrak{a}_{\bf{Q}} , Q_{\mathfrak{a}}) \oplus (\mathfrak{a}_{\bf{Q}}, -Q_{\mathfrak{a}})$, 
we see by inspection that $d(V_A) = d(V_{\mathfrak{a}})^2 \equiv 1 \bmod ({\bf{Q}}^{\times})^2$.
That is, $d(V_A) \in ({\bf{Q}}^{\times})^2$ is a nonzero rational square, hence trivial.
Using the relation $\delta(V_A)^2 = 2^{-4} d(V_A)$ of Proposition \ref{Clifford} (i), 
we deduce that the volume form $\delta(V_A) \in {\bf{Q}}^{\times}$ must be rational. 
We then deduce from Proposition \ref{Clifford} (ii) that $Z(C^0(V_A)) = {\bf{Q}} + \delta(V_A) {\bf{Q}} = {\bf{Q}}$, 
and hence that the even Clifford algebra $C^0(V_A)$ of $\dim_{\bf{Q}}(C^0(V_A)) = 8$
must be a direct sum of two isomorphic copies of a quaternion algebra $B$ over ${\bf{Q}}$. 
Using the classifications of Clifford algebras over ${\bf{R}}$, we see that $C(V_A \otimes {\bf{R}}) \cong C_{2,2}({\bf{R}}) \cong M_4({\bf{R}})$
and $C^0(V_A \otimes {\bf{R}}) \cong C_{2,2}^0({\bf{R}}) \cong M_2({\bf{R}}) \oplus M_2({\bf{R}})$. 
Hence, $B$ must be indefinite. Since the discriminant $d(V_A) =1$ is trivial, 
we deduce that $B=M_2({\bf{Q}})$, with the Clifford norm corresponding to the reduced norm
homomorphism $\operatorname{nrd}: B \rightarrow {\bf{Q}}$ given by the determinant $\det = \operatorname{nrd}$. 
The claimed isomorphisms for the spin groups then follow from the characterization given in Lemma \ref{SL},
which implies that $\operatorname{GSpin}(V_A)$ can be identified with a subgroup of $\operatorname{GL}_2^2$,
and the exact sequence $(\ref{spinSES})$, which implies that $\dim \operatorname{GSpin}(V_A) = \dim \operatorname{SO}(V_A) + 1 =
\dim \operatorname{SO}(2,2) + 1 = 6+1 = 7$. \end{proof}

\begin{corollary}\label{lattices} 

Fix $N \in {\bf{Z}}_{\geq 1}$. Let  $L_A = L_A(N) \subset V_A$ denote the lattice whose adelization $L_A \otimes \widehat{\bf{Z}}$
is stabilized under the action via conjugation by $\operatorname{GSpin}(V_A)({\bf{A}}_f) \cong  \operatorname{GL}_2({\bf{A}}_f) \times_{{\bf{G}}_m} \operatorname{GL}_2({\bf{A}}_f)$
by the compact open subgroup $K_0(N) \times K_0(N)$, where 
$K_0(N) \subset \operatorname{GL}_2(\widehat{\bf{Z}}) \subset \operatorname{GL}_2({\bf{A}}_f)$ denotes the congruence subgroup 
\begin{align*} K_0(N) &= \left\lbrace \left(\begin{array}{cc} a & b \\ c & d \end{array} \right) \in \operatorname{GL}_2(\widehat{\bf{Z}}): c \equiv 0 \bmod N \right\rbrace.\end{align*} \begin{itemize}

\item[(i)] The lattice is $L_A = L_A(N) = N^{-1} \mathfrak{a} \oplus N^{-1} \mathfrak{a}$, with dual lattice 
$L_A^{\vee} = L_A(N)^{\vee} = \mathfrak{d}_k^{-1} N^{-1} \mathfrak{a} \oplus  \mathfrak{d}_k^{-1} N^{-1} \mathfrak{a}$

\item[(ii)] The level of the lattice is $N = \lbrace \min a \in {\bf{Z}}_{\geq 1}: aQ_A(\lambda) \in {\bf{Z}} \forall \lambda \in L_A^{\vee} \rbrace$.

\end{itemize}

\end{corollary}

\begin{proof}

Consider $\operatorname{GSpin}(V_A)({\bf{A}}_f) \cong  \operatorname{GL}_2({\bf{A}}_f) \times_{{\bf{G}}_m} \operatorname{GL}_2({\bf{A}}_f)$ acting on 
$V_A = \mathfrak{a}_{\bf{Q}} \oplus \mathfrak{a}_{\bf{Q}}$ by conjugation. 
Here, we use the canonical embedding $V_A \rightarrow C(V_A)$ and the identification 
$C^0(V_A) \cong M_2({\bf{Q}}) \oplus M_2({\bf{Q}})$. Writing
\begin{align*} R(N) 
&= \left\lbrace \left( \begin{array}{cc} a & b \\ c & d \end{array} \right) : c \equiv 0 \bmod N \right\rbrace \subset M_2({\bf{Z}})\end{align*} 
to denote the Eichler order of level $N$, we can characterize $L_A = L_A(N)$ as the lattice stabilized under conjugation by invertible elements of $R(N) \oplus R(N)$. 
We claim that the conjugation action $g \cdot v = gvg^{-1}$ for $g = (g_1, g_2) \in \operatorname{GSpin}(V_A)({\bf{A}}_f) \cong  \operatorname{GL}_2({\bf{A}}_f) \times_{{\bf{G}}_m} \operatorname{GL}_2({\bf{A}}_f)$ 
and $v = (v_1, v_2) \in \mathfrak{a}_{{\bf{A}}_f} \oplus \mathfrak{a}_{{\bf{A}}_f}$ takes the simpler form 
$(g_1, g_2) \cdot (v_1, v_2) = (g_1 v_1 g_1^{-1}, g_2 v_2 g_2^{-1})$,
for $g_i \in \operatorname{GL}_2({\bf{A}}_f)$ and $v_i \in \mathfrak{a}_{{\bf{A}}_f} = [\alpha_{\mathfrak{a}}, z_{\mathfrak{a}}] {\bf{A}}_f$ for $i=1,2$. 
We can then see by inspection 
that $L_A = L_A(N) = N^{-1} \mathfrak{a} \oplus N^{-1} \mathfrak{a}$ is the stabilized lattice,
with dual lattice $L_A^{\vee} = L_A(N)^{\vee} = \mathfrak{d}_k^{-1} N^{-1} \mathfrak{a} \oplus \mathfrak{d}_k^{-1} N^{-1} \mathfrak{a}$, and 
that this lattice has level $N$. \end{proof}

\section{$\operatorname{GSpin}$ Shimura varieties}

We now describe the Shimura varieties and special cycles that appear. 
We start with the general setting, then focus on the case corresponding to the quadratic spaces $(V_A, Q_A)$ of signature $(2,2)$.

\subsection{Complex Shimura varieties}

Let $(V, Q)$ be any rational quadratic space of signature $(n,2)$ with inner product $(x,y) = \frac{1}{2} \lbrace Q(x+y) - Q(x)-Q(y) \rbrace$.
Write $\operatorname{GSpin}(V)$ for the corresponding general spin group. 
We consider the Grassmannian $D(V) = D^{\pm}(V)$ of oriented\footnote{Although we drop it from the notation
henceforth, we write $D^+(V)$ to denote the negative $2$-planes with positive orientation, and $D^-(V)$ the $2$-planes with negative 
orientation, so that $D^{\pm}(V)$ denotes one of these choices -- which we fix consistently.} negative two-planes in $V( {\bf{R}})$,
\begin{align*} D(V) = \left\lbrace z \subset V_{{\bf{R}}}: \dim(z)=2, Q\vert_z <0 \right\rbrace. \end{align*} 
Extending $(\cdot, \cdot)$ to $V_{\bf{C}}$, the real manifold $D(V)$ is isomorphic to the complex $n$-fold defined by 
\begin{align*} \mathcal{Q}(V) &=  \left\lbrace w \in V_{\bf{C}} \backslash \lbrace 0 \rbrace: 
(w,w) =0, (w, \overline{w}) <0 \right\rbrace / {\bf{C}}^{\times} \subset {\bf{P}}(V({\bf{C}})). \end{align*}
In this way, $D(V)$ acquires the structure of a complex manifold. Here, the isomorphism sends an 
oriented $2$-plane $z = [x, y]$ with basis $[x, y]$ such that $Q(x) = Q(y)$ and hence $(x, y) = 0$ to $w= x + iy \in V_{\bf{C}}$.

Note that $(\operatorname{GSpin}(V), D(V))$ determines a Shimura datum. We have a natural embeddings of ${\bf{R}}$-algebras 
${\bf{C}} \rightarrow C(z) \rightarrow C(V_{\bf{R}})$ for any $2$-plane $z = [x, y] \subset V_{\bf{R}}$, with the first induced by the map 
\begin{align*} i \longmapsto \frac{x y}{ \sqrt{Q(x) Q(y)}}. \end{align*}
The induced map ${\bf{C}}^{\times} \rightarrow C(V_{\bf{R}})^{\times}$ takes values in $\operatorname{GSpin}(V)({\bf{R}})$,
and arises from a morphism of real algebraic groups 
$\alpha_z: \operatorname{Res}_{ {\bf{C}} /{\bf{R}}} {\bf{G}}_m \rightarrow \operatorname{GSpin}(V)({\bf{R}})$.
In this way, we can identify $D(V)$ with a conjugacy class in 
$\operatorname{Hom}(\operatorname{Res}_{ {\bf{C}} /{\bf{R}}} {\bf{G}}_m, \operatorname{GSpin}(V)({\bf{R}}) )$.
Hence, we can associate a Shimura variety to $(\operatorname{GSpin}(V), D(V))$.
 
Any choice of integral lattice $L \subset V$ determines a compact open subgroup 
\begin{align*} K = K_L := \operatorname{GSpin}(V)({\bf{A}}_f) \cap C(\widehat{L}) \subset \operatorname{GSpin}(V)({\bf{A}}_f),
\quad \widehat{L} = L_{\widehat{\bf{Z}}}. \end{align*}
We write $L^{\vee} = \lbrace x \in V: (x, L) \subset {\bf{Z}} \rbrace$ for the dual lattice, and $L^{\vee}/L \cong \widehat{L}^{\vee}/ \widehat{L}$ for the discriminant group. 
Note that $K = K_L$ acts trivially on $\widehat{L}$. Fixing such a choice, we consider the corresponding Shimura variety 
\begin{align}\label{complexpoints} X_K({\bf{C}}) = \operatorname{GSpin}(V)({\bf{Q}}) \backslash D(V) \times \operatorname{GSpin}(V)({\bf{A}}_f) /K
\cong \coprod\limits_{ h \in \operatorname{GSpin}(V)({\bf{Q}}) \backslash \operatorname{GSpin}(V)({\bf{A}}_f) /K } \Gamma_h \backslash D(V) \end{align}
for arithmetic subgroups $\Gamma_h = \operatorname{GSpin}(V)({\bf{Q}}) \cap h K h^{-1}$. This complex orbifold $X_K({\bf{C}})$ has the structure of a 
quasiprojective variety $X_K$ of dimension $n$ over ${\bf{Q}}$ which is projective if and only if $V$ is anisotropic. It is smooth if $K = K_L$ is neat. 
We refer to \cite[$\S$2]{AGHMP}, \cite{KuAC}, and \cite[$\S$1]{KuBF} for more background.
 
\subsection{Special divisors}

Given a vector $x \in V$ with $Q(x) >0$, we define a divisor
\begin{align*} D(V)_x &= \left\lbrace z \in D(V) : z \perp x \right\rbrace. \end{align*}
For each $\mu \in L^{\vee}/L$ and $m \in {\bf{Q}}_{>0}$, we consider the divisor $Z(\mu, m)$ on $X_K$ given by the complex orbifold 
\begin{align}\label{special} Z(\mu, m)({\bf{C}}) &= 
\coprod\limits_{ h \in \operatorname{GSpin}(V)({\bf{Q}}) \backslash \operatorname{GSpin}(V)({\bf{A}}_f) /K } \Gamma_h \Big\backslash
\left(  \coprod\limits_{ x \in \mu_h + L_h \atop Q(x) = m } D(V)_x \right). \end{align}
Here, for any element $h \in \operatorname{GSpin}(v)({\bf{A}}_f)$, we write $L_h \subset V$
for the lattice determined by $\widehat{L}_h = h \cdot \widehat{L}$, and $\mu_h = h \cdot \mu \in L_h^{\vee}/L_h$.
As explained in \cite[$\S2$]{AGHMP} (cf.~\cite{KuAC}, \cite[$\S$1]{KuBF}), these $Z(\mu, m)({\bf{C}}) \rightarrow X_K ( { \bf{C} } )$ 
determine effective Cartier divisors, and admit a moduli description given in terms of the Kuga-Satake abelian scheme over $X_K$. 

\subsection{CM cycles and real geodesic cycles}

Let $V_0 \subset V$ be any rational quadratic subspace of signature $(0,2)$ with corresponding lattice $L_0 = V_0 \cap L$.
The Clifford algebra $C(L_0)$ then determines an order in a quaternion algebra over ${\bf{Q}}$, and its even part $C^0(L_0)$
an order in some imaginary quadratic field $k(V_0)$ determined by $V_0$. The corresponding spin group
$\operatorname{GSpin}(V_0) \cong \operatorname{Res}_{ k(V_0)/{\bf{Q}}} {\bf{G}}_m$ forms a rank-two torus $T(V_0)$
in $\operatorname{GSpin}(V)$. Fixing an embedding $k(V_0) \subset {\bf{C}}$, the left multiplication in $V_0({\bf{R}})$
gives $V_0({\bf{R}})$ the structure of a complex vector space, and determines an orientation. In this way, we see that 
each of the two oriented negative definite subspaces
$z_0 = z_0^{\pm} = V_0({\bf{R}})$ determines a point in $D(V) = D^{\pm}(V)$, 
and $(T(V_0), z_0^{\pm})$ a Shimura datum associated to the zero-dimensional complex orbifold 
\begin{align}\label{CMcycle} Z(V_0)({\bf{C}}) &= T(V_0)({\bf{Q}}) \backslash \left\lbrace z_0 \right\rbrace \times T(V_0)({\bf{A}}_f) / K_0, 
\quad K_0 = K_{L_0} = T(V_0)({\bf{A}}_f) \cap K_L. \end{align}
We call the corresponding zero cycle $Z(V_0) \subset X_K$ the {\it{CM cycle associated to $V_0$}}.

As explained in \cite{AGHMP}, if we assume that $C^0(V_0) \cong \mathcal{O}_{k(V_0)}$ is the maximal order, 
then the ${\bf{Z}}/2{\bf{Z}}$-grading on $C(L_0)$ takes the form $C(L_0) \cong \mathcal{O}_{k(V_0)} \oplus L_0$,
where $L_0$ is both a left and right $\mathcal{O}_{k(V_0)}$-module. In this case, there exists a proper fractional 
$\mathcal{O}_{k(V_0)}$-ideal $\mathfrak{b}$ and left $\mathcal{O}_{k(V_0)}$-module isomorphism 
$L_0 \cong \mathfrak{b}$ which identifies the corresponding quadratic form $Q_0 = Q\vert_{V_0}$ with the norm form 
$Q_0(\cdot) = {\bf{N}}_{k(V_0)/{\bf{Q}}}(\cdot)/{\bf{N}} \mathfrak{b}$. The dual lattice $L_0^{\vee}$ is then 
identified with $\mathfrak{d}_{k(V_0)}^{-1} L_0$ for $\mathfrak{d}_{k(V_0)}^{-1}$ the inverse 
different of $k(V_0)$. In this setting, the zero cycle $Z(V_0)$ can be reinterpreted as the moduli 
stack of elliptic curves with complex multiplication by $\mathcal{O}_{k(V_0)}$. 

Motivated by the study of real quadratic fields, 
we consider sets attached to rational quadratic subspaces $W \subset V$ of signature $(1,1)$. 
Let $D(W) = \left\lbrace z \subset W({\bf{R}}): \dim(z) = 1, 
\text{orientation $\pm$}, Q_W \vert_z <0 \right\rbrace $ denote the domain 
of oriented hyperbolic lines $z = z_W = z_W^{\pm}$ in $W({\bf{R}})$, and fix a connected component $D^{\pm}(W)$. 
We first consider the finite sets defined by
\begin{align*}\mathfrak{G}(W) &= \operatorname{GSpin}(W)({\bf{Q}}) \backslash \operatorname{GSpin}(W)({\bf{A}}_f)/K_W, \quad K_W := K_L \cap \operatorname{GSpin}(W)({\bf{A}}_f). \end{align*}
We call $\mathfrak{G}(W)$ the {\it{geodesic set associated to $W$}}. Given a class $[h] \in \mathfrak{G}(W)$
represented by a finite adelic point $h \in \operatorname{GSpin}(W)({\bf{A}}_f)$, and writing
$\operatorname{GSpin}(W)({\bf{R}})^0$ to denote the connected component of the identity of $\operatorname{GSpin}(W)({\bf{R}})$,
we then consider the real one-dimensional locally symmetric space defined by 
\begin{align*} C_h &= \Gamma_h \backslash D^{\pm}(W), \quad \Gamma_h := \operatorname{GSpin}(W)({\bf{Q}}) \cap \operatorname{GSpin}(W)({\bf{R}})^0 h K_W h^{-1},\end{align*}
together with the corresponding real geodesic cycle defined by the finite dijoint union 
\begin{equation}\begin{aligned}\label{geodesic}
\mathcal{G}(W )&= \coprod\limits_{[h] \in \mathfrak{G}(W) \atop h \in \operatorname{GSpin}(W)({\bf{A}}_f)} C_h
=  \coprod\limits_{[h] \in \mathfrak{G}(W) \atop h \in \operatorname{GSpin}(W)({\bf{A}}_f)} \Gamma_h \backslash D^{\pm}(W). \end{aligned} \end{equation}

\subsection{The $\operatorname{GSpin}$ Shimura varieties $X(\operatorname{GSpin}(V_A), D(V_A))$}

Fix a class $A$ in $C(\mathcal{O}_k)$, together with an integral ideal representative $\mathfrak{a}$,
and consider the corresponding rational quadratic space $(V_A, Q_A)$ of signature (2,2) introduced in $(\ref{V_A})$. 
By Proposition \ref{Clifford2}, we have an accidental isomorphism
\begin{align}\label{acc} \zeta: \operatorname{GSpin}(V_A) \cong  \operatorname{GL}_2 \times_{{\bf{G}}_m} \operatorname{GL}_2 \end{align}
of algebraic groups over ${\bf{Q}}$. Write $L_A \subset V_A$
for the integral lattice whose corresponding compact open subgroup 
$K_A = K_{L_A} \subset \operatorname{GSpin}(V_A)({\bf{A}}_f)$
given by $K_A = \prod_{p < \infty} K_{A, p} = \prod_{p < \infty} K_{\Lambda_A, p}$
has the property that each $K_{A,p} \subset \operatorname{GSpin}(V_A)({\bf{Q}}_p)$
corresponds under $(\ref{acc})$ to the Cartesian product of congruence subgroups
\begin{align}\label{ll} \zeta(K_{A, p}) &= K_{0, p}(N)^2, \quad K_{0, p}(N) :=
\left\lbrace \left(\begin{array}{cc}  a & b \\ c & d \end{array} \right) \in \operatorname{GL}_2({\bf{Z}}_p):
c \in N {\bf{Z}}_p \right\rbrace \subseteq \operatorname{GL}_2({\bf{Z}}_p) \end{align}
for some fixed integer $N \geq 1$. That is, we assume that $K_A$ is identified under $(\ref{acc})$
with the product $K_0(N)^2$ of the congruence subgroup $K_0(N) = \prod_{p < \infty} K_{0, p}(N)$ of 
$\operatorname{GL}_2(\widehat{\bf{Z}}) \subset \operatorname{GL}_2({\bf{A}}_f)$.

\subsubsection{Hermitian symmetric domains}

Recall that we consider the Grassmannian 
\begin{align*} D(V_A) = D^{\pm}(V_A) 
&= \left\lbrace z \subset V_A({\bf{R}}): \dim(z) = 2, ~~~\text{orientation $\pm$}~~~, Q_A\vert_z <0 \right\rbrace \end{align*} 
of oriented negative definite $2$-planes in 
$V_A( {\bf{R}} ) \cong \mathfrak{a}_{ {\bf{R}} } + \mathfrak{a}_{ {\bf{R}} }$. 
Extending the bilinear pairing $(\cdot, \cdot)_A$ to ${\bf{C}}$, we saw that the real manifold $D(V_A)$ is isomorphic to the complex surface defined by the quadric
\begin{align*} \mathcal{Q}(V_A) &= \left\lbrace w \in V_A({\bf{C}}): (w, w)_A = 0, (w, \overline{w})_A < 0 \right\rbrace / {\bf{C}}^{\times} \subset {\bf{P}}(V_A({\bf{C}})) \end{align*}
via the isomorphism sending a properly oriented $2$-plane $z$ with standard basis $z = [x, y] \in D(V_A)$ such that $(x,x)_A = (y,y)_A = -1$ 
and $(x,y)_A =0$ to the complex point $w  = w(z): = x + iy \in \mathcal{Q}(V_A)$. Here, we remark that the quadric $\mathcal{Q}(V_A)$ determines
a complex surface with two connected components $\mathcal{Q}^{\pm}(V_A)$. Our choice of orientation $D^{\pm}(V_A)$ determines one of these,
so that we have the corresponding identification $D^{\pm}(V_A) \cong \mathcal{Q}^{\pm}(V_A)$.
This identification is sometimes referred to as the {\it{projective model}} for $D(V_A) = D^{\pm}(V_A)$.
It is useful for identifying the complex structure on $D(V_A)$, which makes it a hermitian symmetric domain.

We also have the following equivalent description. Fix a Witt decomposition $V_A = W_A \oplus {\bf{Q}} e_1 \oplus {\bf{Q}} e_2$,
with nonzero isotropic basis vectors $e_1$ and $e_2$ chosen so that $(e_1,e_1)_A = (e_2, e_2)_A = 0$ and $(e_1, e_2)_A = 1$. 
Hence, $W_A \subset V_A$ denotes the Lorentzian rational quadratic subspace of signature $(1,1)$ determined by the intersection 
$W_A = V_A \cap e_1^{\perp} \cap e_2^{\perp}$. We can then identify $D(V_A) \cong \mathcal{Q}(V_A)$ with the corresponding tube domain
\begin{align*} \mathcal{H}(V_A) &= \left\lbrace \mathfrak{z} \in W_A({\bf{C}}) : Q_A( \Im( \mathfrak{z}) ) < 0  \right\rbrace. \end{align*}
To be more precise, given an element $w \in V_A({\bf{C}}) = W_A ({\bf{C}}) \oplus {\bf{C}} e_1 \oplus {\bf{C}} e_2$, let us write its 
corresponding Witt decomposition $w = \mathfrak{z} + ae_1 + b e_2$ for $\mathfrak{z} \in W_A({\bf{C}})$ and $a, b \in {\bf{C}}$ as $w = (\mathfrak{z}, a, b)$. 
Given an element $w \in V_A({\bf{C}})$, we also write $[w]$ to denotes its image in $\mathcal{Q}(V_A) \subset {\bf{P}}(V_A({\bf{C}}))$.
We have a biholomorphic map 
\begin{align*} \mathcal{H}(V_A) \cong \mathcal{Q}(V_A), \quad \mathfrak{z} \longmapsto \left[ \mathfrak{z} + e_1 -Q_A(\mathfrak{z}) \right]. \end{align*}
The domain $\mathcal{H}(V_A) \subset W_A({\bf{C}}) \cong {\bf{C}}^2$ has two connected components
$\mathcal{H}^{\pm}(V_A)$ corresponding to the two cones of negative norm vectors in $W_A({\bf{R}})$,
and we have identifications $\mathcal{H}^{\pm}(V_A) \cong \mathfrak{H}^{\pm} = \mathfrak{H}^+ \coprod \mathfrak{H}^{-1}$ 
with the union of Poincar\'e upper--half and lower-half planes
with products of two copies of the upper-half planes $\mathfrak{H}^{\pm}$. The corresponding identification 
$D^{\pm}(V_A) \cong \mathcal{H}^{\pm}(V_A) \cong \mathfrak{H}^{\pm}$ is sometimes referred to as the {\it{tube domain model}}.

\subsubsection{Classical description as a product of modular curves}

Via the identifications of hermitian symmetric domains 
$D^{\pm}(V_A) \cong \mathcal{Q}^{\pm}(V_A) \cong \mathcal{H}^{\pm}(V_A) \cong \mathfrak{H}^{\pm}$
and reductive algebraic groups $(\ref{acc})$ with level $(\ref{ll})$, we obtain  
\begin{equation}\begin{aligned}\label{SSVid} X_{K_A}({\bf{C}}) 
&=  \operatorname{GSpin}(V_A)({\bf{Q}}) \backslash D(V_A) \times \operatorname{GSpin}(V_A)({\bf{A}}_f) /K_A \\
&\cong  \operatorname{GL}_2({\bf{Q}}) \times_{{\bf{G}}_m} \operatorname{GL}_2({\bf{Q}}) \backslash \mathfrak{H}^{\pm} \times \mathfrak{H}^{\pm} \times  \operatorname{GL}_2({\bf{A}}_f) \times_{{\bf{G}}_m} \operatorname{GL}_2({\bf{A}}_f) / \zeta(K_A) 
= Y_0(N) \times Y_0(N), \end{aligned}\end{equation} where  
\begin{align*} Y_0(N) = \Gamma_0(N) \backslash \mathfrak{H} 
\cong \operatorname{GL}_2({\bf{Q}}) \backslash \mathfrak{H}^{\pm} \times \operatorname{GL}_2({\bf{A}}_f) / K_0(N), \quad 
K_0(N) := \prod_{ p < \infty} K_{0, p}(N) \end{align*}
denotes the noncompactified modular curve of level $\Gamma_0(N) \subset \operatorname{SL}_2({\bf{Z}})$.
Hence, we can identify each spin Shimura variety $X_{K_A}({\bf{C}})$ with the surface 
given by the product of modular curves $Y_0(N) \times Y_0(N)$.

\subsubsection{Hirzebruch-Zagier divisors} 

We see from $(\ref{acc})$, $(\ref{ll})$, and $(\ref{SSVid})$ that each divisor $Z(\mu, m) = Z_A(\mu, m)$ defined in 
$(\ref{special})$ above for $\mu \in L_A^{\vee}/L_A$ and $m \in {\bf{Q}}_{>0}$ is given more explicitly by the analytic divisor 
\begin{equation*}\begin{aligned} &Z_A(\mu, m)({\bf{C}}) 
= \coprod\limits_{ h \in \operatorname{GSpin}(V_A)({\bf{Q}}) \backslash \operatorname{GSpin}(V_A)({\bf{A}}_f) /K_A } \Gamma_h \Big\backslash
\left(  \coprod\limits_{ x \in \mu_h + {L_A}_h \atop Q_A(x) = m } D(V_A)_x \right) \\
&\cong \Gamma_0(N)^2 \Big\backslash \coprod\limits_{ x \in \mu + L_A \atop Q_A(x) = m } D(V_A)_x
= \Gamma_0(N)^2 \Big\backslash \coprod\limits_{ x \in \mu + L_A \atop Q_A(x) = m } \left\lbrace z \in D^{\pm}(V_A): (z, x)_A =0 \right\rbrace \subset Y_0(N)({\bf{C}})^2. \end{aligned}\end{equation*}
Note that these divisors can be viewed as embeddings of modular curves into $Y_0(N) \times Y_0(N)$.
This is apparent from the description above, as well as their more intrinsic characterization as analytic divisors in \cite[$\S$2]{KuAC}. 
That is, we choose a positive norm vector $x \in V_A$, or more precisely, an element of the hyperboloid 
\begin{align*} \Omega_{A, \mu, m}({\bf{Q}}) &= \left\lbrace x \in \mu + L_A: Q_A(x) =m \right\rbrace. \end{align*}
We consider the corresponding one-dimensional subspace 
$V_{A, +} := {\bf{Q}} x \subset V_A$, with its orthogonal complement $U_A := V_{A,+}^{\perp} \subset V_A$.
Hence, $U_A \subset V_A$ determines a subspace of signature $(1,2)$. 
Its spin group $\operatorname{GSpin}(U_A)$ is isomorphic  
to the stabilizer of $V_{A,+}$ in $\operatorname{GSpin}(V_A) \cong  \operatorname{GL}_2 \times_{{\bf{G}}_m} \operatorname{GL}_2$ (\cite[Lemma 2.1]{KuAC}).
The natural subspace embedding $U_A \subset V_A$ gives rise to an embedding of reductive algebraic groups
$\operatorname{GSpin}(U_A) \rightarrow \operatorname{GSpin}(V_A)$.
Writing $D(U_A) = D^{\pm}(U_A) = \left\lbrace z \subset U_A({\bf{R}}): \dim(z) = 2, Q_A\vert_z < 0 \right\rbrace$ for the 
corresponding Grassmannian, and $K_{A,U} = K_A \cap \operatorname{GSpin}(U_A)({\bf{A}}_f)$ the corresponding level, we can identify $Z_A(\mu, m)$ with the Shimura subcurve 
\begin{equation*}\begin{aligned} Z_A(\mu, m)({\bf{C}}) 
= \operatorname{GSpin}(U_A)({\bf{Q}}) \backslash D(U_A) \times \operatorname{GSpin}(U_A)({\bf{A}}_f) / K_{A,U} 
&\longrightarrow X_A({\bf{C}}) \cong Y_0(N) \times Y_0(N) \\
\operatorname{GSpin}(U_A)({\bf{Q}})(z, h) K_{A,U} &\longmapsto 
\operatorname{GSpin}(V_A)({\bf{Q}})(z, h) K_A. \end{aligned}\end{equation*}
To be more precise, we know from the discussion above that $Z_A(\mu, m)$
can be identified with the modular curve $\Gamma(U_A) \backslash D(U_A)$, 
where $D(U_A) = D^{\pm}(U_A) \cong \mathfrak{H}$ and 
$\Gamma(U_A) = \operatorname{GSpin}(U_A)({\bf{Q}}) \cap K_{A,U} \subseteq \Gamma_0(N)^2$ is a congruence subgroup. 
As can be seen through this description, the sums over cosets $\mu \in L_A^{\vee}/L_A$ of these divisors $Z_A(\mu,m)$ 
give the classical Hirzebruch-Zagier divisors of the forms described in \cite{HY} and \cite[$\S$2]{Br-123}. We shall return
to this relation to the classical Hirzebruch-Zagier divisors on $Y_0(N) \times Y_0(N)$ given in terms of the moduli discussion
(see e.g.~\cite{HY}) in our discussion of arithmetic heights below. 

\subsubsection{CM cycles}

Let $V_{0, A} \subset V_A$ be any rational quadratic subspace of signature $(0, 2)$, 
with corresponding lattice $L_{0, A} = V_{0,A} \cap L_A$ and quadratic form $Q_{A, 0} = Q_A\vert_{V_{A, 0}}$. 
The even Clifford algebra $C^0(L_{A, 0}) \subset C^0(V_{A, 0})$ determines 
an order in the imaginary quadratic field $k(V_{A,0})$ determined by $V_{A,0}$. 
%Recall that any such subspace $V_{A,0} \subset V_A$ determines a rank-two torus 
%%$\operatorname{GSpin}(V_{0,A}) \cong \operatorname{Res}_{k(V_{A,0})/{\bf{Q}}} {\bf{G}}_m$ and 
%the even Clifford algebra $C^0(L_{A,0}) \cong \mathcal{O}$ an order in $\mathcal{O}_{k(V_{A,0})}$. 
%Again, if $C^0(L_{A,0}) \cong \mathcal{O}_{k(V_{A,0})}$ is maximal, 
%then the ${\bf{Z}}/2{\bf{Z}}$-grading on $C(L_{A,0})$ takes the form $C(L_{A,0}) \cong \mathcal{O}_{k(V_{A,0})} \oplus L_{A,0}$,
%with $L_{A,0}$ being both a left and right $\mathcal{O}_{k(V_{A,0})}$-module, and
%there exists a fractional $\mathcal{O}_{k(V_{A,0})}$-ideal $\mathfrak{b}$ and an isomorphism
%$L_{A,0} \cong \mathfrak{b}$ of left $\mathcal{O}_{k(V_{A,0})}$-modules which identifies the quadratic form
%$Q_{A,0}$ on $L_{A,0}$ with the norm form $- {\bf{N}}_{k(V_{A,0})/{\bf{Q}}}(\cdot)/{\bf{N}} \mathfrak{b}$ and 
%$L_{A,0}^{\vee} \cong \mathfrak{d}_{k(V_{A,0})}^{-1}L_{A, 0}$.

Each sublattice $L_{A,0} \subset L_A$ of signature $(0, 2)$ gives rise to a group scheme $T_{A}$ over ${\bf{Z}}$ with functor 
of points $T_A(R) = (C^0(L_{A,0}) \otimes_{\bf{Z}} R)^{\times}$ for any ${\bf{Z}}$-algebra $R$. This gives a rank-two torus 
$T_A \otimes_{\bf{Z}} {\bf{Q}} = \operatorname{GSpin}(V_{A,0})$ which appears as a subgroup of 
$\operatorname{GSpin}(V_A) \cong  \operatorname{GL}_2 \times_{{\bf{G}}_m} \operatorname{GL}_2$. Writing $L_{A, 0}^{\perp} \subset L_A$
for the complement of the lattice $L_{A,0}$, this maximal subgroup acts trivially on the corresponding 
subspace $V_{A,0}^{\perp} = L_{A,0} \otimes_{\bf{Z}} {\bf{Q}} \subset V_A$. Let $K_{A,0} = T_A({\bf{A}}_f) \cap K_{L_A}$ 
denote the corresponding compact open subgroup of $T_A({\bf{A}}_f) = \operatorname{GSpin}(V_{A,0})({\bf{A}}_f)$.

Fixing an embedding $\mathcal{O}_{k(V_{A,0})} \subset {\bf{C}}$, we can view $V_{A,0}({\bf{R}}) = V_{A,0} \otimes_{\bf{Q}} {\bf{R}}$ 
as an oriented negative $2$-plane of ${\bf{C}}$, and hence as a point $z_{A,0} = V_{A,0}({\bf{R}}) \subset V_A({\bf{R}})$
in $D(V_A) \cong \mathcal{Q}(V_A) \cong \mathcal{H}(V_A) \cong \mathfrak{H}^2$.
This makes $(T_A \otimes_{\bf{Z}} {\bf{Q}}, z_{A, 0}) = (\operatorname{GSpin}(V_{A,0}), z_{A, 0})$ 
a Shimura datum with reflex field $k(V_{A,0})$. The corresponding orbifold 
\begin{align*} Z(V_{A,0})({\bf{C}}) &= T_A({\bf{Q}}) \backslash \left\lbrace z_{A,0} \right\rbrace \times T_A({\bf{A}}_f) / K_{A,0} \\
&= \operatorname{GSpin}(V_{A,0})({\bf{Q}}) \backslash \left\lbrace V_{0,A}({\bf{R}}) \right\rbrace \times \operatorname{GSpin}(V_{A,0})({\bf{A}}_f)/
\left( \operatorname{GSpin}(V_{A,0})({\bf{A}}_f) \cap K_{L_A}  \right)\end{align*} 
can be viewed as the complex points of a zero-dimensional Shimura variety $Z(V_{A,0}) \longrightarrow \operatorname{Spec}(k(V_{A,0}))$,
or a complex fibre on the moduli stack of elliptic curves with complex multiplication by 
$\mathcal{O} \cong C^0(L_{A,0}) \subset \mathcal{O}_{k(V_{A,0})}$ and $\Gamma_0(N)$-level structure. 

\subsubsection{Real geodesic cycles} 

Let $W_{A} \subset V_{A}$ be any Lorentzian quadratic subspace of signature $(1,1)$, with lattice $M_A = W_A \cap L_A$.
Note that the complement $N_A = M_A^{\perp} \subset L_A$, also determines a Lorentzian subspace 
$U_A = N_A \otimes_{\bf{Z}} {\bf{Q}} \subset V_A$
of signature $(1,1)$. We consider the corresponding domain 
\begin{align*} D(W_A) = D^{\pm}(W_A) 
&= \left\lbrace \mathfrak{y} = [\alpha,\beta] \subset W_A({\bf{R}}) : \dim(\mathfrak{y})=1, 
\text{orientation $\pm$}, Q_A\vert_{ W_A}(\mathfrak{y}) < 0 \right\rbrace \end{align*}
of oriented hyperbolic lines $\mathfrak{y} = [\alpha, \beta] \equiv [\alpha: \beta] \in {\bf{P}}^1({\bf{R}})$, given equivalently as a space of projective lines
\begin{align*} D(W_A) = D^{\pm}(W_A) 
&= \left\lbrace \mathfrak{y} = [\alpha: \beta] \subset {\bf{P}}^1({\bf{R}}) : \text{orientation $\pm$}, Q_A\vert_{W_A}(\alpha, \beta) < 0 \right\rbrace. \end{align*}

Recall that after fixing an oriented basis $z = [x,y]$ of each negative definite $2$-plane $z \subset V_A({\bf{R}})$, and fixing a Witt 
decomposition $V_A = W_A \oplus {\bf{Q}} e_1 \oplus {\bf{Q}} e_2$ corresponding to $W_A$, we have identifications
\begin{align*} D^{\pm}(V_A) \cong \mathcal{Q}^{\pm}(V_A) \cong \mathcal{H}^{\pm}(V_A),
\quad z = [x, y] \mapsto [w(z) = x + iy] \mapsto \mathfrak{z}(w) = \mathfrak{z}(w(z)). \end{align*}
%Here, we write the corresponding Witt decomposition of a point $w \in V_A({\bf{C}})$ as $w = \mathfrak{z}(w) + a(w) e_1 +b(w) e_1.$
%Note that while the point $w(\mathfrak{y}) = \alpha+ i\beta \in {\bf{C}}$ determined by a hyperbolic line $\mathfrak{y} = [\alpha: \beta] \in {\bf{P}}^1({\bf{R}})$
%does not lie in the upper-half plane $\mathfrak{H}$, the roots of the quadratic polynomial $Q_{A} \vert_{W_A}(X, 1) = 0$ 
%(or $Q_A\vert_{W_A}(1, Y) = 0$) determine endpoints of a geodesic arc in $\mathfrak{H}$. 
Again, we consider the corresponding finite sets 
\begin{align*} \mathfrak{G}(W_A) &:= \operatorname{GSpin}(W_A)({\bf{Q}}) \backslash \operatorname{GSpin}(W_A)({\bf{A}}_f) / \overline{K}_A,
\quad \overline{K}_A := K_A \cap \operatorname{GSpin}(W_A)({\bf{A}}_f),\end{align*}
and for each $[h] \in \mathfrak{G}(W_A)$ represented by some $h \in \operatorname{GSpin}(W_A)({\bf{A}}_f)$, the corresponding symmetric spaces 
\begin{align*} C_{A, h} &= \Gamma_{A,h} \backslash D^{\pm}(W_A), 
\quad \Gamma_{A,h} := \operatorname{GSpin}(W_A)({\bf{Q}}) \cap \operatorname{GSpin}(W_A)({\bf{R}})^0 h \overline{K}_A h^{-1}. \end{align*}
We then consider the corresponding real geodesic cycles given by the finite disjoint union over classes
\begin{equation*}\begin{aligned}
\mathcal{G}(W_A)  &= \coprod\limits_{[h] \in \mathfrak{G}(W_A) \atop h \in \operatorname{GSpin}(W_A)({\bf{A}}_f)} C_{A,h}
=  \coprod\limits_{[h] \in \mathfrak{G}(W_A) \atop h \in \operatorname{GSpin}(W_A)({\bf{A}}_f)} \Gamma_{A,h} \backslash D^{\pm}(W_A). \end{aligned}\end{equation*}

\section{Green's functions for special divisors}

We describe the automorphic Green's functions that can be constructed from regularized theta lifts for the special divisors $Z(\mu, m)$. 
We start with the general setting, following \cite{AGHMP, KuBF, Br-123, BY, VO}, then specialize to the case parametrized by the rational quadratic spaces $V_A$ of signature $(2,2)$.

\subsection{Siegel theta functions}

Fix $(V, Q)$ a rational quadratic space of signature $(n, 2)$, with integral lattice $L \subset V$.
We write $L^{\vee}/L$ for the discriminant group, and $\mathfrak{S}_L$ the finite-dimensional 
space of {\bf{C}}-valued functions on $L^{\vee}/L$. Writing $\operatorname{Mp}_2$
for the two-fold metaplectic cover of $\operatorname{SL}_2 \cong \operatorname{Sp}_2$, we consider the Weil representation
\begin{align*} \omega_L: \operatorname{Mp}_2({\bf{Z}}) \longrightarrow \mathfrak{S}_L, \end{align*} 
which for $n \geq 1$ even factors through $\operatorname{SL}_2({\bf{Z}})$ as 
$\omega_L: \operatorname{SL}_2({\bf{Z}}) \rightarrow \mathfrak{S}_L$. 
We define the conjugate action $\overline{\omega}_L$ by 
$\overline{\omega}_L (\gamma) \Phi = \overline{\omega_L(\gamma) \overline{\Phi}}$, and let $\omega_L^{\vee}$ denote
contragredient action of $\operatorname{Mp}_2({\bf{Z}})$ on the complex linear dual $\mathfrak{S}_L^{\vee}$.

We now describe how for each $h \in \operatorname{GSpin}(V)({\bf{A}}_f)/K_L$, we can use $\omega_L$ to construct a Siegel theta function
\begin{align*} \theta_L(\tau, z) : \mathfrak{H} \times D(V) &\longrightarrow \mathfrak{S}_L^{\vee},\end{align*}
which in the variable $z \in D(V) = D^{\pm}(V)$ is $\Gamma_h$-invariant, and in the variable $\tau = u + i v \in \mathfrak{H}$
transforms as a nonholomorphic modular form of weight $\frac{n}{2}-1$ and representation $\omega_L^{\vee}$. We give 
the precise definition in $(\ref{Stheta3})$.

\subsubsection{Theta kernels}

Let $\psi = \otimes_v \psi_v$ denote the standard additive character of ${\bf{A}}/{\bf{Q}}$, with archimedean component $\psi_{\infty}(x) = e(x) = \exp(2 \pi i x)$ for $x \in {\bf{R}}$.
Recall that $\operatorname{Mp}_2$ of $\operatorname{Sp}_2 \cong \operatorname{SL}_2$ fits into the exact sequence 
\begin{align*} 1 \longrightarrow \left\lbrace \pm 1 \right\rbrace \longrightarrow \operatorname{Mp}_2 
\longrightarrow \operatorname{Sp}_2 \longrightarrow 1, \end{align*}
and that $\operatorname{GSpin}(V)$ fits into the exact sequence $(\ref{spinSES})$.
Both groups act on the space of Schwartz-Bruhat functions $\Phi = \otimes_v \Phi_v \in \mathcal{S}(V({\bf{A}}_f))$ by the Weil representation 
\begin{align*} \omega_L = \omega_{L, \psi}:  \operatorname{Mp}_2({\bf{A}}) \times \operatorname{GSpin}(V)({\bf{A}}) 
\longrightarrow \mathcal{S}(V({\bf{A}})). \end{align*}
This gives a natural theta kernel, defined on $g \in \operatorname{Mp}_2({\bf{A}})$, 
$h \in \operatorname{GSpin}(V)({\bf{A}})$, and $\Phi = \otimes_v \Phi_v \in \mathcal{S}(V({\bf{A}}))$ by  
\begin{align}\label{thetakernel} \vartheta_L(g, h; \Phi) &= \sum\limits_{x \in V({\bf{Q}})} \left( \omega_L(g, h) \Phi \right) (x). \end{align}
This function $\vartheta_L(g, h; \Phi)$ is seen by inspection to be left $\operatorname{GSpin}(V)({\bf{Q}})$-invariant, 
and by Poisson summation to be left $\operatorname{Mp}_2({\bf{Q}})$-invariant. 
It is referred to as the {\it{theta kernel associated to the Weil representation $\omega_L$}}.

\subsubsection{Choice of local Schwartz functions}

We choose the following Schwartz functions $\Phi = \otimes_v \Phi_v \in \mathcal{S}(V({\bf{A}}))$ to construct  
theta functions from the theta kernel $(\ref{thetakernel})$; see \cite{Bo, Br, BY}.

We first define the following Gaussian function $\Phi_{\infty} \in \mathcal{S}(V({\bf{R}}))$.
Given a negative $2$-plane $z \in D(V) = D^{\pm}(V)$, we define the corresponding majorant 
$(x, x)_z = (x_{z^{\perp}}, x_{z^{\perp}}) - (x_z, x_z)$, which can be viewed as a positive 
definite quadratic form on $V({\bf{R}})$. We then define the Gaussian function 
\begin{align}\label{Gaussian} \Phi_{\infty}(x, z) &= \exp \left( - (x, x)_z \right), \quad z \in D(V) = D^{\pm}(V), ~~~ x \in V({\bf{R}}). \end{align}
As a function $x \in V({\bf{R}})$, this determines an archimedean local Schwartz function $\Phi_{\infty} \in \mathcal{S}(V({\bf{R}}))$.
This Gaussian function satisfeis the transformation property $\Phi_{\infty}(hx, hz) = \Phi_{\infty}(x, z)$ for all $h \in \operatorname{GSpin}(V)({\bf{R}})$,
and has weight $\frac{n}{2}-1$ under the action of the maximal compact subgroup of $\operatorname{Mp}_2({\bf{R}})$.

For the remaining finite part $\Phi_f = \otimes_{v < \infty} \Phi_v \in \mathcal{S}(V({\bf{A}})_f)$, we shall later take the characteristic functions 
\begin{align*} \Phi_f &= {\bf{1}}_{\mu} := \operatorname{char} \left( \mu + L \otimes \widehat{{\bf{Z}}} \right)
\quad \text{for a coset $\mu \in L^{\vee}/L$}. \end{align*}

\subsubsection{Construction of Siegel theta functions} 

Fix a basepoint $z_0 \in D(V) = D^{\pm}(V)$. For any finite archimedean Schwartz function 
$\Phi_f = \otimes_{v < \infty} \Phi_v \in \mathcal{S}(V({\bf{A}})_f)$, we can define from $(\ref{thetakernel})$ the theta function
\begin{align}\label{Stheta} \theta_L(g, h; \Phi_f) &:= \vartheta_L(g, h; \Phi_{\infty}(\cdot, z_0) \otimes \Phi_f(\cdot)). \end{align}
We obtain a classical Siegel theta series on $\mathfrak{H} \times D(V)$ from this as follows. Given any oriented $2$-plane 
$z = D(V)  = D^{\pm}(V)$, we choose an element $h_z \in \operatorname{GSpin}(V)({\bf{R}})$ for which $h_z z_0 = z$. Note that 
\begin{align*} \omega_L(h_z) \Phi_{\infty}(\cdot, z_0) &= \Phi_{\infty}(\cdot, z). \end{align*}
Choosing $i \in \mathfrak{H}$ as the basepoint, let us for any $\tau = u + iv \in \mathfrak{H}$ write $g_{\tau}$ to denote the mirabolic matrix 
\begin{align*} g_{\tau}  &= \left( \begin{array}{cc} 1 & u \\ ~& 1 \end{array} \right) 
\left( \begin{array}{cc} v^{\frac{1}{2}} & ~\\ ~& v^{-\frac{1}{2}} \end{array} \right) \in \operatorname{SL}_2({\bf{R}}), \end{align*}
and $g_{\tau}' = (g_{\tau}, 1)$ its image in $\operatorname{Mp}_2({\bf{R}})$. 
Note that $g_{\tau}' \cdot i = \tau$. Via $(\ref{Stheta})$, we can define the Siegel theta series 
\begin{equation*}\begin{aligned} \theta_L(\tau, z, h_f; \Phi_f) 
&= v^{-\frac{n}{4} + \frac{1}{2}}  \vartheta_L (g_{\tau}', h_z h_f; \Phi_{\infty}(\cdot, z_0) \otimes \Phi_f(\cdot)) 
= v^{-\frac{n}{4} + \frac{1}{2}} \sum\limits_{x \in V({\bf{Q}})} 
\omega_L(g_{\tau}') \left( \Phi_{\infty}(\cdot, z) \otimes \omega(h_f) \Phi_f \right)(x) \end{aligned}\end{equation*}
for $\tau = u + iv \in \mathfrak{H}$, $z \in D(V) = D^{\pm}(V)$, $h_f \in \operatorname{GSpin}(V)({\bf{A}}_f)$, 
and $\Phi_f = \otimes_{v<\infty} \Phi_v \in \mathcal{S}(V({\bf{A}}_f))$. Since 
\begin{align*} v^{-\frac{n}{4} + \frac{1}{2}} \omega_L(g_{\tau}') \left( \Phi_{\infty}(\cdot, z) \right)(x) 
&= v e\left( Q(x_{z^{\perp}}) \tau + Q(x_z) \overline{\tau}  \right), \end{align*} we have the more explicit expansion  
\begin{align}\label{Stheta2} \theta_L(\tau, z, h_f; \Phi_f) &= v \sum\limits_{x \in V({\bf{Q}})}
e\left( Q(x_{z^{\perp}}) \tau + Q(x_z) \overline{\tau}  \right) \otimes \Phi_f (h_f^{-1} x).\end{align}
This theta series satisfies a transformation law for the metaplectic group; see \cite{KuBF} or \cite[(2.5)]{BY}.
Viewing $\theta_L(\tau, z, h_f; \cdot)$ as a function on $\tau \in \mathfrak{H}$ taking values in the dual space $\mathcal{S}(V({\bf{A}}_f))^{\vee}$
of $\mathcal{S}(V({\bf{A}}_f))$, we see that $\theta_L(\tau, z, h_f; \cdot)$ determines a nonholomorphic modular form of weight $\frac{n}{2} - 1$ 
and representation $\omega_L^{\vee}$. 
  
Let $\mathfrak{S}_L$ denote the subspace of $\mathcal{S}(V({\bf{A}}_f))$ which are supported on $L^{\vee} \otimes \widehat{\bf{Z}}$,
and constant on cosets of $L \otimes \widehat{ {\bf{Z}} }$. For instance, $\mathfrak{S}_L$ contains the characteristic function
${\bf{1}}_{\mu} = \operatorname{char}(\mu + L \otimes \widehat{ {\bf{Z}} })$ for a given coset $\mu \in L^{\mu}/L$. In fact, these
functions form a basis for the space, and we have the decompositon
\begin{align*} \mathfrak{S}_L &= \bigoplus_{\mu \in L^{\vee}/L} {\bf{C}} {{\bf{1}}}_{\mu} \subset \mathcal{S}(V({\bf{A}}_f)). \end{align*}
In particular, it follows that $\dim_{\bf{C}} \mathfrak{S}_L = \vert L^{\vee}/L \vert$ is finite. Writing $\mathfrak{e}_{\mu}$ for the standard
basis element in ${\bf{C}}[L^{\vee}/L]$, we also have a natural identification $\mathfrak{S}_L \cong {\bf{C}}[L^{\vee}/L], {\bf{1}}_{\mu} \mapsto \mathfrak{e}_{\mu}$.
This space $\mathfrak{S}_L$ is stable under the image of $\operatorname{SL}_2(\widehat{\bf{Z}})$ in $\operatorname{Mp}_2({\bf{A}}_f)$.
We define from $(\ref{Stheta2})$ the corresponding $\mathfrak{S}_L$-valued Siegel theta series 
\begin{equation}\begin{aligned}\label{Stheta3}
\theta_L(\tau, z, h_f) &= \sum\limits_{\mu \in L^{\vee}/L} \theta_L(\tau, z, h_f; {\bf{1}}_{\mu}) {\bf{1}}_{\mu} \end{aligned}\end{equation}
We refer to \cite[$\S$2]{BY} for more on these theta series, which coincide with those considered by Borcherds in \cite{Bo}.
   
\subsection{Harmonic weak Maass forms}

Fix a half-integer $l \in \frac{1}{2}{\bf{Z}}$. Recall that a twice-differentiable function $f: \mathfrak{H} \longrightarrow \mathfrak{S}_L$ 
is said to be a {\it{harmonic weak Maass form of weight $l$ and representation $\omega_L$}} if

\begin{itemize}

\item[(i)] $f\vert_{l, \omega_L} \gamma = f$ for all $\gamma \in \Gamma = \operatorname{SL}_2{\bf{Z}}$, where $\vert_{l, \omega_L}$
denotes the Petersson weight-$l$ operator. \\

\item[(ii)] There exists an $\mathfrak{S}_L$-value Fourier polynomial
\begin{align*} P_f(\tau) &= \sum\limits_{\mu \in L^{\vee}/L} \sum\limits_{m \geq 0} c_f^+(\mu, m) e(m \tau) {\bf{1}}_{\mu},
\quad {\bf{1}}_{\mu} := \operatorname{char}(\mu + L \otimes \widehat{{\bf{Z}}}) \end{align*}
known as the {\it{principal part of $f$}} for which $f(\tau) = P_f(\tau) + O(e^{-\varepsilon v})$ for some $\varepsilon >0$ as $v = \Im(\tau) \rightarrow \infty$. \\

\item[(iii)] The function is harmonic: $\Delta_l f =0$ for $\Delta_l$ the hyperbolic Laplacian of weight $l$ defined by 
\begin{align*} \Delta_l := -v^2 \left( \frac{\partial^2}{ \partial u^2 } + \frac{ \partial^2 }{ \partial v^2 }\right) 
+ i l \left( \frac{\partial}{\partial u} + i \frac{\partial}{\partial v}\right), \quad \tau = u + iv \in \mathfrak{H}. \end{align*}

\end{itemize} We write $H_l(\omega_L)$ to denote the ${\bf{C}}$-vector space of harmonic weak Maass forms of weight $l$ and representation. 
Each harmonic weak Maass form $f \in H_l(\omega_L)$ has a unique decomposition $f = f^+ + f^-$ where 
\begin{align*} f^+(\tau) &= \sum\limits_{\mu \in  L^{\vee}/L} \sum\limits_{m \in {\bf{Q}} \atop m \gg - \infty} c_f^+(\mu, m) e(m \tau) {\bf{1}}_{\mu} \end{align*}
and 
\begin{align*} f^-(\tau) &= \sum\limits_{\mu \in  L^{\vee}/L} \sum\limits_{m \in {\bf{Q}} \atop m < 0} c_f^-(\mu, m) W_l(2 \pi m v) e(m \tau) {\bf{1}}_{\mu}, \end{align*}
where $W_l(a):= \int_{-2a}^{\infty} e^{-t} t^{-l} dt = \Gamma(1-l, 2 \vert a \vert)$ for $a<0$ denotes the Whittaker function given by the partial Gamma function,
and $e(\tau) = \exp(2 \pi i \tau)$ for $\tau = u + iv \in \mathfrak{H}$. We call $f^+$ the {\it{holomorphic part of $f$}} and $f^{-}$ the {\it{non-holomorphic part of $f$}}.
We consider the subspace $M_l^!(\omega_L) \subset H_l(\omega_L)$ of {\it{weakly holomorphic forms}} whose poles are supported at the cusps,
as well as the subspace of holomorphic forms $M_l(\omega_L) \subset M_l^!(\omega_L)$, and the subspace of holomorphic cusp forms 
$S_l(\omega_L) \subset M_l(\omega_L) \subset M_L^!(\omega_L) \subset H_l(\omega_L)$.

Recall that we have the Maass weight-lowering operator $L_l$ and the Maass weight-raising operator $R_l$, 
\begin{align}\label{Maassops} L_l := - 2 i v^2 \cdot \frac{\partial }{ \partial \overline{\tau}}, \quad R_l := 2 i \cdot \frac{\partial}{ \partial \tau} + l \cdot v^{-1}. \end{align}
Bruinier and Funke \cite{BF} define an antilinear differential operator
\begin{align}\label{xi} \xi_l: H_l(\omega_L) \longrightarrow S_{2-l}(\overline{\omega}_L), \quad f(\tau) \longmapsto v^{l-2} \overline{L_l f(\tau)}, \end{align}
and show that it sits in a short exact sequence 
\begin{align*}\begin{CD} 0 @>>> M_l^!(\omega_L) @>>> H_l(\omega) @>{\xi_l}>> S_{2-l}(\overline{\omega}_L) @>>> 0 \end{CD} \end{align*}
so that $\ker(\xi_l) = M_l^!(\omega_L)$. We refer to \cite{BF} and \cite[$\S$3]{BY} for more details and basic properties.
  
\subsubsection{Definition of the divisor $Z(f)$} Given $f \in H_{1-\frac{n}{2}}(\omega_L)$, we define the corresponding divisor 
\begin{align}\label{Z(f)} Z(f) &= \sum\limits_{\mu \in L^{\vee}/L} \sum\limits_{m \in {\bf{Q}} \atop m >0} c_f^+(\mu, -m)Z(\mu, m) \end{align}
on $X_K = X_{K_L}$. In the special case of the quadratic spaces $(V_A, Q_A)$ of signature $(2,2)$ with the integral lattices 
$L_A = L_A(N) \subset V_A$ described above, we consider for any $f_0 = f_{0, A} \in H_0(\omega_{L_A})$ the corresponding divisors on $X_{A} = X_{K_A} \cong Y_0(N)^2$ defined by
\begin{align}\label{Z(f_A)} Z_A(f_0) &= \sum\limits_{\mu \in L_A^{\vee}/L_A} \sum\limits_{m \in {\bf{Q}} \atop m >0} c_{f_0}^+(\mu, -m)Z_A(\mu, m).\end{align}

\subsection{Regularized theta lifts}

We describe the regularized theta lifts $\Phi(f, z, h)$ associated to $f \in H_{1-\frac{n}{2}}(\omega_L)$.

\subsubsection{The tautological pairing} Let $\langle \langle \cdot, \cdot \rangle \rangle: \mathfrak{S}_L \times \mathfrak{S}_L^{\vee} \longrightarrow {\bf{C}}$
denote the tautological pairing. Hence, given
\begin{align*} f(\tau) = \sum\limits_{\mu \in L^{\vee}/L} f_{\mu}(\tau) {\bf{1}}_{\mu} \in H_l(\omega_L) \quad\text{and}\quad
g(\tau) = \sum\limits_{\mu \in L^{\vee}/L} g_{\mu}(\tau) {\bf{1}}_{\mu} \in H_{-l}(\omega_L^{\vee}), \end{align*}
we have
\begin{align*} \langle \langle f(\tau), g(\tau) \rangle \rangle &= \sum\limits_{\mu \in L^{\vee}/L} f_{\mu}(\tau) g_{\mu}(\tau). \end{align*}

\subsubsection{Regularized theta integrals}

Given $f \in H_{1-\frac{n}{2}}(\omega_L)$ a harmonic weak Maass form of weight $1 - \frac{n}{2}$ and representation $\omega_L$, we define the 
corresponding regularized theta lift $\Phi(f, z, h)$ for $z \in D(V) = D^{\pm}(V)$ 
and $h \in \operatorname{GSpin}(V)({\bf{A}}_f)$ by the regularized theta integral 
\begin{align*} \Phi(f, z, h) &= \int_{\mathcal{F}}^{\star} \langle \langle f(\tau), \theta_L(\tau, z, h) \rangle \rangle d \mu(\tau)
= \operatorname{CT}_{s=0} \left\lbrace \lim_{T \rightarrow \infty} \int_{\mathcal{F}_T} 
 \langle \langle f(\tau), \theta_L(\tau, z, h) \rangle \rangle v^{-s} d \mu(\tau) \right\rbrace. \end{align*}
Here, we write $\mathcal{F} = \left\lbrace \tau \in \mathfrak{H}: -1/2 \leq \Re(\tau) \leq 1/2, \tau \overline{\tau} \geq 1 \right\rbrace$ to denote 
the standard fundamental domain for the action of $\operatorname{SL}_2({\bf{Z}})$ on $\mathfrak{H}$, and $\mu(\tau) = \frac{du dv}{v^2}$ 
the Poincar\'e measure on $\mathfrak{H}$. The regularized theta integral $\Phi(f, z, h)$ 
is given by the constant term in the Laurent series around $s=0$ of the function
\begin{align*} \lim_{T \rightarrow \infty} \int_{\mathcal{F}_T} 
\langle \langle f(\tau), \theta_L(\tau, z, h) \rangle \rangle v^{-s} d \mu(\tau), \end{align*}
where the limit is taken over truncated domains 
$\mathcal{F}_T = \left\lbrace \tau \in \mathfrak{H}: -1/2 \leq \Re(\tau) \leq 1/2, \tau \overline{\tau} \geq 1, \Im(\tau) \leq T \right\rbrace$.

\subsubsection{Arithmetic automorphic forms and Petersson norms}

Let $\mathcal{L}_{D(V)}$ be the restriction to $D(V) \cong \mathcal{Q}(V)$ of the tautological bundle on ${\bf{P}}(V({\bf{C}}))$.
The natural action of $\operatorname{O}(V)({\bf{R}})$ on $V({\bf{C}})$ induces one of the connected component of the identity
$\operatorname{GSpin}(V)({\bf{R}})^0 \subset \operatorname{GSpin}(V)({\bf{R}})$ on $\mathcal{L}_{D(V)}$. This gives a holomorphic line bundle
\begin{align*} \mathcal{L} &= \operatorname{GSpin}(V)({\bf{Q}}) \backslash \mathcal{L}_{D(V)} \times
\operatorname{GSpin}(V)({\bf{A}}_f) / K \longrightarrow X_K, \end{align*}
which has canonical model defined over ${\bf{Q}}$ by Harris \cite{Ha}.
%Note that on the component $\Gamma_h \backslash D(V)$, it takes the form $\Gamma_h \backslash \mathcal{L}_{D(V)}$.
We define a hermitian metric $h_{\mathcal{L}_{D(V)}}$ on $\mathcal{L}_{D(V)}$ by 
\begin{align*} h_{\mathcal{L}_{D(V)}}(w_1, w_2) &= \frac{1}{2} \cdot (w_1, \overline{w}_2). \end{align*}
Observe that this metric is fixed by the action of $\operatorname{O}(V)({\bf{R}})$, and hence descends to $\mathcal{L}$. 

We now describe the Petersson inner product on sections of $\mathcal{L}^{\otimes l}$ for $l \in {\bf{Z}}$ any integer. 
Fix a Witt decomposition $V = W \oplus {\bf{Q}} e_1 \oplus {\bf{Q}} e_2$ for basis vectors $e_1, e_2$ satisfying $(e_1, e_2)=1$
and $(e_1, e_1) = (e_2, e_2) = 0$, so that $W = V \cap e_1^{\perp} \cap e_2^{\perp}$ determines a rational quadratic subspace of signature $(n-1, 1)$.
Given any vector $w \in V({\bf{C}})$, we then write the corresponding decomposition for the Witt decomposition as $w = \mathfrak{z} + a e_1 + b e_2$.
Note that $D(V) \cong \mathcal{Q}(V)$ is isomorphic to the tube domain 
\begin{align*} \mathcal{H}(V) &= \left\lbrace \mathfrak{z} \in W({\bf{C}}) : \Im(\mathfrak{z}) \in C^-(V) \right\rbrace, 
\quad C^-(V) := \left\lbrace y \in W: (y, y) < 0 \right\rbrace \end{align*} via
\begin{align*} \mathcal{H}(V) \longrightarrow V({\bf{C}}) = W({\bf{C}}) + {\bf{C}} e_1 + {\bf{C}} e_2 \longrightarrow \mathcal{Q}(V), 
\quad \mathfrak{z} \longmapsto w(\mathfrak{z}) := \mathfrak{z} + e_1 - Q(\mathfrak{z}) e_2 \longmapsto [w(\mathfrak{z})]. \end{align*}
The map $\mathfrak{z} \longmapsto w(\mathfrak{z}) := \mathfrak{z} + e_1 - Q(\mathfrak{z}) _2$ can be viewed as a nowhere vanishing
holomorphic section of $\mathcal{L}_{D(V)}$. Observe that this section has norm for the hermitian metric $h_{\mathcal{L}_{D(V)}}$ given by
\begin{align*} \vert \vert w(\mathfrak{z}) \vert \vert ^2 = - \frac{1}{2} \cdot \left( w(\mathfrak{z}), w(\overline{\mathfrak{z}}) \right) 
= - \left( \Im(\mathfrak{z}), \Im(\mathfrak{z}) \right) = -(y,y) =: \vert y \vert^2. \end{align*}

Let us now write $z=[x, y] \in D(V) = D^{\pm}(V)$ for the basis $[x, y]$ of an oriented negative $2$-plane $z \subset V({\bf{R}})$ 
$w(z) := x + iy$, and $[w(z)]$ its image in $\mathcal{Q}(V) = \mathcal{Q}^{\pm}(V) \cong D^{\pm}(V)$. 
Given $h \in \operatorname{GSpin}(V)({\bf{R}})$, we have  
\begin{align*} h \cdot w(z) &= w(h z) j(h, z) \end{align*}
for a holomorphy factor \begin{align*} j: \operatorname{GSpin}(V)({\bf{R}}) \times D(V) \longrightarrow {\bf{C}}. \end{align*}
In this way, we can identify the holomorphic sections of $\mathcal{L}^{\otimes l}$ with functions 
\begin{align*} \Psi: D(V) \times \operatorname{GSpin}(V)({\bf{A}}_f) \longrightarrow {\bf{C}} \end{align*}
satisfying the transformation properties
\begin{itemize}
\item $\Psi(z, hk) = \Psi(z, h)$ for all $k \in K$ 
\item $\Psi(\gamma z, \gamma h)  = j(\gamma, z)^l \Psi(z, h)$ for all $\gamma \in \operatorname{GSpin}(V)({\bf{Q}})$. \end{itemize}
The norm of the section 
\begin{align*} (z, h) &\longmapsto \Psi(z, h) \cdot w(z)^{\otimes l}  \end{align*} 
corresponding to any such function $\Psi$ is given by 
\begin{align*} \vert \vert \Psi(z, h) \vert \vert^2 &= \vert \Psi(z, h) \vert^2 \vert y \vert^{2l}, \end{align*}
and referred to as the {\it{Petersson norm of $\Psi$}}.

\subsubsection{Borcherd's products and automorphic Green's functions}

We now summarize several key results. 

\begin{theorem}[Borcherds]\label{Borcherds}

Let $f \in M_{1 - \frac{n}{2}}^!(\omega_L)$ be a weakly holomorphic form with Fourier series expansion
\begin{align*} f(\tau) &= \sum\limits_{\mu \in L^{\vee}/L} 
\sum\limits_{m \in {\bf{Q}} \atop m \gg - \infty} c_f(\mu, m) e(m \tau) {\bf{1}}_{\mu}, \quad c_f(\mu, m) \in {\bf{Z}}.\end{align*}
Then,
\begin{align}\label{BF} \Phi(f, z, h) &= - 2 \log \vert \Psi(f, z, h) \vert^2 - c_f^+(0, 0) \cdot \left( 2 \log \vert y \vert + \Gamma'(1) \right) \end{align}
for $\Psi(f, z, h)$ a meromorphic modular form on $D(V) \times \operatorname{GSpin}(V)({\bf{A}}_f)$ of weight $\frac{1}{2} \cdot c_f^+(0,0)$ with divisor 
\begin{align*} \operatorname{Div} \left( \Psi(f, \cdot)^2 \right) 
&= Z(f) := \sum\limits_{\mu \in L^{\vee}/L} \sum\limits_{m \in {\bf{Q}} \atop m < 0} c_f^+(\mu, m) Z(\mu,m). \end{align*}

\end{theorem}

\begin{proof} See \cite[Theorem 13.3]{Bo} with \cite[Theorems 1.2 and 1.3]{KuBF} and relevant discussions in Bruiner \cite{Br}, \cite{Br-123}. \end{proof}

\begin{theorem}[Borcherds/Bruinier]\label{Bruinier} 

Let $f \in H_{1-\frac{n}{2}}(\omega_L)$ be any harmonic weak Maass form of weight $1-\frac{n}{2}$ and representation $\omega_L$.
The regularized theta lift $\Phi(f, \cdot)$ is an automorphic Green's function in the sense of Arakelov theory for the divisor $Z(f)$ on $X_K$.
That is, $\Phi(f, z, h)$ satifies the following characterizing properties: \\

\begin{itemize}

\item[(i)] $\Phi(f, z, h)$ is a smooth function on $X_K\backslash Z(f)$ with a logarithmic singularity along $-2 \log Z(f)$. \\

\item[(ii)] The $(1,1)$-form $d d^c \Phi(f, z, h)$ has a smooth extension to all of $X_K$, and satisfies the Green's current equation
$d d^c [\Phi(f, z, h)] = \delta_{Z(f)} + [dd^c \Phi(f, z, h)]$ for $\delta_{Z(f)}$ the Dirac current for $Z(f)$. \\

\item[(iii)] $\Phi(f, z, h)$ is an eigenvector for the generalized Laplacian $\Delta_z$ on $z \in D(V)$, and more precisely 
\begin{align*} \Delta_z \Phi(f, z, h) &= \frac{n}{4} \cdot c_f^+(0,0) \cdot \Phi(f, z, h). \end{align*}

\item[(iv)] $\Phi(f, z, h) \in L^{1 + \varepsilon}(X_K)$ for some $\varepsilon >0.$ 

\end{itemize}

\end{theorem}

\begin{proof} See \cite[Theorem 4.3]{BY} and more generally \cite{Br}. \end{proof}

As explained in \cite[$\S 1.1$]{HMP}, the Shimura variety $X_K = X_{K_L}$ comes equipped with a metrized line bundle
\begin{align*} \mathcal{L} \in \widehat{\operatorname{Pic}}(X_K) \end{align*}
of weight one modular forms, which has an extension to the integral model $\widehat{\operatorname{Pic}}(\mathcal{X}_K)$.

\begin{theorem}[Borcherds/Howard-Madapusi Pera]\label{HM-P}

Fix $f \in M_{1-\frac{n}{2}}^!(\omega_L)$ with integral Fourier coefficients,
\begin{align*} f(\tau) &= \sum\limits_{\mu \in L^{\vee}/L} 
\sum\limits_{m \in {\bf{Q}} \atop m \gg - \infty} c_f(\mu, m) e(m \tau) {\bf{1}}_{\mu}, \quad c_f(\mu, m) \in {\bf{Z}}.\end{align*}
Replacing $f$ by a suitable integer multiple if needed, there exists a rational section $\Psi(f)$ of the line bundle 
$\mathcal{L}^{c_f(0,0)}$ on $\mathcal{X}_K$ whose norm under the metric defined by $\vert \vert z \vert \vert = \frac{(z, \overline{z})}{4 \pi e^{\gamma}}$,
satisfies the relation 
\begin{align*} - 2 \log \vert \vert \Psi(f) \vert \vert &= \Phi(f) \end{align*}
and hence 
\begin{align*} \operatorname{Div}(\Psi(f)) &= \sum\limits_{\mu \in L^{\vee}/L} \sum\limits_{m \in {\bf{Q}} \atop m>0 } c_f(\mu, -m) Z(\mu, m). \end{align*}
In particular, the Borcherds product $\Phi(f)$ is defined over ${\bf{Q}}$, and takes algebraic values.
To be more precise, $\Phi(f, z, h)$ takes values in the algebraic number field to which the point $(z, h) \in X_K$ belongs. \end{theorem}

\begin{proof} See \cite[Theorem 9.1.1]{HMP}, which refines the original theorem of Borcherds \cite{Bo} (cf.~\cite[Theorem 5.2.2]{HMP}). \end{proof}

Finally, we remark that the special divisors $Z(\mu, m)$ on $X=X_K$ can be viewed as arithmetic divisors corresponding
to automorphic Green's functions $\Phi_{\mu, m}(z, h) = \Phi(F_{\mu, m}, z, h)$ for vector-valued Maass-Niebur Poincar\'e series $F_{\mu, m}(\tau, s)$.
We refer to Bruinier \cite{Br} for details.

\subsubsection{Extension to compactifications}

Fix $f \in H_{1-n/2}(\omega_L)$ a harmonic weak Maass form whose holomorphic part $f^+$ has integral Fourier coefficients. 
Fix a compactification $X^{\star}$ of the Shimura variety $X = X_K$.  
For the divisor $Z(f) \subset X$ defined in $(\ref{Z(f)})$ above, there exists a divisor $C(f)$ supported
on the boundary $\partial X^{\star} = X^{\star} \backslash X$ such that $\Phi(f, \cdot)$ is the automorphic 
Green's function in the sense of Theorem \ref{Bruinier} for the corresponding divisor 
\begin{align}\label{Z^c(f)} Z^c(f) &= Z(f) + C(f) \end{align} of degree zero on $X$.
For a more precise description of this in the setting of the modular curve, with the 
quadratic space given in Example $\ref{KSAVn=1}$ below, we refer to the discussion in \cite[$\S$7.3]{BY}.

\section{Summation along isotropic quadratic subspaces}

We now compute the regularized theta lifts $\Phi(f, z, h)$ along CM cycles $Z(V_0)$ and geodesic cycles $\mathcal{G}(W)$. 

\subsection{Eisenstein series and Siegel-Weil formulae}

Fix $V_0 \subset V$ any quadratic subspace of signature $(0,2)$, writing $L_0 = V_0 \cap L$ for the corresponding lattice
and $Q_0 = Q\vert_{V_0}$ the corresponding quadratic form. We also fix a rational quadratic subspace $W \subset V$ of 
signature $(1,1)$, writing $L_W = W \cap L$ and $Q_W = Q\vert_{W}$.

\subsubsection{Langlands Eisenstein series and the Siegel-Weil formula}

We first describe the construction in more general terms. 
Let $(U, Q)$ be any anisotropic rational quadratic space of even dimension $\dim(U)$ and signature $(p(U), q(U))$.
Fix a lattice $L \subset U$, and consider the corresponding Weil representation 
$\omega_{L}: \operatorname{SL}_2({\bf{Z}}) \rightarrow \mathfrak{S}_{L}$.

Let $P = MN \subset \operatorname{SL}_2$ denote the parabolic group of upper-triangular matrices, 
with Levi subgroup $M$ and unipotent radical $N$. We use the shorthand notations
\begin{align*}\begin{aligned} 
M &= \left\lbrace m(a) : a \in {\bf{G}}_m \right\rbrace, &\quad m(a): &= \left( \begin{array}{cc} a^{\frac{1}{2}} & ~\\ ~& a^{-\frac{1}{2}}\end{array} \right) \\
N &= \left\lbrace n(b) : b \in {\bf{G}}_a \right\rbrace, &\quad n(b): &= \left( \begin{array}{cc} 1 & b \\ ~& 1 \end{array} \right). \end{aligned}\end{align*}
Hence, writing $K_{\infty} = \operatorname{SO}_2({\bf{R}})$ for the maximal compact subgroup of $\operatorname{SL}_2({\bf{R}})$
and $K = \operatorname{SL}_2(\widehat{\bf{Z}})$ for the maximal compact subgroup of $\operatorname{SL}_2({\bf{A}}_f)$, 
we have the Iwasawa decomposition $\operatorname{SL}_2({\bf{A}}) = N({\bf{A}}) M({\bf{A}}) K_{\infty} K$.

Let $\chi_U$ denote the quadratic idele class character of ${\bf{Q}}$ given on $x \in {\bf{A}}^{\times}$ by 
\begin{align*} \chi_U(x) &= \left( x, (-1)^{ \frac{ \dim(U) }{2} } \det(U) \right)_{\bf{A}}, \end{align*}
where $\left( \cdot, \cdot \right)_{ \bf{A} }$ denotes the Hilbert symbol on ${\bf{A}}^{\times}$, and $\det(U)$ the Gram determinant of $U$.
Given $s \in {\bf{C}}$, let $I(s, \chi_U)$ denote the principal series representation of $\operatorname{SL}_2({\bf{A}})$
induced by the quasicharacter $\chi_U (\cdot) \vert \cdot \vert^s$. Hence, $I(s, \chi_U)$ consists of all smooth functions
$\phi(g, s)$ on $g \in \operatorname{SL}_2({\bf{A}})$ satisfying the transformation property 
\begin{align*} \phi( n(b) m(a) g, s ) 
\chi_{U}(a) \vert a \vert^{s+1} \phi(g, s)  \end{align*} 
for all $a \in {\bf{A}}^{\times}$ and $b \in {\bf{A}}$. 
There is an $\operatorname{SL}_2({\bf{A}})$-intertwining map 
\begin{align*} \lambda: \mathcal{S}(U({\bf{A}})) 
\longrightarrow I(s_0(U), \chi_U), \quad \lambda(\Phi)(g) := \left( \omega_L(g) \Phi \right)(0)
\quad \text{ for ~~$s_0(U) := \frac{\dim(U)}{2}  -1$}. \end{align*}

Recall that a section $\phi(s) \in I(s, \chi_U)$ is {\it{standard}} if its restriction to the maximal compact subgroup $K_{\infty} K \subset \operatorname{SL}_2({\bf{A}})$
does not depend on $s$. Via the Iwasawa decomposition, we see that any $\lambda(\Phi) \in I(s_0(U), \chi_U)$ has a unique extension to a standard 
section $\lambda(\Phi)(s) \in I(s, \chi_U)$ such that $\lambda(\Phi)(s_0(U)) = \lambda(\Phi)$. We consider the following standard sections.
Let us for any $l \in {\bf{Z}}$ write $\chi_l$ to denote the character of $K_{\infty}$ defined by 
\begin{align*} \chi_l(k_{\theta}) = e^{i l \theta} = \exp(i l \theta), 
\quad k_{\theta} = \left( \begin{array}{cc} \cos \theta & \sin \theta \\ - \sin \theta & \cos \theta\end{array} \right) \in K_{\infty}. \end{align*} 
Let $\phi_{\infty}^{l}(s) \in I(s, \chi_U)$ be the unique standard section for which $\phi_{\infty}^l(k_{\theta}, s) = \chi_l(k_{\theta}) = e^{il \theta}$.
In terms of the Iwasawa decomposition, this section can be characterized by the transformation property 
\begin{align*} \phi_{\infty}^l(n(b) m(a) k_{\infty}, s)  =  \chi_{U}(a) \vert a \vert^{s+1} e^{i l \theta}  \end{align*}
Recall we defined the Gaussian $\Phi_{\infty} \in \mathcal{S}(U({\bf{R}}))$ in $(\ref{Gaussian})$ for $(V, Q)$ of signature $(n,2)$. 
More generally, writing $D(U) = \left\lbrace z \subset U({\bf{R}}): \dim(u) = p(U), Q\vert_z <0 \right\rbrace$  
for the corresponding domain, and defining for a given $z \in D(U)$ the 
corresponding majorant $(x, x)_z = (x_{z^{\perp}}, x_{z^{\perp}}) - (x_z, x_z)$ on $x \in U({\bf{R}})$, we let
\begin{align*} \Phi_{\infty}(x, z) &= \exp \left( - (x, x)_z \right). \end{align*}
Again, we see that $\Phi_{\infty}(h x, h z) = \Phi_{\infty}(x, z)$ for all $h \in \operatorname{GSpin}(U)({\bf{R}})$. 
Viewed as a function of $x \in U({\bf{R}})$, we obtain an archimedean local Schwartz function $\Phi_{\infty} \in \mathcal{S}(U({\bf{R}}))$.
Through the Weil representation, we also know that $K_{\infty}$ acts on $\Phi_{\infty}(x, z)$ with weight $\frac{p(U)-q(U)}{2}$. 
Hence, we see in general that  \begin{align*} \lambda_{\infty}(\Phi_{\infty}(\cdot, z)) &= \phi_{\infty}^{\frac{p(U)-q(U)}{2}}(s_0(U)),
\quad s_0(U) := \frac{\dim(U)}{2}-1. \end{align*}

Given any standard section $\phi(s) \in I(s, \chi_U)$, we consider the corresponding Langlands Eisenstein 
series on $g \in \operatorname{SL}_2({\bf{A}})$, defined first for $\Re(s) >1$ by the summation
\begin{align*} E_L(g, s; \phi) &= \sum\limits_{\gamma \in P({\bf{Q}}) \backslash \operatorname{SL}_2({\bf{Q}})} \phi(\gamma g, s).\end{align*}
This sum $E_L(g, s; \phi)$ has a meromorphic continuation to all $s \in {\bf{C}}$ 
via the Langlands functional equation, which relates $E_L(g, s; \phi)$ to $E_L(g, -s; M \phi)$ for the corresponding intertwining operator $M$.
It determines an automorphic form on $g \in \operatorname{SL}_2({\bf{A}})$. 
Its value at $s_0(U)$ is known classically to be holomorphic, as it is given as an average of the corresponding 
Siegel theta series $\vartheta_L(g, h; \Phi)$ by the following classical theorem.

\begin{theorem}[Siegel-Weil]\label{SW}

Let $(U, Q)$ be any anisotropic quadratic space of signature $(p(U), q(U))$ dimension $p(U) + q(U) = 2$.
Let $L\subset U$ be any integral lattice, and $\Phi \in \mathcal{S}(U({\bf{A}}))$ any Schwartz function. 
The Eisenstein series $E_L(g, s; \lambda(\Phi))$ is holomorphic at $s_0(U) = (p(U)+q(U))/2-1$, and given by the formula 
\begin{align*} \int\limits_{ \operatorname{SO}(U)({\bf{Q}}) \backslash  \operatorname{SO}(U)({\bf{A}})} \vartheta_L(g,h; \Phi) dh 
&= E_L(g, s_0(U); \lambda(\Phi)).\end{align*}
Here, we write $dh$ to denote the Tamagawa Haar measure on $\operatorname{SO}(U)({\bf{A}})$. \end{theorem}

\begin{proof} See e.g. \cite[Theorem 4.1]{KuBF} or \cite[Theorem 2.1]{BY}. \end{proof}

\subsubsection{Classical description and images under Maass operators}

The Eisenstein series $E_L(g, s; \phi_{\infty}^l \otimes \lambda({\bf{1}}_{\mu}))$ for each coset $\mu \in L^{\vee}/L$ has the following classical description (cf.~\cite[$\S$2.2]{BY}).
Writing $\Gamma = \operatorname{SL}_2({\bf{Z}})$ and $\Gamma_{\infty} = P({\bf{Q}}) \cap \Gamma = \lbrace n(b): b \in {\bf{Z}} \rbrace$,
we see that $P({\bf{Q}}) \backslash \operatorname{SL}_2({\bf{Q}}) = \Gamma_{\infty} \backslash \Gamma$. 
Using the Iwasawa decomposition, we can write the action of each matrix in the sum as 
\begin{align*} \gamma = \left( \begin{array}{cc} a & b \\ c & d \end{array} \right) \in \Gamma  \quad \implies \quad
\gamma g_{\tau} = n(\beta) m(\alpha) k_{\infty} \quad &\text{  for some $\alpha \in {\bf{R}}_{>0}$ and $\beta, \theta \in {\bf{R}}$.}\end{align*}
A direct computation reveals that 
\begin{align*} \alpha = v^{\frac{1}{2}} \vert c \tau + d \vert^{-1} \quad \text{and} \quad e^{i \theta} = \frac{ c \overline{\tau} + d }{ \vert c \tau + d \vert }\end{align*}
and hence 
\begin{align*} \phi_{\infty}^l(\gamma g_{\tau}) &= v^{\frac{s}{2} + \frac{1}{2}} (c \tau + d)^{-l} \vert c \tau + d \vert^{l - s- 1}, \end{align*}
so that  
\begin{equation*}\begin{aligned} E_L(g_{\tau}, s; \phi_{\infty}^l \otimes \lambda_f({\bf{1}}_{\mu})) &= 
\sum\limits_{  \gamma = \left( \begin{array}{cc} a & b \\ c & d \end{array} \right) \in \Gamma_{\infty} \backslash \Gamma } (c \tau + d)^{-l} \cdot
\frac{ v^{\frac{s}{2} + \frac{1}{2}}  }{ \vert c \tau + d \vert^{s + 1 - l}   } \cdot \lambda_f({\bf{1}}_{\mu})(\gamma) \\
&= \sum\limits_{  \gamma = \left( \begin{array}{cc} a & b \\ c & d \end{array} \right) \in \Gamma_{\infty} \backslash \Gamma } (c \tau + d)^{-l} \cdot
\frac{ v^{\frac{s}{2} + \frac{1}{2}}  }{ \vert c \tau + d \vert^{s + 1 - l}   } 
\cdot \langle {\bf{1}}_{\mu}, \left( \omega_L^{-1}(\gamma)\right) {\bf{1}}_0  \rangle.\end{aligned}\end{equation*}
It follows that 
\begin{align}\label{E-classical} E_L(\tau, s; l) := \sum\limits_{ \mu \in L{\vee}/L} E_L(g_{\tau}, s; \phi_{\infty}^l \otimes \lambda_f({\bf{1}}_{\mu}))
&= \sum\limits_{ \gamma \in \Gamma_{\infty} \backslash \Gamma } \left[ \Im(\tau)^{\frac{s + 1 - l}{2}} {\bf{1}}_0 \right] \Big\vert_{l, \omega_{L}} \gamma. \end{align}

Recall that we consider the Maass weight raising and lowering operators $R_l$ and $L_l$, as defined in $(\ref{Maassops})$.
A simple computation with the series expansion on the right-hand side of $(\ref{E-classical})$ reveals that 
\begin{equation}\begin{aligned}\label{Maassopreln} L_l E_{L}(\tau, s; l) &= \frac{1}{2}(s+1-l) \cdot E_{L}(\tau, s; l-2) \\ R_l E_{L}(\tau, s; l) &= \frac{1}{2}(s + 1 + l) \cdot E_{L}(\tau, s; l+2).\end{aligned}\end{equation} 

\subsubsection{The CM case} 

We now consider the special case of the negative definite subspace $(L_0, Q_0)$.
Hence, we consider for each integer $l \in {\bf{Z}}$ the $\mathfrak{S}_{L_0}$-valued Langlands Eisenstein series 
\begin{align}\label{E-CM} E_{L_0}(\tau, s; l) &:= \sum\limits_{ \mu \in L_0^{\vee}/L_0 } 
E_{L_0}(g_{\tau}, s; \phi_{\infty}^l \otimes \lambda_f({\bf{1}}_{\mu})) {\bf{1}}_{\mu}, 
\quad g_{\tau} = n(u)m(v) \in \operatorname{SL}_2({\bf{R}}), \tau = u + iv \in \mathfrak{H}. \end{align}
We can then describe Theorem \ref{SW} in terms of the Siegel theta series $(\ref{Stheta3})$ as 
\begin{align}\label{SW-CM} \int\limits_{  \operatorname{SO}(V_0)({\bf{Q}}) \backslash  \operatorname{SO}(V_0)({\bf{A}}_f) }
\theta_{L_0}(\tau, z_0, h_f) dh &= E_{L_0}(\tau, 0; -1). \end{align}
Here (cf.~\cite[Proposition 2.2]{BY}), we let $z_0 \in D(V_0) = D^{\pm}(V_0)$ denote the oriented negative $2$-plane determined by $V_0({\bf{R}})$. 
We use Tamagawa Haar measure on $\operatorname{SO}(V_0)({\bf{A}})$ (described below), for which $\operatorname{vol}(\operatorname{SO}(V_0)({\bf{R}}))=1$. 

Now, we see from the relation $(\ref{Maassopreln})$ that 
\begin{align}\label{FI-CM} L_1 E_{L_0}(\tau, s; 1) &= \frac{s}{2} \cdot E_{L_0}(\tau, s; -1). \end{align}
Since the Eisenstein series $E_{L_0}(\tau, s; -1)$ on the right-hand side of $(\ref{FI-CM})$ is holomorphic
at $s = s_0(V_0) = 0$ by the Siegel-Weil formula $(\ref{SW-CM})$, we deduce that the Eisenstein series
$E_{L_0}(\tau, s; 1)$ appearing on the left-hand side must vanish at its central point $s=0$ for its functional equation.
We also obtain from $(\ref{FI-CM})$ the relation 
\begin{align}\label{FI2-CM} L_1 E_{L_0}'(\tau, 0; 1) &= \frac{1}{2} E_{L_0}(\tau, 0; -1)\end{align}
for the derivative $E_{L_0}'(\tau, 0; 1) = \frac{d}{ds} E_{L_0}(\tau, s; 0) \big\vert_{s=0}$ at $s=0$.
Writing $\partial$ and $\overline{\partial}$ to denote the Dolbeault operators, so that the exterior
derivative on differential forms on $\mathfrak{H}$ is given by $d = \partial + \overline{\partial}$,
and again $d \mu(\tau) = \frac{du dv}{v^2}$ for $\tau= u+i v \in \mathfrak{H}$, we can express the relation $(\ref{FI2-CM})$ in terms of differential forms as 
\begin{align}\label{FI-CM3} - 2 \overline{\partial} \left(  E_{L_0}'(\tau, 0; 1) d \tau \right) &= E_{L_0}(\tau, 0; -1) d \mu(\tau). \end{align}
More generally, we have the following useful description of the operator $L_l$.

\begin{lemma}\label{diff} 

The Maass weight-lowering operator $L_l$ can be described in terms of differential forms as
\begin{align*} \overline{\partial}(f d \tau) &= - v^{2-l} \overline{\xi_l(f)} d \mu(\tau) = - L_l f d \mu(\tau).\end{align*} \end{lemma}

\begin{proof} See \cite[Lemma 2.5]{Eh} and \cite[Lemma 2.3]{BY}. \end{proof}

Let us now consider the Fourier series expansion of the Eisenstein series $E_{L_0}(\tau, s; 1)$, which we write as
\begin{align*} E_{L_0}(\tau, s; 1) &= \sum\limits_{\mu \in L_0^{\vee}/L_0} \sum\limits_{ m \in {\bf{Q}} } 
A_{L_0}(s, \mu, m, v) e(m \tau) {\bf{1}}_{\mu}. \end{align*}
Following the discussion of Kudla \cite[Theorem 2.12]{KuBF} (cf.~\cite[$\S$2.2]{BY}), and using the fact that $E_{L_0}(\tau, 0; 1) =0$ 
by $(\ref{FI-CM})$, we compute the Laurent series expansions of each of the coefficients $A_{L_0}(s, \mu, m, v)$ around $s=0$ as 
\begin{align*} A_{L_0}(s, \mu, m, v) &= b_{L_0}(\mu, m, v) s + O(s^2). \end{align*}
We deduce from this that $E_{L_0}'(\tau, 0; 1)$ has the Fourier series expansion 
\begin{align*} E_{L_0}'(\tau, 0; 1) &= \sum\limits_{\mu \in L_0^{\vee}/L_0} \sum\limits_{m \in {\bf{Q}}} b_{L_0}(\mu, m, v) e(m \tau) {\bf{1}}_{\mu}. \end{align*}
Viewing $E_{L_0}'(\tau, 0; 1) = E_{L_0}^{\prime +}(\tau, 0; 1) + E_{L_0}^{\prime -}(\tau, 0; 1) \in H_1(\omega_{L_0}^{\vee})$ 
as a harmonic weak Maass form of weight $1$ and representation $\omega_{L_0}^{\vee}$,
we also use the general calculation of Kudla \cite[Theorem 2.12]{KuBF} to compute the Fourier series expansion
of the principal/holomorphic part $\mathcal{E}_{L_0}(\tau) := E_{L_0}^{\prime +}(\tau, 0; 1)$ as 
\begin{align}\label{E+CM}\mathcal{E}_{L_0}(\tau) = E_{L_0}^{\prime +}(\tau, 0; 1) 
&= \sum\limits_{\mu \in L_0^{\vee}/L_0} \sum_{m \in {\bf{Q}}} \kappa_{L_0}(\mu, m) e(m \tau) {\bf{1}}_{\mu},\end{align}
where the coefficients are given explicitly by the convergent limits 
\begin{align*} \kappa_{L_0}(\mu, m) &= \begin{cases} \lim_{v \rightarrow \infty} b_{L_0}(\mu, m, v) &\text{ if $\mu \neq 0$ or $m \neq 0$} \\
\lim_{v \rightarrow \infty} b_{L_0}(0, 0, v) - \log(v) &\text{ if $\mu=0$ and $m=0$}. \end{cases} \end{align*}

Let us now specialize the setting we consider below, where the negative definite space $(L_0, Q_0)$
is incoherent in that it is constructed from an ideal $L_0 = \mathfrak{a} \subset \mathcal{O}_k$
in an imaginary quadratic field $k = k(V_0)$, with its positive definite norm form 
$Q_{\mathfrak{a}}(\cdot) := {\bf{N}}_{k/{\bf{Q}}}(\cdot)/{\bf{N}} \mathfrak{a}$, but we take $(V_0, Q_0) = (\mathfrak{a}_{\bf{Q}}, - Q_{\mathfrak{a}})$ 
to get a negative definite space of signature $(0,2)$. This construction amounts to taking the positive definite quadratic space 
$(\mathfrak{a}_{\bf{Q}}, Q_{\mathfrak{a}})$ at all of the finite places, but then switching invariants at the real place in replacing by $(\mathfrak{a}_{{\bf{Q}}}, - Q_{\mathfrak{a}})$. 

\begin{proposition}\label{LFE-CM}

Suppose $(L_0, Q_0) = (\mathfrak{a}, -Q_{\mathfrak{a}})$ for $\mathfrak{a} \subset \mathcal{O}_k$ a nonzero integral 
ideal of an imaginary quadratic field $k = k(V_0)$ of discriminant $d_k$ and odd quadratic Dirichlet character $\eta_k(\cdot) = \left( \frac{d_k}{\cdot} \right)$,
with $Q_{\mathfrak{a}}(\cdot) = {\bf{N}}_{k/{\bf{Q}}}(\cdot)/{\bf{N}} \mathfrak{a}$ the corresponding positive definite norm form. 
Let $E_{L_0}(\tau; s; l)$ denote either of the corresponding Eisenstein series of weight $l=\lbrace 1, -1\rbrace$ defined in $(\ref{E-CM})$ and $(\ref{E-classical})$ above. Writing
\begin{align*} \Lambda(s, \eta_k) 
&= \vert d_k \vert^{\frac{s}{2}} \Gamma_{\bf{R}} (s+1) L(s, \eta_k), \quad \Gamma_{\bf{R}}(s) := \pi^{-\frac{s}{2}} \Gamma(s)\end{align*}
to denote the completed Dirichlet $L$-function $L(s, \eta_k)$ of the character $\eta_k$, the completed Eisenstein series
\begin{align*} E_{L_0}^{\star}(\tau, s; l) &:= \Lambda(s+1, \eta_k) E_{L_0}(\tau, s; l) = \Lambda(s+1, \eta_k) E_{L_0}(\tau, s; l)\end{align*}
satisfies the odd, symmetric functional equation 
\begin{align*} E_{L_0}^{\star}(\tau,s: l) &= -E_{L_0}^{\star}(\tau, -s; l). \end{align*}

\end{proposition}

\begin{proof} See \cite[Proposition 2.5]{BY} for the case of $l=-1$. The functional equation is deduced from the Langlands functional equation of 
each of the constituent incoherent Eisenstein series $E_{L_0}(g_{\tau}, s; \phi_{\infty}^1 \otimes \lambda_f({\bf{1}}_{\mu}))$,
where switching invariants at infinity, as described above, leads to switching the corresponding archimedean local sign to $-1$. 
We deduce the corresponding case of weight $l=1$ from the functional identity $(\ref{FI-CM})$ above. \end{proof}

\subsubsection{The geodesic case} 

We now fix a subspace $(L_W, Q_W)$ of signature $(1,1)$.
Consider for each $l \in {\bf{Z}}$ the $\mathfrak{S}_{L_W}$-valued Eisenstein 
series on $\tau = u + iv \in \mathfrak{H}$ and $s \in {\bf{C}}$ defined by 
\begin{align}\label{E-geo} E_{L_W}(\tau, s; l) &:= \sum\limits_{ \mu \in L_W^{\vee} / L_W } 
E_{L_W}(g_{\tau}, s; \phi_{\infty}^l \otimes \lambda_f({\bf{1}}_{\mu})) {\bf{1}}_{\mu}. \end{align}
We can then describe Theorem \ref{SW} in terms of the Siegel theta series $(\ref{Stheta3})$ as 
\begin{align}\label{SW-geo} \int\limits_{  \operatorname{SO}(W)({\bf{Q}}) \backslash  \operatorname{SO}(W)({\bf{A}}) }
\theta_{L_W}(\tau, z_W, h_f) dh &= E_{L_W}(\tau, 0; 0). \end{align}
Here, we fix a line $z_W \in D^{\pm}(W)$, and take the Tamagawa Haar measure on $\operatorname{SO}(V)({\bf{A}})$ (as described later). 

We now consider the quadratic space $(L_W, Q_W)$ given by a nonzero integer ideal $L_W = \mathfrak{a} \subset \mathcal{O}_k$ in a real quadratic field $k = k(W)$, 
with $Q_W$ given by the norm form $Q_W(\cdot) = Q_{\mathfrak{a}}(\cdot) = {\bf{N}}_{k/{\bf{Q}}}(\cdot)/{\bf{N}} \mathfrak{a}$.
Unlike in the CM setup, where the negative definite space was obtained by multiplying the positive definite norm form 
$Q_{\mathfrak{a}}$ by $-1$ (which gives an incoherent space), these quadratic spaces $(L_W, Q_W) = (\mathfrak{a}, Q_{\mathfrak{a}})$ are coherent. 

\begin{proposition}\label{LFE-geo}

Suppose $(L_W, Q_W) = (\mathfrak{a}, Q_{\mathfrak{a}})$ for $\mathfrak{a} \subset \mathcal{O}_k$ a nonzero
integral ideal of a real quadratic field $k = k(W)$ of discriminant $d_k$ and even quadratic Dirichlet character $\eta_k(\cdot) = \left(\frac{d_k}{\cdot} \right)$,
with $Q_{\mathfrak{a}}(\cdot) = {\bf{N}}_{k/{\bf{Q}}}(\cdot)/{\bf{N}} \mathfrak{a}$ the corresponding indefinite norm form. 
Let $E_{L_W}(\tau, s) = E_{L_W}(\tau; s; 0)$ denote the corresponding Eisenstein series of weight $l=0$ defined in $(\ref{E-geo})$ 
and $(\ref{E-classical})$ above. Writing
\begin{align*} \Lambda(s, \eta_k) 
&= \vert d_k \vert^{\frac{s}{2}} \Gamma_{\bf{R}} (s) L(s, \eta_k), \quad \Gamma_{\bf{R}}(s) := \pi^{-\frac{s}{2}} \Gamma(s)\end{align*}
to denote the completed Dirichlet $L$-function $L(s, \eta_k)$ of the character $\eta_k$, the completed Eisenstein series
\begin{align*} E_{L_W}^{\star}(\tau, s) &:= \Lambda(s+1, \eta_k) E_{L_W}(\tau, s) = \Lambda(s+1, \eta_k) E_{L_W}(\tau, s; 0)\end{align*}
satisfies the even, symmetric functional equation $E_{L_W}^{\star}(\tau,s) = E_{L_W}^{\star}(\tau, -s)$.

\end{proposition}

\begin{proof} See \cite[Proposition 4.9]{VO}. This can be deduced simply from the Langlands functional equation for the 
Eisenstein series on the right-hand side of  $(\ref{E-geo})$ for the coherent quadratic space $(L_W, Q_W) = (\mathfrak{a}, Q_{\mathfrak{a}})$. \end{proof}

Fixing such a coherent choice of Lorentzian quadratic space $(L_W, Q_W)$ henceforth, we consider the images of the 
Eisenstein series $E_{L_W}(\tau, s; l)$ under the Maass operators $(\ref{Maassops})$. Hence by $(\ref{Maassopreln})$, we have 
\begin{align}\label{FI-geo} L_2 E_{L_W}(\tau, s; 2) &= \frac{1}{2} \cdot (s-1)\cdot E_{L_W}(\tau, s; 0). \end{align}
Now, the Eisenstein series $E_{L_W}(\tau, s; 0)$ on the right-hand side of $(\ref{FI-geo})$ is holomorphic at
$s = s_0(W) =0$, and so we can evaluate this relation at $s=0$ to obtain the identity 
\begin{align*} L_2 E_{L_W}(\tau, 0; 2) &= - \frac{1}{2} \cdot E_{L_W}(\tau, 0; 0). \end{align*}
On the other hand, we can differentiate each side of $(\ref{FI-geo})$ with respect to $s$ to obtain the relation
\begin{align*} L_2 E_{L_W}'(\tau, s; 2) &= \frac{1}{2} \cdot (s-1) \cdot E'_{L_W}(\tau, s; 0) + \frac{1}{2} \cdot E_{L_W}(\tau, s;0),\end{align*}
then evaluate this latter relation at $s=0$ to obtain the identity 
\begin{align*} L_2 E_{L_W}'(\tau, 0; 2) &= \frac{1}{2} \cdot E_{L_W}(\tau, 0;0) - \frac{1}{2} \cdot E^{\prime}_{L_W}(\tau, 0; 0), \end{align*}
equivalently 
\begin{align}\label{FI-geo2} 2 L_2 E_{L_W}'(\tau, 0; 2) &= E_{L_W}(\tau, 0;0) -  E^{\prime}_{L_W}(\tau, 0; 0). \end{align}
Recall that we write $k = k(W)$ to denote the real quadratic field associated to the space $L_W$, with discriminant $d_k$,
quadratic Dirichlet character $\eta_k$ and completed Dirichlet $L$-function $\Lambda(s, \eta_k) = d_k^{\frac{s}{2}}\Gamma_{\bf{R}}(s)L(1, \eta_k)$.

\begin{proposition}\label{vanishing} (i) The weight-two Eisenstein series $E_{L_W}(\tau, s; 2)$ is incoherent; its analytic continuation
$E^{\star}_{L_W}(\tau, s; 2) := \Lambda(s+1, \eta_k) E_{L_W}(\tau, s; 2)$ satisties the odd functional equation 
\begin{align*} E^{\star}_{L_W}(\tau, s; 2) = - E_{L_W}^{\star}(\tau, -s;2), \end{align*}
and hence $E_{L_W}(\tau, 0; 2) = 0$.  (ii) We have that $E_{L_W}^{\star \prime}(\tau, 0;0) = \Lambda(1, \eta_k) E'_{L_W}(\tau, 0; 0) + \Lambda'(1, \eta_k) E_{L_W}(\tau, 0; 0) =0$, 
and hence via $(\ref{FI-geo})$ the relation $-2 L_2 E_{L_W}'(\tau, 0; 2) =-  E_{L_W}(\tau, 0; 0)$.
Expressed in terms of differential forms according to Lemma \ref{diff}, we obtain the relation 
\begin{align*} - 2 L_2 E_{L_W}'(\tau, 0; 2) d \mu(\tau) = 2 \overline{\partial} \left( E_{L_W}'(\tau, 0; 2) d\tau \right) &= - E_{L_W}(\tau, 0; 0),\end{align*} 
equivalently 
\begin{align}\label{FI-geo3} E_{L_W}(\tau, 0; 0) d \mu(\tau) 
&= - 2 \overline{\partial} \left( E_{L_W}'(\tau, 0; 2) d \tau\right). \end{align}

 \end{proposition}

\begin{proof} See \cite[Propositions 4.10 and 4.12]{VO}. The first claim is a well-known consequence for Eisenstein series constructed 
from incoherent quadratic spaces, where in this case we switch Hasse invariants at the real place $v=\infty$ (see \cite[Definition 2.1]{KuEis}). 
The second claim is simple to deduce from the analytic properties of $E_{L_W}(\tau, s; 0)$ (Proposition \ref{LFE-geo}).
In particular, using the even functional equation $E_{L_W}^{\star}(\tau, s; 0) = E^{\star}_{L_W}(\tau, -s;0)$, we can compare Taylor series expansions around $s=0$
to deduce that $E_{L_W}^{\star \prime}(\tau, 0;0) =0$. From this, we deduce the completed functional identity 
$2 L_2 E_{L_W}^{\star \prime}(\tau, 0; 2) =  E^{\star}_{L_W}(\tau, 0;0) - E^{\star \prime}_{L_W}(\tau, 0; 0) = E^{\star}_{L_W}(\tau, 0;0)$.
On the other hand, since $E_{L_W}(\tau, 0; 2) = 0$ by (i), we deduce that $E_{L_W}^{\star \prime}(\tau, 0;2) = \Lambda(1, \eta_k)E_{L_W}^{\prime}(\tau, 0; 2)$.
Hence, the completed functional identity $2 \Lambda(1, \eta_k) L_2 E_{L_W}^{\prime}(\tau, 0; 2) = \Lambda(1, \eta_k) E_{L_W}(\tau, 0; 0)$. This implies
the stated functional identity after dividing out by the scalar value $\Lambda(1, \eta_k)$. \end{proof}

Let us now consider the Fourier series expansion of the Eisenstein series $E_{L_W}(\tau, s; 2)$, 
\begin{align*} E_{L_W}(\tau, s; 2) &= \sum\limits_{ \mu \in L_W^{\vee}/L_W } \sum\limits_{ m \in {\bf{Q}} } A_{L_W}(s, \mu, m, v) e(m \tau) {\bf{1}}_{\mu}. \end{align*}
Following \cite[Theorem 2.12]{KuBF}, we write the Laurent series expansions around $s=0$ of the coefficients as 
\begin{align*} A_{L_W}(s, \mu, m, v) &= a_{L_W}(\mu, m, v) + b_{L_W}(\mu, m, v) s + O(s^2),\end{align*}
and deduce that the derivative Eisenstein series $E_{L_W}'(\tau, 0; 2)$ at $s=0$ has the Fourier series expansion 
\begin{align*} E_{L_W}'(\tau, 0; 2) &= \sum\limits_{ \mu \in L_W^{\vee}/L_W } \sum\limits_{ m \in {\bf{Q}} } b_{L_W}(\mu, m, v) e(m \tau) {\bf{1}}_{\mu}. \end{align*}
Viewing this derivative Eisenstein series as a harmonic weak Maass form
\begin{align*} E_{L_W}'(\tau, 0; 2) = E_{L_W}^{\prime ~ + }(\tau, 0; 2) + E_{L_W}^{\prime  -}(\tau, 0; 2) \in H_2(\omega_{L_W}^{\vee}) \end{align*}
of weight $2$ and representation $\omega_{L_W}^{\vee}$, 
we consider the principal/holomorphic part $E_{L_W}(\tau) := E_{L_W}^{\prime ~ +}(\tau, 0;2)$.
Using the argument of \cite[Theorem 2.12]{KuBF} again, we can compute the coefficients in its Fourier series expansion 
\begin{align}\label{E+geo} E_{L_W}(\tau) = E_{L_W}^{\prime ~ +}(\tau, 0;2) 
&= \sum\limits_{\mu \in L_W^{\vee}/L_W} \sum\limits_{ m \in {\bf{Q}} } \kappa_{L_W}(\mu, m) e(m \tau) {\bf{1}}_{\mu} \end{align}
as the convergent limits
\begin{align*} \kappa_{L_W}(\mu, m) &= \begin{cases} \lim_{v \rightarrow \infty} b_{L_W}(\mu, m, v) &\text{ if $\mu \neq 0$ or $m \neq 0$} \\
\lim_{v \rightarrow \infty} b_{L_W}(0, 0, v) - \log(v) &\text{ if $\mu=0$ and $m=0$}. \end{cases} \end{align*}

\subsection{Summation formulae}

Fix $f \in H_{1-\frac{n}{2}}(\omega_L)$. Write $l = 1 - \frac{n}{2}$ for simplicity. 

\subsubsection{Decompositions of theta series}

We first justify how to decompose the Siegel theta series $\theta_L(\tau, z, h)$ for later calculation. 
Let $L_j$ for $j=1,2$ be any pair of even lattices, with corresponding Weil representations 
\begin{align*} \omega_{L_j}: \operatorname{Mp}_2({\bf{Z}}) \longrightarrow \mathfrak{S}_{L_j} \cong {\bf{C}}[L_j^{\vee}/L_j]. \end{align*}
The Weil representation of the direct sum $L_1 \oplus L_2$ is given by the tensor product
 $\omega_{L_1} \otimes \omega_{L_2}$. Given 
\begin{align*} f(\tau) &= \sum\limits_{\mu \in L_1^{\vee}/L_1} f_{\mu}(\tau) {\bf{1}}_{\mu} \in H_{l_1}(\omega_{L_1}) 
\quad \text{and} \quad g(\tau) = \sum\limits_{\nu \in L_2^{\vee}/L_2} g_{\nu}(\tau) {\bf{1}}_{\nu} \in H_{l_2}(\omega_{L_2}) \end{align*}
harmonic weak Maass forms of weights $l_j$ and representations $\omega_{L_j}$, the corresponding tensor product 
\begin{align*} f(\tau) \otimes g(\tau) &= \sum\limits_{ \mu \in L_1^{\vee}/ L_1 \atop \nu \in L_2^{\vee}/ L_2} f_{\mu}(\tau) g_{\nu}(\tau) {\bf{1}}_{\mu + \nu}
\in H_{l_1 + l_2}(\omega_{L_1 \oplus L_2}) = H_{l_1 + l_2}(\omega_{L_1} \otimes \omega_{L_2}) \end{align*}
determines a harmonic weak Maass form of weight $l_1 + l_2$ and representation 
$\omega_{L_1 \oplus L_2} = \omega_{L_1} \otimes \omega_{L_2}$.

Suppose now that $M \subset L$ is any sublattice of finite index. Observe that we have inclusions 
\begin{align*} M \subset L \subset L^{\vee} \subset M^{\vee} \quad \implies \quad
L /M \subset L^{\vee}/M \subset M^{\vee}/M,\end{align*}
and hence an inclusion of spaces $H_l(\omega_L) \subset H_l(\omega_M)$ for any weight $l \in \frac{1}{2} {\bf{Z}}$.
Consider the natural map \begin{align*} L^{\vee}/M \longrightarrow L^{\vee}/L, \quad \mu \longmapsto \overline{\mu}. \end{align*}

\begin{lemma}\label{lattice} 

Let $M \subset L$ be any sublattice of finite index. We have natural restrictions and trace maps 
\begin{align*} \operatorname{res}_{L/M} : H_l(\omega_L) \longrightarrow H_l(\omega_M), 
\quad f(\tau) = \sum\limits_{\overline{\mu} \in L^{\vee}/L} f_{\overline{\mu}}(\tau) {\bf{1}}_{\overline{\mu}} 
\longmapsto f_M (\tau) = \sum\limits_{\mu \in M^{\vee}/M} f_{M, \mu} (\tau) {\bf{1}}_{\mu} \end{align*} 
and 
\begin{align*} \operatorname{tr}_{L/M} : H_l(\omega_M) \longrightarrow H_l(\omega_L), 
\quad g (\tau) = \sum\limits_{\mu \in M^{\vee}/M} g_{\mu} (\tau) {\bf{1}}_{\mu} 
\longmapsto g^L (\tau) = \sum\limits_{\overline{\mu} \in L^{\vee}/L} g^L_{\overline{\mu}}(\tau) {\bf{1}}_{\overline{\mu}} \end{align*}
such that for any pair of vector-valued forms $f \in H_l(\omega_L)$ and $g \in H_l(\omega_M)$, we have 
\begin{align*} \langle \langle f(\tau), \overline{g}^L(\tau) \rangle \rangle &= \langle \langle f_M(\tau), \overline{g}(\tau) \rangle \rangle. \end{align*}
Explicitly, the restriction map is given for any $\mu \in M^{\vee}/M$ and $f \in H_l(\omega_L)$ by
\begin{align*} f_{M, \mu}(\tau) &= \begin{cases} f_{\overline{\mu}}(\tau)
&\text{ if $\mu \in L^{\vee}/M$} \\ 0 & \text{ if $\mu \notin L^{\vee}/M$} \end{cases} \end{align*}
The trace map is given for any $\overline{\mu} \in L^{\vee}/L$ with fixed preimage $\mu \in L^{\vee}/M$ and $g \in H_l(\omega_M)$ by 
\begin{align*} g^L_{\overline{\mu}}(\tau) &= \sum\limits_{\nu \in L/M} g_{\nu + \mu}(\tau). \end{align*} \end{lemma}

\begin{proof} See \cite[Lemma 3.1]{BY}. \end{proof}

As explained in \cite[Remark 3.2]{BY}, we have for the Siegel theta series we consider the relation 
\begin{align}\label{relation} \theta_L &= (\theta_M)^L. \end{align}
We shall use this relation $(\ref{relation})$ for the finite-index subgroups $M = L_0 \oplus L_0^{\perp} \subset L$ and $M = L_W \oplus L_W^{\perp} \subset L$. 
In particular, we obtain from $(\ref{relation})$ and Lemma $(\ref{lattice})$ the corresponding relations 
\begin{equation}\begin{aligned}\label{IPrelations} 
\langle \langle f, \theta_L \rangle \rangle &= \langle \langle f, (\theta_{L_0 \oplus L_0^{\perp}})^L \rangle \rangle 
= \langle \langle f_{L_0 \oplus L_0^{\perp}}, \theta_{L_0 \oplus L_0^{\perp}} \rangle \rangle
= \langle \langle f_{L_0 \oplus L_0^{\perp}}, \theta_{L_0} \otimes \theta_{L_0^{\perp}} \rangle \rangle \\
\langle \langle f, \theta_L \rangle \rangle &= \langle \langle f, (\theta_{L_W \oplus L_W^{\perp}})^L \rangle \rangle 
= \langle \langle f_{L_W \oplus L_W^{\perp}}, \theta_{L_W \oplus L_W^{\perp}} \rangle \rangle
= \langle \langle f_{L_W \oplus L_W^{\perp}}, \theta_{L_W} \otimes \theta_{L_W^{\perp}} \rangle \rangle. \end{aligned}\end{equation}
We shall take these relations $(\ref{IPrelations})$ for granted in what follows, 
and drop various subscripts and bars from the notations for simpler reading. That is, we shall simply write 
$\langle \langle f, \theta_{L_0} \otimes \theta_{L_0^{\perp}} \rangle \rangle = \langle \langle f, \theta_{L_W} \otimes \theta_{L_W^{\perp}} \rangle \rangle$
to denote the right-hand side(s) $\langle \langle f_{L_0 \oplus L_0^{\perp}}, \theta_{L_0} \otimes \theta_{L_0^{\perp}} \rangle \rangle 
= \langle \langle f_{L_W \oplus L_W^{\perp}}, \theta_{L_W} \otimes \theta_{L_W^{\perp}} \rangle \rangle$ of $(\ref{IPrelations})$ from now on.  

\subsubsection{Measure normalizations and preliminary calculations}

We first relate our sums to the integrals appearing in the Siegel-Weil formula (Theorem \ref{SW}), $(\ref{SW-CM})$ and $(\ref{SW-geo})$. 
Recall we consider the CM cycle $Z(V_0)$ on $X_K$ with complex points as described in $(\ref{CMcycle})$,
as well as the geodesic cycles $\mathcal{G}(W)$ as described in $(\ref{geodesic})$. Let 
\begin{equation*}\begin{aligned} T_0 &= T(V_0) := \operatorname{GSpin}(V_0) \cong \operatorname{Res}_{k(V_0)/{\bf{Q}}} {\bf{G}}_m 
\quad \quad \text{ $[k(V_0): {\bf{Q}}]=2$ imaginary quadratic} \\
T_W &= T(W) :=  \operatorname{GSpin}(W) \cong \operatorname{Res}_{k(W)/{\bf{Q}}} {\bf{G}}_m 
\quad \quad \text{ $[k(W):{\bf{Q}}] =2$ real quadratic} \end{aligned}\end{equation*}
denote the corresponding maximal tori in $\operatorname{GSpin}(V)$.
We again write $U$ to denote either of these subspaces $V_0, W \subset V$, 
with $T_U = \operatorname{GSpin}(U) = \operatorname{Res}_{k(U)/{\bf{Q}}} {\bf{G}}_m$, with $k=k(U)$ the corresponding quadratic field. 
Let $w_k = \# \mu(k)$ denote the number of roots of unity in $k$.
Hence, by Dirichlet's unit theorem, we know that $\# \mathcal{O}_k = w_k$ when $k = k(U) = k(V_0)$ is an imaginary quadratic field. 
When $k = k(U)= k(W)$ is a real quadratic field, $\mathcal{O}_k^{\times} = \langle \varepsilon_k \rangle \times \mu(k)$ for $\varepsilon_k$ the fundamental unit, 
so the solution $\varepsilon_k = \frac{1}{2}(t + u \sqrt{d_k})$ with $u$ minimal to Pell's equation $t^2 - d_k u^2 = 4$. 
Taking adelic modulo in $(\ref{spinSES})$ for $T_U = \operatorname{GSpin}(U)$, we obtain the Hilbert exact sequence for $k=k(U)$:
\begin{align}\label{HilbertSES} &1 \longrightarrow {\bf{A}}^{\times} \longrightarrow {\bf{A}}_{k}^{\times} \longrightarrow {\bf{A}}_{k}^{1} \longrightarrow 1. \end{align}

We fix the Tamagawa Haar measure on $\operatorname{SO}(U)({\bf{A}}) \cong {\bf{A}}_k^1$, 
so that $\operatorname{vol}(\operatorname{SO}(U)({\bf{Q}}) \backslash \operatorname{SO}(U)({\bf{A}}_f))=2$.
Hence, we normalize the measure so that $\operatorname{vol}(\operatorname{SO}(U)({\bf{R}})) = 1$ if $U=V_0$ is negative definite, and otherwise take 
the Haar measure to be the multiplicative Lebesgue measure $\frac{dx}{x}$ on $k_{\infty}^1 = \lbrace (t, t^{-1}): t \in k_{\infty}, t>0 \rbrace \cong {\bf{R}}_{>0}$ if $U = W$ is Lorentzian. 
We fix the standard Haar measure on ${\bf{A}}^{\times}$ with $\operatorname{vol}({\bf{Z}}_p^{\times}) = 1$ for each prime $p$ so that 
$\operatorname{vol}(\widehat{\bf{Z}})=1$, as well as $\operatorname{vol}({\bf{Q}}^{\times} \backslash {\bf{A}}^{\times})=1$,
and $\operatorname{vol}({\bf{A}}_f^{\times}/{\bf{Q}}^{\times}) = \frac{1}{2}$. 
This determines a measure on $T_U({\bf{A}}) \cong {\bf{A}}_{k(U)}^{\times}$ via the exact sequence $(\ref{HilbertSES})$, with 
\begin{align*} \operatorname{vol} ( k^{\times} \backslash {\bf{A}}_{k}^{\times} )
&=  \operatorname{vol} (  {\bf{Q}}^{\times} \backslash {\bf{A}}^{\times} ) \cdot 
\operatorname{vol} (  k^1 \backslash {\bf{A}}_{k}^1 ) = 1 \cdot 
\operatorname{vol} \left( \operatorname{SO}(U)({\bf{Q}}) \backslash \operatorname{SO}(U)({\bf{A}}) \right) =2. \end{align*}
If $k = k(U) = k(V_0)$ is imaginary quadratic with $V_0$ negative definite, then we also have that 
\begin{align*} \operatorname{vol} ( k^{\times} \backslash {\bf{A}}_{k, f}^{\times} )
&=  \operatorname{vol} (  {\bf{Q}}^{\times} \backslash {\bf{A}}_f^{\times} ) \cdot 
\operatorname{vol} (  k^1 \backslash {\bf{A}}_{k, f}^1 ) = \frac{1}{2} \cdot 
\operatorname{vol} \left( \operatorname{SO}(U)({\bf{Q}}) \backslash \operatorname{SO}(U)({\bf{A}}_f) \right) =1 \end{align*}
and hence 
\begin{equation}\begin{aligned}\label{measures} 1 &= \int\limits_{ k^{\times} \backslash {\bf{A}}_{k, f}^{\times} } d^{\times} x = 
\int\limits_{ k^{\times} \backslash {\bf{A}}_{k, f}^{\times}/ \widehat{\mathcal{O}}_k^{\times} } 
\int\limits_{ \mathcal{O}_k^{\times} \backslash \widehat{\mathcal{O}}_k^{\times} } d^{\times} x 
= \frac{h_k}{w_k} \cdot \operatorname{vol} (\widehat{\mathcal{O}}_k^{\times}) \quad \implies \quad
\operatorname{vol}(\widehat{\mathcal{O}}_k^{\times}) &= \frac{w_k}{h_k} . \end{aligned}\end{equation} 

Let us now turn to the sums we wish to compute, which we denote by 
\begin{equation}\begin{aligned}\label{Asums} 
\Phi(f, Z(V_0)) &:= \sum\limits_{ (z_0, h) \in Z(V_0)({\bf{C}}) } \frac{\Phi(f, z_0, h)}{\# \operatorname{Aut}(z_0, h)} \\
\Phi(f, \mathcal{G}(W)) &:= \sum\limits_{ [h] \in \mathfrak{G}(W) \atop h \in \operatorname{GSpin}(W)({\bf{A}}_f) } 
\frac{1}{\# \operatorname{Aut}(h)} \int\limits_{C_h = \Gamma_h \backslash D^{\pm}(W)}  \Phi(f, z_W, h) d \nu(z_W).
\end{aligned}\end{equation}
Note that we have two orientations $z_0^{\pm} \in D(V_0)$ and $z_W^{\pm} \in D(W)$ in each case. We fix one choice throughout.

\begin{lemma}\label{schofer}

We have the following expressions for $(\ref{Asums})$ as integrals over adelic quotients of $\operatorname{SO}(U)$.

\begin{itemize}

\item[(i)] If $V_0$ is a rational quadratic space of signature $(0,2)$, then  
 \begin{equation*}\begin{aligned}
\Phi(f, Z(V_0)) &= \frac{2}{\operatorname{vol}(K_0)} 
\int\limits_{ h \in \operatorname{SO}(V_0)({\bf{Q}}) \backslash \operatorname{SO}(V_0)({\bf{A}}_f)  } \Phi(f, z_0, h) dh \\
&= \frac{\operatorname{deg}(Z(V_0))}{2} 
\int\limits_{ h \in \operatorname{SO}(V_0)({\bf{Q}}) \backslash \operatorname{SO}(V_0)({\bf{A}}_f)  } \Phi(f, z_0, h) dh, \end{aligned}\end{equation*}
where \begin{align*}\operatorname{deg}(Z(V_0)) &= \frac{4}{ \operatorname{vol}(K_0)}. \end{align*}

\item[(ii)] If $W$ is a rational quadratic space of signature $(1,1)$, then 
\begin{align*} \Phi(f, \mathcal{G}(W)) 
&= \frac{1}{ \operatorname{vol}(K_W)} \int\limits_{ (z_W, h) \in \operatorname{SO}(W)({\bf{Q}}) \backslash \operatorname{SO}(W)({\bf{A}})  } \Phi(f, z_W, h) dh. \end{align*}

\end{itemize}

\end{lemma}

\begin{proof} See \cite[Lemma 2.13]{Sch} with \cite[Lemma 4.4]{BY} for (i). We can thus assume $U =W$ is Lorentzian for (ii).
Let $B(h)$ denote the function of $h  = h_{\infty} h_f \in T_W({\bf{A}})$ defined by $\Phi(f, h_{\infty} z, h_f)$. Note that this function 
is both left $\operatorname{SO}(W)({\bf{Q}})$-invariant and right $K_W$-invariant. We have the relation 
\begin{align*} \int\limits_{ \operatorname{SO}(W)({\bf{Q}}) \backslash \operatorname{SO}(W)({\bf{A}})  } B(h)dh &= 
\operatorname{vol}(K_W) \sum\limits_{  h \in T_W({\bf{Q}}) \backslash T_W({\bf{A}}_f)/ K_W  } \frac{1}{\#(\Gamma_h)_{\operatorname{tors}}} 
\int\limits_{C_h = \Gamma_h \backslash D^{\pm}(W)} B(h) d \nu(z). \end{align*} 
Indeed, fix a set of idele representatives $h$ for the finite set $T_W({\bf{Q}}) \backslash T_W({\bf{A}}_f)/K_W$, and 
consider the projection $\pi: \operatorname{GSpin}(W) \rightarrow \operatorname{SO}(W)$. 
We partition $\operatorname{SO}(W)({\bf{Q}}) \backslash \operatorname{SO}(W)({\bf{A}}) \cong k^1 \backslash {\bf{A}}_k^1 \cong {\bf{A}}_k^{\times} / k^{\times} {\bf{A}}^{\times}$
into disjoint pieces $\operatorname{SO}(W)({\bf{Q}}) \backslash \operatorname{SO}(W)({\bf{Q}}) \pi(h) K_W \pi(h^{-1})$,
then pull back to $T_W({\bf{A}}_f)  \cong {\bf{A}}_{k, f}^{\times} $. Since $B(h)$ is left $\operatorname{SO}(W)({\bf{Q}})$-invariant
and right $K_W$-invariant, with each piece having measure $\operatorname{vol}(K_W)/\#(\Gamma_h)_{\operatorname{tors}}$, we obtain the identity. 
Using that $\#(\Gamma_h)_{\operatorname{tors}} = \# \lbrace \pm 1 \rbrace =2$, we deduce the claimed identity. \end{proof}

We have the following more convenient expression for the regularized integrals defining the sums $(\ref{Asums})$.

\begin{proposition}\label{kudla}

We have the following expressions for the regularized theta integral $\Phi(z, h)$ as limits of truncated sums of integrals. 
Here, we take for granted the relation of scalar products $(\ref{IPrelations})$.\\

\begin{itemize}

\item[(i)] In the CM case with negative definite lattice $L_0 \subset L$, we have for any $(z_0, h) \in D(V_0) \times T_0({\bf{A}}_f)$ that  
\begin{align*} \Phi(f, z_0, h) &= \lim_{T \rightarrow \infty} \left[ \int_{\mathcal{F}_T} \langle \langle f(\tau), 
\theta_{L_0^{\perp}}(\tau) \otimes \theta_{L_0}(\tau, z_0, h) \rangle \rangle d \mu(\tau) - A_0 \log(T) \right], \end{align*} where 
\begin{align*} A_0 &= \operatorname{CT} \langle \langle f ^+(\tau), 
\theta_{L_0^{\perp}}(\tau) \otimes {\bf{1}}_{0 + L_0} \rangle \rangle.\end{align*}
Here, we write $\theta_{L_0^{\perp}}(\tau) = \theta_{L_0^{\perp}}(\tau, 1, 1)$, and note that the underlying theta
series $\theta_{L_0^{\perp}}(\tau, z_0, h)$ for the positive definite lattice $L_0^{\perp}$ of signature $(n, 0)$ is holomorphic in the variable $\tau \in \mathfrak{H}$. \\

\item[(ii)] In the case with the signature $(1,1)$ lattice $L_W \subset L$, we have for any $(z_W, h) \in D(W) \times T_0({\bf{A}}_f)$ that  
\begin{align*} \Phi(f, z_W, h) &= \lim_{T \rightarrow \infty} \left[ \int_{\mathcal{F}_T} \langle \langle f(\tau), 
\theta_{L_W^{\perp}}(\tau) \otimes \theta_{L_W}(\tau, z_W, h) \rangle \rangle d \mu(\tau) - A_0 \log(T) \right], \end{align*} where 
\begin{align*} A_0 &=\sum\limits_{m \in M^{\vee}/M} \sum\limits_{x \in U^{\perp}({\bf{Q}})} c_{f}^+(\mu, -Q(x)) .\end{align*} \end{itemize}

\end{proposition}

\begin{proof} See Kudla \cite[Proposition 2.5]{KuBF}. 
The same argument works in each case. See \cite[Lemma 4.5]{BY} for (i) and \cite[Lemma 4.14]{VO} for (ii). 
\end{proof}

\begin{corollary}\label{preliminary}

We have the following preliminary expressions for the sums $(\ref{Asums})$.\\

\begin{itemize}

\item[(i)] In the CM case with the negative definite lattice $L_0 \subset L$ of signature $(0,2)$, we have 
\begin{align*} \Phi(f, Z(V_0)) &= \lim_{T \rightarrow \infty} \left[
 \frac{1}{\operatorname{vol}(K_0)} \int_{\mathcal{F}_T} \langle \langle f(\tau), 
\theta_{L_0^{\perp}}(\tau) \otimes E_{L_0}(\tau, 0, -1) \rangle \rangle d \mu(\tau) - A_0 \log(T) \right]. \end{align*}  

\item[(ii)] In the geodesic case with the Lorentzian lattice $L_W \subset L$ of signature $(1,1)$, we have 
\begin{align*} \Phi(f, \mathcal{G}(W)) &= \lim_{T \rightarrow \infty} \left[  \frac{1}{\operatorname{vol}(K_W)} \int_{\mathcal{F}_T} \langle \langle f(\tau), 
\theta_{L_W^{\perp}}(\tau) \otimes E_{L_W}(\tau, 0; 0) \rangle \rangle d \mu(\tau) - A_0 \log(T) \right].\end{align*}  

\end{itemize}

\end{corollary}

\begin{proof} 

We start with the expressions of Lemma \ref{schofer}, evaluating the regularized theta integrals according to Proposition \ref{kudla}.
Switching the order of summation and applying the Siegel-Weil formula (Theorem \ref{SW}), so $(\ref{SW-CM})$ for (i) and $(\ref{SW-geo})$ for (ii),
we obtain the stated formulae. \end{proof}

\subsubsection{Summation along CM cycles}

We now evaluate $\Phi(f, Z(V_0))$. Let $g(\tau) = g_f(\tau) := \xi_{l}(f) = \xi_{1-\frac{n}{2}}(f)(\tau)$ 
be the holomorphic modular form of weight $2 - l = 1 + \frac{n}{2}$ obtained by applying the antilinear 
differential operator $\xi_l$ to the initial harmonic weak Maass form $f \in H_l(\omega_L)$. 
We consider the Rankin-Selberg $L$-function 
\begin{align*} L(s, g \times \theta_{L_0^{\perp}}) &:= \langle g(\tau), \theta_{L_0^{\perp}}(\tau) \otimes E_{L_0}(\tau, s; 1)  \rangle, \end{align*}
as well as its completion 
\begin{align*} L^{\star}(s, g \times \theta_{L_0^{\perp}}) &:= \Lambda(s+1, \eta_k)
\langle g(\tau), \theta_{L_0^{\perp}}(\tau) \otimes E_{L_0}(\tau, s; 1)  \rangle 
= \langle g(\tau), E^{\star}_{L_0}(\tau, s) \rangle. \end{align*}
Here, we write $k = k(V_0)$ for the imaginary quadratic field attached to the incoherent quadratic space $(V_0, Q_0) = (\mathfrak{a}, -Q_{\mathfrak{a}})$,
with discriminant $d_k$ and character $\eta_k(\cdot) = \left( \frac{d_k}{\cdot} \right)$, and $E_{L_0}(s, \tau; 1)$ for the corresponding (incoherent) Eisenstein series with completion
$E_{L_0}^{\star}(s, \tau) := \Lambda (s+1, \eta_k) E_{L_0}(s, \tau; 1) = - E_{L_0}^{\star}(-s, \tau)$. 
Writing 
\begin{align*} g(\tau) &= \sum\limits_{ \mu \in L^{\vee}/L } \sum\limits_{  m \in {\bf{Q}} \atop  m > 0} 
c_g(\mu, m) e(m \tau) {\bf{1}}_{\mu} \end{align*}
and 
\begin{align*} \theta_{L_0^{\perp}}(\tau) 
&= \sum\limits_{ \mu \in (L_0^{\perp})^{\vee}/L_0^{\perp}  } \sum\limits_{ m \in {\bf{Q}} \atop m >0  } 
c_{\theta_{L_0^{\perp}}}(\mu, m) e(m \tau) {\bf{1}}_{\mu}
= \sum\limits_{ \mu \in (L_0^{\perp})^{\vee}/L_0^{\perp}  } \sum\limits_{ m \in {\bf{Q}} \atop m >0  } 
r_{L_0^{\perp}}(\mu, m) e(m \tau) {\bf{1}}_{\mu} \end{align*}
for the Fourier series expansions of the holomorphic forms $g(\tau) \in S_{2-l}(\omega_L)$
and $\theta_{L_0^{\perp}}(\tau) \in H_{\frac{n}{2}}(\omega_{L_0}^{\vee})$, the $L$-function
$L(s, g \times \theta_{L_0^{\perp}}) = \langle g, \theta_{L_0^{\perp}} \otimes E_{L_0}(\cdot, s; 1) \rangle$ 
has for $\Re(s) \gg 1$ the Dirichlet series expansion
\begin{align}\label{RS-CM} L(s, g \times \theta_{L_0^{\perp}}) 
&= (4 \pi)^{- \left( \frac{s+n}{2} \right)} \Gamma \left( \frac{s+n}{2} \right)
\sum\limits_{ \mu \in (L_0^{\perp})^{\vee}/L_0^{\perp} }\sum\limits_{ m \geq 1 } \overline{c_g(\mu, m)} r_{L_0^{\perp}}(\mu, m) m^{- \left( \frac{s+n}{2} \right)}. \end{align}

\begin{theorem}[Bruinier-Yang]\label{BY4.7} We have that
\begin{align*} \Phi(f, Z(V_0)) &= - \frac{2}{\operatorname{vol}(K_0)} \left(
\operatorname{CT} \langle \langle f^+(\tau), \theta_{L_0^{\perp}}(\tau) \otimes \mathcal{E}_{L_0}(\tau) \rangle \rangle
+ L'(0, g \times \theta_{L_0^{\perp}}) \right) \\
&= - \frac{\operatorname{deg}(Z(V_0))}{2} \left( \operatorname{CT} \langle \langle f^+(\tau), \theta_{L_0^{\perp}}(\tau) \otimes \mathcal{E}_{L_0}(\tau) \rangle \rangle
+ L'(0, g \times \theta_{L_0^{\perp}}) \right). \end{align*} \end{theorem}

\begin{proof}

See \cite[Theorem 4.7]{BY}. As there seems to be at least one sign error in their 
formula\footnote{See also \cite[Theorem 5.7.1]{AGHMP}, where the same sign error 
for the contribution of $L'(0, \xi_{1-n/2}(f) \times \theta_{L_0^{\perp}})$ in \cite[Theorem 4.7]{BY} 
is acknowledged. That is, the integral in the last line of \cite[p.~654]{BY} should be evaluated
using the differential forms identity $\overline{\partial}(f \tau) = - v^{2-l} \overline{\xi_l(f)} d \mu(\tau) = - L_l fd \mu(\tau)$,
and the substitution made implicity for the first identity in \cite[p.~655]{BY} misses the sign change. 
Moreover, the application of Stokes' theorem for the remaining integral does not involve a change of 
sign after identifying the boundary $\partial \mathcal{F}_T$ with the interval $[iT, 1 + iT]$.}, we supply a detailed proof. 
We know from Corollary \ref{preliminary} that  
\begin{align*} \Phi(f, Z(V_0)) &= \lim_{T \rightarrow \infty} \left[ \frac{1}{\operatorname{vol}(K_0)} \cdot I_T(f) - A_0 \log(T) \right], 
\quad I_T(f) := \int\limits_{\mathcal{F}_T} \langle \langle f(\tau), \theta_{L_0^{\perp}}(\tau) \otimes E_{L_0}(\tau, 0; -1) \rangle \rangle  d \mu (\tau) \end{align*}
with
\begin{align*} A_0 &= \operatorname{CT} \langle \langle f ^+(\tau), \theta_{L_0^{\perp}}^+(\tau) \otimes {\bf{1}}_{0 + L_0} \rangle \rangle.\end{align*}
To evaluate the integral $I_T(f)$, we first use the identity $(\ref{FI-CM3})$ and the relation $d = \partial + \overline{\partial}$ to compute 
\begin{align*} I_T(f) &= -2 \int\limits_{\mathcal{F}_T} \langle \langle f(\tau), \theta_{L_0^{\vee}} 
\otimes \overline{\partial} E_{L_0}^{\prime}(\tau, 0; 1) \rangle \rangle d \tau 
= -2 \int\limits_{\mathcal{F}_T} \langle \langle \partial f(\tau), \theta_{L_0^{\perp}} 
\otimes E_{L_0}^{\perp}(\tau, 0; 1) \rangle \rangle d \tau \\
&= -2 \int\limits_{\mathcal{F}_T} d \langle \langle f(\tau), \theta_{L_0^{\perp}} 
\otimes E_{L_0}^{\perp}(\tau, 0; 1) \rangle \rangle d \tau 
+ 2 \int\limits_{\mathcal{F}_T} \langle \langle \overline{\partial} f(\tau), \theta_{L_0^{\perp}} 
\otimes E_{L_0}^{\prime}(\tau, 0; 1) \rangle \rangle d \tau, \end{align*}
which after using Lemma \ref{diff} to compute the second integral in the latter expression becomes 
\begin{align*} I_T(f) &= -2 \int\limits_{\mathcal{F}_T} d \langle \langle f(\tau), \theta_{L_0^{\perp}} 
\otimes E_{L_0}^{\prime}(\tau, 0; 1) \rangle \rangle d \tau 
- 2 \int\limits_{\mathcal{F}_T} \langle \langle \xi_l(f)(\tau), \theta_{L_0^{\perp}} 
\otimes E_{L_0}^{\prime}(\tau, 0; 1) \rangle \rangle v^{2-l} d \mu(\tau), \end{align*}
and which after using Stokes' theorem to evaluate the first integral becomes 
\begin{align*} I_T(f) &= -2 \int\limits_{ \partial \mathcal{F}_T} \langle \langle f(\tau), \theta_{L_0^{\perp}} 
\otimes E_{L_0}^{\prime}(\tau, 0; 1) \rangle \rangle d \tau 
- 2 \int\limits_{\mathcal{F}_T} \langle \langle \xi_l(f)(\tau), \theta_{L_0^{\perp}} 
\otimes E_{L_0}^{\prime}(\tau, 0; 1) \rangle \rangle v^{2-l} d \mu(\tau) \\ 
&= -2 \int\limits_{\tau = iT}^{1 + iT} \langle \langle f(\tau), \theta_{L_0^{\perp}} 
\otimes E_{L_0}^{\prime}(\tau, 0; 1) \rangle \rangle d \tau 
- 2 \int\limits_{\mathcal{F}_T} \langle \langle \xi_{1 - \frac{n}{2}}(f)(\tau), \theta_{L_0^{\perp}} 
\otimes E_{L_0}^{\prime}(\tau, 0; 1) \rangle \rangle v^{1 + \frac{n}{2}} d \mu(\tau). \end{align*}
Inserting this back into the initial formula, we obtain
\begin{align*} \Phi(f, Z(V)_0) 
&= - \frac{2}{\operatorname{vol}(K_0)} \cdot \langle g(\tau), \theta_{L_0^{\perp}}(\tau) \otimes E_{L_0}^{\prime}(\tau, 0; 1) \rangle \\
&- \lim_{T \rightarrow \infty} \left[ \frac{2}{\operatorname{vol}(K_0) }  \int\limits_{\tau = iT}^{1 + iT} \langle \langle f(\tau), \theta_{L_0^{\perp}} 
\otimes E_{L_0}^{\prime}(\tau, 0; 1) \rangle \rangle d \tau  - A_0 \log(T)  \right]. \end{align*} 

To evaluate the limiting term in this latter expression, we first  split the integral into parts according to the decomposition 
$f(\tau) = f^+(\tau) + f^-(\tau)$ of $f(\tau)$ into principal/holomorphic and nonholomorphic parts as 
\begin{align*} &\lim_{T \rightarrow \infty} \int\limits_{\tau = i T}^{1 + i T} \langle \langle f(\tau), \theta_{L_0^{\perp}}(\tau)
\otimes E_{L_0}'(\tau, 0; 1) \rangle \rangle d \tau \\ &= 
\lim_{T \rightarrow \infty} \int\limits_{\tau = i T}^{1 + i T} \langle \langle f^+(\tau), \theta_{L_0^{\perp}}(\tau)
\otimes E_{L_0}'(\tau, 0; 1) \rangle \rangle d \tau 
+ \lim_{T \rightarrow \infty} \int\limits_{\tau = i T}^{1 + i T} \langle \langle f^-(\tau), \theta_{L_0^{\perp}}(\tau) 
\otimes E_{L_0}'(\tau, 0; 1) \rangle \rangle d \tau. \end{align*}
We argue that the second integral on the right of this latter expression vanishes (cf.~\cite[Theorem 3.5]{Eh}).
To be more precise, let us write the Fourier series expansion as 
\begin{align*} \langle \langle f^-(\tau), \theta_{L_0^{\perp}}(\tau) \otimes E_{L_0}'(\tau, 0; 1) \rangle \rangle 
&= \sum\limits_{ m \in {\bf{Z}} } a(m, iv) e(m \tau), \quad \tau = u + iv \in \mathfrak{H}. \end{align*}
Using the orthogonality of additive characters, we find that 
\begin{align*} \int_{\tau = i T}^{1 + i T}\langle \langle f^-(\tau), \theta_{L_0^{\perp}}(\tau) \otimes E_{L_0}'(\tau, 0; 1) \rangle \rangle d \tau
&= \int_0^1 \langle \langle f^-(u + iT), \theta_{L_0^{\perp}}(u + iT) \otimes E_{L_0}'(u + iT, 0; 1) \rangle \rangle du \\
&= \sum\limits_{m \in {\bf{Z}}} a(m, iT) \int_0^1 e(m u) du = a(0, iT). \end{align*}
Here,  \begin{align*} a(0, iT) &= \sum\limits_{ \mu \in L^{\vee}/L }  \sum\limits_{ m \in {\bf{Q}}\atop m>0 } c_f^- (\mu, -m) W_l(- 2 \pi m v) c_F(\mu, m, v) \end{align*}
denotes the constant coefficient of the scalar-valued form 
$\langle \langle f^-(\tau), \theta_{L_0^{\perp}}(\tau) \otimes E_{L_0}'(\tau, 0; 1) \rangle \rangle$, and we write 
\begin{align*} F(\tau) &:= \theta_{L_0^{\perp}}(\tau) \otimes E^{\prime}_{L_0}(\tau, 0; 1) = 
\sum\limits_{ \mu \in (L_0^{\perp} \oplus L_0)^{\vee}/(L_0^{\perp} \oplus L_0) } \sum\limits_{ m \in {\bf{Q}} } c_F(\mu, m, v) e(m \tau) {\bf{1}}_{\mu}. \end{align*}
Recall that the Whittaker coefficients $W_l(y):= \int_{-2y}^{\infty} e^{-t} t^{-l} dt = \Gamma(1-l, 2 \vert y \vert)$ decay rapidly for $y \rightarrow -\infty$.
Using this together with standard bounds for the Fourier coefficients of $f^-(\tau)$ and $F(\tau)$, 
we deduce that for some integer $M>0$ and constant $C>0$ we have for each integer $m \geq M$ the bounds 
\begin{align*} c_f^{-}(\mu, -m) W_l(- 2 \pi m v) c_F(\mu, m, v) &= O(e^{-m C v}). \end{align*}
Hence, via geometric series, we derive the bound 
\begin{align*} a(0, iT) &= O \left(  \frac{e^{-CT}}{ \left( 1 - e^{-CT} \right)  } \right). \end{align*}
It is then apparent that 
\begin{align*} \lim_{T \rightarrow \infty} a(0, iT) &= \lim_{T \rightarrow \infty} 
\int_{\tau = i T}^{1 + i T}\langle \langle f^-(\tau), \theta_{L_0^{\perp}}(\tau) \otimes E_{L_0}'(\tau, 0; 1) \rangle \rangle d \tau = 0. \end{align*}
Hence, it remains to evaluate
\begin{align*}\lim_{T \rightarrow \infty} \left[ \frac{2}{\operatorname{vol}(K_0) } 
\int\limits_{\tau = iT}^{1 + iT} \langle \langle f^+(\tau), \theta_{L_0^{\perp}} 
\otimes E_{L_0}^{\prime}(\tau, 0; 1) \rangle \rangle d \tau  - A_0 \log(T)  \right]. \end{align*}
Here, we use the calculation of coefficients $(\ref{E+CM})$ as in \cite[Theorem 4.7]{BY} to find
\begin{align*} &\lim_{T \rightarrow \infty} \left[ \int\limits_{\tau = iT}^{1 + iT} \langle \langle f^+(\tau), \theta_{L_0^{\perp}} 
\otimes E_{L_0}^{\prime}(\tau, 0; 1) \rangle \rangle d \tau  - A_0 \log(T) \right] \\
&= \lim_{T \rightarrow \infty} \int_{\tau = iT}^{1 + i T} \langle \langle f^+(\tau), \theta_{L_0^{\perp}}(\tau) \otimes
\sum\limits_{\mu \in L_0^{\vee}/L_0} \sum\limits_{m \in {\bf{Q}} } 
\left( b_{L_0}(\mu, m, v) - \delta_{\mu, 0} \delta_{m, 0} \log(v) \right) e(m\tau) {\bf{1}}_{\mu} \rangle \rangle d \tau \\
&= \operatorname{CT} \langle \langle f^+(\tau), \theta_{L_0^{\perp}}(\tau) \otimes \mathcal{E}_{L_0}(\tau) \rangle \rangle. \end{align*} \end{proof}

\subsubsection{Summation along real geodesic cycles}

We now evaluate the sum along the real geodesic cycle $\Phi(f, \mathcal{G}(W))$. Again, we consider $g(\tau) = \xi_{1-\frac{n}{2}}(f)(\tau)$ the 
holomorphic modular form of weight $2 - l = 1 + \frac{n}{2}$ obtained by applying $\xi_l$ to $f \in H_l(\omega_L)$. We consider the Rankin-Selberg $L$-function 
\begin{align*} L(s, g \times \theta_{L_W^{\perp}}) &:= \langle g, \theta_{L_W^{\perp}} \otimes E_{L_W}(\cdot, s; 2)  \rangle, \end{align*}
which has an analytic continuation to all $s \in {\bf{C}}$ given by the completed $L$-function
\begin{align*} L^{\star}(s, g \times \theta_{L_W^{\perp}}) &:= \Lambda(s+1, \eta_k)
\langle g, \theta_{L_W^{\perp}} \otimes E_{L_W}(\cdot, s; 2)  \rangle 
= \langle g, \theta_{L_W^{\perp}} \otimes E^{\star}_{L_W}(\tau, s; 2) \rangle. \end{align*}
As explained in Proposition \ref{vanishing} (i), the Eisenstein series satisfies the odd, symmetric functional equation 
\begin{align*} E^{\star}_{L_W}(\tau, s; 2) := \Lambda (s+1, \eta_k) E_{L_W}(s, \tau; 2) = - E^{\star}_{L_W}(\tau, -s; 2),\end{align*}
from which we find $L^{\star}(s, g \times \theta_{L_W^{\perp}}) = - L^{\star}(-s, g \times \theta_{L_W^{\perp}})$. We also consider the regularized inner product  
\begin{equation*}\begin{aligned} I(s, f \times \xi_{\frac{n-2}{2}}(\theta_{L_W^{\perp}})) &=
\int_{ \mathcal{F} }^{\star} \langle \langle f (\tau), \overline{ \xi_{\frac{n-2}{2}}(\theta_{ L_W^{\perp}   } )(\tau)} \otimes E_{L_W}(\tau, s; 2) \rangle \rangle v^{3 - \frac{n}{2}} d \mu (\tau) \\
&= \lim_{T \rightarrow \infty} \int\limits_{\mathcal{F}_T}  \langle \langle f (\tau), \overline{\xi_{\frac{n-2}{2}}(\theta_{L_W^{\perp}})(\tau)} \otimes E_{L_W}(\tau, s; 2) \rangle \rangle v^{3 - \frac{n}{3}} d \mu(\tau). 
\end{aligned}\end{equation*} 
Here, $\xi_{\frac{n-2}{2}}(\theta_{L_W^{\perp}})$ denotes is the holomorphic ``shadow" theta series of weight $2 - (n-2)/2 = 3-n/2$ associated to $\theta_{L_W^{\perp}}$,
and $I(s, f \times \xi_{\frac{n-2}{2}}(\theta_{L_W^{\perp}}))$ has an analytic continuation $I^{\star}(s, f \times \xi_{\frac{n-2}{2}}(\theta_{L_W^{\perp}}))$ to all $s \in {\bf{C}}$,
and via the appearance of $E^{\star}_{L_W^{\perp}}(\tau, s; 2)$ satisfies the odd, symmetric functional equation 
\begin{align}\label{IFE} I^{\star}(s, f \times \xi_{\frac{n-2}{2}}(\theta_{L_W^{\perp}})) 
&:= \Lambda(s, \eta_k) I(s, f \times \xi_{\frac{n-2}{2}}(\theta_{L_W^{\perp}})) = -I^{\star}(-s, f \times \xi_{\frac{n-2}{2}}(\theta_{L_W^{\perp}})).\end{align}

\begin{theorem}\label{realquad} 

Assume the signature $(n-1, 1)$ space $(W^{\perp}, Q) = (W^{\perp}, Q\vert_{W^{\perp}})$ is anisotropic. We have that
\begin{align*} \Phi(f, \mathcal{G}(W)) &= - \frac{2}{ \operatorname{vol}(K_W)} \left(
\operatorname{CT} \langle \langle f^+(\tau), {\bf{1}}_{L_W^{\perp} + 0} \otimes \mathcal{E}_{L_W}(\tau) \rangle \rangle
+ L'(0, g \times \theta_{L_W^{\perp}}) +I'(0, f \times \xi_{\frac{n-2}{2}}(\theta_{L_W^{\perp}})) \right). \end{align*} 

\end{theorem}

\begin{proof} See \cite[Theorem 4.15]{VO}, which we generalize. We know from Corollary \ref{preliminary} that  
\begin{align*} \Phi(f, \mathcal{G}(W)) &= \lim_{T \rightarrow \infty} \left[ \frac{1}{\operatorname{vol}(K_W)} \cdot I_T(f) - A_0 \log(T) \right], \quad
 I_T(f) := \int\limits_{\mathcal{F}_T} \langle \langle f(\tau), \theta_{L_W^{\perp}}(\tau) \otimes E_{L_W}(\tau, 0; 0) \rangle \rangle  d \mu (\tau). \end{align*}
To evaluate the integral $I_T(f)$, we first use the functional identity $(\ref{FI-geo3})$ and the relation $d = \partial + \overline{\partial}$ to compute 
\begin{equation*}\begin{aligned} &\overline{\partial} \langle \langle f(\tau), \theta_{L_W^{\perp}}(\tau) \otimes  E_{L_W}^{\prime}(\tau, 0; 2) d \tau \rangle \rangle
= d \langle \langle f(\tau), \theta_{L_W^{\perp}}(\tau) \otimes  E_{L_W}^{\prime}(\tau, 0; 2) d \tau \rangle \rangle \\
&=\langle \langle \overline{\partial} f(\tau), \theta_{L_W^{\perp}}(\tau) \otimes E_{L_W}^{\prime}(\tau, 0; 2) d \tau \rangle \rangle
+ \langle \langle f(\tau), \overline{\partial} \theta_{L_W^{\perp}}(\tau) \otimes \overline{\partial} E_{L_W}^{\prime}(\tau, 0; 2) d \tau \rangle \rangle
+ \langle \langle f(\tau),  \theta_{L_W^{\perp}}(\tau) \otimes \overline{\partial} E_{L_W}^{\prime}(\tau, 0; 2) d \tau \rangle \rangle\end{aligned}\end{equation*} 
and hence 
\begin{align*} I_T(f) &= -2 \int\limits_{\mathcal{F}_T} \langle \langle f(\tau), \theta_{L_W^{\perp}} \otimes \overline{\partial} E_{L_W}^{\prime}(\tau, 0; 2) \rangle \rangle d \tau 
= -2 \int\limits_{\mathcal{F}_T} d \langle \langle f(\tau), \theta_{L_W^{\perp}} \otimes E_{L_W}^{\prime}(\tau, 0; 2) \rangle \rangle d \tau \\
&+ 2 \int\limits_{\mathcal{F}_T} \langle \langle \overline{\partial} f(\tau), \theta_{L_W^{\perp}} \otimes E_{L_W}^{\prime}(\tau, 0; 2) \rangle \rangle d \tau 
+ 2 \int\limits_{\mathcal{F}_T} \langle \langle f(\tau), \overline{\partial} \theta_{L_W^{\perp}}(\tau) \otimes E_{L_W}'(\tau, 0; 2) \rangle \rangle d \tau. \end{align*}
Using Lemma \ref{diff} with $l=(2-n)/2$ to compute the second integral in the latter expression, we then obtain 
\begin{align*} I_T(f) &= -2 \int\limits_{\mathcal{F}_T} d \langle \langle f(\tau), \theta_{L_W^{\perp}} \otimes E_{L_W}^{\prime}(\tau, 0; 2) \rangle \rangle d \tau \\
&- 2 \int\limits_{\mathcal{F}_T} \langle \langle L_{\frac{2-n}{2}} f(\tau), \theta_{L_W^{\perp}}(\tau) \otimes E_{L_W}(\tau, 0; 2) \rangle \rangle d \mu (\tau)
- 2 \int\limits_{\mathcal{F}_T} \langle \langle f(\tau), L_{\frac{n-2}{2}}(\theta_{L_W^{\perp}})(\tau) \otimes E_{L_W}(\tau, 0; 2) \rangle \rangle d \mu(\tau) \\
&= -2 \int\limits_{\mathcal{F}_T} d \langle \langle f(\tau), \theta_{L_W^{\perp}} \otimes E_{L_W}^{\prime}(\tau, 0; 2) \rangle \rangle d \tau \\
&-2 \int\limits_{\mathcal{F}_T} \langle \langle \overline{\xi_l(f)(\tau)}, \theta_{L_W^{\perp}}(\tau) \otimes E_{L_W}^{\prime}(\tau, 0; 2) \rangle \rangle v^{2-l} d \mu(\tau) 
-2 \int\limits_{\mathcal{F}_T} \langle \langle f(\tau), \overline{\xi_{-l}(\theta_{L_W^{\perp}})(\tau)} \otimes E_{L_W}^{\prime}(\tau, 0; 2) \rangle \rangle v^{2 +l} d \mu (\tau). \end{align*}
Using Stokes' theorem, we identify the first integral in this latter expression as  
\begin{align*} -2 \int\limits_{\tau = iT}^{1 + iT} \langle \langle f(\tau), \theta_{L_W^{\perp}}(\tau) 
\otimes E_{L_W}^{\prime}(\tau, 0; 2) \rangle \rangle d \tau 
= -2 \int_0^1 \langle \langle f(u + iT), \theta_{L_W^{\perp}}(u + iT) \otimes E_{L_W}'(u + iT) \rangle \rangle du. \end{align*}
Hence, we obtain 
\begin{align*} &I_T(f) = -2 \int_0^1 \langle \langle f(u + iT), \theta_{L_W^{\perp}}(u + iT) \otimes E_{L_W}'(u + iT) \rangle \rangle du \\
&- 2 \int\limits_{\mathcal{F}_T} \langle \langle \overline{\xi_{1 - \frac{n}{2}}(f)(\tau)}, \theta_{L_W^{\perp}} 
\otimes E_{L_W}^{\prime}(\tau, 0; 2) \rangle \rangle v^{1 + \frac{n}{2}} d \mu(\tau)
-2 \int\limits_{\mathcal{F}_T}
\langle \langle f(\tau), \overline{\xi_{\frac{n}{2}-1}(\theta_{L_W^{\perp}})(\tau)} \otimes E_{L_W}^{\prime}(\tau, 0; 2) \rangle \rangle v^{3 - \frac{n}{2}} d \mu (\tau). \end{align*}
Inserting this back into the initial formula, we obtain the preliminary formula
\begin{align*} \Phi(f, \mathcal{G}(W)) 
&= - \frac{2}{\operatorname{vol}(K_W)} \left(  \left\langle g, \theta_{L_W^{\perp}}(\tau) \otimes E_{L_W}^{\prime}(\cdot, 0; 2) \right\rangle
+ I' \left( 0, f \times \xi_{\frac{n-2}{2}} \theta_{L_W^{\perp}} \right) \right) \\
&- \lim_{T \rightarrow \infty} \left[ \frac{2}{\operatorname{vol}(K_W) }  \int\limits_{\tau = iT}^{1 + iT} \langle \langle f(\tau), \theta_{L_W^{\perp}} 
\otimes E_{L_W}^{\prime}(\tau, 0; 2) \rangle \rangle d \tau  - A_0 \log(T)  \right]. \end{align*} 
To evaluate the second limiting term with the constant coefficients, we argue as in \cite[Theorem 4.7]{BY} and \cite[Theorem 4.15]{VO}.
Here, we first argue that only the holomorphic part of $f$ contributes. 
To see this, we split the constant coefficient term in the preliminary expression into parts 
\begin{equation*}\begin{aligned} &\lim_{T \rightarrow \infty} \int_{\tau = i T}^{iT +1} 
\langle \langle f(\tau), \theta_{L_W^{\perp}}(\tau) \otimes E^{\prime}_{L_W}(\tau, 0; 2)\rangle \rangle d \tau \\
&= \lim_{T \rightarrow \infty} \int_{\tau = i T}^{iT +1} 
\langle \langle f^+(\tau), \theta_{L_W^{\perp}}(\tau) \otimes E^{\prime }_{L_W}(\tau, 0; 2)\rangle \rangle d \tau + \lim_{T \rightarrow \infty} 
\int_{\tau = i T}^{iT +1} \langle \langle f^-(\tau), \theta_{L_W^{\perp}}(\tau) \otimes E^{\prime}_{L_W}(\tau, 0; 2)\rangle \rangle d \tau. \end{aligned}\end{equation*}
We argue in the same way as for the proof of Theorem \ref{BY4.7} that the second integral on the right-hand side vanishes,
as $f^-(\tau)$ is rapidly decreasing as $v \rightarrow \infty$. To evaluate the remaining integral
\begin{equation*}\begin{aligned} \lim_{T \rightarrow \infty} \left[ \frac{2}{\operatorname{vol}(K_W) }  
\int\limits_{\tau = iT}^{1 + iT} \langle \langle f^+(\tau), \theta_{L_W^{\perp}} \otimes E_{L_W}^{\prime}(\tau, 0; 2) \rangle \rangle d \tau  - A_0 \log(T)  \right], \end{aligned}\end{equation*} 
we use $(\ref{E+geo})$ to calculate the remaining contribution explicitly. We first expand 
\begin{align*} \theta_{L_W^{\perp}}(\tau) \otimes E_{L_W}'(\tau, 0; 2) 
&= \sum\limits_{ \lambda_1 \in (L_W^{\perp})^{\vee}/L_W^{\perp} } \sum\limits_{\lambda_2 \in L_W^{\vee}/L_W}
\theta_{L_W^{\perp}, \lambda_1}(\tau) E_{L_W, \lambda_2}'(\tau, 0; 2) {\bf{1}}_{\lambda_1 + \lambda_2} \\
&=  \sum\limits_{ \lambda_1 \in (L_W^{\vee})^{\vee}/L_W^{\vee} } \sum\limits_{\lambda_2 \in L_W^{\vee}/L_W}
\sum\limits_{ x \in W^{\perp}({\bf{Q}}) \atop x \in \lambda_1 + L_W^{\perp} } e \left( Q(x_{z_0^{\perp}}) \tau + Q(x_{z_0}) \overline{\tau} \right) 
\sum\limits_{m \in {\bf{Q}}} b_{L_W}(\lambda_2, m, v) e(m \tau) {\bf{1}}_{\lambda_1 + \lambda_2}, \end{align*}
fixing a basepoint $z_0 \in D(W^{\perp})$ corresponding to the identity splitting $W^{\perp}({\bf{R}}) = z_0 \oplus z_0^{\perp}$. We compute 
\begin{equation*}\begin{aligned} &\int\limits_{\tau = i T}^{iT + 1} \langle \langle f^+(\tau), \theta_{L_W^{\perp}}(\tau) \otimes E_{L_W}'(\tau, 0; 2) \rangle \rangle d \tau
= \int\limits_0^1 \langle \langle f^+(u + iT), \theta_{L_W^{\perp}}(u + iT) \otimes E_{L_W}'(u + iT, 0; 2)  \rangle \rangle du \\
&= \sum\limits_{ \lambda_1 \in (L_W^{\perp})^{\vee} / L_W^{\perp} \atop \lambda_2 \in L_W^{\vee}/L_W  }
\sum\limits_{n \gg - \infty} c_{f}^+(\lambda_1 + \lambda_2, n) e(niT)
\sum\limits_{x \in W^{\perp}({\bf{Q}}) \atop x \in \lambda_1 + L_W^{\perp}} e \left( Q(x_{z_0^{\perp}})iT - Q(x_{z_0})iT \right)
\sum\limits_{m \in {\bf{Q}}} b_{L_W}(\lambda_2, m, v) e(miT) \\
&\times \int\limits_0^1 e\left( nu + Q(x_{z_0^{\perp}})u + Q(x_{z_0})u + \mu \right) du \\ 
&= \sum\limits_{ \lambda_1 \in (L_W^{\perp})^{\vee} / L_W^{\perp} \atop \lambda_2 \in L_W^{\vee}/L_W  }
\sum\limits_{n \gg - \infty} c_{f}^+(\lambda_1 + \lambda_2, n) e(niT)
\sum\limits_{x \in W^{\perp}({\bf{Q}}) \atop x \in \lambda_1 + L_W^{\perp}} e \left( Q(x_{z_0^{\perp}})iT - Q(x_{z_0})iT \right)
\sum\limits_{m \in {\bf{Q}} \atop n + Q(x_{z_0^{\perp}}) + Q(x_{z_0}) + m=0 } b_{L_W}(\lambda_2, m, T) e(miT)\\
&=  \sum\limits_{ \lambda_1 \in (L_W^{\perp})^{\vee} / L_W^{\perp} \atop \lambda_2 \in L_W^{\vee}/L_W  } 
\sum\limits_{x \in W^{\perp} ({\bf{Q}}) \atop x \in \lambda_1 + L_W^{\perp}} 
\sum\limits_{n \gg - \infty, m \in {\bf{Q}} \atop n + Q(x_{z_0^{\perp}}) + m= - Q(x_{z_0}) } 
c_{f}^+(\lambda_1 + \lambda_2, n) e \left( n iT + Q(x_{z_0^{\perp}})iT + miT - Q(x_{z_0})iT \right) b_{L_W}(\lambda_2, m, T)  \\
&= \sum\limits_{ \lambda_1 \in (L_W^{\perp})^{\vee} / L_W^{\perp} \atop \lambda_2 \in L_W^{\vee}/L_W  } 
\sum\limits_{x \in W^{\perp} ({\bf{Q}}) \atop x \in \lambda_1 + L_W^{\perp}} 
\sum\limits_{n \gg - \infty, m \in {\bf{Q}} \atop n + Q(x_{z_0^{\perp}}) + m= - Q(x_{z_0}) } 
c_{f}^+(\lambda_1 + \lambda_2, -Q(x_{z_0^{\perp}}) - Q(x_{z_0}) - m) \\ &\times e \left( -2Q(x_{z_0}) iT \right)
b_{L_W} \left( \lambda_2, -n-Q(x_{z_0^{\perp}}) - Q(x_{z_0}), T \right) \\
&=  \sum\limits_{ \lambda_1 \in (L_W^{\perp})^{\vee} / L_W^{\perp} \atop \lambda_2 \in L_W^{\vee}/L_W  } 
\sum\limits_{x \in W^{\perp}({\bf{Q}}) \atop x \in \lambda_1 + L_W^{\perp}} 
\sum\limits_{n \gg - \infty, m \in {\bf{Q}} \atop n + m = -Q(x)} c_{f}^+(\lambda_1+ \lambda_2, -Q(x) - m) 
b_{L_W}(\lambda_2, -Q(x) - n, T) e \left( - 2Q(x_{z_0}) iT \right) \\
&= \sum\limits_{ \lambda_1 \in (L_W^{\perp})^{\vee} / L_W^{\perp} \atop \lambda_2 \in L_W^{\vee}/L_W  } 
\sum\limits_{x \in W^{\perp}({\bf{Q}}) \atop x \in \lambda_1 + L_W^{\perp}} 
\sum\limits_{m \in {\bf{Q}} \atop -m -Q(x) \gg - \infty} c_{f}^+(\lambda_1+ \lambda_2, -Q(x) - m) 
b_{L_W}(\lambda_2, m, T) e \left( - 2Q(x_{z_0}) iT \right). \end{aligned}\end{equation*} 
Recall that from Proposition \ref{kudla}, we have 
\begin{align*} A_0 &= \sum\limits_{ \lambda_1 + \lambda_2 \in (L_W^{\perp} \oplus L_W)^{\vee}/(L_W^{\perp} \oplus L_W) } 
\sum\limits_{ x \in W^{\perp}({\bf{Q}}) \atop x \in \lambda_1 + L_W^{\perp}} c_{f}^+(\lambda_1 + \lambda_2, -Q(x)). \end{align*}
We then use this to compute the difference as 
\begin{equation}\begin{aligned}\label{lim}& \lim_{T \rightarrow \infty} 
\left[ 2 \int_{\tau = i T}^{iT +1} \langle \langle f^+(\tau), \theta_{L_W^{\perp}}(\tau) \otimes E^{\prime}_{L_W}(\tau, 0; 2)\rangle \rangle d \tau -  A_0 \log(T) \right] \\ 
&=  \lim_{T \rightarrow \infty} 
2 \sum\limits_{ \lambda_1 \in (L_W^{\perp})^{\vee} / L_W^{\perp} \atop \lambda_2 \in L_W^{\vee}/L_W  } 
\sum\limits_{x \in W^{\perp}({\bf{Q}}) \atop x \in \lambda_1 + L_W^{\perp}} 
\left( e \left( -2Q(x_{z_0}) iT \right) \sum\limits_{m \in {\bf{Q}} } c_{f}^+(\lambda_1+ \lambda_2, -Q(x) - m) b_{L_W}(\lambda_2, m, T)  -   c_{f}^+(\lambda_1 + \lambda_2, -Q(x))  \log(T)    \right)  \\ 
&=  \lim_{T \rightarrow \infty}   2 \sum\limits_{ \lambda_1 \in (L_W^{\perp})^{\vee} / L_W^{\perp} \atop \lambda_2 \in L_W^{\vee}/L_W  } 
\sum\limits_{x \in W^{\perp}({\bf{Q}}) \atop x \in \lambda_1 + L_W^{\perp}} 
\left( e^{4 \pi Q(x_{z_0}) T} \sum\limits_{m \in {\bf{Q}} } c_{f}^+(\lambda_1+ \lambda_2, -Q(x) - m) b_{L_W}(\lambda_2, m, T)  
-   c_{f}^+(\lambda_1 + \lambda_2, -Q(x))  \log(T)    \right).\end{aligned}\end{equation}
Now, we can calculate the $x=0$ contributions in this latter expression $(\ref{lim})$ using the coefficient formula $(\ref{E+geo})$  
for the holomorphic part  $\mathcal{E}_{L_W}(\tau) = E_{L_W}'(\tau,0; 2)$ as 
\begin{align*} &\lim_{T \rightarrow \infty} 
2 \sum\limits_{ \lambda_1 \in (L_W^{\perp})^{\vee} / L_W^{\perp} \atop \lambda_2 \in L_W^{\vee}/L_W  } 
\left( \sum\limits_{m \in {\bf{Q}} } c_{f}^+(\lambda_1+ \lambda_2, - m) b_{L_W}(\lambda_2, m, T)  
-   c_{f}^+(\lambda_1 + \lambda_2, 0)  \log(T)    \right)  \\ 
 &= 2 \sum\limits_{ \lambda_1 \in (L_W^{\perp})^{\vee} / L_W^{\perp} \atop \lambda_2 \in L_W^{\vee}/L_W  } 
\sum\limits_{m \in {\bf{Q}} } c_{f}^+(\lambda_1+\lambda_2, -m) \kappa_{L_W}(\lambda_2, m)
= 2 \operatorname{CT} \langle \langle f^+(\tau), {\bf{1}}_{L_W^{\perp} + 0} \otimes \mathcal{E}_{L_W}(\tau) \rangle \rangle. \end{align*}
For the remaining contributions over $x \neq 0$, observe that since 
\begin{align*} z_0 \in D(W^{\perp}) = \lbrace z \subset W^{\perp}({\bf{R}}) : \dim(z)= n-1, Q\vert_{z_0} < 0 \rbrace \end{align*}
is negative, we have $Q(x_{z_0}) <0$. Since $(W^{\perp}, Q) = (W^{\perp}, Q\vert_{W^{\perp}})$ is anisotropic, the remaining contributions 
\begin{equation*}\begin{aligned} &\lim_{T \rightarrow \infty} 
2 \sum\limits_{ \lambda_1 \in (L_W^{\perp})^{\vee} / L_W^{\perp} \atop \lambda_2 \in L_W^{\vee}/L_W  } 
\sum\limits_{x  \neq 0 \in W^{\perp}({\bf{Q}}) \atop x \in \lambda_1 + L_W^{\perp}} 
\left( e^{4 \pi Q(x_{z_0}) T} \sum\limits_{m \in {\bf{Q}} } c_{f}^+(\lambda_1+ \lambda_2, -Q(x) - m) b_{L_W}(\lambda_2, m, T)  
-   c_{f}^+(\lambda_1 + \lambda_2, -Q(x))  \log(T)    \right)  \\
&= \lim_{T \rightarrow \infty } [ 2 \sum\limits_{ \lambda_1 \in (L_W^{\perp})^{\vee} / L_W^{\perp} \atop \lambda_2 \in L_W^{\vee}/L_W  } 
\sum\limits_{x  \neq 0 \in W^{\perp}({\bf{Q}}) \atop x \in \lambda_1 + L_W^{\perp}} 
e^{4 \pi Q(x_{z_0}) T}  c_{f}^+(\lambda_1 + \lambda_2, -Q(x)) \left( b_{L_W}(\lambda_2, 0, T) - \log(T) \right)  \\ 
&+  \lim_{T \rightarrow \infty} 2 \sum\limits_{ \lambda_1 \in (L_W^{\perp})^{\vee} / L_W^{\perp} \atop \lambda_2 \in L_W^{\vee}/L_W  } 
\sum\limits_{x  \neq 0 \in W^{\perp}({\bf{Q}}) \atop x \in \lambda_1 + L_W^{\perp}}  e^{4 \pi Q(x_{z_0}) T} 
\sum\limits_{m \in {\bf{Q}} \atop m \neq 0 } c_{f}^+(\lambda_1+ \lambda_2, -Q(x) - m) b_{L_W}(\lambda_2, m, T) \end{aligned}\end{equation*}
decay exponentially with $T\rightarrow \infty$, and hence to vanish in the limit. This gives the stated formula. \end{proof}

\section{Integral presentations of Rankin-Selberg $L$-functions}

We now explain how to identify the Rankin-Selberg $L$-functions $L^{\star}(s, \xi_{1-n/2}(f) \times \theta_{L_U^{\perp}})$ 
appearing in Theorems \ref{BY4.7} and \ref{realquad} with standard Rankin-Selberg $L$-functions
 for $\operatorname{GL}_2({\bf{A}}) \times \operatorname{GL}_2({\bf{A}})$ when $n=2$.

Let $k$ be any (real or imaginary) quadratic field of discriminant $d_k$ and corresponding Dirichlet character 
$\eta_k(\cdot) = \left( \frac{d_k}{\cdot} \right)$. We consider the ideal class group\footnote{More generally,
we could consider the ring class group $C(\mathcal{O})$ of any order $\mathcal{O} \subset \mathcal{O}_k$ 
for all of the analytic/archimedean discussion here. However, since the discussion of integral models and arithmetic 
heights in \cite{AGHMP} is so far only understood for the maximal order $\mathcal{O}_k$, we stick to this case for simplicity.} 
$C(\mathcal{O}_k)$ of $k$. 
Recall that we fix an integer ideal representative $\mathfrak{a} \subset \mathcal{O}_k$ for each class
$A = [\mathfrak{a}] \in C(\mathcal{O}_k)$, and write $Q_{\mathfrak{a}}(z) = {\bf{N}}_{k/{\bf{Q}}}(z)/{\bf{N}} \mathfrak{a}$ 
for the corresponding norm form. Again, each space $(\mathfrak{a}, Q_{\mathfrak{a}})$ has signature $(2,0)$
when $k$ is an imaginary quadratic field, and signature $(1,1)$ when $k$ is a real quadratic field. 
We consider for each class $A \in C(\mathcal{O}_k)$ the rational quadratic space $(V_A, Q_A)$ of signature $(2,2)$
given by $V_A = \mathfrak{a}_{\bf{Q}} \oplus \mathfrak{a}_{\bf{Q}}$ and quadratic form 
$Q_A(z) = Q_A(z_1, z_2) = Q_{\mathfrak{a}}(z_1) - Q_{\mathfrak{a}}(z_2)$. We fix a level $N$ prime to $d_k$,
and consider the lattice $L_A = L_A(N)\subset V_A$ whose adelization corresponds
to the compact open subgroup $K_A$ of $\operatorname{GSpin}(V_A)({\bf{A}}_f) \cong  \operatorname{GL}_2({\bf{A}}_f) \times_{{\bf{G}}_m} \operatorname{GL}_2({\bf{A}}_f)$
given by $K_0(N)^2$, as in Corollary \ref{lattices}.
Hence, the corresponding spin Shimura variety $X_{K_A}$ can be identified with $Y_0(N) \times Y_0(N)$. 
In this setting, we explain two ways to associate with a cuspidal newform 
$\phi \in S_l(\Gamma_0(N))$ a vector-valued cusp form $g_{\phi} = g_{\phi, A} \in S_l(\omega_{L_A})$.
As we explain below, we can use the Doi-Naganuma lift (see e.g.~\cite[$\S 3.1$]{Br-123}, \cite{Za})
to show the existence of such a form. We can also use the theorem of Str\"omberg 
\cite[Theorem 5.4]{Str} -- see also Scheithauer \cite[Theorem 3.1]{Scheit}, 
Zhang \cite[Theorem 4.15]{YZ}, and Bruinier-Bundschuh \cite{BB} -- to construct such a form more explicitly. 
We then show that we have identifications of completed Rankin-Selberg $L$-functions 
\begin{align}\label{E1} L^{\star}(2s-2, g_{\phi, A} \times \theta_A) &= \Lambda(s-1/2, \phi \times \theta_A),\end{align}
where $\theta_A$ denotes the Hecke theta series associated to the class $A \in C(\mathcal{O}_k)$, and hence that 
\begin{align}\label{E2} \sum\limits_{A \in C(\mathcal{O}_k)} \chi(A)
L^{\star}(2s-2, g_{\phi, A} \times \theta_A) &= 
\sum\limits_{A \in C(\mathcal{O}_k)} \chi(A) \Lambda(s-1/2, \phi \times \theta_A)= \Lambda(s-1/2, \phi \times \theta(\chi)).\end{align}
We then use this to reinterpret the calculations of Theorems \ref{BY4.7} and \ref{realquad} in terms of $\Lambda(s, \phi \times \theta(\chi))$.

\subsection{Equivalences of $L$-functions}

We now show the identifications $(\ref{E1})$ and $(\ref{E2})$.

\subsubsection{Hecke theta series associated to class group characters of quadratic fields}

Given a quadratic field $k$ as above and a class group character
$\chi : C(\mathcal{O}_k) \longrightarrow {\bf{C}}^{\times}$, we consider the corresponding Hecke theta series 
\begin{align*} \theta(\chi)(\tau) &= \sum\limits_{A \in C(\mathcal{O}_k)} \chi(A) \theta_A(\tau). \end{align*}
Here, each $\theta_A(\tau)$ denotes the theta series associated to the class 
$A \in C(\mathcal{O}_k)$ and quadratic space $(\mathfrak{a}, Q_{\mathfrak{a}})$.
Hence, when $k$ is an imaginary quadratic field, this theta series has the explicit expansion 
\begin{align*} \theta_A(\tau) &= \frac{1}{w_k}\sum\limits_{\lambda \in \mathfrak{a}} e \left( Q_{\mathfrak{a}}(\lambda) \tau \right)
= \sum\limits_{m \in {\bf{Z}}_{\geq 0}} r_A(m) e(m \tau), \end{align*}
where $w_k = \# \mu(k)$ denotes the number of roots of unity in $k$, and $r_A(m)$ the counting function 
\begin{align*} r_A(m) &= \frac{1}{w_k} \cdot \# \left \lbrace \lambda \in \mathfrak{a}: Q_{\mathfrak{a}}(\lambda) = m \right\rbrace. \end{align*}
A classical theorem of Hecke shows that this theta series $\theta_A \in M_1(\Gamma_0(\vert d_k \vert), \eta_k)$
is a modular form of weight $1 = (2-0)/2$, level $\Gamma_0(\vert d_k \vert)$, 
and character $\eta_k$. Hence, $\theta(\chi) \in M_1(\Gamma_0(\vert d_k \vert), \eta_k)$ when $k$ is imaginary quadratic. 
When $k$ is a real quadratic field, the unit group $\mathcal{O}_k^{\times} \cong {\bf{Z}} \times \mu(k) = \langle \varepsilon_k \rangle \times \mu(k)$ 
is no longer torsion, and we must fix a fundamental domain $\mathfrak{a}^{\star}$ for the action of 
$\mathcal{O}_k^{\times}/ \mu(k) = \langle \varepsilon_k \rangle$ on the lattice 
$\mathfrak{a} \subset k_{\infty} \cong {\bf{R}}^2$. We can then describe the 
corresponding theta series via the explicit expansion 
\begin{align*} \theta_A(\tau) &= \frac{1}{w_k}\sum\limits_{\lambda \in \mathfrak{a}^{\star}}  \sqrt{v} K_0 \left( 2 \pi \vert Q_{\mathfrak{a}}(\lambda) \vert v \right) 
e \left( Q_{\mathfrak{a}}(\lambda) u \right)
= \sum\limits_{m \in {\bf{Z}}_{\geq 0}} r_A(m) \sqrt{v} K_0(2 \pi \vert m \vert v) e(m u), \end{align*}
where $K_0$ denotes the $K$-Bessel function, and $r_A(m)$ denotes the counting function 
\begin{align*} r_A(m) &= \frac{1}{w_k} \cdot \# \left \lbrace \lambda \in \mathfrak{a}^{\star}: Q_{\mathfrak{a}}(\lambda) = m \right\rbrace. \end{align*}
The theorem of Hecke shows that this $\theta_A \in A_0(\Gamma_0(d_k ), \eta_k)$ 
is a nonholomorphic Maass form of weight $0 = (1-1)/2$, level $\Gamma_0( d_k )$, and character $\eta_k$. 
Hence, $\theta(\chi) \in A_0(\Gamma_0(d_k), \eta_k)$ when $k$ is real quadratic. 

Let us henceforth fix one of these theta series $\theta(\chi)$, hence 
\begin{equation}\begin{aligned}\label{Hecketheta} \theta(\chi) \in \begin{cases} M_{1}(\Gamma_0(\vert d_k \vert), \eta_k) &\text{if $k$ imaginary quadratic}\\
A_0(\Gamma_0(\vert d_k \vert), \eta_k) &\text{if $k$ real quadratic}.\end{cases}\end{aligned}\end{equation}
We also denote the weight by 
\begin{align}\label{weight} l(k) := \begin{cases} 1 &\text{ if $k$ is imaginary quadratic} \\
0 &\text{ if $k$ is real quadratic} \end{cases}. \end{align}

\subsubsection{Rankin-Selberg $L$-functions}

Let $\phi \in S_{l(\phi)}(\Gamma_0(N))$ be a holomorphic cusp form of weight $l(\phi)$ on $\Gamma_0(N)$. We write the Fourier series expansion as 
\begin{align*} \phi(\tau) &= \sum\limits_{ m \geq 1} c_{\phi}(m) e(m \tau) = \sum\limits_{m \geq 1} a_{\phi}(m) m^{\frac{l(\phi)-1}{2}} e(m \tau). \end{align*}
so that the finite part $L(s, \phi)$ of the standard $L$-function $\Lambda(s, \phi) = L_{\infty}(s, \phi) L(s, \phi)$ has the series expansion 
\begin{align*} L(s, \phi) = \sum_{m \geq 1} a_{\phi}(m)m^{-s} = \sum_{m \geq 1} c_{\phi}(m) m^{-(s + \frac{l(\phi)-1}{2})}, \quad \Re(s)>1. \end{align*}
Let us also write the Fourier series expansion of the theta series $\theta(\chi)$ as 
\begin{align*} \theta(\chi)(\tau) &= \sum\limits_{m \gg - \infty} c_{\theta(\chi)}(m) e(m \tau) 
= \sum\limits_{m \gg - \infty} a_{\theta(\chi)}(m)m^{\frac{l(k)-1}{2}} e(m \tau), \end{align*}
where
\begin{align*} c_{\theta(\chi)}(m) &= \sum\limits_{A \in C(\mathcal{O}_k)} \chi(A) c_{\theta_A}(m) = \sum\limits_{A \in C(\mathcal{O}_k)} \chi(A) r_A(m). \end{align*}
We look at the corresponding Rankin-Selberg $L$-functions %correct arithmetic normalizations
\begin{align*} L(s, \phi \times \theta(\chi)) 
&= L(2s, \eta_k) \sum\limits_{m \geq 1} c_{\phi}(m) c_{\theta(\chi)}(m) m^{ - \left( s + \left\lbrace \frac{l(\phi) + l(k) }{2}  \right\rbrace -1 \right)} \\
&= L(2s, \eta_k) \sum\limits_{m \geq 1} a_{\phi}(m) a_{\theta(\chi)}(m) m^{-s}. \end{align*}
Implicitly, we consider the corresponding partial Rankin-Selberg $L$-functions, defined first for $\Re(s) > 1$ by 
\begin{align*} L(s, \phi \times \theta_A) &= 
L(2s, \eta_k) \sum\limits_{m \geq 1} c_{\phi}(m) r_A(m) m^{ - \left( s + \left\lbrace \frac{l(\phi) + l(k) }{2}  \right\rbrace -1 \right)} \end{align*}
whose $\chi$-twisted linear combinations give the $\chi$-twisted Rankin-Selberg $L$-function 
\begin{align*} L(s, \phi \times \theta(\chi)) &= L(2s, \eta_k) \sum\limits_{A \in C(\mathcal{O}_k)} \chi(A) \sum\limits_{m \geq 1} c_{\phi}(m) r_A(m) 
m^{ - \left( s + \left\lbrace \frac{l(\phi) + l(k) }{2}  \right\rbrace -1 \right)}. \end{align*}

By the method of unfolding (cf.~e.g.~\cite[$\S$V.1]{GZ}), we have the integral presentation
\begin{align}\label{IP} \frac{ \Gamma \left( s + \left\lbrace \frac{l(\phi) + l(k) }{2}  \right\rbrace -1 \right) }{(4 \pi)^{s + \left\lbrace \frac{l(\phi) + l(k) }{2}  \right\rbrace -1}} 
\sum\limits_{m \geq 1} c_{\phi}(m) c_{{\theta}_A}(m) m^{ - \left( s + \left\lbrace \frac{l(\phi) + l(k) }{2}  \right\rbrace -1 \right)} 
&=\langle \phi, \theta_A E_{A, \eta_k}(\cdot, s; l(\phi) - l(k)) \rangle\end{align}
for $E_{A, \eta_k}(\tau, s; l(\phi) - l(k)) \in M_{l(\phi) - l(k)}(\Gamma_0( \operatorname{lcm}( \vert d_k \vert, N)), \eta_k)$ some uniquely-determined 
Eisenstein series of weight $l(\phi)-l(k)$, level $\Gamma_0( \operatorname{lcm}(\vert d_k \vert , N))$, and character $\eta_k$. 
That is, the method of Rankin-Selberg convolution shows that these Rankin-Selberg $L$-functions 
have analytic continuations given by a functional equation inherited from the Eisenstein series 
appearing in $(\ref{IP})$. Here, we have the following more precise result. Let $\Gamma_{\bf{C}}(s) = 2(2 \pi)^{-s} \Gamma(s)$
and $\Gamma_{\bf{R}}(s) = \pi^{-\frac{s}{2}}\Gamma(\frac{s}{2})$. Note that we have the relation $\Gamma_{\bf{C}}(s) = \Gamma_{\bf{R}}(s)\Gamma_{\bf{R}}(s+1)$.

\begin{proposition}\label{FE}

Let $\phi \in S_{l(\phi)}(\Gamma_0(N))$ be a normalized newform.
Assume $(N, d_k)=1$ and $l(\phi) \geq l(k)$, where $l(k) = \lbrace 0, 1 \rbrace$ as in $(\ref{weight})$ denotes the weight of the 
Hecke theta series $\theta(\chi)$ defined in $(\ref{Hecketheta})$. Put 
\begin{align*} L_{\infty}(s, \phi \times \theta(\chi)) &= \Gamma_{\bf{C}} \left( s +  \frac{l(\phi) -1}{2}  \right)^2
= \Gamma_{\bf{R}} \left( s +  \frac{l(\phi) -1}{2}  \right)^2 \Gamma_{\bf{R}}\left( s + \frac{l(\phi) +1}{2} \right)^2 \end{align*}
Then, the completed $L$-function 
\begin{align*} \Lambda(s, \phi \times \theta(\chi)) := L_{\infty}(s, \phi \times \theta(\chi)) L(s, \phi \times \theta(\chi)) \end{align*}
satisfies the symmetric functional equation 
\begin{align*} \Lambda(s, \phi \times \theta(\chi)) 
&= \eta_k(-N) \vert d_k N \vert^{1-2s}  \Lambda(1-s, \phi \times \theta(\chi)). \end{align*} \end{proposition}

\begin{proof} The proof is well-known for the more general setup of $\operatorname{GL}_2({\bf{A}}) \times \operatorname{GL}_2({\bf{A}})$
Rankin-Selberg $L$-functions given by Jacquet \cite{Ja} and Jacquet-Langlands \cite{JL}; cf.~Li \cite[Theorem 2.2, Example 2]{Li}. Here,
we use the quadratic basechange equivalence described in $(\ref{RSBC})$ below with the fact that each $\chi$ is a wide ray class character to describe
the archimedean local factor. That is, the archimedean factor $L_{\infty}(s, \phi \times \theta(\chi))$ does not depend 
on the choice of the wide ray class character $\chi$, nor on the choice of quadratic field $k$. \end{proof}

\subsubsection{Quadratic basechange equivalences}

Fix a cuspidal newform $\phi \in S_{l(\phi)}(\Gamma_0(N))$ of weight $l(\phi) > l(k)$ on $\Gamma_0(N)$ of trivial central character 
as in Proposition \ref{FE}. Let $\pi(\phi) = \otimes_v \pi(\phi)_v$ denote the cuspidal automorphic representation of 
$\operatorname{GL}_2({\bf{A}})$ determined by $\phi$, with $\Lambda(s, \pi(\phi)) = \prod_{v \leq \infty} L(s, \pi(\phi)_v)$
its standard $L$-function. Note that this coincides with the corresponding completed $L$-function 
$\Lambda(s, \phi) = L_{\infty}(s, \phi) L(s, \phi)$ of $\phi$. Let us also write $\pi(\chi)$ to denote the automorphic
representation of $\operatorname{GL}_2({\bf{A}})$ determined by the class group character $\chi \in C(\mathcal{O}_k)^{\vee}$,
equivalently by the corresponding theta series $\theta(\chi) \in M_{l(k)}(\Gamma_0(d_k), \eta_k)$. 
Hence, we consider the corresponding Rankin-Selberg $L$-function 
\begin{align*} \Lambda(s, \pi(\phi) \times \pi(\chi)) = \prod_{v \leq \infty} L(s, \pi(\phi)_v \times \pi(\chi)_v) 
= \Lambda(s, \phi \times \theta(\chi)) = L_{\infty}(s, \phi \times \theta(\chi)) L(s, \phi \times \theta(\chi)).\end{align*}

Let us now consider the quadratic basechange lifting 
\begin{align*} \Pi(\phi) = \otimes_w \Pi(\phi)_w = \operatorname{BC}_{k/{\bf{Q}}}(\pi(\phi)) \end{align*} 
of $\pi(\phi)$ to a cuspidal automorphic representation of $\operatorname{GL}_2({\bf{A}}_k)$. Such a lifting exists by the theta 
lifting construction of Shintani (c.f.~\cite[$\S$2.7]{Br-123}), and more generally for any $\operatorname{GL}_2({\bf{A}})$-automorphic 
representation by Langlands \cite{La}. We refer to \cite{GL} for additional background.
Writing $\Lambda(s, \Pi(\phi)) = \prod_{w \leq \infty} L(s, \Pi(\phi)_w)$ to denote the corresponding completed 
standard $L$-function of $\Pi(\phi)$, we have an equivalence of $L$-functions 
\begin{align}\label{RSBC} \Lambda(s, \Pi(\phi) \otimes \chi) &= \Lambda(s, \pi(\phi) \times \pi(\chi)) = \Lambda(s, \phi \times \theta(\chi)).\end{align}
for any idele class character $\chi$ of $k$.
To be clear, $(\ref{RSBC})$ relates the $\operatorname{GL}_2({\bf{A}}_k) \times \operatorname{GL}_1({\bf{A}}_k)$
automorphic $L$-function $\Lambda(s, \Pi(\phi) \otimes \chi)$ to the $\operatorname{GL}_2({\bf{A}}) \times \operatorname{GL}_2({\bf{A}})$ 
Rankin-Selberg $L$-function $ \Lambda(s, \pi(\phi) \times \pi(\chi)) = \Lambda(s, \phi \times \theta(\chi))$.

\subsubsection{Vector-valued lifts of cuspidal eigenforms via the Doi-Naganuma lift}

Let us return to the quadratic spaces $(L_A, Q_A)$ of signature $(2,2)$ described in Proposition \ref{Clifford} and Corollary \ref{Clifford2}. 
Hence, we fix an integer $N \geq 1$ prime to $d_k$. We consider the lattice $L_A = L_A(N) = N^{-1} \mathfrak{a} \oplus N^{-1} \mathfrak{a} \subset V_A$
of level $N$ and square/trivial discriminant $d(L_A)$ corresponding to the compact open subgroup $K_0(N)^2 \subset  \operatorname{GL}_2({\bf{A}}_f) \times_{{\bf{G}}_m} \operatorname{GL}_2({\bf{A}}_f)$.

We now introduce the Doi-Naganuma lift (see e.g.~\cite[$\S$3.1]{Br-123}) to describe how to construct from the  
cusp form $\phi \in S_{l(\phi)}(\Gamma_0(N))$ and the classical Siegel theta series 
$\Theta_{L_A, l(\phi)}(\tau, z)$ associated to $(L_A, Q_A)$ a Hilbert modular form $F_{\phi, L_A}(z)$ of parallel
weight $l(\phi)$ on $X_A({\bf{C}}) \cong Y_0(N) \times Y_0(N)$. This construction implies the existence of a 
vector-valued cusp form $g_{\phi} = g_{\phi, A} \in H_{l(\phi)}(\omega_{L_A})$ which lifts $\phi$, and 
consequently for which $L^{\star}(s, g_{\phi} \times \theta_{L_{A, U}^{\vee}})$ describes the partial completed Rankin-Selberg $L$-function $\Lambda(s, \phi \times \theta_{A})$. 

Let $\Theta_{L_A, l}(\tau, z)$ denote the Siegel theta seires of weight $l \in {\bf{Z}}$ associated to $L_A$; cf.~\cite[$\S$2.6]{Br-123}.
Since $L_A = L_A(N)$ has square discriminant $d(V_A)$, its corresponding Dirichlet character $\eta_{d(V_A)}(\cdot) = ( \frac{d(V_A)}{\cdot} )$ is principal/trivial. 
Given $z \in D(V_A)$ and $\lambda \in V_A({\bf{R}})$, we have a unique decomposition $\lambda = \lambda_z + \lambda_{z^{\perp}}$,
where $\lambda_z$ and $\lambda_{z^{\perp}}$ denote the corresponding projections to $z$ and $z^{\perp}$.
Let $Q_A(\lambda)_z := Q_A(\lambda_z) - Q_A(\lambda_{z^{\perp}})$ denote the corresponding majorant.
We consider the theta function $\Theta_{L_A, l}: \mathfrak{H} \times D(V_A) \rightarrow {\bf{C}}$ 
of weight $l$ associated to $L_A$, defined on $\tau = u + iv \in \mathfrak{H}$ and $z \in D(V_A)$ by 
\begin{align*} \Theta_{L_A, l}(\tau, z) &= v \sum\limits_{ \lambda \in L_A^{\vee} } \frac{(\lambda, z)_A^l}{(z,z)_A^l} \cdot 
e \left( Q_A(\lambda_{z^{\perp}}) N \tau + Q_A(\lambda_{z}) N \overline{\tau} \right). \end{align*}
This series converges normally, is nonholomorphic in both variables, and satisfies the transformation property
\begin{align*} \Theta_{L_A, l}(\gamma \tau, z) &= \eta_{d(V_A)}(d) (c \tau + d)^l \Theta_{L_A, l}(\tau, z) 
= (c \tau + d)^l \Theta_{L_A, l}(\tau, z) \quad \quad \forall \quad
\gamma = \left( \begin{array}{cc} a & b \\ c & d \end{array}\right) \in \Gamma_0(N). \end{align*}

\begin{theorem}[Doi-Naganuma]\label{D-N}

Let $\Theta_{L_A, l}(\tau, z)$ denote the Siegel theta series of weight $l$ associated to the lattice 
$L_A \subset V_A$ of Corollary \ref{lattices}, whose adelization is fixed by the compact open subgroup $K_0(N) \oplus K_0(N)$ of 
$\operatorname{GL}_2(\widehat{\bf{Z}})^2 \subset  \operatorname{GL}_2({\bf{A}}_f) \times_{{\bf{G}}_m} \operatorname{GL}_2({\bf{A}}_f) \cong \operatorname{GSpin}(V_A)({\bf{A}}_f)$.
Hence, we identify this function $\Theta_{L_A, l}(\tau, \cdot)$ in the variable $\tau \in \mathfrak{H}$ 
as a nonholomorphic modular form of weight $l$, level $\Gamma_0(N)$, and trivial character.
Let $\varphi \in S^{\operatorname{new}}_l(\Gamma_0(N))$ be a cuspidal holomorphic newform of the same weight, level, and character. 
Assume that $\varphi$ lies in the simultaneous $+1$-eigenspace for the Atkin-Lehner involutions $W_p$ for all prime divisors $p \mid N$; equivalently, 
\begin{align*}\varphi \in S_{l}^+(\Gamma_0(N))  := \left\lbrace f \in S_l(\Gamma_0(N)) :  W_p f = f \text{ for all $p \mid N$} \right\rbrace 
\subset S^{\operatorname{new}}_{l}(\Gamma_0( N)). \end{align*}
Then, the theta lift defined on $z \in D(V_A) = D^{\pm}(V_A) \cong \mathfrak{H}^2$ by the convergent integral
\begin{align*} F_{\varphi, L_A}(z) &= \int\limits_{\mathcal{F}} \varphi(\tau) \Theta_{L_A, l}(\tau, z) v^l  \frac{du dv}{v^2}\end{align*}
determines a cuspidal eigenform of parallel weight $l$ on $Y_0(N) \times Y_0(N)$.
%Here again, we write $\mathcal{F}$ to denote the standard fundamental domain for the action of $\operatorname{SL}_2({\bf{Z}})$ on $\mathfrak{H}$. 
\end{theorem}

\begin{proof} This is a special case of the Doi-Naganuma lifting for the setup of Proposition \ref{Clifford2} and Corollary \ref{lattices}. 
See \cite[$\S$3.1]{Br-123}, as well as Doi-Naganuma \cite{DN}, Naganuma \cite{Na}, van der Geer \cite[$\S$4]{vdG}, and Zagier \cite{Za}. \end{proof}

Let us now return to the setup of Theorems \ref{BY4.7} and \ref{realquad} with $L_A \subset V_A$. 
Consider the Siegel theta function $\theta_{L_A, l}: \mathfrak{H} \times D(V_A) \longrightarrow \mathfrak{S}_{L_A}^{\vee}$ of weight $l$ constructed 
from the Weil representation in $(\ref{Stheta})$. Here, we make the following modification to the choice of 
Gaussian archimedean Schwartz function $\Phi_{\infty}(x, z) := \exp(-(x,x)_{A,z})$ for $z \in D(V_A) \cong \mathfrak{H}^2$ 
and $x \in V_A({\bf{R}})$ in $(\ref{Gaussian})$. Let $\mathcal{P}_l(x, z)$ be a weight $l$ harmonic polynomial, so that 
$\omega_{L_A}(k_{\theta}) \mathcal{P}_l(x, z) = e^{il \theta} \mathcal{P}_l(x, z)$ for all $k_{\theta} \in \operatorname{SO}_2({\bf{R}})$.
Let $\Phi_{\infty}^{(l)}(x, z) = \mathcal{P}_l(x, z) \Phi_{\infty}(x, z)$. Hence, we obtain an archimedean local Schwartz function 
$\Phi_{\infty}(x, \cdot) \in \mathcal{S}(V({\bf{R}}))$ which transforms with weight $l$ under the action of $\operatorname{SO}_2({\bf{R}})$.
We then define the corresponding theta series 
\begin{align*} \theta_{L_A, l}(\tau, z) = \theta_{L_A, l}(\tau, z, 1) 
&= \sum\limits_{\mu \in L_A^{\vee}/L_A} \vartheta_{L_A}(g_{\tau}, 1; \Phi^{(l)}_{\infty}(\cdot, z_0) \otimes {\bf{1}}_{\mu}) {\bf{1}}_{\mu}\end{align*}
from the theta kernel 
\begin{align*} \vartheta_{L_A}(g, h; \Phi) &= \sum\limits_{x \in V_A({\bf{Q}})} \left( \omega_{L_A}(g, h) \Phi \right)(x). \end{align*}
Here again, we fix $z_0 \in D(V_A) \cong \mathfrak{H}^2$, 
and write $g \in \operatorname{SL}_2({\bf{A}})$, $h \in \operatorname{GSpin}(V_A)({\bf{A}})$, and $\Phi = \otimes_v \Phi_v \in \mathcal{S}(V_A({\bf{A}}))$. 
 
\begin{corollary}\label{existence}

Fix a cuspidal newform $\phi \in S_{l}^{\operatorname{new}}(\Gamma_0(N))$. Assume $\phi$ lies in the simultaneous $+1$-eigenspace
for all the Atkin-Lehner involutions $W_p$ with $p \mid N$ a prime divisor: $\phi \in S_{l}^+(\Gamma_0(N)) \subset S_{l}(\Gamma_0(N))$.
There exists a unique $g_{\phi} = g_{\phi, A} \in H_{l}(\omega_{L_A}^{\vee})$ for which 
\begin{align}\label{DNrel} \langle \langle g_{\phi}(\tau), \overline{\theta}_{L_A, l}(\tau, z) \rangle \rangle 
&= \Theta_{L_A, l}(\tau, z) \phi  (\tau),\end{align}
so that the Doi-Naganuma lifting $F_{\phi, L_A}(z)$ can be characterized equivalently as the theta integral 
\begin{align*} F_{\phi, L_A}(z) &= \int\limits_{\mathcal{F}} 
\langle \langle g_{\phi}(\tau), \overline{\theta}_{L_A, l}(\tau, z) \rangle \rangle v^l  \frac{du dv}{v^2}. \end{align*} \end{corollary}

\begin{remark} The lifting  $g_{\phi} \in H_{l}(\omega_{L_A}^{\vee})$ of $\phi \in S_l( \Gamma_0(N) )$ can be described explicitly 
in special cases by Zhang \cite[Theorem 4.15]{YZ} and Scheithauer \cite[Theorem 3.1]{Scheit}. 
More generally, the theorem of Str\"omberg\footnote{taking the isotropic subgroup $S_0 = \lbrace 0 \rbrace \subset L^{\vee}/L$} \cite[Theorems 5.2 and 5.4]{Str}
allow us to construct such a lift of any modular form $\phi \in M_l(\Gamma_0(M(L)), \eta_{\vert d(L) \vert })$ of level $M(L)$ equal to that of the lattice
$L$ and quadratic character $\eta_{\vert d(L) \vert }(\cdot) = ( \frac{\vert d(L) \vert}{\cdot} )$ with $d = d(L)$ the discriminant of the lattice
to a vector-valued form $g_{\phi} \in M_l(\omega_L)$ via the expansion 
\begin{align*} g_{\phi}(\tau) &= \sum\limits_{\gamma \in \Gamma_0(M(L)) \backslash \operatorname{SL}_2({\bf{Z}})}
\omega_L(\gamma)^{-1} {\bf{1}}_0 \phi(\tau) \vert_l \gamma. \end{align*}\end{remark}

\subsubsection{Equivalences of Rankin-Selberg $L$-functions}

We now return to Theorems \ref{BY4.7} and \ref{realquad} for the space $(V_A, Q_A)$ of signature $(2,2)$ with lattice 
$L_A = L_A(N)$ corresponding to $K_0(N)^2 \subset  \operatorname{GL}_2({\bf{A}}_f) \times_{{\bf{G}}_m} \operatorname{GL}_2({\bf{A}}_f)$ (Corollary \ref{lattices}).

\begin{proposition}\label{RS-equiv} 

Fix a holomorphic cuspidal newform $\phi \in S^{\operatorname{new}}_2(\Gamma_0(N))$ of weight $2$, level $\Gamma_0(N)$, and trivial character. 
Let $g_{\phi, A} \in S_2(\omega_{L_A}^{\vee})$ denote the lifting of $\phi$ to a vector-valued cusp form of weight $2$ 
and conjugate Weil representation $\omega_{L_A}^{\vee}$. We have the following identifications of completed Rankin-Selberg $L$-functions. \\

\begin{itemize}

\item[(i)] If $k$ is the imaginary quadratic field associated to the negative definite subspace $V_{A, 0} \subset V_A$
with $L_{A,0} = L_A \cap V_{A,0}$, we have the identifications of completed Rankin-Selberg $L$-functions 
\begin{align*} L^{\star}(2s-2, g_{\phi, A} \times \theta_{L_{A, 0}^{\perp}}) 
= \Lambda(s-1/2, \phi \times \theta_A) \end{align*} 
for each class $A \in C(\mathcal{O}_k)$, and for each class group character $\chi \in C(\mathcal{O}_k)^{\vee}$ the identification 
\begin{align*} \sum\limits_{A \in C(\mathcal{O}_k)} \chi(A) L^{\star}(2s-2, g_{\phi, A} \times \theta_{L_{A, 0}^{\perp}}) 
= \Lambda(s-1/2, \phi \times \theta(\chi)). \end{align*} 

\item[(ii)] If $k$ is the real quadratic field associated to the Lorentzian subspace 
$W_A \subset V_A$ with $L_{A,W} = W_A \cap L_A$, we have the identifications of completed Rankin-Selberg $L$-functions 
\begin{align*} L^{\star}(2s-2, g_{\phi, A} \times \theta_{L_{A, W}^{\perp}}) &= \Lambda(s-1/2, \phi \times \theta_A)\end{align*} 
for each class $A \in C(\mathcal{O}_k)$, and for each class group character $\chi \in C(\mathcal{O}_k)^{\vee}$ that 
\begin{align*} \sum\limits_{A \in C(\mathcal{O}_k)} \chi(A) L^{\star}(2s-2, g_{\phi, A} \times \theta_{L_{A, W}^{\perp}}) 
&= \Lambda(s-1/2, \phi \times \theta(\chi)).\end{align*}

\end{itemize}

\end{proposition}

\begin{proof} Cf.~\cite[Corollary 4.18]{VO}. Fix any class $A \in C(\mathcal{O}_k)$. 
If $k$ is imaginary quadratic as in (i), we write the Fourier series expansions of the corresponding holomorphic vector-valued forms as 
\begin{align*} g_{\phi, A}(\tau) 
&= \sum\limits_{ \mu \in L_A^{\vee}/L_A } \sum\limits_{m >0} c_{\phi, A}(\mu, m) e(m \tau) {\bf{1}}_{\mu}  \end{align*}
and
\begin{align*} \theta_{L_{A,0}^{\perp}}(\tau)  
&= \sum\limits_{  \mu \in (L_{A,0}^{\perp})^{\vee} / L_{A, 0}^{\perp} } \sum\limits_{m \geq 0} r_{L_{A, 0}^{\perp}}(\mu, m) e(m \tau) {\bf{1}}_{\mu}. \end{align*}
We consider the Dirichlet series of the corresponding Rankin-Selberg $L$-function, defined for $\Re(s) \gg 1$ by 
\begin{equation*}\begin{aligned}
L(s, g_{\phi, A} \times \theta_{L_{A,0}^{\perp}}) 
&= \frac{ \Gamma \left( \frac{s+2}{2} \right)  }{ (4 \pi)^{ \frac{s+2}{2}  } }
\sum\limits_{ \mu \in (L_{A,0}^{\perp})^{\vee} / L_{A,0}^{\perp} } \sum\limits_{m \geq 1} 
\frac{  c_{g_{\phi, A}}(\mu, m) r_{ L_{A,0}^{\perp} }(\mu, m) }{ m^{\frac{s+2}{2} } }. \end{aligned}\end{equation*}
Here, we can identify the discriminant group as 
\begin{align*} (L_{A,0}^{\perp})^{\vee} / L_{A,0}^{\perp} \cong \mathfrak{d}_k^{-1} N^{-1} \mathfrak{a} / N^{-1} \mathfrak{a} 
\cong \mathfrak{d}_k^{-1} \mathcal{O}_k / \mathcal{O}_k, \end{align*}
and the counting functions appearing in the Fourier series expansion of the theta series as
\begin{align*} r_{L_{A,0}^{\perp}}(\mu, m) 
&= \frac{1}{w_k} \cdot \# \left\lbrace \lambda \in \mu + L_{A,0}^{\perp}  : Q_A\vert_{L_{A,0}^{\perp} }(\lambda) = m \right\rbrace
=\frac{1}{w_k} \cdot \# \left\lbrace \lambda \in \mu + N^{-1} \mathfrak{a}: Q_{\mathfrak{a}}(\lambda) = m \right\rbrace. \end{align*}
It is easy to see from this that we have the identification\footnote{More formally, we use that 
$[N^{-1} \mathfrak{a}] = [(N^{-1}) \mathfrak{a}] = [\mathfrak{a}] \in C(\mathcal{O}_k) = I(k)/P(k)$.} of counting functions 
\begin{align*} \sum\limits_{ \mu \in (L_{A,0}^{\perp})^{\vee} / L_{A,0}^{\perp} } r_{L_{A,0}^{\perp}}(\mu, m)
&= \frac{1}{w_k} \cdot \# \left\lbrace \lambda \in N^{-1} \mathfrak{a}: Q_{\mathfrak{a}}(\lambda) = m \right\rbrace = r_A(m). \end{align*}
Similarly, as a consequence of the relation
\begin{align*}\langle \langle g_{\phi, A}(\tau), \overline{\theta}_{L_A, 2}(\tau, z) \rangle \rangle 
&= \Theta_{L_A, 2}(\tau, z) \phi (\tau) \end{align*}
implied by $(\ref{DNrel})$, we deduce that we have the relation of Fourier coefficients 
\begin{align*} \sum\limits_{ \mu \in (L_{A,0}^{\perp})^{\vee} / L_{A,0}^{\perp} } c_{\phi, A}(\mu, m) &= c_{\phi}(m), \end{align*}
and more generally, that we have an identification of scalar-valued forms 
\begin{align*} \langle \langle g_{\phi, A}(\tau), \theta_{L_{A,0}^{\perp}}(\tau) \otimes E_{L_{A,0}}(\tau, s; 1) \rangle \rangle
&= \phi (\tau) \theta_A(\tau) E_A(\tau, s; 1), \end{align*}
where $E_{A}(\tau, s; 1)$ denotes the Eisenstein series in the Rankin-Selberg integral presentation $(\ref{IP})$
corresponding to $E_{L_{A,0}}(\tau, s; 1) \in H_1(\omega_{L_{A,0}})$.
This implies the corresponding identification of Rankin-Selberg products 
\begin{equation*}\begin{aligned} L(s, g_{\phi, A} \times \theta_{L_{A,0}^{\perp}}) 
&= \int\limits_{\mathcal{F}} \langle \langle g_{\phi, A}(\tau), \theta_{L_{A,0}^{\perp}}(\tau) \otimes E_{L_{A,0}}(\tau, s; 1) \rangle \rangle \\
&= \frac{ \Gamma \left( \frac{s+2}{2} \right)  }{ (4 \pi)^{ \frac{s+2}{2}  } }
\sum\limits_{ \mu \in (L_{A,0}^{\perp})^{\vee} / L_{A,0}^{\perp} } \sum\limits_{m \geq 1} 
\frac{  c_{g_{\phi, A}}(\mu, m) r_{ L_{A,0}^{\perp} }(\mu, m) }{ m^{ \frac{s+2}{2} } } =  \frac{ \Gamma \left( \frac{s+2}{2} \right)  }{ (4 \pi)^{ \frac{s+2}{2}  } }
\sum\limits_{m \geq 1} \frac{  c_{\phi}(m) r_A(m) }{ m^{\frac{s+2}{2}} }. \end{aligned}\end{equation*}
We then deduce from $(\ref{IP})$ that Proposition \ref{FE} that we have the identification of completed $L$-functions
\begin{align*} L^{\star}(2s-2, g_{\phi, A} \times \theta_{L_{A,0}^{\perp}}) 
&= \Lambda(s-1/2, \phi \times \theta_A) \end{align*}
and hence that 
\begin{align*} \sum_{A \in C(\mathcal{O}_k)} \chi(A) L^{\star}(2s-2, g_{\phi, A} \times \theta_{L_{A,0}^{\perp}}) &= 
\sum_{A \in C(\mathcal{O}_k)} \chi(A) \Lambda(s-1/2, \phi \times \theta_A) = \Lambda(s-1/2, \phi \times \theta(\chi)). \end{align*}

If $k$ is real quadratic as for (ii), we again open up the Dirichlet series expansion (first for $\Re(s) \gg 1$) 
\begin{equation*}\begin{aligned} L(s, g_{\phi, A} \times \theta_{L_{A,W}^{\perp}}) 
&= \frac{ \Gamma \left( \frac{s+2}{2} \right)  }{ (4 \pi)^{ \frac{s+2}{2}  } }
\sum\limits_{ \mu \in (L_{A,W}^{\perp})^{\vee} / L_{A,W}^{\perp} } \sum\limits_{m \geq 1} 
\frac{  c_{g_{\phi, A}}(\mu, m) r_{ L_{A, W}^{\perp} }(\mu, m) }{ m^{\frac{s+2}{2} } }. \end{aligned}\end{equation*}
Again, we can identify the discriminant group as 
\begin{align*} (L_{A,W}^{\perp})^{\vee} / L_{A,W}^{\perp} \cong \mathfrak{d}_k^{-1} N^{-1} \mathfrak{a} / N^{-1} \mathfrak{a} 
\cong \mathfrak{d}_k^{-1} \mathcal{O}_k / \mathcal{O}_k, \end{align*}
and the counting functions 
\begin{align*} r_{L_{A,W}^{\perp}}(\mu, m) 
&= \frac{1}{w_k} \cdot \# \left\lbrace \lambda \in \mu + L_{A,W}^{\perp} / \langle \varepsilon_k \rangle  : Q_A\vert_{L_{A,W}^{\perp} }(\lambda) = m \right\rbrace
=\frac{1}{w_k} \cdot \# \left\lbrace \lambda \in \mu + N^{-1} \mathfrak{a}^{\star}: Q_{\mathfrak{a}}(\lambda) = m \right\rbrace \end{align*}
so that 
\begin{align*} \sum\limits_{ \mu \in (L_{A,W}^{\perp})^{\vee} / L_{A,W}^{\perp} } r_{L_{A,W}^{\perp}}(\mu, m)
&= \frac{1}{w_k} \cdot \# \left\lbrace \lambda \in N^{-1} \mathfrak{a}^{\star} : Q_{\mathfrak{a}}(\lambda) = m \right\rbrace = r_A(m). \end{align*}
We obtain from the corresponding relation $(\ref{DNrel})$ the identification of Fourier coefficients 
\begin{align*} \sum\limits_{ \mu \in (L_{A,W}^{\perp})^{\vee} / L_{A,W}^{\perp} } c_{\phi, A}(\mu, m) &= c_{\phi}(m), \end{align*}
and more generally, the identification of scalar-valued forms 
\begin{align*} \langle \langle g_{\phi, A}(\tau), \theta_{L_{A,W}^{\perp}}(\tau) \otimes E_{L_{A,W}}(\tau, s; 2) \rangle \rangle
&= \phi (\tau) \theta_A(\tau) E_A(\tau, s; 2), \end{align*}
where $E_{A}(\tau, s; 2)$ denotes the Eisenstein series in the Rankin-Selberg integral presentation $(\ref{IP})$
corresponding to $E_{L_{A,W}}(\tau, s; 2) \in H_2(\omega_{L_{A,W}})$.
Taking Petersson inner products, we obtain the identifications 
\begin{equation*}\begin{aligned} L(s, g_{\phi, A} \times \theta_{L_{A,W}^{\perp}}) &=
\langle  g_{\phi, A}(\tau), \theta_{L_{A,W}^{\perp}}(\tau) \otimes E_{L_{A,W}}(\tau, s; 2) \rangle
= \int\limits_{\mathcal{F}} \langle \langle g_{\phi, A}(\tau), \theta_{L_{A,W}^{\perp}}(\tau) \otimes E_{L_{A,W}}(\tau, s; 2) \rangle \rangle v^2 d\mu(\tau) \\
&= \frac{ \Gamma \left( \frac{s+2}{2} \right)  }{ (4 \pi)^{ \frac{s+2}{2}  } }
\sum\limits_{ \mu \in (L_{A,W}^{\perp})^{\vee} / L_{A,W}^{\perp} } \sum\limits_{m \geq 1} 
\frac{  c_{g_{\phi, A}}(\mu, m) r_{ L_{A,W}^{\perp} }(\mu, m) }{ m^{ \frac{s+2}{2} } } =  \frac{ \Gamma \left( \frac{s+2}{2} \right)  }{ (4 \pi)^{ \frac{s+2}{2}  } }
\sum\limits_{m \geq 1} \frac{  c_{\phi}(m) r_A(m) }{ m^{\frac{s+2}{2}} } \end{aligned}\end{equation*}
of the corresponding Rankin-Selberg inner products. 
We then deduce from $(\ref{IP})$ that Proposition \ref{FE} that we have the identification of completed $L$-functions
\begin{align*} L^{\star}(2s-2, g_{\phi, A} \times \theta_{L_{A,W}^{\perp}}) &= \Lambda(s-1/2, \phi \times \theta_A) \end{align*}
and hence that 
\begin{align*} \sum_{A \in C(\mathcal{O}_k)} \chi(A) L^{\star}(2s-2, g_{\phi, A} \times \theta_{L_{A,W}^{\perp}}) &= 
\sum_{A \in C(\mathcal{O}_k)} \chi(A) \Lambda(s-1/2, \phi \times \theta_A) = \Lambda(s-1/2, \phi \times \theta(\chi)). \end{align*}
\end{proof}

\subsection{Relations to sums of Green's functions along anisotropic subspaces}

Putting all of these observations together, we derive the following consequences of Theorems \ref{BY4.7} and \ref{realquad}.

\begin{theorem}\label{RSIP}

We retain the setup of Proposition \ref{RS-equiv}. For each class $A \in C(\mathcal{O}_k)$, 
let $f_{0, A} \in H_0(\omega_{L_A})$ be any harmonic weak Maass form 
whose image under the antilinear differential operator $\xi_0$ equals $g_{\phi, A}$, so
\begin{align*} \xi_0(f_{0,A})(\tau) &= g_{\phi, A}(\tau). \end{align*}
We have the following integral presentations of completed Rankin-Selberg $L$-functions, 
given in terms of sums over CM cycles or geodesic sets as in Theorems \ref{BY4.7} and \ref{realquad} above, respectively. \\

\begin{itemize}

\item[(i)] If $k$ is imaginary quadratic and $\eta_k(-N) = -\eta_k(N) = -1$, then 
\begin{align*} &\Lambda'(1/2, \phi \times \theta(\chi)) \\ &= - \Lambda(1, \eta_k) \sum\limits_{A \in C(\mathcal{O}_k)} \chi(A) 
\left[ \left( \frac{\operatorname{vol}(K_{A,0})  }{ 2} \right) \Phi(f_{0, A}, Z(V_{A, 0})) 
+ \operatorname{CT} \langle \langle f_{0, A}^+(\tau), \theta_{L_{A, 0}^{\perp}}(\tau) \otimes \mathcal{E}_{L_{A, 0}}(\tau) \rangle \rangle \right] \\
&= - \Lambda(1, \eta_k) \sum\limits_{A \in C(\mathcal{O}_k)} \chi(A) 
\left[ \frac{ \operatorname{deg}(Z(V_{A, 0}))}{2}  \Phi(f_{0, A}, Z(V_{A, 0})) 
+ \operatorname{CT} \langle \langle f_{0, A}^+(\tau), \theta_{L_{A, 0}^{\perp}}(\tau) \otimes \mathcal{E}_{L_{A, 0}}(\tau) \rangle \rangle \right]. \end{align*}

\item[(ii)] If $k$ is real quadratic and $\eta_k(-N) = \eta_k(N) = -1$, then 
\begin{align*} &\Lambda'(1/2, \phi \times \theta(\chi)) = - \Lambda(1, \eta_k) \\  &\times \sum\limits_{A \in C(\mathcal{O}_k)} \chi(A) 
\left[ \left( \frac{\operatorname{vol}(K_{A,W})  }{ 2} \right) \Phi(f_{0, A}, \mathcal{G}(W_A)) 
+ \operatorname{CT} \langle \langle f_{0, A}^+(\tau), \theta_{L_{A, W}^{\perp}}(\tau) \otimes \mathcal{E}_{L_{A, W}}(\tau) \rangle \rangle
+ I'(0, f_{0,A} \times \xi_0(\theta_{L_{A,W}^{\perp}})) \right]. \end{align*}

\end{itemize}

\end{theorem}

\begin{proof}

For (i), we have for each class $A \in C(\mathcal{O}_k)$ the relation
\begin{align*} \Phi(f_{0, A}, Z(V_{A, 0})) &= - \frac{2}{\operatorname{vol}(K_{A,0})} \left(
\operatorname{CT} \langle \langle f_{0, A}^+(\tau), \theta_{L_{A, 0}^{\perp}}(\tau) \otimes \mathcal{E}_{L_{A, 0}}(\tau) \rangle \rangle
+ L'(0, g_{\phi, A} \times \theta_{L_{A, 0}^{\perp}}) \right) \end{align*}
and hence 
\begin{align}\label{CMIP} - \left( \frac{ \operatorname{vol}(K_{A,0})  }{ 2} \right) \Phi(f_{0, A}, Z(V_{A, 0})) 
- \operatorname{CT} \langle \langle f_{0, A}^+(\tau), \theta_{L_{A, 0}^{\perp}}(\tau) \otimes \mathcal{E}_{L_{A, 0}}(\tau) \rangle \rangle 
&= L'(0, g_{\phi, A} \times \theta_{L_{A, 0}^{\perp}})  \end{align}
by Theorem \ref{BY4.7}. Observe that since 
$L^{\star}(s, g_{\phi, A} \times \theta_{L_{A,0}^{\perp}}) = - L^{\star}(s, g_{\phi, A} \times \theta_{L_{A,0}^{\perp}})$
by the odd, symmetric functional equation $E_{L_{A,0}}^{\star}(\tau, s; 1) = - E_{L_{A,0}}^{\star}(\tau, -s; 1)$ described in Proposition \ref{LFE-CM}, 
we have the vanishing of the central value 
$L^{\star}(0, g_{\phi, A} \times \theta_{L_{A,0}^{\perp}}) = L(0, g_{\phi, A} \times \theta_{L_{A,0}^{\perp}}) = 0$, and hence that 
\begin{align}\label{CCDV} \Lambda^{\star \prime}(0, g_{\phi, A} \times \theta_{L_{A,0}^{\perp}}) 
= \Lambda(1, \eta_k) L'(0, g_{\phi, A} \times \theta_{L_{A,0}^{\perp}}). \end{align} 
Moreover, observe that by the equivalence of $L$-functions shown in Proposition \ref{RS-equiv} (i)
with the functional equation for $\Lambda(s, \phi \times \theta(\chi))$ described in Proposition \ref{FE}, 
we are only in the non-degenerate situation when $\eta_k(-N) = -\eta_k(N) = -1$.
Hence, we can multiply each side of $(\ref{CMIP})$ to obtain the corresponding relation
\begin{equation} \begin{aligned}\label{CMIP2}  &- \Lambda(1, \eta_k) \left[ \left( \frac{ \operatorname{vol}(K_{A,0})  }{ 2} \right) \Phi(f_{0, A}, Z(V_{A, 0})) 
+ \operatorname{CT} \langle \langle f_{0, A}^+(\tau), \theta_{L_{A, 0}^{\perp}}(\tau) \otimes \mathcal{E}_{L_{A, 0}}(\tau) \rangle \rangle \right] 
= L^{\star \prime}(0, g_{\phi, A} \times \theta_{L_{A, 0}^{\perp}}). \end{aligned}\end{equation}
Taking a twisted linear combination of the $L$-values on each side of $(\ref{CMIP2})$ and using Proposition \ref{RS-equiv}, we obtain 
\begin{equation} \begin{aligned}\label{CMIP3} 
&- \Lambda(1, \eta_k) \sum\limits_{A \in C(\mathcal{O}_k)} \chi(A) \left[ \left( \frac{ \operatorname{vol}(K_{A,0})  }{ 2} \right) \Phi(f_{0, A}, Z(V_{A, 0})) 
+ \operatorname{CT} \langle \langle f_{0, A}^+(\tau), \theta_{L_{A, 0}^{\perp}}(\tau) \otimes \mathcal{E}_{L_{A, 0}}(\tau) \rangle \rangle \right] \\
&= \sum\limits_{A \in C(\mathcal{O}_k)} \chi(A) L^{\star \prime}(0, g_{\phi, A} \times \theta_{L_{A, 0}^{\perp}})
= \Lambda'(1/2, \phi \times \theta(\chi)) = \Lambda'(1/2, \phi \times \theta(\chi)). \end{aligned}\end{equation}

For (ii), we have for each class $A \in C(\mathcal{O}_k)$ the relation
\begin{align*} &\Phi(f_{0, A}, \mathcal{G}(W_A)) \\ &= - \frac{2}{ \operatorname{vol}(K_{A,W})} \left(
\operatorname{CT} \langle \langle f_{0, A}^+(\tau), \theta_{L_{A, W}^{\perp}}^+(\tau) \otimes \mathcal{E}_{L_{A, W}}(\tau) \rangle \rangle
+ L'(0, g_{\phi, A} \times \theta_{L_{A, W}^{\perp}}) +I'(0, f_{0,A} \times \xi_0(\theta_{L_{A,W}^{\perp}})) \right) \end{align*}
and hence 
\begin{equation}\begin{aligned}\label{geoIP} &- \left( \frac{ \operatorname{vol}(K_{A,W})  }{ 2} \right) \Phi(f_{0, A}, \mathcal{G}(W_A)) 
- \operatorname{CT} \langle \langle f_{0, A}^+(\tau), \theta^+_{L_{A, W}^{\perp}}(\tau) \otimes \mathcal{E}_{L_{A, W}}(\tau) \rangle \rangle 
- I'(0, f_{0,A} \times \xi_0(\theta_{L_{A,W}^{\perp}})) \\
&= L'(0, g_{\phi, A} \times \theta_{L_{A, W}^{\perp}})  \end{aligned}\end{equation}
by Theorem \ref{realquad}. Using the identification of completed $L$-functions of Proposition \ref{RS-equiv} (ii)
with the functional equation of Proposition \ref{FE}, we see that $L^{\star}(s, g_{\phi, A} \times \theta_{L_{A, W}^{\perp}})$
satisfies an odd symmetric functional equation when $\eta_k(-N) = \eta_k(N) = -1$, whence
$L^{\star}(0, g_{\phi, A} \times \theta_{L_{A,W}^{\perp}}) = L(0, g_{\phi, A} \times \theta_{L_{A,W}^{\perp}}) = 0$, and 
\begin{align*} \Lambda^{\star \prime}(0, g_{\phi, A} \times \theta_{L_{A,W}^{\perp}}) 
= \Lambda(1, \eta_k) L'(0, g_{\phi, A} \times \theta_{L_{A,W}^{\perp}}). \end{align*}
Hence, we can multiply each side of $(\ref{geoIP})$ to obtain the corresponding relation
\begin{equation} \begin{aligned}\label{geoIP2} 
&- \Lambda(1, \eta_k) \left[ \left( \frac{ \operatorname{vol}(K_{A,W})  }{ 2 } \right) \Phi(f_{0, A}, \mathcal{G}(W_A)) 
+ \operatorname{CT} \langle \langle f_{0, A}^+(\tau), \theta_{L_{A, W}^{\perp}}(\tau) \otimes \mathcal{E}_{L_{A, W}}(\tau) \rangle \rangle 
+ I'(0, f_{0,A} \times \xi_0(\theta_{L_{A,W}^{\perp}})) \right] \\
&= L^{\star \prime}(0, g_{\phi, A} \times \theta_{L_{A, W}^{\perp}}). \end{aligned}\end{equation}
Taking a twisted linear combination of the $L$-values on each side of $(\ref{geoIP2})$ and using Proposition \ref{RS-equiv}, we obtain 
\begin{equation} \begin{aligned}\label{geoIP3} 
&- \Lambda(1, \eta_k) \sum\limits_{A \in C(\mathcal{O}_k)} \chi(A) \left[ \left( \frac{ \operatorname{vol}(K_{A,W})  }{ 2 } \right) \Phi(f_{0, A}, \mathcal{G}(W_A)) 
+ \operatorname{CT} \langle \langle f_{0, A}^+(\tau), \theta_{L_{A, W}^{\perp}}(\tau) \otimes \mathcal{E}_{L_{A, W}}(\tau) \rangle \rangle 
+ I'(0, f_{0,A} \times \xi_0(\theta_{L_{A,W}^{\perp}})) \right] \\
&= \sum\limits_{A \in C(\mathcal{O}_k)} \chi(A) L^{\star \prime}(0, g_{\phi, A} \times \theta_{L_{A, W}^{\perp}})
= \Lambda'(1/2, \phi \times \theta(\chi)) = \Lambda'(1/2, \phi \times \theta(\chi)). \end{aligned}\end{equation} \end{proof}

Now, we can make Theorem \ref{RSIP} more explicit by using the Dirichlet analytic class number formula 
\begin{align}\label{Dirichlet} L(1, \eta_k) &= \begin{cases} \frac{2 \pi h_k}{w_k \sqrt{\vert d_k \vert}} &\text{ if $k$ is imaginary quadratic} \\
\frac{2 \log(\varepsilon_k) h_k}{\sqrt{d_k}} &\text{ if $k$ is real quadratic} \end{cases} \end{align}
to evaluate 
\begin{align}\label{CDirichlet} \Lambda(1, \eta_k) 
%=\begin{cases} \vert d_k \vert^{\frac{1}{2}} \Gamma_{\bf{R}}(2) L(1, \eta_k) &\text{ if $k$ is imaginary quadratic} \\ 
%\vert d_k \vert^{\frac{1}{2}} \Gamma_{\bf{R}}(1) L(1, \eta_k) &\text{ if $k$ is real quadratic} \end{cases}
&= \begin{cases} \frac{2 h_k}{w_k} &\text{ if $k$ is imaginary quadratic} \\ 2 \log(\varepsilon_k) h_k &\text{ if $k$ is real quadratic} \end{cases}.\end{align}
%Here, we observe that each of the compact open subgroups 
%$K_{A, 0} \subset T_{A,0}({\bf{A}}_f)$ and $K_{A, W} \in T_{A, W}({\bf{A}}_f)$ must be the maximal compact group $\mathcal{O}_k^{\times}$. 
%We then calculate $\operatorname{vol}(K_{A,0}) = w_k/h_k$ and $\operatorname{vol}(K_{A, W}) = (w_k \log(\varepsilon_k))/h_k$
%for each class $A \in C(\mathcal{O}_k)$ using Lemma \ref{measures} to obtain  %%%USE THIS TO SIMPLIFY?

\begin{corollary}\label{simplified} We have the following identities for the central derivative value $\Lambda'(1/2, \phi \times \theta(\chi))$.

\begin{itemize}

\item[(i)] If $k$ is imaginary quadratic and $\eta_k(-N) = -\eta_k(N) = -1$, then 
\begin{align*} & \Lambda'(1/2, \phi \times \theta(\chi))  \\
&=- \frac{2 h_k}{w_k} \sum\limits_{A \in C(\mathcal{O}_k)} \chi(A) 
\left[ \left(  \frac{\operatorname{vol}(K_{A,0})  }{ 2}\right) \Phi(f_{0, A}, Z(V_{A, 0})) 
+ \operatorname{CT} \langle \langle f_{0, A}^+(\tau), \theta_{L_{A, 0}^{\perp}}(\tau) \otimes \mathcal{E}_{L_{A, 0}}(\tau) \rangle \rangle \right]. \end{align*}

\item[(ii)] If $k$ is real quadratic and $\eta_k(-N) = \eta_k(N) = -1$, then 
\begin{align*} &\Lambda'(1/2, \phi \times \theta(\chi)) =-2 \log(\varepsilon_k) h_k \\ &\times \sum\limits_{A \in C(\mathcal{O}_k)} \chi(A) 
\left[  \left( \frac{\operatorname{vol}(K_{A,W})  }{ 2} \right) \Phi(f_{0, A}, \mathcal{G}(W_A)) 
+ \operatorname{CT} \langle \langle f_{0, A}^+(\tau), \theta_{L_{A, W}^{\perp}}(\tau) \otimes \mathcal{E}_{L_{A, W}}(\tau) \rangle \rangle 
+ I'(0, f_{0,A} \times \xi_0(\theta_{L_{A,W}^{\perp}})) \right]. \end{align*}

\end{itemize}

\end{corollary}

\section{Arithmetic implications}

We now explain how to compute the Faltings heights of arithmetic divisors $Z(f)$ along zero cycles to prove ''higher Gross-Zagier formulae'' following Bruinier-Yang \cite{BY} and 
Andreatta-Goren-Howard-Madapusi Pera \cite{AGHMP}. Via the integral presentations described above, we use this to derive a new proof of the theorem of 
Gross-Zagier \cite[$\S$I. (6.3)]{GZ} (Theorem \ref{Gross-Zagier} above) in terms of arithmetic Hirzebruch-Zagier divisors on $Y_0(N) \times Y_0(N)$. 
We also explain some applications to the refined conjecture of Birch and Swinnerton-Dyer.

\subsection{Arithmetic heights and higher Gross-Zagier formulae}

We first explain how to derive from \cite[Theorem 4.7]{BY} (Theorem \ref{BY4.7} above) and \cite[Theorem A]{AGHMP} the following 
''higher Gross-Zagier formula'', relating the central derivative values $L'(0, \xi_{1-\frac{n}{2}}(f) \times \theta_{L_0^{\perp}})$ of the Rankin-Selberg $L$-functions 
\begin{align*} L(s, \xi_{1-\frac{n}{2}}(f) \times \theta_{L_0^{\perp}}) = \langle \xi_{1- \frac{n}{2}}(f)(\tau), \theta_{L_0^{\perp}}(\tau) \otimes E_{L_0}(\tau, s; 1) \rangle,
\quad \quad f \in H_{1- \frac{1}{2}}(\omega_L) \end{align*}
to Faltings heights of arithmetic divisors $\widehat{Z}(f) = (Z(f), \Phi(f, \cdot))$ along the CM cycles $Z(V_0)$ of $(\ref{complexpoints})$ and $(\ref{CMcycle})$.

\subsubsection{Extension to integral models}

Let us now fix the level $K = K_L \subset \operatorname{GSpin}(V)({\bf{A}}_f)$ corresponding to some lattice $L \subset V$, and write $X = X_K$ for the corresponding spin Shimura variety. 
As explained in \cite{AGHMP}, $X({\bf{C}})$ can be viewed as the space of complex points of an algebraic
Mumford-Deligne stack $X \longrightarrow \operatorname{Spec}({\bf{Q}})$. In general, apart from some cases of small dimension ($n \leq 3$), 
the Shimura variety $X$ is not of PEL type, and so does not generally represent a moduli space of abelian varieties with PEL structure.
However, it is of Hodge type, and so the theorems of Kisin \cite{Ki}, Madapusi Pera \cite{MP}, and Kim-Madapusi Pera \cite{KMP} apply to 
show the existence of a regular, flat integral model $\mathcal{X} \longrightarrow \operatorname{Spec}({\bf{Z}})$.

Recall from $(\ref{special})$ that for each $\mu \in L^{\vee}/L$ and $m \in {\bf{Q}}$ for which $\Omega_{\mu, m}({\bf{Q}}) := \left\lbrace x \in \mu + L: Q(x) = m \right\rbrace$
is nonempty, we have the special divisor $Z(\mu, m) \rightarrow X$. 
Each of these divisors admits an extension $\mathcal{Z}(\mu, m) \rightarrow \mathcal{X}$ to the integral model. 
Roughly speaking, this is obtained through the Kuga-Satake abelian scheme $\mathcal{A} \rightarrow \mathcal{X}$.
That is, $X = X_K$ comes equipped with a family of Kuga-Satake abelian varieties $A_z \rightarrow X$ indexed by points $z \in D(V)$. 
To describe the abelian variety $A_z \rightarrow X$, consider the natural functor 
\begin{align*} \left\lbrace \text{algebraic representations of $\operatorname{GSpin}(V)$} \right\rbrace 
&\longrightarrow \left\lbrace \text{local systems of ${\bf{Q}}$-vectorspaces on $X({\bf{C}})$}  \right\rbrace \\ 
(\operatorname{GSpin}(V) \rightarrow \operatorname{GL}(W)) &\longmapsto (W_{\operatorname{Betti}, {\bf{Q}}} \rightarrow X({\bf{C}}))\end{align*}
where 
\begin{align*} W_{\operatorname{Betti}, {\bf{Q}}} &:= \operatorname{GSpin}(V)({\bf{Q}}) \backslash \left( W \times D(V) \right) 
\times \operatorname{GSpin}({\bf{A}}_f)/ K.\end{align*}
This allows us to associate with each algebraic representation
$\operatorname{GSpin}(V) \rightarrow \operatorname{GL}(W)$ a pair 
$(W_{\operatorname{dR}}, \nabla)$ consisting of a locally free $\mathcal{O}_{X({\bf{C}})}$-module 
$W_{\operatorname{dR}} = W_{\operatorname{Betti}, {\bf{Q}}} \otimes \mathcal{O}_{X({\bf{C}})}$ and 
a connection $\nabla = 1 \otimes d$. Each such pair can be viewed as a vector bundle $W_{\operatorname{dR}}$ 
with integrable connection $\nabla$ such that $W_{\operatorname{dR}}^{\nabla = 0} = W_{\operatorname{Betti}, {\bf{Q}}} \otimes_{\bf{Q}} {\bf{C}}$ and:\\

\begin{itemize}

\item[(i)] For all $z \in D(V)$, the map $h_z: {\bf{S}} \longrightarrow \operatorname{GSpin}(V) ({\bf{R}}) \longrightarrow \operatorname{GL}(W_{\bf{R}})$ 
induces a map ${\bf{S}}({\bf{C}}) \rightarrow \operatorname{GL}(W_{\bf{C}})$. \\

\item[(ii)]The fibre $W_{\operatorname{dR}, z}$ at $z \in D(V)$ has a bigradation $W_{\operatorname{dR}, z} = \bigoplus_{p, q} W_{\operatorname{dR}, z}^{p,q}$ 
induced by the action of ${\bf{S}}({\bf{C}})$. \\

\item[(iii)] The $\mathcal{O}_{X({\bf{C}})}$-module $W_{\operatorname{dR}}$ is endowed with a 
decreasing filtration $\operatorname{Fil}_J(W_{\operatorname{dR}}) \subseteq W_{\operatorname{dR}}$
of submodules, defined pointwise by 
$\operatorname{Fil}^J(W_{\operatorname{dR}, z}) = \bigoplus_{p \geq J} W_{\operatorname{dR}, z}^{p, q}$. \\ \end{itemize}
Here, have have natural identifications ${\bf{S}}({\bf{C}}) = {\bf{C}}^{\times} \times {\bf{C}}^{\times}$ 
and $\operatorname{GL}(W_{\bf{C}}) = \operatorname{GL}(W_{\bf{Q}} \otimes {\bf{C}})$
Now, consider the representation of the Clifford algebra $C(V)$ on $\operatorname{GSpin}(V)$ induced by the inclusion 
\begin{align*} \operatorname{GSpin}(V) \subset C^0(V)^{\times} := C^0(V) \backslash \lbrace 0 \rbrace, \end{align*} 
with action given by left multiplication of $C^0(V)^{\times}$ on $\operatorname{GSpin}(V)$. The corresponding vector bundle
$C(V)_{\operatorname{dR}}$ gives rise to the following variation of Hodge structures: For each $z \in D(V)$, we have 
\begin{align}\label{VHS} C(V)_{\operatorname{dR}, z} &= C(V)_z^{-1, 0} \oplus C(V)_{z}^{0, -1}. \end{align}
Note that having such a variation of Hodge structures $(\ref{VHS})$ for each $z \in D(V)$ is equivalent to 
having a complex structure on the Clifford algebra $C(V_{\bf{R}}) = C(V_{\bf{Q}} \otimes {\bf{R}})$ for each $z \in D(V)$. 
In particular, we obtain from this complex structure for each $z \in D(V)$ a corresponding abelian variety $A_z = C(V_{\bf{R}})/C(L)$
of dimension $2^{n+1}$ known as the {\it{Kuga-Satake abelian variety associated to $X = X_{K_L}$ at a point $z \in D(V)$}}.
As explained in \cite{AGHMP}, this construction extends\footnote{They also show that the associated vector bundle with 
connection $(C(V)_{\operatorname{dR}}, \nabla)$ extends to the integral model $\mathcal{X}$, with vector bundle $C(V)_{\operatorname{dR}}$
given by the relative de Rham cohomology $H^1_{\operatorname{dR}}(\mathcal{A})$, connection $\nabla$ given by the Gauss-Manin
connection, and filtration $\operatorname{Fil}^J(C(V)_{\operatorname{dR}, z})$ given by the Hodge filtration 
$0 \longrightarrow R_{!} \pi_*(\mathcal{O}_{\mathcal{A}}) \longrightarrow H^1_{\operatorname{dR}}(\mathcal{A}) 
\longrightarrow \pi_*(\Omega_{\mathcal{A}}^1) \longrightarrow 0$.} to give an abelian scheme $\mathcal{A} \rightarrow \mathcal{X}$.

\begin{example}\label{KSAVn=0} 

Consider a rational quadratic space $(V_0, Q_0)$ of signature $(0, 2)$ with integral lattice $L_0 \subset V_0$.
The submodule $C^0(L_0) \subset C^0(V_0)$ then corresponds to an order $\mathcal{O} \subset \mathcal{O}_{k(V_0)}$ of the imaginary quadratic 
field $k(V_0)$ associated to $V_0$. Each $z \in D(V_0)$ determines an abelian surface $A_z = A_z^+ \times A_z^- = C(V_0({\bf{R}}))/C(L_0)$, 
where $A_z^+$ is an elliptic curve with complex multiplication by the order
$\mathcal{O} \cong C^0(V_0)$, and $A_z$ is the elliptic curve with CM by $\mathcal{O}$ given by 
$A_z^- = A_z^+ \otimes_{\mathcal{O}} L_0 = A_z^+ \otimes_{C^0(L_0)} L_0$. 

\end{example}

\begin{example}\label{KSAVn=1} 

Consider the rational quadratic space $(V,Q)$ of signature $(1,2)$ 
given by $V = \operatorname{Mat}_2^{\operatorname{tr} = 0}({\bf{Q}})$ and $Q(\cdot) = N \det(\cdot)$. 
As explained in \cite[$\S$7.3]{BY} and Appendix A, we have an exceptional isomorphism
$\operatorname{GSpin}(V) \cong \operatorname{GL}_2$ of algebraic groups over ${\bf{Q}}$.
Consider the lattice $L \subset V$ and dual lattice $L^{\vee}$ given by
\begin{align*} L = \left\lbrace \left( \begin{array}{cc} b & -a/N \\ c & -b \end{array} \right) : a,b,c \in {\bf{Z}} \right\rbrace,
\quad L^{\vee} = \left\lbrace \left( \begin{array}{cc} b/2n & -a/N \\ c & -b/2N \end{array} \right) : a,b,c \in {\bf{Z}} \right\rbrace, \end{align*}
so that the discriminant group can be identified as 
\begin{align*} {\bf{Z}}/2N {\bf{Z}} &\longrightarrow L^{\vee}/L, \quad r \longmapsto \left( \begin{array}{cc} r/2N & ~\\~&-r/2N \end{array} \right). \end{align*}
Hence, $L$ has level $4N$, and the corresponding hyperboloid $\Omega_{\mu, m}({\bf{Q}})$
is nonempty unless $Q(\mu) \equiv m \bmod 1$. The corresponding compact open subgroup $K = K_L$ is given by the congruence subgroup $K_0(N) \subset \operatorname{GL}_2({\bf{A}}_f)$.
%\begin{align*} K = \prod_{p<\infty} K_p \subset \operatorname{GSpin}(V)({\bf{A}}_f) \cong \operatorname{GL}_2({\bf{Q}}_p),
%\quad K_p := \left\lbrace \left( \begin{array}{cc} a & b \\ c & d \end{array} \right) \in \operatorname{GL}_2({\bf{Z}}_p) : c \in N {\bf{Z}}_p \right\rbrace. \end{align*}
In this setting, we have an isomorphism of Shimura curves 
\begin{align*} Y_0(N) := \Gamma_0(N) \backslash \mathfrak{H} &\cong X_K({\bf{C}}), 
\quad \Gamma_0(N) z \longmapsto \operatorname{GSpin}(V)({\bf{Q}})(z,1) K. \end{align*}
In the other direction, using the moduli description of the noncompactified modular curve $Y_0(N)$, we have 
\begin{align*} X_K({\bf{C}}) := \operatorname{GSpin}(V)({\bf{Q}}) \backslash D(V) \times \operatorname{GSpin}(V)({\bf{A}}_f)/K \cong Y_0(N),
\quad z \longmapsto (E_z \rightarrow E_z') \end{align*}
for $(E_z, E_z')$ a pair of elliptic curves with CM by some order $\mathcal{O} \subset \mathcal{O}_{k(V_0)}$
in the imaginary quadrtic field $k(V_0)$ determined by a negative definite subspace $V_0 \subset V$ (given by a Heegner embedding). 
In this case, each point $z \in D(V) = D^{\pm}(V) \cong \mathfrak{H}$ has the corresponding Kuga-Satake abelian fourfold $A_z = C(V({\bf{R}}))/C(L)$ given by
\begin{align*} A_z = A_z^+ \times A_z^-, \quad \quad A_z^+ = A_z^- = E_z \times E_z'. \end{align*}
We shall return to this example in Appendix A below to explain how to recover the formula of Gross-Zagier \cite[Theorem I (6.3)]{GZ}
from \cite[Theorem 4.7, Theorem 7.7]{BY} and \cite[Theorem A]{AGHMP}. \end{example}

The Kuga-Satake abelian scheme $\mathcal{A} \rightarrow \mathcal{X}$ comes equipped with an action of the Clifford algebra 
$C(L) = C^0(L) \oplus C^1(L)$, and acquires from this a ${\bf{Z}}/2{\bf{Z}}$-grading $\mathcal{A} = \mathcal{A}^+ \times \mathcal{A}^-$.
For each scheme $S \rightarrow \mathcal{X}$, we can associate to the pullback 
$\mathcal{A}_S$ a distinguished ${\bf{Z}}$-module of {\it{special endomorphisms}} 
\begin{align*} V(\mathcal{A}_S) \subset \operatorname{End}(\mathcal{A}_S),\end{align*}
together with an associated quadratic form $q: V(\mathcal{A}_S) \longrightarrow {\bf{Z}}$ defined by $x \circ x = q(x) \cdot \operatorname{id}$.
More generally, for each coset $\mu \in L^{\vee}/L$, we can associate to the pullback $\mathcal{A}_S$ a distinguished
subset $V_{\mu}(\mathcal{A}_S) \subset V(\mathcal{A}_S) \otimes {\bf{Q}}$ with the property that $V_0(\mathcal{A}_S) = V(\mathcal{A}_S)$. 
We can define from each of these subsets $V_{\mu}(\mathcal{A}_S)$ a divisor on $\mathcal{X}$ as follows: 
For each $\mu \in L^{\vee}/L$ and $m \in {\bf{Q}}$ for which the hyperboloid $\Omega_{\mu, m}({\bf{Q}})$ is nonempty, let
\begin{align*} \mathcal{Z}(\mu, m) \longrightarrow \mathcal{X} \end{align*}
denote the moduli stack that assigns to each $\mathcal{X}$-scheme $S \rightarrow \mathcal{X}$ the set 
\begin{align*} \mathcal{Z}(\mu, m)(S) &:= \left\lbrace x\in V_{\mu}(\mathcal{A}_S): q(x) =m \right\rbrace. \end{align*}
This morphism $\mathcal{Z}(\mu, m) \rightarrow \mathcal{X}$ is finite and relatively representable, and determines a Cartier divisor on $\mathcal{X}$. 
It agrees on the generic fibre with the special cycles $(\ref{special})$ defined above, $\mathcal{Z}(\mu, m)({\bf{C}}) = Z(\mu, m)$.

Recall that we write $T_0 = \operatorname{GSpin}(V_0) = \operatorname{Res}_{k(V_0)/{\bf{Q}}} {\bf{G}}_m$ to denote the 
torus corresponding to a rational quadratic subspace $(V_0, Q_0)$ of signature $(0, 2)$ and associated imaginary quadratic field $k(V_0)$. 
We consider the corresponding zero-dimensional Shimura variety $Z(V_0) \longrightarrow \operatorname{Spec}(k(V_0))$ with complex points 
given by $(\ref{CMcycle})$. 
Note that we can identify the corresponding compact open subgroup $K_0:= K \cap T_0({\bf{A}}_f)$ with $K_0 =\widehat{O}_{k(V_0)}^{\times}$, 
and that this acts trivially on the discriminant group $L_0^{\vee}/L_0$. Hence, we can identify the complex points
\begin{align*} Z(V_0)({\bf{C}}) 
&= k(V_0)^{\times} \backslash \lbrace z_{V_0}^{\pm} \rbrace  \times {\bf{A}}_{k(V_0), f}^{\times} / \widehat{O}_{k(V_0)}^{\times} \end{align*}
with two copies of the ideal class group $C(\mathcal{O}_{k(V_0)})$. Observe that by Lemma \ref{measures} (cf.~\cite[Lemma 6.3]{BY}) 
with the Dirichlet analytic class number formula $(\ref{Dirichlet})$, we have that the degree $\operatorname{deg} Z(V_0) = 4/\operatorname{vol}(K_0)$
of $Z(V_0)$ as defined in Lemma \ref{schofer} (i)  is given by the relation 
\begin{align*} \frac{ \operatorname{deg} Z(V_0) }{4} 
&= \frac{1}{ \operatorname{vol}(K_0)} = \frac{h_{k(V_0)}}{w_{k(V_0)}} = \frac{ \vert d_{k(V_0)} \vert^{\frac{1}{2}} }{ 2 \pi} \cdot L(1, \eta_{k(V_0)}).\end{align*}
Viewing $Z(V_0) \rightarrow \operatorname{Spec}(k(V_0))$ as the moduli space of elliptic curves with complex
multiplication by $\mathcal{O}_{k(V_0)}$, we obtain a smooth integral model $\mathcal{Z}(V_0) \rightarrow \operatorname{Spec}(\mathcal{O}_{k(V_0)})$.
If $k = k(V_0)$ has odd discriminant $d_{k(V_0)}$, then the embedding 
$T_0 \subset \operatorname{GSpin}(V)$ of reductive groups induced by the embedding of quadratic spaces 
$V_0 \subset V$ gives a finite, relatively representable, unramified morphism $\mathcal{Z}(V_0) \longrightarrow \mathcal{X}$.
%This algebraic stack has its own Kuga-Satake abelian scheme $\mathcal{A}_0 \longrightarrow \mathcal{Z}(V_0)$ 
%equipped with an action of the Clifford algebra $C(L_0) = C^0(L_0) \oplus C^1(L_0)$ and hence a ${\bf{Z}}/2{\bf{Z}}$-grading
%$\mathcal{A}_0 \cong \mathcal{A}_0^+ \times \mathcal{A}_0^-$. Here, $\mathcal{A}_0^+$ can be identified with the universal
%elliptic curve with CM by $\mathcal{O}_{k(V_0)}$, with $\mathcal{A}_0 \cong \mathcal{A}_0^+ \otimes_{ \mathcal{O}_{k(V_0)} } C(L_0)$
%and $\mathcal{A}_0^- \cong \mathcal{A}_0^+ \otimes_{ \mathcal{O}(k(V_0)) } L_0$. Moreover, this Kuga-Satake abelian scheme 
%$\mathcal{A}_0 \longrightarrow \mathcal{Z}(V_0)$ is related to the Kuga-Satake abelian scheme $\mathcal{A} \longrightarrow \mathcal{X}$ 
%by a $C(L)$-linear isomorphism 
%\begin{align*} \mathcal{A} \vert_{\mathcal{Z}(V_0)} \cong \mathcal{A}_0 \vert_{\mathcal{Z}(L_0)}  \otimes_{C(L_0)} C(L). \end{align*} 

\subsubsection{Arithmetic degrees along CM cycles and central derivative Rankin-Selberg $L$-values}

By Theorem \ref{BY4.7}, we have for any harmonic weak Maass form $f \in H_{1- \frac{n}{2}}(\omega_L)$ the formula
\begin{align}\label{BYf} \Phi(f, Z(V_0)) &= - \frac{\operatorname{deg}(Z(V_0))}{2} \cdot \left( 
\operatorname{CT} \langle \langle f^+(\tau), \theta_{L_0^{\perp}}(\tau) \otimes \mathcal{E}_{L_0}(\tau) \rangle \rangle 
+ L'(0, \xi_{1- n/2}(f) \times \theta_{L_0^{\perp}} \right).  \end{align}

We can now describe this formula in terms of arithmetic heights, according to the calculations of \cite[$\S$5-6]{BY} 
and more generally \cite[Theorem A]{AGHMP}, which we now summarize. Recall that an arithmetic divisor 
$\widehat{x} = (x, G_x)$ on the integral model $\mathcal{X} \rightarrow \operatorname{Spec}({\bf{Z}})$ consists 
of a divisor $x$ on $\mathcal{X}$ and a corresponding Green's function $G_x$ for the the divisor $x({\bf{C}})$ 
induced by $x$ on the complex variety $\mathcal{X}({\bf{C}} ) = X({\bf{C}})$. That is, the $G_x$ is a smooth 
function on $\mathcal{X}({\bf{C}}) \backslash x({\bf{C}})$ with a logarithmic singularity along $x({\bf{C}})$ 
which satisfies the Green's current equation \begin{align*} d d^c [G_x] + \delta_{x({\bf{C}})} = [\Omega_x] \end{align*}
for some smooth $(1,1)$-form $\Omega_x$ on $\mathcal{X}({\bf{C}})$.
Let $\widehat{\operatorname{Ch}}^1(\mathcal{X})$ denote the first arithmetic Chow group of $\mathcal{X}$,
so the free abelian group generated by arithmetic divisors on $\mathcal{X}$ modulo rational equivalence. Let
\begin{align*} \left[ \cdot, \cdot \right]: \widehat{\operatorname{Ch}}^1(\mathcal{X}) \times Z^n(\mathcal{X}) \longrightarrow {\bf{R}}\end{align*}
denote the height pairing defined in Bost-Gillet-Soul\'e \cite[$\S$ 2.3]{BGS}.
Given an arithmetic divisor $\widehat{x}  \in \widehat{\operatorname{Ch}}^1(\mathcal{X})$
and an $n$-cycle $y \in Z^n(\mathcal{X})$ intersecting properly, we know that this pairing is given by the Faltings height 
\begin{align*} \left[ \widehat{x}, y\right] = \left[ \widehat{x}, y \right]_{\operatorname{Fal}} 
&= \left[ x, y \right]_{\operatorname{fin}} + \left[ \widehat{x}, y\right]_{\infty}.\end{align*}
Here, the archimedean component is given by the value of the Green function $G_x$ along $y({\bf{C}})$, 
\begin{align*} \left[ \widehat{x}, y\right]_{\infty} &= G_x \left( y({\bf{C}}) \right). \end{align*}

The regularized theta lifts $\Phi(f) = \Phi(f, \cdot)$ provide a supply of arithmetic divisors 
$\widehat{\mathcal{Z}}(f) = (\mathcal{Z}(f), \Phi(f))$ on $\mathcal{X} \rightarrow \operatorname{Spec}({\bf{Z}})$. 
Each CM-cycle $\mathcal{Z}(V_0) \longrightarrow \mathcal{X}_{\mathcal{O}_{k(V_0)}}$ determined by a negative definite subspace $V_0 \subset V$ 
is a zero-cycle $y \in Z^0(\mathcal{X})$ which intersects $\widehat{\mathcal{Z}}(f) = (\mathcal{Z}(f), \Phi(f))$ properly. 
In particular, $(\ref{BYf})$ gives us
\begin{equation}\begin{aligned}\label{BYf2} \left[ \widehat{\mathcal{Z}}(f), \mathcal{Z}(V_0) \right]_{\infty} &= \Phi(f, Y) 
&= - \frac{\deg( Z(V_0) )}{2} \cdot \left(   \operatorname{CT}  \langle\langle f^+(\tau), \theta_{L_0^{\perp}}(\tau) 
\otimes \mathcal{E}_{L_0}(\tau) \rangle\rangle + L'(0, \xi_{1-n/2}(f), \theta_{\Lambda})  \right). \end{aligned}\end{equation}

\begin{remark}\label{boundary} 

When the Shimura variety $X = X_K$ is not compact, we can add suitable boundary components 
$\mathcal{C}(f)$ to the divisor $\mathcal{Z}(f)$ as in $(\ref{Z^c(f)})$ to get an arithmetic divisor
\begin{align*} \widehat{\mathcal{Z}}^c(f) = (\mathcal{Z}^c(f), \Phi) \in \widehat{\operatorname{Ch}}^1(\mathcal{X}^{\star})\end{align*}
on the integral model $\mathcal{X}^{\star}$ of the compactification $X^{\star}$. See \cite[$\S$5-7]{BY} for more details. 
When $f \in H_{1-n/2}(\omega_L)$ is not cuspidal, we also have to work with generalized arithmetic Chow groups in the sense of \cite{BKK}. \end{remark}

Suppose that $f = f^+ + f^- \in H_{1-n/2}(\omega_L)$ has integral holomorphic part $f^+$, so that the Fourier coefficients 
$c_f^+(\mu, m)$ are integers for all $m \in {\bf{Q}}$ and $\mu \in L^{\vee}/L$. The divisor $Z(f) \subset X$ defined in $(\ref{Z(f)})$ has an extension
\begin{align*} \mathcal{Z}(f) = \sum\limits_{\mu \in L^{\vee}/L} \sum\limits_{m \in {\bf{Q}} \atop m > 0} c_f^+(\mu, -m) \mathcal{Z}(\mu, m)\end{align*}
to a Cartier divisor on the integral model $\mathcal{X}$, and with corresponding arithmetic divisor 
\begin{align*} \widehat{\mathcal{Z}}(f) = (\mathcal{Z}(f), G_{\mathcal{Z}(f)}) \in \widehat{\operatorname{Ch}}^1(\mathcal{X}). \end{align*}

If $f \in \operatorname{ker}(\xi_{1-n/2}) \cong M_{1-n/2}^!(\omega_L)$ is weakly holomorphic, 
then we expect $\widehat{Z}(f) = (\mathcal{Z}(f), \Phi(f, \cdot))$ to be rationally equivalent to a torsion element, 
by the relation given by the Borcherds lift $\Psi(f, \cdot)$ described in $(\ref{BF})$ above (Theorem \ref{Borcherds}). 
If this rational equivalence to zero were known to be true, we would derive the 
corresponding vanishing of the Faltings height the relation 
\begin{align}\label{CV} \left[ \widehat{Z}(f), \mathcal{Z}(V_0) \right] 
= \left[ \mathcal{Z}(f), \mathcal{Z}(V_0) \right]_{\operatorname{fin}} + \Phi(f, Z(V_0)) = 0, \end{align}
from which it would follow that the nonarchimedean height pairing is given by the constant coefficient term 
\begin{align}\label{relation} \left[ \mathcal{Z}(f), \mathcal{Z}(V_0) \right]_{\operatorname{fin}}  &= - \frac{\deg(V_0)}{2} \cdot 
\operatorname{CT} \langle\langle f(\tau), \theta_{L_0^{\perp}}(\tau) \otimes \mathcal{E}_{L_0}(\tau) \rangle\rangle. \end{align}
Expanding out both sides of this relation $(\ref{relation})$ leads to the following expectation for the general case.

\begin{conjecture}[Bruinier-Yang]\label{BYC}

Let $f = f^+ + f^- \in H_{1-n/2}(\omega_L)$ be any weakly harmonic Maass form whose holomorphic part has integral Fourier coefficients,
\begin{align*} f^+(\tau) &= \sum\limits_{\mu \in L^{\vee}/L} 
\sum\limits_{m \in {\bf{Q}} \atop m \gg - \infty} c^+_f(\mu, m) e(m \tau) {\bf{1}}_{\mu}, \quad c^+_f(\mu, m) \in {\bf{Z}}.\end{align*}
We have for each coset $\mu \in L^{\vee}/L$ and each positive rational $m \in Q(\mu) + {\bf{Z}}$ that the nonarchimedean
local height $\left[ \mathcal{Z}(\mu, m), \mathcal{Y} \right]_{\operatorname{fin}}$ is given by the $(\mu, m)$-th Fourier coefficient
of the modular form $\theta_{L_0^{\perp}} (\tau) \otimes \mathcal{E}_{L_0}(\tau)$, 
\begin{align*} \left[ \mathcal{Z}(\mu, m), \mathcal{Z}(V_0) \right]_{\operatorname{fin}} &=  
- \sum\limits_{ \mu_1 \in (L_0^{\perp})^{\vee} / L_0^{\perp}, \mu_2 \in L_0^{\vee}/L_0 \atop \mu_1 + \mu_2 \equiv \mu \bmod L} 
\sum\limits_{m_1, m_2 \in {\bf{Q}}_{\geq 0} \atop m_1 + m_2 = m } r_{L_0^{\perp}}(\mu_1, m_1) \kappa_{L_0}(\mu_2, m_2). \end{align*}
Putting this into $(\ref{BYf})$ (Theorem \ref{BY4.7}), taking the sum over $m >0$, we obtain the arithmetic height formula 
\begin{align}\label{AHF} \left[ \widehat{\mathcal{Z}}(f), \mathcal{Z}(V_0)\right] = \left[ \widehat{\mathcal{Z}}(f), \mathcal{Z}(V_0)  \right]_{\operatorname{Fal}} 
&= - \frac{\deg(Z(V_0))}{2} \cdot \left( c_f^+(0, 0) \cdot \kappa_{L_0}(0,0) + L'(0, \xi_{1-n/2}(f), \theta_{L_0^{\perp}}) \right). \end{align} \end{conjecture}

\begin{theorem}[Bruinier-Yang, Andreatta-Goren-Howard-Madapusi Pera]\label{BYAGHMP}  

Let $(V_0, Q_0) = (\mathfrak{a}_{\bf{Q}}, -Q_{\mathfrak{a}}(\cdot))$ be a rational quadratic subspace of signature $(0, 2)$
given by $\mathfrak{a}_{\bf{Q}}$ in an imaginary quadratic field $k(V_0)$ of odd discrimant $d_{k(V_0)}$
determined by a nonzero integral ideal $\mathfrak{a} \subset \mathcal{O}_{k(V_0)}$. Let $L_0 = \mathfrak{a}$ denote the corresponding lattice.
Assume that the even part $C^0(L_0)$ of the Clifford algebra $C(L_0)$ is identified with the maximal order $\mathcal{O}_{k(V_0)} \cong C^0(V_0)$. 
Then, Conjecture \ref{BYC} is true. In particular, the arithmetic height formula $(\ref{AHF})$ is true. 
Equivalently, writing $\widehat{\bf{T}} \in \widehat{\operatorname{Ch}}^1(\mathcal{X})$ 
to denote the metrized cotautological defined in \cite[$\S$5.3]{AGHMP}, 
\begin{align*} \left[ \widehat{\mathcal{Z}}(f): \mathcal{Z}(V_0) \right] + c_f^+(0,0) \cdot \left[ \widehat{\bf{T}} : \mathcal{Z}(f) \right]
&= - \frac{h_k}{w_k} \cdot L'(0, \xi_{1-n/2}(f) \times \theta_{L_0^{\perp}}). \end{align*} \end{theorem}

\begin{proof} This follows from the combined results of \cite[Theorem 1.2]{BY} and \cite[Theorem A, Theorem 5.7.3]{AGHMP}. \end{proof}

\begin{remark} Note that Conjecture \ref{BYC} is not yet established in general; see \cite[Conjectures 5.1 and 5.2]{BY}. 
That is, the conjecture is posed more generally for $(V_0, Q_0)$ any negative definite quadratic subspace of signature $(0, 2)$. 
In particular, while it should be possible to take $(V_0, Q_0) = (\mathfrak{a}, -Q_{\mathfrak{a}}(\cdot))$ with  $C^+(L_0) \cong \mathcal{O}$ 
any (non-maximal) order $\mathcal{O} \subset \mathcal{O}_{k(V_0)}$, and without any condition on the partity of the discriminant $d_{k(V_0)}$, this is not yet known. 
It is for this reason that we restrict to class group twists $\chi \in C(\mathcal{O}_k)^{\vee}$ in this work. \end{remark}

\subsection{Arithmetic heights of Hirzebruch-Zagier divisors on $X_0(N) \times X_0(N)$}

We now return to the quadratic spaces $(V_A, Q_A)$ of signature $(2,2)$ parametrizing $X_{K_A} = Y_0(N) \times Y_0(N)$.
Here, we give a geometric interpretation of the formulae of Theorem \ref{RSIP} and Corollary \ref{simplified} above.
In case (i) where $k$ is imaginary quadratic, this will give a new integral presentation for the $L$-values described in 
Theorem \ref{Gross-Zagier}, including a comparison of the arithmetic heights of Heegner divisors on $X_0(N)$ and the corresponding arithmetic heights 
of Hirzebruch-Zagier divisors on $X_0(N) \times X_0(N)$. To be clear, we are going to deduce from a special case of the higher Gross-Zagier formula for the product of modular curves
$X_0(N) \times X_0(N)$ with the integral presentations of Proposition \ref{RSIP} and Corollary \ref{simplified} a novel integral presentation for the central derivative values 
$\Lambda'(E/k, \chi, 1) = \Lambda'(1/2, \phi \times \theta(\chi))$ in terms of Faltings heights of arithmetic Hirzebruch-Zagier divisors on $X_0(N) \times X_0(N)$.
For the convenience of the reader, we also explain in Appendix A how the Gross-Zagier formula \cite[Theorem I (6.3)]{GZ} can be derived by a variation of the proof of 
Bruinier-Yang \cite[Theorem 7.7]{BY}, developing Theorem \ref{BYAGHMP} for the special case of signature $(1,2)$ described in Example \ref{KSAVn=1}. 

\subsubsection{Setup}

For each class $A \in C(\mathcal{O}_k)$, recall that we fix a representative $\mathfrak{a} \subset \mathcal{O}_k$, 
and consider the quadratic space $(V_A, Q_A)$ of signature (2,2) defined by  
$V_A = \mathfrak{a}_{\bf{Q}} \oplus \mathfrak{a}_{\bf{Q}}$ and quadratic form $Q_A(z_1, z_2) = Q_{\mathfrak{a}}(z_1) - Q_{\mathfrak{a}}(z_2)$
(for $Q_{\mathfrak{a}}(z) := {\bf{N}}_{k/{\bf{Q}}}(z)/ {\bf{N}} \mathfrak{a}$ the norm form),
so that $(V_{A,0}, Q_{A,0}) = (\mathfrak{a}, -Q_{\mathfrak{a}})$ determines a quadratic space of signature $(0,2)$. 
Recall we have an accidental isomorphism $\operatorname{GSpin}(V_A) \cong  \operatorname{GL}_2 \times_{{\bf{G}}_m} \operatorname{GL}_2$ of algebraic 
groups over ${\bf{Q}}$ by Proposition \ref{Clifford2}, and that we take $L_A =L_A(N) \subset V_A$ to be the lattice  
corresponding to the compact open subgroup $K_A = K_{L_A} \cong K_0(N)^2 \subset \operatorname{GL}_2(\widehat{\bf{Z}})^2$ 
so that $X_A \cong Y_0(N) \times Y_0(N)$ as in Corollary \ref{lattices}. 
Consider the integral model $\mathcal{X}_A \cong \mathcal{Y}_0(N) \times \mathcal{Y}_0(N)$.
Fix a compactification $X_A^{\star} \cong X_0(N) \times X_0(N)$ (see \cite[$\S$1.2 and 2.4]{Br-123}).
We can identify the corresponding integral model $\mathcal{X}_A^{\star}$ with $\mathcal{X}_0(N) \times \mathcal{X}_0(N)$. 

\subsubsection{Moduli description}

Recall (see e.g.~\cite{GZ}) that for any scheme $S$ over {\bf{Q}}, $Y_0(N)(S)$ represents the isomorphism class of triples $(E, E', \varphi)$
consisting of a pair of elliptic curves $E/S$, $E'/S$ and an isogeny $\varphi: E \rightarrow E'$ of degree $N$.
Hence, $\varphi$ is finite flat of degree $N$, and its kernel $\ker(\varphi) \cong {\bf{Z}} / N{\bf{Z}}$ is a finite locally free group scheme over $S$.
The compactified modular curve $X_0(N)(S)$ represents triples $(E, E', \varphi)$ of generalized elliptic curves $E/S$, $E'/S$ and a cyclic isogeny $\varphi: E \rightarrow E'$ of degree $N$. 
Hence, $X_A(S) = Y_0(N)(S) \times Y_0(N)(S)$ represents pairs of triples $(E_1, E_1', \varphi_1)$, $(E_2, E_2', \varphi_2)$. 
More precisely, $Y_0(N)(S) \times Y_0(N)(S)$ represents the 
moduli space of triples ${\bf{A}} = (A, \kappa_A(\varphi), \lambda_A)$ made up of the abelian surface $A = E_1 \times E_2$, the endomorphism
$\kappa_A(\varphi)$ of $A$ determined by the isogeny $\varphi = \varphi_1 \times \varphi_2: A \rightarrow A'$ for $A' = E_1' \times E_2'$ 
and its dual $\varphi^{\vee}: A' \rightarrow A$, and the product principal polarization $\lambda_A = \lambda_{E_1} \times \lambda_{E_2}$.
Similarly, $X_0(N)(S) \times X_0(N)(S)$ represents the 
moduli space of triples ${\bf{A}} = (A, \kappa_A(\varphi), \lambda_A)$ made up of the abelian surface $A = E_1 \times E_2$
with $E_1/S$, $E_2/S$ generalized elliptic curves, special endomorphism $\kappa_A(\varphi) \in \operatorname{End}(A)$
determined by the isogeny $\varphi: A \rightarrow A'$ and its dual $\varphi^{\vee}: A'\rightarrow A$, and the product principal 
polarization $\lambda_A$. We can then describe the special arithmetic (Hirzebruch-Zagier) divisor $Z_A(\mu, m)$ in either 
of these spaces as the moduli of triples 
\begin{align*}(A, \kappa_A(x), \lambda_A) = (E_1 \times E_2, \kappa_{E_1 \times E_2}(x), \lambda_{E_1} \times \lambda_{E_2})\end{align*}
with endomorphism $\kappa_A(x)$ of degree $\deg(\kappa_A(\varphi)) = m$ supported on $\mu + L_A$. 

We also have the following modular description of the CM cycles $\mathcal{Z}(V_{A,0})$ (cf.~\cite[Proposition 7.2]{BY}, \cite{GZ}).
We can identify the CM cycles $\mathcal{Z}(V_{A,0})$ 
with Heegner divisors corresponding in the moduli descriptions of $Y_0(N)$ and $X_0(N)$ to a triple $(E, E', \varphi)$ consisting 
of elliptic curves $E$ and $E'$ with complex multiplication by $\mathcal{O}_k$ and a cyclic isogeny $\varphi: E \rightarrow E'$
of degree $N$ annihilated by a primitive ideal $\mathfrak{n} = [N, (r+\sqrt{d_k})/2]$. We then deduce that the CM
cycles $\mathcal{Z}(V_{A,0})$ we consider will correspond to triples ${\bf{A}} = (A, \kappa_A(\varphi), \lambda_A)$ with 
$A = E \times E$ the self-product of the elliptic curve $E$ with CM by $\mathcal{O}_k$, and $\kappa_A(\varphi)$ the endomorphism
corresponding to the cyclic isogeny $\varphi: E \rightarrow E'$ and its dual $\varphi^{\vee}: E' \rightarrow E$.
We refer to the discussions in \cite{vdG} and \cite{HY} for a more general description of these moduli spaces 
of abelian surfaces with special endomorphisms.

\subsubsection{Arithmetic height formula}

As in Remark \ref{boundary}, we extend each arithmetic divisor 
\begin{align*} \widehat{\mathcal{Z}}_A(f) = (\mathcal{Z}_A(f), \Phi(f, \cdot)), \quad
\widehat{\mathcal{Z}}_A(\mu, m) = (\mathcal{Z}_A(\mu, m), \Phi_{\mu, m}(\cdot))
\in \widehat{\operatorname{Ch}}^1 \left( \mathcal{Y}_0(N) \times \mathcal{Y}_0(N) \right) \end{align*}
to the compatification
\begin{align}\label{Hilbertboundary} \widehat{\mathcal{Z}}_A^c(f) := (\mathcal{Z}_A^c(f), \Phi(f, \cdot)), \quad
\widehat{\mathcal{Z}}_A^c(\mu, m) := (\mathcal{Z}_A^c(\mu, m), \Phi_{\mu, m}(\cdot))
\in \widehat{\operatorname{Ch}}^1 \left( \mathcal{X}_0(N) \times \mathcal{X}_0(N) \right).\end{align}
We derive the following consequence of Theorem \ref{BYAGHMP} in this setting, 
using Proposition \ref{RS-equiv} and Theorem \ref{RSIP}.

\begin{theorem}\label{main}

Let $\phi \in S_2^{\operatorname{new}}(\Gamma_0(N))$ be a cuspidal newform of level $N$ and trivial character. 
Let $k$ be an imaginary quadratic field of odd discriminant $d_k$ and (odd) quadratic Dirichlet character $\eta_k(\cdot) = (\frac{d_k}{\cdot})$.
Assume that $(N, d_k)=1$, and that $\eta_k(-N) = - \eta_k(N) = -1$. 
Let $g_{\phi, A} \in S_2(\overline{\omega}_{L_A})$ denote the vector-valued lift of $\phi$ described lift in Corollary \ref{existence}. 
Let $f_{0, A} \in H_0(\omega_{L_A})$ be a harmonic weak Maass form for which 
\begin{align*} \xi_0(f_{0,A})(\tau) &= g_{\phi, A}(\tau) \in S_2(\overline{\omega}_{L_A}), \end{align*}
where $\xi_0: H_0(\omega_{L_A}) \rightarrow S_2(\overline{\omega}_{L_A})$ denotes the antilinear differential operator defined in $(\ref{xi})$.
Then, for any class group character $\chi \in C(\mathcal{O}_k)^{\vee}$, we have the central derivative value formula 
\begin{align*} \Lambda'(1/2, \phi \times \theta(\chi)) &= - 2 \sum\limits_{A \in C(\mathcal{O}_k)} \chi(A) 
\left[ \widehat{\mathcal{Z}}_A(f_{0, A}): \mathcal{Z}(V_{A,0})\right] \end{align*}
for the completed Rankin-Selberg $L$-function $\Lambda(s, \phi \times \theta(\chi))$ of $\phi$ times
the Hecke theta series $\theta(\chi)$, where each term on the right-hand side denotes the arithmetic height of the arithmetic special divisor 
\begin{align*} \widehat{\mathcal{Z}}_A(f_{0, A}) 
&= \sum\limits_{\mu \in L_A^{\vee}/L_A} \sum\limits_{m \in {\bf{Q}} \atop m > 0} c_{f_{0, A}}^+(\mu, -m) \mathcal{Z}_A(\mu, m) \end{align*} on the 
integral model $\mathcal{X} = \mathcal{Y}_0(N) \times \mathcal{Y}_0(N)$ of 
$X = Y_0(N) \times Y_0(N)$ evaluated along the corresponding CM cycle 
$\mathcal{Z}(V_{A,0}) \subset \mathcal{X} = \mathcal{Y}_0(N) \times \mathcal{Y}_0(N)$. Here, each $\mathcal{Z}_A(\mu, m)$
is the Hirzebruch-Zagier divisor 
$\widehat{\mathcal{Z}}_A(\mu, m) = (\mathcal{Z}_A(\mu, m), \Phi^{L_A}_{\mu, m})$
on $\mathcal{X} = \mathcal{Y}_0(N) \times \mathcal{Y}_0(N)$ described above.
We can also extend arithmetic divisors to the compactification 
$\mathcal{X}^{\star} \cong \mathcal{X}_0(N) \times \mathcal{X}_0(N)$ 
as described in $(\ref{Hilbertboundary})$ to get the corresponding formula 
\begin{align*} \Lambda'(1/2, \phi \times \theta(\chi)) &= - 2 \sum\limits_{A \in C(\mathcal{O}_k)} \chi(A) 
\left[ \widehat{\mathcal{Z}}_A^c(f_{0, A}): \mathcal{Z}(V_{A,0})\right]. \end{align*}

\end{theorem}

\begin{proof} 

For each class $A \in C(\mathcal{O}_k)$, we have by Theorem \ref{BYAGHMP} the arithmetic height formula
\begin{align*} \left[ \widehat{Z}_A(f_{0, A}): \mathcal{Z}(V_{A, 0}) \right] 
&= - \frac{h_k}{w_k} \cdot L'(0, \xi_{1-n/2}(f_{0, A}) \times \theta_{L_{A,0}^{\perp}})  
= - \frac{h_k}{w_k} \cdot L'(0, g_{ \phi, A } \times \theta_{L_{A,0}^{\perp}}).\end{align*}
Using the relation $(\ref{CCDV})$ from the proof of Theorem \ref{RSIP}, we can extend this to the completed Rankin-Selberg $L$-function
$L^{\star}(s, g_{\phi, A} \times \theta_{ L_{A,0}^{\perp} }) = \Lambda(1+s, \eta_k) L(s, g_{\phi, A} \times \theta_{ L_{A,0}^{\perp}})$
to get the corresponding formula 
\begin{align*} \Lambda(1, \eta_k )\left[ \widehat{Z}_A(f_{0, A}): \mathcal{Z}(V_{A, 0}) \right] 
&= - \frac{h_k}{w_k} \cdot \Lambda(1, \eta_k) L'(0, g_{\phi, A} \times \theta_{L_{A,0}^{\perp}})  
= - \frac{h_k}{w_k} \cdot L^{\star \prime}(0, g_{ \phi, A } \times \theta_{L_{A,0}^{\perp}}).\end{align*}
By Proposition \ref{RS-equiv} (i), this is the same as 
\begin{align*} \Lambda(1, \eta_k )\left[ \widehat{Z}_A(f_{0, A}): \mathcal{Z}(V_{A, 0}) \right] 
&= - \frac{h_k}{w_k} \cdot \Lambda^{\prime}(1/2, \phi \times \theta_A),\end{align*}
and which after evaluating $\Lambda(1, \eta_k) = 2 h_k/w_k$ via $(\ref{CDirichlet})$ then dividing each side by $-h_k/w_k$ is the same as 
\begin{align*} -2 \left[ \widehat{Z}_A(f_{0, A}): \mathcal{Z}(V_{A, 0}) \right] &= \Lambda^{\prime}(1/2, \phi \times \theta_A). \end{align*}
Taking the twisted sum $\Lambda'(1/2, \phi \times \theta(\chi)) = \sum\limits_{A \in C(\mathcal{O}_k)} \chi(A) \Lambda'(1/2, \phi \times \theta_A)$
then gives the stated formula. \end{proof}

\begin{corollary}\label{ec}

Let $E$ be an elliptic curve defined over ${\bf{Q}}$, parametized by a newform 
$\phi = \phi_E \in S_2(\Gamma_0(N))$, so that the Hasse-Weil $L$-function $L(E,s)$ has an analytic continuation 
$\Lambda(E, s) = \Lambda(s-1/2, \phi)$ given by a shift of the standard $L$-function $\Lambda(s, \phi) = L_{\infty}(s, \phi)L(s, \phi)$ of $\phi$.
Let $k$ be an imaginary quadratic field of odd discriminant $d_k$ and (odd) quadratic Dirichlet character $\eta_k(\cdot) = (\frac{d_k}{\cdot})$.
Assume that $(N, d_k)=1$, and that $\eta_k(-N) = - \eta_k(N) = -1$. 
Then, for any class group character $\chi \in C(\mathcal{O}_k)^{\vee}$, we have the formula 
\begin{align*} \Lambda'(E/k, \chi, 1) &= - 2 \sum\limits_{A \in C(\mathcal{O}_k)} \chi(A) 
\left[ \widehat{\mathcal{Z}}_A(f_{0, A}): \mathcal{Z}(V_{A,0})\right] \end{align*}
for the central derivative value of the $L$-function $\Lambda(E/k, \chi, s) = \Lambda(s-1/2, \phi \times \theta(\chi))$ of $E$ over $K$ twisted by $\chi$
in terms of arithmetic divisors on $\mathcal{Y}_0(N) \times \mathcal{Y}_0(N) \longrightarrow \operatorname{Spec}({\bf{Z}})$.
Extending to the compactification $\mathcal{X}_0(N) \times \mathcal{X}_0(N) \longrightarrow \operatorname{Spec}({\bf{Z}})$, we also have the central derivative value formula
\begin{align*} \Lambda'(E/k, \chi, 1) &= - 2 \sum\limits_{A \in C(\mathcal{O}_k)} \chi(A) 
\left[ \widehat{\mathcal{Z}}_A^c(f_{0, A}): \mathcal{Z}(V_{A,0})\right]. \end{align*} \end{corollary}

Note that by comparing with the formula of Gross-Zagier \cite[Theorem I (6.3)]{GZ}, we also obtain a relation of 
arithmetic heights on $\mathcal{X}_0(N)$ and $\mathcal{X}_0(N) \times \mathcal{X}_0(N)$. We explain this in more detail in Appendix A below. 

\subsection{Relations to Birch-Swinnerton-Dyer constants and periods}

We now return to Theorem \ref{BSD}.

\begin{proof}[Proof of Theorem \ref{BSD}] Cf.~\cite[Theorem 5.1]{VO}. In either case on $k$, we use the product rule with the Artin decomposition 
$\Lambda(E/k, s) = \Lambda(E,s)\Lambda(E^{(d_k)}, s) = \Lambda(E/{\bf{Q}},s)\Lambda(E^{(d_k)}/{\bf{Q}}, s)$ to compute 
\begin{align*} \Lambda'(E/k, 1) &=  \Lambda'(E, 1)\Lambda(E^{(d_k)}, 1) + \Lambda'(E^{(d_k)},1)\Lambda(E, 1), \end{align*} equivalently 
\begin{align*} \Lambda'(1/2, \Pi(\pi)) &=  \Lambda'(1/2, \phi)\Lambda(1/2, \phi \otimes \eta_k) + \Lambda'(1/2, \phi \otimes \eta_k)\Lambda(1/2, \phi), \end{align*}
where exactly one of the summands on the right vanishes. For the nonvanishing summand, we can take for granted the refined conjecture
of Birch and Swinnerton-Dyer up to powers of $2$ and $3$ using the combinations of various theorems on the Iwasawa main conjectures and subsequent Euler
characteristic calculations; see \cite{BST} and \cite{BST2} for details. In brief, we use the combined works of 
Kato \cite{Ka}, Kolyvagin \cite{Ko}, Rohrlich \cite{Ro}, and Skinner-Urban \cite{SU} to establish the cyclotomic main 
conjectures\footnote{We also consider the anticyclotomic main conjecture for the rank one factor, after passing to an imaginary quadratic field,
then descend back down to ${\bf{Q}}$ using converse theorems and Euler characteristic calculations -- see \cite{BST} and \cite{BST2} for a desciption of
the state of the art.}, followed by the Euler characteristic calculations of Burungale-Skinner-Tian \cite{BST}, \cite{BST2}, 
and Castella \cite{Ca} for the rank zero factor, and those of Jetchev-Skinner-Wan \cite{JSW}, Skinner-Zhang \cite{SkZh}, and Zhang \cite{WZ} 
for the rank one factor. This allows us to deduce that $\Lambda'(E/k, 1) \approx \kappa_E({\bf{Q}}) \cdot \kappa_{E^{(d_k)}}({\bf{Q}})$. 
We then identify the central derivative values according to Theorem \ref{RSIP} and Corollary \ref{simplified}, as well as Theorem \ref{main}
and Corollary \ref{ec} when $k$ is imaginary quadratic. \end{proof}

Recall that a complex number $\alpha = \sigma + it$ is said to be a period if its real and imaginary parts $\sigma$ and $it$
can be expressed as integrals of rational functions, over domains in ${\bf{R}}^n$ given by polynomial inequalities with rational 
coefficients. We write $\mathcal{P} \subset {\bf{C}}$ to denote the set of all such numbers.
We refer to the paper of Kontsevich-Zagier \cite{KZ} for a definitive expository account of this countable subring 
$\mathcal{P}$ of {\bf{C}}, which contains the algebraic numbers $\overline{{\bf{Q}}}$ and their logarithms 
$\log \overline{\bf{Q}}$ (for instance). This paper also describes the conjecture of Birch-Swinnerton-Dyer from
this perpective, including the conjecture \cite[Question 4]{KZ} that the central derivative value 
$\Lambda^{(r_E(k))}(E/k, 1)$ should lie in the ring of periods $\mathcal{P}$. 
Assuming the finiteness of the Tate-Shafarevich group, the argument of \cite[$\S 3.5$]{KZ}
shows that the Birch-Swinnerton-Dyer constant $\kappa_E({\bf{Q}}) \in \mathcal{P}$ is a period.
The same argument works for the more general setting of number fields, to show that $\kappa_E(k) \in \mathcal{P}$.

\begin{corollary}\label{periods} We retain the setup of Theorem \ref{BSD}. \\

\begin{itemize}

\item[(i)] If $k$ is imaginary quadratic with $\eta_k(-N) = -\eta_k(N) = -1$, then the central derivative value
\begin{align*} \Lambda'(E/k, 1) %&=- \frac{1 }{2} \sum\limits_{A \in C(\mathcal{O}_k)} \Phi(f_{0, A}, Z(V_{A, 0})) 
= -2 \sum\limits_{A \in C(\mathcal{O}_k)} \left[ \widehat{\mathcal{Z}}_A^c(f_{0, A}): \mathcal{Z}(V_{A,0})\right] \end{align*}
lies in the ring of periods $\mathcal{P}$. \\

\item[(ii)] If $k$ is real quadratic with $\eta_k(-N) = \eta_k(N) = -1$, then the central derivative value 
\begin{align*} &\Lambda'(E/k, 1) =-2 \log(\varepsilon_k) h_k  \\ &\times\sum\limits_{A \in C(\mathcal{O}_k)} 
\left[ \left( \frac{ \operatorname{vol}(K_{A,W}) }{2 } \right) \Phi(f_{0, A}, \mathcal{G}(W_A)) 
+ \operatorname{CT} \langle \langle f_{0, A}^+(\tau), \theta_{L_{A, W}^{\perp}}(\tau) \otimes \mathcal{E}_{L_{A, W}}(\tau) \rangle \rangle
+ I'(0, f_{0,A} \times L_0(\theta_{L_{A,W}^{\perp}})) \right] \end{align*}
lies in the ring of periods $\mathcal{P}$.

\end{itemize}

\end{corollary}

\begin{proof} We use the argument of \cite[$\S 3.5$]{KZ} to deduce that $\kappa_E(k) \in \mathcal{P}$, up to powers of $2$ and $3$.
That is, we take for granted the conditions of Theorem \ref{BSD} so that various theorems on Iwasawa main conjectures
described above allow us to bound the $p$-primary Tate-Shafarevich groups $\Sha(E/{\bf{Q}})[p^{\infty}]$ and 
$\Sha(E^{(d_k)}/{\bf{Q}})[p^{\infty}]$ for primes $p \geq 5$ through the corresponding bounds for the $p$-primary Selmer groups. 
The argument \cite[$\S$ 3.5]{KZ} then shows that each of the corresponding Birch-Swinnerton-Dyer constants $\kappa_E({\bf{Q}})$
and $\kappa_{E^{(d_k)}}({\bf{Q}})$, up to powers of $2$ and $3$, lies in the ring of periods $\mathcal{P}$. \end{proof}

\appendix

\section{Gross-Zagier via the signature (1,2) setting}

We now explain how a variation of the argument of Bruinier-Yang \cite[Theorem 7.7 and Corollary 7.8]{BY} 
can be used to deduce the full Gross-Zagier formula \cite[Theorem I (6.3)]{GZ} (Theorem \ref{Gross-Zagier} above), for twists by class group characters. 
This generalizes \cite[Theorem 7.7]{BY}, which recovers the formula of \cite[Theorem I(6.3)]{GZ} for the case of trivial/principal class group character 
$\chi_0 = {\bf{1}} \in C(\mathcal{O}_k)^{\vee}$. Although perhaps well-known to experts, we include details for lack of reference, as well as to compare with Theorem \ref{main} and Corollary \ref{ec}

\subsection{$X_0(N)$ as spin Shimura variety}

See \cite[$\S 2.4$]{BO} and \cite[$\S$7.3]{BY}. Fix an integer $N \geq 1$. Let $(V, Q)$ be the rational quadratic space with underlying vector space  
\begin{align*} V = \operatorname{Mat}_{2 \times 2}^{\operatorname{tr}=0}({\bf{Q}}) \end{align*} 
given by $2 \times 2$ matrices with rational coordinates and trace zero, and quadratic form given by 
$Q(x) = N \det(x)$. The bilinear form is then given by $ - N \operatorname{tr}(xy)$ for $x, y \in V$. This quadratic space $(V, Q)$ has signature $(1, 2)$. 
The group $\operatorname{GL}_2({\bf{Q}})$ acts on the trace zero matrices $V$ by conjugation 
$\gamma \cdot x = \gamma x \gamma^{-1}$ for $x \in V$ and $\gamma \in \operatorname{GL}_2({\bf{Q}})$. This action
leaves the form $Q$ invariant, and induces isomorphisms 
$\operatorname{GSpin}(V) \cong \operatorname{GL}_2$ and  $\operatorname{Spin}(V) \cong \operatorname{SL}_2$
of algebraic groups over ${\bf{Q}}$. The Grassmannian $D(V) = D^{\pm}(V)$ can be identified with $\mathfrak{H} \cup \overline{\mathfrak{H}}$ via the map
\begin{align*} z = x + i y \in \mathfrak{H} \longmapsto {\bf{R}} \Re \left( \begin{array}{cc} z & -z^2 \\ 1 & -z \end{array} \right) 
+ {\bf{R}} \Im \left( \begin{array}{cc} z & -z^2 \\ 1 & -z \end{array} \right) \in D(V). \end{align*}
Note that $\operatorname{GSpin}(V)({\bf{R}})$ acts on $D(V) \cong \mathfrak{H} \cup \overline{\mathfrak{H}}$ by fractional linear transformation. 
The congruence subgroup $\Gamma_0(N) \subset \operatorname{SL}_2({\bf{Z}})$ determines both a lattice $L \subset V$ and a 
compact open subgroup $K = K_L = \prod_p K_p$ of $\operatorname{GSpin}(V)({\bf{A}}_f)$. To be more concrete, we consider 
\begin{align}\label{Ln=1} L &= \left\lbrace \left( \begin{array}{cc} b & - a/N \\ c & - b \end{array} \right) : a, b, c \in {\bf{Z}} \right\rbrace,
\quad L^{\vee} = \left\lbrace \left( \begin{array}{cc} b/2N & -a/N \\ c & -b/2N \end{array} \right) : a, b, c \in {\bf{Z}} \right\rbrace. \end{align}
We have a natural identification of the corresponding discriminant group
\begin{align}\label{natural} ({\bf{Z}}/ 2 N {\bf{Z}}) \cong L^{\vee}/L, \quad r \longmapsto \mu_r := \left( \begin{array}{cc} r/2N & ~~\\~& - r/2N \end{array} \right). \end{align}
The lattice $L \subset V$ has level $4N$, and the quadratic form on $L^{\vee}/L$ can be identified with $x \mapsto -x^2$.
The corresponding compact open subgroup $K = K_L \subset \operatorname{GSpin}(V)({\bf{A}}_f) \cong \operatorname{GL}_2({\bf{A}}_f)$
is given by $K= \prod_p K_p$, with each $K_p \subset \operatorname{GSpin}(V)({\bf{Z}}_p) \cong \operatorname{GL}_2({\bf{Z}}_p)$ defined by 
\begin{align*} K_p &= \left\lbrace \left( \begin{array}{cc} a & b \\c & d \end{array} \right) \in \operatorname{GL}_2({\bf{Z}}_p) : c \in N {\bf{Z}}_p \right\rbrace. \end{align*}
In this way, we obtain the identification of Shimura curves 
\begin{align}\label{ncMC} Y_0(N) = \Gamma_0(N) \backslash \mathfrak{H} &\longrightarrow X_K({\bf{C}}) 
= \operatorname{GSpin}(V)({\bf{Q}}) \backslash D(V) \times \operatorname{GSpin}(V)({\bf{A}}_f)/K, 
\quad \Gamma_0(N) z \longmapsto \operatorname{GSpin}(V)({\bf{Q}}) [z, 1] K. \end{align}

\subsection{Heegner divisors as special divisors}

\subsubsection{Special divisors and CM cycles associated to the lattice $L$}

We have the following correspondence between the special divisors $Z(\mu, m)$ 
and Heegner divisors $P_{D, r}$ described\footnote{To be more precise, let $\tau \in \mathfrak{H}$ be a root of the quadratic equation $a \tau^2 + b \tau + c =0$
for $a, b, c \in {\bf{Z}}$, $a>0$, $a \equiv 0 \bmod N$, $b \equiv r \bmod 2N$ and $D = b^2 - 4ac$. The image $\tau_{a, b, c}$ of such a root 
in $X_0(N)$ is rational over the Hilbert class field $k[1]$ of the imaginary quadratic field $k$ of discriminant $D$, and the
Galois group $\operatorname{Gal}(k[1]/k) \cong C(\mathcal{O}_k)$ permutes these images simply transitively. We then define 
$P_{D, r} = \frac{w_k}{2} \sum_{[a, b, c] \in \mathcal{Q}_D \cong C(\mathcal{O}_k)} \tau_{a, b, c}$ as $\frac{w_k}{2}$ times the sum of these 
$h_k$ points. In the moduli description of $X_0(N)$, this point $P_{D,r}$ corresponds to a triple $(E, E', \varphi)$ of elliptic curves $E$ and $E'$
with complex multiplication by $\mathcal{O}_k$ and $\varphi: E \rightarrow E'$ is an isogeny of kernel annihilated by the primitive ideal 
$\mathfrak{n} = [N, (r + \sqrt{D}) /2] $ of norm $N$.} 
 in Gross-Kohnen-Zagier \cite{GKZ}.
Given $m \in {\bf{Q}}_{>0}$ and a coset $\mu \in L^{\vee}/L$ such that $Q(\mu) \equiv m \bmod 1$, we again consider the hyperboloid
\begin{align*} \Omega_{\mu, m}({\bf{Q}}) &= \left\lbrace x \in \mu + L : Q(x) = m \right\rbrace. \end{align*}
Note that $\Omega_{\mu, m}({\bf{Q}})$ is empty unless $Q(\mu) \equiv m \bmod 1$.

Let us for each $m \in {\bf{Q}}_{>0}$ and $\mu \in L^{\vee}/L$ with $Q(\mu) \equiv m \bmod 1$ consider the fundamental discriminant 
\begin{align*} D = -4N m \in {\bf{Z}}. \end{align*} 
Given an integer $r \in {\bf{Z}}$ with coset representative $\mu = \mu_r$ under the natural bijection $(\ref{natural})$,
\begin{align*} \mu = \mu_r = \left( \begin{array}{cc} r/2N & ~~\\ ~~& -r/2N \end{array}  \right) \in \Omega_{\mu, r}({\bf{Q}}), \end{align*}
we have that $D \equiv r^2 \bmod 4N$. In this way, we produce a positive norm vector in the hyperboloid
\begin{align}\label{x} x = x(\mu, m) = x(\mu_r, -D/4N) 
&= \left( \begin{array}{cc} r/2N & 1/N \\ (D - r^2)/4N & - r/2N \end{array} \right) \in \Omega_{\mu, m}({\bf{Q}}). \end{align}
Conversely, given integers $D < 0$ and $r$ such that $D \equiv r^2 \bmod 4N$, let $m  = -D /4N$ and $\mu = \mu_r$. 
Observe that $m \in Q(\mu) + {\bf{Z}}$ is positive. As in \cite[$\S$7.1]{BY}, we take this identification for granted, 
and note that the corresponding special divisor $Z(\mu, m) = Z(\mu_r, -D/4N)$ 
defined in $(\ref{special})$ above can be identified with a sum of Heegner divisors $P_{D, r} + P_{D, -r}$ 
defined in Gross-Kohnen-Zagier \cite[IV.1(1)]{GKZ}. 
We remark that each of these Heegner divisors $P_{D, \pm r}$ has degree equal to one quarter of the Hurwitz class number $H(D) = 2h(D)/w(D)$,
\begin{align*} \deg (P_{D, \pm r}) &= \frac{H(D)}{4} = \frac{h(D)}{2w(D)}. \end{align*}
Here, $h(D)$ denotes the class number of the imaginary quadratic field ${\bf{Q}}(\sqrt{D})$, equivalently the cardinality 
of the class group of positive definite binary quadratic forms of discriminant $D$. 
We also write $w(D)$ to denote the number of roots of unity in ${\bf{Q}}(\sqrt{D})$. 
We deduce that for a pair $(\mu, m) = (\mu_r, -D/4N)$ corresponding to $(D, r)$ in this way, we have the relation 
\begin{align*} \deg Z(\mu, m)  = \deg Z(\mu_r, -D/4N) = \deg(P_{D, r} + P_{D, -r}) = \frac{H(D)}{2} = \frac{h(D)}{w(D)}. \end{align*} 

Fixing a vector $x = x(\mu,m) \in \Omega_{\mu, m}({\bf{Q}})$ as in $(\ref{x})$, we consider the positive and negative definite subspaces 
\begin{align*} V_+ &:= {\bf{Q}} x, \quad U := V \cap x^{\perp} \end{align*}
of $(V, Q) = (\operatorname{Mat}_{2 \times 2}^{\operatorname{tr=0}}({\bf{Q}}), N \cdot \det(\cdot))$, 
as well as the respective positive definite and negative definite lattices 
\begin{align*} \mathcal{P} := L \cap V_+, \quad \mathcal{N} := L \cap x^{\perp}. \end{align*}
Notice that we can present the negative definite lattice $\mathcal{N} \subset L$ more explicitly as 
\begin{align*} \mathcal{N} 
&= {\bf{Z}} \left( \begin{array}{cc} 1 & 0 \\ -r & 1 \end{array} \right) \oplus {\bf{Z}} \left( \begin{array}{cc} 0 & 1/N \\ (r^2-D)/4N & 0 \end{array} \right), \end{align*}
and also that $\mathcal{N}$ has determinant $-D$.  Writing $t = \gcd(r, 2N)$, 
the positive definite lattice $\mathcal{P} \subset L$ and its dual lattice $\mathcal{P}^{\vee}$ can be presented more explicitly as 
\begin{align*} \mathcal{P} 
&= {\bf{Z}} \left( \begin{array}{cc} r & 2 \\ (D-r^2)/2 & -r  \end{array} \right) = {\bf{Z}} \frac{2N}{t} x, 
\quad (L_0^{\perp})^{\vee}= \mathcal{P}^{\vee} = {\bf{Z}} \frac{t}{D} x. \end{align*}

Let us now consider the ideal $\mathfrak{n} = [N, (r+\sqrt{D})/2]$ of ${\bf{Z}}[(D + \sqrt{D})/2]$.
This ideal has norm $N$, and we can associate with it the quadratic form given by the corresponding norm form
\begin{align*} Q_{\mathfrak{n}}(z) &:= \frac{z \overline{z}}{N} = \frac{ {\bf{N}}(z) }{ {\bf{N}} \mathfrak{n} }.\end{align*} 
As shown in \cite[Lemma 7.1]{BY}, if $D$ is a fundamental discriminant of the imaginary quadratic field 
$k={\bf{Q}}(\sqrt{D})$, then we have an isomorphism of quadratic lattices 
\begin{align*} (\mathfrak{n}, -Q_{\mathfrak{n}}) &\longrightarrow (\mathcal{N}, -Q_{\mathfrak{n}}), \quad xN + y \left( \frac{r + \sqrt{D}}{2} \right) 
\longmapsto \left( \begin{array}{cc} x & -y/N \\ -rx - y(r^2-D)/4N & -x \end{array} \right). \end{align*}
Both lattices are equivalent to the integral quadratic form defined by \begin{align*}[-N, -r, -(r^2-D)/4N] = -Nx^2 - rxy - (r^2-D)/4N y^2. \end{align*}

Now, recall that the spin group $\operatorname{GSpin}(U) = \operatorname{GSpin}(U(x))$ 
can be identified as the multiplicative group $T_U = \operatorname{GSpin}(U) \cong k^{\times}$,
with $K_T = K \cap T \cong \widehat{\mathcal{O}}_k^{\times}$ maximal. 
According to \cite[Proposition 7.2]{BY}, if the fundamental discriminant
$D$ is coprime to $N$, then we have an identification of zero cycles $Z(U) = Z(m, \mu)$. 
 
\subsubsection{Ideal class representatives} 
 
Let $k$ be any imaginary quadratic field of discriminant $d_k$ and class group $C(\mathcal{O}_k)$.  
Let $\mathcal{Q}_{d_k}$ denote the class group of binary quadratic forms $q_{a,b,c}(x,y)=ax^2 +bxy + c y^2$ 
of discriminant $d_k = b^2 - 4ac$. Write $[a, b, c] = [q_{a, b, c}] \in \mathcal{Q}_{d_k}$ to denote the class
represented by a binary quadratic form $q_{a, b, c}(x, y)$ of discriminant $d_k = b^2 - 4ac$.
A classical theorem shows that we have an isomorphism of class groups $\mathcal{Q}_{d_k} \cong C(\mathcal{O}_k)$. 
For instance (see e.g.~\cite[Theorem 7.7]{Cox}), we have the explicit isomorphism
\begin{align*} \varphi: \mathcal{Q}_{d_k} \cong C(\mathcal{O}_k), \quad \quad [a, b, c] &\longmapsto [ a, (-b + \sqrt{d_k})/2 ]. \end{align*}
 
Recall that for each $A \in C(\mathcal{O}_k)$ we fix a representative $\mathfrak{a} \subset \mathcal{O}_k$. Consider the sublattice $L_A \subset L$ given by 
\begin{align}\label{LAn=1} L_A &= \left\lbrace \left( \begin{array}{cc} b & - a/N \\ c & - b \end{array} \right) : a, b, c \in {\bf{Z}}, \quad
N \det \left( \begin{array}{cc} b & - a/N \\ c & - b \end{array} \right)  \Bigg\vert_{L_{A,U}} \equiv - q_{a,b,c} \text{ for } \varphi([a, b, c]) = A \right\rbrace,\end{align}
with dual lattice $L_A^{\vee}$.
%\begin{align*} L_A^{\vee} &= \left\lbrace \left( \begin{array}{cc} b/2N & -a/N \\ c & -b/2N \end{array} \right) : a, b, c \in {\bf{Z}}, \quad
%N \det \left( \begin{array}{cc} b & - a/N \\ c & - b \end{array} \right)  \Bigg\vert_{L_{A,U}} \equiv - q_{a,b,c} \text{ for } \varphi([a, b, c]) = A \right\rbrace. \end{align*}
Again, we have a natural identification of the corresponding discriminant group
\begin{align*} ({\bf{Z}}/2 N {\bf{Z}}) \cong L_A^{\vee}/L_A, \quad r \longmapsto \mu_r := \left( \begin{array}{cc} r/2N & ~~\\~& - r/2N \end{array} \right), \end{align*}
and the quadratic form on $L_A^{\vee}/L_A$ can again be identified with $x \mapsto -x^2$.
The corresponding compact open subgroup $K_A = K_{L_A} \subset \operatorname{GSpin}(V)({\bf{A}}_f) \cong \operatorname{GL}_2({\bf{A}}_f)$
is given by $K_A= \prod_p K_{A, p}$, with each $K_{A, p} = K_p$ for $K_p \subset \operatorname{GSpin}(V)({\bf{Z}}_p) \cong \operatorname{GL}_2({\bf{Z}}_p)$ 
as defined above. That is, we have identified the corresponding compact open subgroups $K = K_A$ for each $A \in C(\mathcal{O}_k)$.

Let us now assume $k$ is imaginary quadratic with discriminant $d_k= D$.
Adapting the discussion above, we consider for each $m \in {\bf{Q}}_{>0}$ and $\mu \in L_A^{\vee}/L_A$ 
with $Q(\mu) \equiv m \bmod 1$ the corresponding hyperboloid 
\begin{align*} \Omega_{A, \mu, m}({\bf{Q}}) &= \left\lbrace x \in \mu + L_A : Q(x) = m \right\rbrace. \end{align*}
Given $m \in {\bf{Q}}_{>0}$, let $\mu \in L_A^{\vee}/L_A$ be such that $Q(\mu) \equiv m \bmod 1$.
Hence, $D = -4N m \in {\bf{Z}}$ is again a negative discriminant. If $r \in {\bf{Z}}$ with $Q(mu_r) \equiv m \bmod 1$, then $D \equiv r^2 \bmod N$, and we have a positive norm vector 
\begin{align*} x_A = x_A(\mu, m) = x_A(\mu_r, -D/4N)
&= \left( \begin{array}{cc} r/2N & 1/N \\ (D-r^2)/4N & - r/2N \end{array} \right) \in \Omega_{A, \mu, m}({\bf{Q}}). \end{align*}
Conversely, given $D<0$ and $r$ with $D \equiv r^2 \bmod 4N$, put $m = -D/4N$ and $\mu = \mu_r$. 
Then, $m \in Q(\mu) + {\bf{Z}}$ is positive. The corresponding special divisor
\begin{align*} Z_A(\mu, m) = Z_A(\mu_r, -D/4N) 
&= \Gamma_0(N) \Big\backslash \coprod\limits_{  x \in \mu + L_A \atop Q(x) = m  } D(V)_x  = P_{D,r}^A + P_{D, -r}^A \end{align*}
corresponds to the Heegner divisor $P_{D, r}^A + P_{D, -r}^A$, where each point $P_{D, \pm r} \in X_0(N)(k[1])$ in the moduli 
description is represented by a triple $(E, E', \varphi)$ with $E({\bf{C}}) \cong {\bf{C}}/ \mathfrak{a}$ and $E' ({\bf{C}}) \cong {\bf{C}}/ \mathfrak{n}^{-1} \mathfrak{a}$ 
and kernel $\ker(\varphi)$ of the isogeny $\varphi: E \rightarrow E'$ annihilated by the primitive ideal $\mathfrak{n}$ (see e.g.~\cite[$\S$II.1]{GZ}).
In this way, we see that each class $A \in C(\mathcal{O}_{k})$ 
has a representative special (Heegner) divisor $Z_A(\mu, m) = P_{D, r}^A + P_{D, -r}^A$, 
as well as a representative positive norm vector
$x_A = x_A(\mu, m) = x_A(\mu_r, -D/4N) \in \Omega_{A, \mu, m}({\bf{Q}})$.
Let us henceforth fix this set of representative special (Heegner) divisors and positive norm vectors
\begin{align}\label{representatives} \left\lbrace Z_A(\mu, m) = Z_A(\mu_r, -D/4N) \right\rbrace_{A \in C(\mathcal{O}_k)}, \quad
\left\lbrace  x_A(\mu, m) = x_A(\mu_r, -D/4N) \right\rbrace_{A \in C(\mathcal{O}_k)}
&\longleftrightarrow A \in C( \mathcal{O}_k). \end{align}

Fixing such a set of representatives $(\ref{representatives})$, we consider the positive and negative definite subspaces 
\begin{align*} V_{A, +} = V_{A, +}(x_A) &:= {\bf{Q}} x_A, \quad U_A = U_A(x_A) := V \cap x_A^{\perp} \end{align*}
of $(V, Q) = (\operatorname{Mat}_{2 \times 2}^{\operatorname{tr=0}}({\bf{Q}}), N \cdot \det(\cdot))$, 
with corresponding positive and negative definite lattices 
\begin{align*} \mathcal{P} _A= \mathcal{P}_A(x_A) := L_A \cap V_{A, +}, 
\quad \mathcal{N}_A = \mathcal{N}_A(x_A)  = L_A \cap U_A = L_A \cap x_A^{\perp}. \end{align*}
We shall apply Theorems \ref{BY4.7} and \ref{BYAGHMP} to these subspaces 
$V_{A, +} = \mathcal{N}_A \otimes_{\bf{Z}} {\bf{Q}} \subset V$ for each class $A \in C(\mathcal{O}_k)$.

\subsection{Cuspidal eigenforms from vector-valued Shimura lifts} 
 
Let $L \subset V$ be the lattice from $(\ref{Ln=1})$ corresponding to $K_0(N) \subset \operatorname{GL}_2(\widehat{\bf{Z}})$. 
Let $S_{3/2}(\omega_L)$ denote the space of holomorphic cuspidal modular forms of weight $3/2$ 
and representation $\omega_L$. Eichler-Zagier \cite[Theorem 5.1]{EZ} shows that we have the identification 
\begin{align}\label{EZ} S_{3/2}(\omega_L) \cong J_{2, N}^{\operatorname{cusp}} \end{align} 
with the space $J_{2, N}^{\operatorname{cusp}}$ of Jacobi cusp forms of weight $2$ and index $N$. 
There is a theory of Hecke operators and newforms for the space of Jacobi cusp forms $J_{2, N}^{\operatorname{cusp}}$.
We write $J_{2, N}^{\operatorname{new, cusp}} \subset J_{2, N}^{\operatorname{cusp}}$ to denote its subspace of newforms,
with $S_{3/2}^{\operatorname{new}}(\omega_L) \subset S_{3/2}(\omega_L)$ the induced subspace of vector-valued newforms.

To describe this more precisely, we again consider the space $S_2(\Gamma_0(N))$ of scalar-valued 
holomorphic cusp forms of weight $2$ on $X_0(N)$ with trivial nebentype character.
Let $S_2^{\star}(\Gamma_0(N)) \subset S_2(\Gamma_0(N))$ denote the subspace of holomorphic cusp forms 
which are invariant under the Fricke involution $w_N$.
Note that a cusp form $\phi \in S_2(\Gamma_0(N))$ is invariant under the Fricke involution if and only if its corresponding standard 
$L$-function\footnote{Which we normalize here to have central value at $s=1/2$, as in the discussion above, 
but distinct from the classical normalizations used by \cite[$\S$7.3]{BY}, \cite{GKZ}, and \cite{GZ}.} $\Lambda(s, \phi) = L_{\infty}(s, \phi) L(s, \phi)$
has an odd, symmetric functional equation $\Lambda(s, \phi) = - \Lambda(1-s, \phi)$.
Let $S_2^{\operatorname{new}}(\Gamma_0(N)) \subset S_2(\Gamma_0(N))$ denote the subspace of newforms, and  
\begin{align*} S_2^{\operatorname{new}, \star}(\Gamma_0(N)) = S_2^{\star}(\Gamma_0(N)) 
\cap S_2^{\operatorname{new}}(\Gamma_0(N)) \subset S_2(\Gamma_0(N)) \end{align*}
the subspace of Fricke-invariant newforms. The theorem of Skoruppa-Zagier \cite{SZ} 
shows that the Shimura correspondence can be realized explicitly as an isomorphism 
of $J_{2, N}^{\operatorname{new, cusp}}$-Hecke modules
\begin{align}\label{Shimura} S_2^{\operatorname{new}, \star}(\Gamma_0(N)) &\cong J_{2, N}^{\operatorname{new, cusp}}.\end{align}
Explicitly, fix $m_0 \in {\bf{Q}}_{>_0}$ and $\mu_0 \in L^{\vee}/L$ such that $m_0 \equiv Q(\mu_0) \bmod 1$. 
Suppose $D_0 := -4 N m_0 \in {\bf{Z}}$ is a fundamental discriminant, 
with $x \in \Omega_{\mu_0, m_0}({\bf{Q}})$ the corresponding positive norm vector defined in $(\ref{x})$,
and $U = U(x) = V \cap x^{\perp}$ the corresponding negative definite space. 
Consider the space $S_{3/2}(\omega_L)$ of holomorphic cuspidal forms of weight $3/2$ and representation $\omega_L$. 
We have for each such pair $(\mu_0, m_0)$ a linear map 
\begin{align*} \mathcal{S}_{\mu_0, m_0}: S_{3/2}(\omega_L) \longrightarrow S_2(\Gamma_0(N)), \quad g \longmapsto \mathcal{S}_{m_0, \mu_0}(g) \end{align*} 
defined on Fourier series expansions
\begin{align*} g(\tau) = \sum\limits_{\mu \in L^{\vee}/L} \sum\limits_{m \in {\bf{Q}}_{>0}} c_g(\mu, m) e(m \tau) {\bf{1}}_{\mu} \in S_{3/2}(\omega_L)\end{align*}
by the rule 
\begin{align*}\mathcal{S}_{\mu_0, m_0}(g)(\tau) &:= \sum\limits_{n \geq 1} \left(  \sum\limits_{d \mid n} \left( \frac{D_0}{d} \right)   
c_g \left(  \mu_0 \cdot \frac{n}{d}, m_0 \cdot \frac{n^2}{d^2}  \right)  \right) e(n \tau). \end{align*} 
Here, we shall also write the Fourier series expansion of 
$\mathcal{S}_{\mu_0, m_0}(g)(\tau) \in S_2(\Gamma_0(N))$ with the simpler notations
\begin{align*} \mathcal{S}_{\mu_0, m_0}(g)(\tau) &= \sum\limits_{n \geq 1} c_{\mathcal{S}_{\mu_0, m_0}(g)}(n) e(n \tau),
\quad  c_{\mathcal{S}_{\mu_0, m_0}(g)}(n) := \sum\limits_{d \mid n} \left( \frac{D_0}{d} \right)  
c_g \left( \mu_0 \cdot \frac{n}{d}, m_0 \cdot \frac{n^2}{d^2}  \right),\end{align*}
as well as the normalized Fourier series expansion 
\begin{align*} \mathcal{S}_{\mu_0, m_0}(g)(\tau) &= \sum\limits_{n \geq 1} n^{\frac{1}{2}} a_{\mathcal{S}_{\mu_0, m_0}(g)}(n) e(n \tau), \quad 
a_{\mathcal{S}_{\mu_0, m_0}(g)}(n) = c_{\mathcal{S}_{\mu_0, m_0}(g)}(n) n^{-\frac{1}{2}}. \end{align*}
Hence, the standard $L$-function 
$\Lambda(s, \mathcal{S}_{\mu_0, m_0}(g)) = L_{\infty}(s, \mathcal{S}_{\mu_0, m_0}(g)) L(s, \mathcal{S}_{\mu_0, m_0}(g))$ has Dirichlet expansion 
\begin{align*} L(s, \mathcal{S}_{\mu_0, m_0}(g)) = \sum\limits_{n \geq 1} a_{ \mathcal{S}_{\mu_0, m_0}(g) }(n) n^{-s} 
= \sum\limits_{n \geq 1} c_{ \mathcal{S}_{\mu_0, m_0}(g) }(n) n^{-(s + \frac{1}{2})}, \quad \Re(s)>1. \end{align*}
Writing $\eta_{D_0}(\cdot) = \left( \frac{D_0}{\cdot} \right)$ for the quadratic Dirichlet character of discriminant $D_0$, this can also be written as 
\begin{align}\label{DSESL} L(s -1/2, \mathcal{S}_{\mu_0, m_0}(g)) &= \sum\limits_{n \geq 1} c_{\mathcal{S}_{\mu_0, m_0}(g)}(n) n^{-s} 
= L(s, \eta_{D_0}) \sum\limits_{n \geq 1} c_g(\mu_0 n, m_0 n^2 ) n^{-s}. \end{align}
Each of the linear maps $\mathcal{S}_{\mu_0, m_0}: S_{3/2}(\omega_L) \longrightarrow S_2(\Gamma_0(N))$ is Hecke-equivariant, 
and some linear combination of them supplies the isomorphism
$S_{3/2}^{\operatorname{new}}(\omega_L) \cong S_2^{\operatorname{new}, \star}(\Gamma_0(N))$
implicit in the combination of $(\ref{Shimura})$ and $(\ref{EZ})$. 

Observe from the Dirichlet series expansion $(\ref{DSESL})$ that if 
$g \in S_{3/2}^{\operatorname{new}}(\omega_L)$ is related via Shimura correspondence to a scalar-valued 
cusp form $\phi = \phi_g \in S_2^{\operatorname{new}, \star}(\Gamma_0(N))$, then we have the relation of $L$-series 
\begin{align}\label{RLS} L(s, \mathcal{S}_{\mu_0, m_0} (g) ) &= c_g (\mu_0, m_0) \cdot L(s, \phi), \end{align} 
and hence the relation of central derivative values $L'(1/2, \mathcal{S}_{\mu_0, m_0}(g) ) = c_g (\mu_0, m_0) \cdot L'(1/2, \phi)$.

\begin{lemma}\label{7.3}  

Fix $m_0 \in {\bf{Q}}_{>0}$ and $\mu_0 \in L^{\vee}/L$ with $m_0 \equiv Q(\mu_0) \bmod 1$.
Consider a fundamental discriminant $D_0 = -4 N m_0 \in {\bf{Z}}$, with corresponding positive norm vector 
$x_0 \in \Omega_{\mu_0, m_0}({\bf{Q}})$ and negative definite space $U = U(x_0) = V \cap x_0^{\perp}$. 
Consider the $L$-series defined (first for $\Re(s) >1$) by the Dirichlet series
\begin{align*} L(s, g, U) &:= (4 \pi)^{- \left( \frac{s+1}{2} \right)} \Gamma \left( \frac{s+1}{2} \right) 
\sum\limits_{m \geq 1} \sum\limits_{\mu \in \mathcal{P}^{\vee}/\mathcal{P}}
r_{\mathcal{P}}(\mu, m) c_g(\mu, m) m^{- \left( \frac{s+1}{2} \right)}. \end{align*}
Given any vector-valued cuspidal form $g \in S_{3/2}(\overline{\omega}_L)$ whose $(\mu_0, m_0)$ 
Fourier coefficient $c_g(\mu_0, m_0)$ does not vanish, we have the identification of $L$-series
\begin{align*} L(s, g, U) &= (4 \pi m_0)^{- \left( \frac{s+1}{2} \right)} \Gamma \left( \frac{s+1}{2} \right) 
\frac{L(s +1/2, \mathcal{S}_{\mu_0, m_0}(g))}{ L(s+1, \eta_{D_0}) }.\end{align*}
In particular, if $g \in S_{3/2}(\overline{\omega}_L)$ and $\phi \in S_2^{\operatorname{new}, \star}(\Gamma_0(N))$ 
are linked by the Shimura correspondence $(\ref{Shimura})$ via $(\ref{RLS})$, then we have the relation of $L$-series
\begin{align*} L(s, g, U) &= (4 \pi m_0)^{- \left( \frac{s+1}{2} \right)} \Gamma \left( \frac{s+1}{2} \right) 
\frac{ c_g(\mu_0, m_0) \cdot L(s+1/2, \phi) }{ L(s+1, \eta_{D_0}) }, \end{align*}
from which we can derive the identification of central derivative values 
\begin{align*} L'(0, g, U) &= \frac{2 \sqrt{N}}{ \pi } \left( \frac{ c_g(\mu_0, m_0) }{\deg Z(\mu_0, m_0)} \right) \cdot L'(1/2, \phi).\end{align*} \end{lemma} 

\begin{proof} 

See \cite[Lemma 7.3]{BY} with \cite[(4.24)]{BY}, which we state here in the unitary normalization for the standard $L$-function of $\phi \in S_2^{\operatorname{new}, \star}(\Gamma_0(N))$. 
Since we obtain a slightly distinct identification for the central derivative value (by a factor of $2$), we provide details. 
Viewing $g \in S_{3/2}(\overline{\omega}_L)$ as a form of weight $3/2$ and representation $\overline{\omega}_{\mathcal{P} \oplus \mathcal{N}}$ 
via \cite[Lemma 3.1]{BY}, we argue as in \cite[Lemma 7.3]{BY} that $c_g(\lambda, Q(\lambda)) = 0$ for all $\lambda \in \mathcal{P}^{\vee}$ 
unless $\lambda \in \mathcal{P}^{\vee} \cap L^{\vee} = {\bf{Z}} x$ to deduce that 
\begin{align*} L(s, g, U) &= (4 \pi)^{- \left( \frac{s+1}{2} \right)} \Gamma \left( \frac{s+1}{2} \right) 
\sum\limits_{\lambda \in \mathcal{P}^{\vee}} c_g(\lambda, Q(\lambda)) Q(\lambda)^{- \left( \frac{s+1}{2} \right)}.\end{align*}
On the other hand, we deduce from $(\ref{DSESL})$ that we have the relation
\begin{equation*}\begin{aligned} L(s +1/2, \mathcal{S}_{\mu_0, m_0}(g)) 
&= L(s +1, \eta_{D_0}) \sum\limits_{m \geq 1} c_g( \mu_0 m, m_0 m^2) m^{-(s+1)} \\
&=L(s+1, \eta_{D_0})  \sum\limits_{\lambda \in \mathcal{P}^{\vee}} c_g(\mu_0 \lambda, m_0 Q(\lambda)) \cdot \left( m_0 Q(\lambda) \right)^{- \left( \frac{s+1}{2} \right)} \\
&= L(s+1, \eta_{D_0}) \cdot m_0^{- \left( \frac{s+1}{2} \right)}  
\sum\limits_{\lambda \in \mathcal{P}^{\vee}} c_g(\lambda, Q(\lambda)) Q(\lambda)^{- \left( \frac{s+1}{2} \right)}\end{aligned}\end{equation*}
and hence 
\begin{align*} \sum\limits_{\lambda \in \mathcal{P}^{\vee}} c_g(\lambda, Q(\lambda)) Q(\lambda)^{- \left( \frac{s+1}{2} \right)} 
&=m_0^{- \left( \frac{s+1}{2} \right)}  \cdot \frac{L(s +1/2, \mathcal{S}_{\mu_0, m_0}(g))}{ L(s+1, \eta_{D_0}) },\end{align*}
so that 
\begin{align*} L(s, g, U) &= (4 \pi m_0)^{- \left( \frac{s+1}{2} \right)} \Gamma \left( \frac{s+1}{2} \right) 
\frac{L(s +1/2, \mathcal{S}_{\mu_0, m_0}(g))}{ L(s+1, \eta_{D_0}) }. \end{align*}
which by $(\ref{RLS})$ gives the desired relation of $L$-series
\begin{align*} L(s, g, U) &= (4 \pi m_0)^{- \left( \frac{s+1}{2} \right)} \Gamma \left( \frac{s+1}{2} \right) 
\frac{ c_g(\mu_0, m_0) \cdot L(s+1/2, \phi) }{ L(s+1, \eta_{D_0}) }. \end{align*}

Since $L(1/2, \phi)=0$ as $\phi$ is invariant under the Fricke involution, we deduce via the product rule that 
\begin{align}\label{L'(0)} L'(0, g, U) &= (4 \pi m_0)^{-  \frac{1}{2} } \Gamma \left( \frac{1}{2} \right) 
\frac{ c_g(\mu_0, m_0) \cdot L'(1/2, \phi) }{ L(1, \eta_{D_0}) } = \frac{ c_g(\mu_0, m_0) 
\cdot L'(1/2, \phi)}{2 \sqrt{m_0} \cdot L(1, \eta_{D_0})}. \end{align} 
Using the Dirichlet class number formula $(\ref{Dirichlet})$ for the imaginary quadratic field $k={\bf{Q}}(\sqrt{D_0}) $, we evaluate 
\begin{align*} L(1, \eta_{D_0} ) &=  \frac{ 2 \pi h_k  }{ w_k \sqrt{ \vert D_0 \vert }  } 
= \frac{ \pi}{\sqrt{\vert D_0 \vert}} \cdot \frac{2 h_k}{ w_k} = \frac{ \pi}{ \sqrt{\vert D_0 \vert}} \cdot H(D_0)
= \frac{  \pi }{ 2 \sqrt{N m_0} } \cdot \deg Z(\mu_0, m_0). \end{align*} 
Here, $H(D_0) = 2 h_k/w_k$, and we have that $H(D_0) = \deg Z(\mu_0, m_0)/2$ (see \cite[Lemma 6.3]{BY}). Hence,
\begin{align}\label{HurwitzID} L'(0, g, U) &= \frac{ c_g(\mu_0, m_0) \cdot L'(1/2, \phi)}{2 \sqrt{m_0} \cdot L(1, \eta_{D_0})}
= \frac{ 4 \sqrt{N m_0} \cdot c_g(\mu_0, m_0) \cdot L'(1/2, \phi)}{ 2 \sqrt{m_0} \cdot \pi  \cdot \deg Z(\mu_0, m_0) }
= \frac{ 2 N^ {\frac{1}{2}} }{ \pi} \cdot \frac{ c_g(\mu_0, m_0) \cdot L'(1/2, \phi)}{ \deg Z(\mu_0, m_0) } . \end{align}\end{proof}

To relate this to Theorems \ref{BY4.7} and \ref{BYAGHMP} above, we choose $f(\tau) \in H_{1/2}(\omega_L)$ as follows. 

\begin{lemma}\label{7.4} 

Fix a cuspidal form $g \in S_{3/2}^{\operatorname{new}}(\overline{\omega}_L)$,
and let $\phi \in S_2^{\operatorname{new}, \star}(\Gamma_0(N))$ denotes its image under the Shimura correspondence via $(\ref{Shimura})$. 
There exists a Maass form $f(\tau) = f^+(\tau) + f^-(\tau) \in H_{1/2}(\omega_L)$ such that: \\

\begin{itemize}

\item[(i)] We have the relation $\xi_{1/2}(f) = g / \vert \vert g \vert \vert^2$. \\

\item[(ii)] The Fourier coefficients $c_f^+(\mu, m)$ of the holomorphic part $f^+$ of $f$ lie in the Hecke field ${\bf{Q}}(\phi)$ obtained by adjoining to ${\bf{Q}}$
the Fourier coefficients of the cuspidal newform $\phi \in S_2^{\operatorname{new}, \star}(\Gamma_0(N))$. \\

\item[(iii)] The constant Fourier coefficient $c_f^+(0, 0)$ of the holomorphic part $f^+$ of $f$ vanishes. \\

\end{itemize}

\end{lemma}

\begin{proof} See \cite[Lemma 7.4]{BY} or \cite[Lemma 7.3]{BO}.  \end{proof}

We also have the following result, to ensure the nonvanishing of coefficients $c_g(\mu_0, m_0)$ in Lemma \ref{7.3}. 

\begin{lemma}\label{7.5} Fix a newform
\begin{align*} g(\tau) 
&= \sum\limits_{\mu \in L^{\vee}/L} \sum\limits_{m >0} c_g(\mu, m) e(m \tau) {\bf{1}}_{\mu} \in S_{3/2}^{\operatorname{new}}(\overline{\omega}_L). \end{align*}
There exist infinitely many fundamental discriminants $D< 0$ such that  

\begin{itemize}

\item[(i)] Each prime divisor $q \mid N$ splits in the imaginary quadratic extension ${\bf{Q}}(\sqrt{D})$. 

\item[(ii)] The coefficient $c_g(\mu, m)$ does not vanish for $m = - \frac{D}{4N}$ and any $\mu \in L^{\vee}/L$ for which $m \equiv Q(\mu) \bmod 1$. 
 
\end{itemize}
\end{lemma}

\begin{proof} See \cite[Lemma 7.5]{BY}. This is deduced from the nonvanishing theorem of 
Bump-Friedberg-Hoffstein \cite{BFH} together with the Waldspurger formula  
shown in \cite[$\S$II.4 Corollary 1]{GKZ} and \cite{SZ}. \end{proof}

\subsection{Relation to heights}

We now consider the moduli stack $\mathcal{Y}_0(N)$ over ${\bf{Z}}$ of cyclic isogenies of degree $N$ of elliptic curves
$\pi: E \longrightarrow E'$ for which $\ker(\pi)$ meets each irreducible component of each geometric fibre. We also consider
the moduli stack $\mathcal{X}_0(N) $ over ${\bf{Z}}$ of cyclic isogenies of degree $N$ of generalized elliptic curves 
$\pi: E \longrightarrow E'$ for which $\ker(\pi)$ meets each irreducible component of each geometric fibre. 
Hence, we have the relation $\mathcal{X}_0(N)({\bf{C}}) = X_0(N) = X_K^{\star}({\bf{C}})$, and
$\mathcal{X}_0(N)$ is smooth over ${\bf{Z}}[1/N]$, regular away from supersingular points $\underline{x}$ in characteristic $p$ for $p \mid N$ any prime divisor. 

Recall that each of the special divisors $Z(\mu, m)$ has an extension $\mathcal{Z}(\mu, m)$ to the integral model $\mathcal{X} = \mathcal{Y}_0(N)$. 
More precisely, we can view each $\mathcal{Z}(\mu, m)$ as a Deligne-Mumford stack which assigns to a base scheme $S$ over ${\bf{Z}}$
a set of pairs $(\pi: E \longrightarrow E', \iota)$ consisting of 

\begin{itemize}

\item A cyclic isogeny $\pi: E \longrightarrow E'$ of elliptic curves $E, E'$ over $S$ of degree $N$ 

\item An action $\iota: \mathcal{O}_{ {\bf{Q}}(\sqrt{D}) } \hookrightarrow \operatorname{End}(\pi) = 
\left\lbrace \alpha \in \operatorname{End}(E): \pi \alpha \pi^{-1} \in \operatorname{End}(E') \right\rbrace$
of  $\mathcal{O}_{ {\bf{Q}}(\sqrt{D}) }$ on $\pi$ for which $\iota(\mathfrak{n}) \ker(\pi) = 0$. \end{itemize}

Again, we take $\mathfrak{n}$ to be the ideal $\mathfrak{n} = [N, (r + \sqrt{D})/2]$ in $k= {\bf{Q}}(\sqrt{D})$, with
\begin{align*} D = - 4 N m \quad \text{and} \quad \mu = \mu_r = \left( \begin{array}{cc} \frac{r}{2N} & ~~\\~~& - \frac{r}{2N}  \end{array} \right). \end{align*}

\begin{remark} Although $\mathcal{X}_0(N)$ is not regular, we may use intersection theory for the special divisors 
$\mathcal{Z}(\mu, m)$ and for cuspidal divisors on $\mathcal{X}_0(N)$. To justify this, we consider the corresponding forgetful maps
\begin{align*} \mathcal{Z}(\mu, m) \longrightarrow \mathcal{Y}_0(N), \quad (\pi: E \rightarrow E', \iota) \longmapsto (\pi: E \rightarrow E'), \end{align*}
each of which is finite \'etale, and $2:1$. The image of each of these maps consists of the flat closure of $\mathcal{Z}(\mu, m)$ in $\mathcal{X}_0(N)$, 
which does not intersect the boundary $\mathcal{X}_0(N) \backslash \mathcal{Y}_0(N)$, and which lies in the regular locus of $\mathcal{X}_0(N)$. \end{remark}

Let us now fix the Maass cusp form $f(\tau) \in H_{1/2}(\omega_L)$ from Lemma \ref{7.4}. 
Note that if $\phi(\tau) \in S_2^{\operatorname{new}, \star}(\Gamma_0(N))$ 
is the eigenform parametrizing an elliptic curve $E$ over ${\bf{Q}}$, 
then the Fourier coefficients of $f^+(\tau)$ are rational integers, ${\bf{Q}}(\phi) = {\bf{Q}}$. 
Recall that we consider the corresponding divisor $Z(f) = \sum\limits_{\mu \in L^{\vee}/L} \in \operatorname{Div}(Y_0(N))$,
and that the corresponding regularized theta lift $\Phi(f, \cdot) = \Phi(f, z, h)$ can be identified as the automorphic Green's function
$G_{Z(f)}(\cdot) = G_{Z(f)}(z, h)$ with logarithmic singularity along $Z(f)$. 
As explained in \cite[$\S$ 7.3]{BY}, there exists a divisor $C(f)$ on $X_0(N)$ supported on the cusps for which the divisor $Z^c(f) := Z(f) + C(f)$
has degree zero on $X_0(N)$. Moreover, the regularized theta lift $\Phi(f, \cdot)$
can be viewed as the automorphic Green's function $G_{Z^c(f)}(\cdot)$ for this divisor $Z^c(f)$ on the compactification $X_0(N) = X_K$.
We write $\mathcal{Z}^c(f) = \mathcal{Z}(f) + \mathcal{C}(f)$ to denote its flat closure in $\mathcal{X}_0(N)$, and consider the corresponding arithmetic divisor 
\begin{align*} \widehat{\mathcal{Z}}^c(f) &= ( \mathcal{Z}^c(f), \Phi(f, \cdot) ) 
= (\mathcal{Z}^c(f), G_{Z^c(f)}(\cdot)) \in \widehat{\operatorname{Ch}}^1(\mathcal{X}_0(N))_{ {\bf{Q}}(\phi) }. \end{align*}

Given $m \in {\bf{Q}}_{>0}$ and $\mu \in L^{\vee}/L$ as above, we consider the divisor on $X_0(N)$ given by
\begin{align*} y(\mu, m) := Z(\mu, m) - \frac{\deg(Z(\mu, m))}{2} \left( (\infty) +(0) \right). \end{align*}
Note that this divisor $y(\mu, m)$ has degree zero, and is invariant under the Fricke involution. 
Let $\mathcal{Y}(\mu, m)$ denote its flat closure in $\mathcal{X}_0(N)$. As explained in \cite[$\S$ 7.3]{BY},
for each prime $p$ not dividing the discriminant $D = - 4N m$, this latter divisor $\mathcal{Y}(\mu, m)$ has
zero intersection with each fibre component of $\mathcal{X}_0(N)$ over ${\bf{F}}_p$. We also consider the divisor defined by 
\begin{align*} y(f) := \sum\limits_{\mu \in L^{\vee}/L} \sum\limits_{m \in {\bf{Q}} \atop m>0} c_f^+(\mu, -m) y(\mu, m) \in 
\operatorname{Div}(X_0(N)) = \operatorname{Div}(X^{\star}_K), \end{align*}  
and write $\mathcal{Y}(f)$ to denote its flat closure in $\mathcal{X}_0(N)$.

Let $J_0(N)$ denote the Jacobian of $X_0(N)$, with $J_0(N)(F)$ the $F$-rational points for a number field $F$.
Hence, elements of $J_0(N)(F)$ correspond to divisor classes of degree zero on $X_0(N)$ which are rational over $F$.
Now, observe that $y(f)$ is a divisor of degree zero on $X_0(N)$ which differs from the $Z^c(f)$ by a divisor of degree
zero supported at the cusps. We deduce from the Manin-Drinfeld theorem that $y(f)$ and $Z^c(f)$ represent the same
point in $J_0(N) \otimes_{\bf{Z}} \overline{{\bf{Q}}}$.
Keeping with the setup of Lemmas \ref{7.3}, \ref{7.4}, and \ref{7.5} above, let us now consider the generating series 
\begin{align*} \Zha(\tau) &=  \sum\limits_{\mu \in L^{\vee}/L} \sum\limits_{m \in {\bf{Q}} \atop m>0} y(\mu, m) e(m \tau) {\bf{1}}_{\mu}. \end{align*}  
By the theorem\footnote{This was reproven later by Borcherds using Borcherds products for weakly holomorphic forms in the 
space $M_{1/2}^!(\omega_L)$.} of Gross-Kohnen-Zagier \cite{GKZ}, 
this generating series $\Zha(\tau)$ can be viewed as a modular form taking values $J_0(N)({\bf{Q}}) \otimes_{\bf{Z}} {\bf{Q}}$.
Given a normalized newform $\phi \in S_2^{\operatorname{new}, \star}(\Gamma_0(N))$ as above, we can consider the corresponding
projection $\Zha^{\phi}(\tau)$ of $\Zha(\tau)$ to the $\phi$-isotypical component. The coefficients of this projection $\Zha^{\phi}(\tau)$
consist of the projections $y^{\phi}(\mu, m)$ of each of the divisors $y(\mu, m)$ to the $\phi$-isotypical component, 
\begin{align*} \Zha^{\phi}(\tau) &=  \sum\limits_{\mu \in L^{\vee}/L} \sum\limits_{m \in {\bf{Q}} \atop m>0} y^{\phi}(\mu, m) e(m \tau) {\bf{1}}_{\mu}. \end{align*}  

\begin{theorem}\label{7.6} 

Let us retain the setups of Lemmas \ref{7.3}, \ref{7.4}, and \ref{7.5} above, so that $g = \xi_{1/2}(f) \in S_{3/2}(\overline{\omega}_L)$ is the vector-valued 
Shimura lift of the Fricke-invariant newform $\phi \in S_2^{\operatorname{new}, \star}(\Gamma_0(N))$. We have the identity
\begin{align*} \Zha^{\phi}(\tau) &= g(\tau) \otimes y(f) \in S_{3/2}(\omega_L^{\vee}) \otimes J_0(N)({\bf{Q}}). \end{align*}
In particular, the divisor $y(f)$ factors through the $\phi$-isotypical component of the Jacobian $J_0(N)({\bf{Q}}) \otimes_{\bf{Z}} \overline{\bf{Q}}$. \end{theorem}

\begin{proof} See \cite[Theorem 7.6]{BY}, which explains how to deduce this from \cite{GKZ} and \cite[Theorem 7.7]{BO}. \end{proof}
  
\begin{theorem}\label{GZF1} 

Let us retain the setups of Lemmas \ref{7.3}, \ref{7.4}, and \ref{7.5} above,
so that $g = \xi_{1/2}(f) \in S_{3/2}(\overline{\omega}_L)$ is the vector-valued Shimura lift 
of the Fricke-invariant newform $\phi \in S_2^{\operatorname{new}, \star}(\Gamma_0(N))$.
The N\'eron-Tate height $\left[ y(f), y(f) \right]_{\operatorname{NT}}$ of the divisor $y(f)$ is given by the preliminary Gross-Zagier formula
\begin{align*} \left[ y(f), y(f) \right]_{\operatorname{NT}} &=  \frac{\sqrt{N}}{\pi \vert \vert g \vert \vert^2} \cdot L'(1/2, \phi).  \end{align*}  \end{theorem}

\begin{proof} 

See \cite[Theorem 7.7]{BY}. We modify the proof using Lemma \ref{7.3} and Theorem \ref{BY4.7} above as follows. 
Observe that Theorem \ref{7.6} implies the identification of Fourier coefficient divisors 
$c_g(\mu, m) y(f) = y^{\phi}(\mu, m)$ for each of the pairs $(\mu, m)$ we consider.
Using this identification together with the Manin-Drinfeld theorem, we deduce that we have the relation
\begin{align}\label{observation} \left[ y(f), y(f) \right]_{\operatorname{NT}} \cdot c_g(\mu, m) &= \left[ y(f), y^{\phi}(\mu, m) \right]_{\operatorname{NT}}
= \left[ y(f), y(\mu, m) \right]_{\operatorname{NT}} = \left[ Z^c(f), y(\mu, m) \right]_{\operatorname{NT}} \end{align}
for each pair $(\mu, m)$ contributing to the holomorphic part $f^+$ of of $f \in H_{1/2}(\omega_L)$.

Let us now fix two distinct pairs $(\mu_0, m_0)$ and $(\mu_1, m_1)$, and for simplicity write
\begin{align*} d(\mu_j, m_j) &= \deg Z(\mu_j, m_j) \quad \quad \text{for $j=0,1$}. \end{align*}
Define the constant  
\begin{align*} c &= c(\mu_0, \mu_1, m_0, m_1) := d(\mu_1, m_1) c_g(\mu_0, m_0) - d(\mu_0, m_0) c_g(\mu_1, m_1). \end{align*}
Consider the divisor of degree zero on $X_0(N)$ defined by 
\begin{align*} Z &= d(\mu_1, m_1) y(\mu_0, m_0) - d(\mu_0, m_0) y(\mu_1, m_1) 
= d(\mu_1, m_1) Z(\mu_0, m_0) - d(\mu_0, m_0) Z( \mu_1, m_1). \end{align*}
We also write $\mathcal{Z}$ to denote its flat closure in $\mathcal{X}_0(N)$. Observe that $Z$ is supported outside of the cusps of $X_0(N)$. 
Let $M$ denote the least common multiple of all the discriminants of the special divisors $Z(\mu, m)$ in the support of $Z(f)$.
Assume that for each $j=0, 1$ the discriminant $D_j = -4N m_j$ is coprime to $MN$. This ensures that the divisors $Z$ and $Z^c(f)$ are coprime. 
It also ensures for each prime $p$ that $\mathcal{Z}$ and $\mathcal{Z}^c(f)$ have zero intersection with each fibral component of 
$\mathcal{X}_0(N)$ over ${\bf{F}}_p$. Via $(\ref{observation})$, we compute 
\begin{align*} c \cdot \left[ y(f), y(f) \right]_{\operatorname{NT}} 
&= \left[ Z^c(f), d(\mu_1, m_1) Z(\mu_0, m_0) - d(\mu_0, m_0) Z(\mu_1, m_1) \right]_{\operatorname{NT}} \\
&= - d (\mu_1, m_1) \left[ \widehat{\mathcal{Z}}^c(f), \mathcal{Z}(\mu_0, m_0) \right] 
+ d(\mu_0, m_0) \left[ \widehat{\mathcal{Z}}^c(f), \mathcal{Z}( \mu_1, m_1) \right].\end{align*}
Note that the cuspidal divisor $\mathcal{C}(f)$ does not intersect with any of the special divisors $\mathcal{Z}(\mu, m)$. 
We now apply the arithmetic height formula $(\ref{AHF})$ shown\footnote{Here, we could also use the original 
argument of Bruinier-Yang \cite[Theorem 7.7]{BY}, replacing their substitution of the formula \cite[Theorem 4.7]{BY}
with the slightly modified version we derive in Theorem \ref{BY4.7} above to obtain the same result.} in Theorem \ref{BYAGHMP}
for each of the negative definite spaces $U_j = V \cap x(\mu_j, m_j)^{\perp}$ and lattices 
$\mathcal{N}_j = U_j \cap L$ and $\mathcal{P}_j = \mathcal{N}_j^{\perp} \subset L$
with Lemma \ref{7.3} (cf.~\cite[Lemma 7.3]{BY}) and the identification of central derivative $L$-values $(\ref{HurwitzID})$ 
and Lemma \ref{7.4} (i) and (iii) for the cuspidal form $f \in H_{1/2}(\omega_L)$ to each index $j=0,1$ to obtain the arithmetic height formulae 
\begin{equation*}\begin{aligned} \left[ \widehat{\mathcal{Z}}^c(f), \mathcal{Z}(\mu_j, m_j) \right] 
&= \frac{1}{2} \cdot \Phi(f, Z(\mu_j, m_j)) + \left[ \mathcal{Z}(f), \mathcal{Z}(\mu_j, m_j) \right]_{\operatorname{fin}} 
+ \left[ \mathcal{C}(f), \mathcal{Z}(\mu_j, m_j) \right]_{\operatorname{fin}} \\
&= - \frac{d(\mu_j, m_j)}{2} \left( L'(0, \xi_{1/2}(f), U_j) + c_f^+(0,0) \cdot \kappa_{\mathcal{N}_j}(0,0) \right) 
+ \left[ \mathcal{C}(f), \mathcal{Z}(\mu_j, m_j) \right]_{\operatorname{fin}} \\ 
&= - \frac{d(\mu_j, m_j)}{2} \cdot L'(0, \xi_{1/2}(f), U_j) 
= - \frac{d(\mu_j, m_j)}{2 } \cdot \frac{ 2 N^{\frac{1}{2}} }{\vert \vert g \vert \vert^2 \pi}  \frac{c_g(\mu_j, m_j) L'(1/2, \phi)}{\deg Z(\mu_j, m_j)} \\
&= - \frac{ N^{\frac{1}{2}} }{ \pi \vert \vert g \vert \vert^2 } \cdot c_g(\mu_j, m_j) \cdot L'(1/2, \phi) \end{aligned}\end{equation*}
so that 
\begin{equation}\begin{aligned}\label{hc} &c \cdot \left[ y(f), y(f) \right]_{\operatorname{NT}} 
=  - d(\mu_1, m_1) \left[ \widehat{\mathcal{Z}}^c(f), \mathcal{Z}(\mu_0, m_0) \right] 
+ d(\mu_0, m_0) \left[ \widehat{\mathcal{Z}}^c(f), \mathcal{Z}( \mu_1, m_1) \right] \\
&=   d(\mu_1, m_1) \cdot \frac{ N^{\frac{1}{2}} }{  \pi \vert \vert g \vert \vert^2 } \cdot c_g(\mu_0, m_0) \cdot L'(1/2, \phi)
- d(\mu_0, m_0) \cdot \frac{ N^{ \frac{1}{2} } }{  \pi \vert \vert g \vert \vert^2} 
\cdot c_g(\mu_1, m_1) \cdot L'(1/2, \phi).\end{aligned}\end{equation}

Now, it is not hard to show that we can choose the pairs $(\mu_0, m_0)$ and $(\mu_1, m_1)$ in such a way that the constant 
$c = c(\mu_0, \mu_1, m_0, m_1)$ does not vanish. We can then deduce from the calculation $(\ref{hc})$ that 
\begin{align*} c \cdot \left[ y(f), y(f) \right]_{\operatorname{NT}} 
&= c \cdot \frac{\sqrt{N}}{ \pi \vert \vert g \vert \vert^2} \cdot L'(1/2, \phi), \end{align*}
so that the claimed formula follows after dividing out by the nonzero constant $c$. \end{proof} 

\begin{corollary}\label{GZFprincipal} 

For any coset $\mu \in L^{\vee}/L$ and positive integer $m \in Q(\mu) + {\bf{Z}}$, we have for $D:= - 4 N m$ that 
\begin{align*} \left[ y^{\phi}(\mu, m), y^{\phi}(\mu, m) \right]_{\operatorname{NT}} 
&= \frac{\sqrt{\vert D \vert}}{8 \pi^2 \vert \vert \phi \vert \vert^2} \cdot L(1/2, \phi \otimes \eta_D) \cdot L'(1/2, \phi). \end{align*}

 \end{corollary}

\begin{proof} Cf.~\cite[Corollary 7.8]{BY}. We deduce this from Theorem \ref{GZF1} 
using the relation $y^{\phi}(\mu, m) = c_g(\mu, m) \cdot y(f)$ with the Waldspurger-like 
formula theorem shown in Gross-Kohnen-Zagier \cite[II, $\S$4 Corollary 1]{GKZ}:
\begin{align*}\frac{c_{j_g}(\mu, m)^2}{\langle j_g, j_g  \rangle} 
&= \frac{ \sqrt{\vert D \vert} }{2 \pi }  \cdot \frac{L(1/2, \phi \otimes \eta_D)}{\langle \phi, \phi \rangle},\end{align*}
where $j_g \in J_{2, N}^{\operatorname{new}, \operatorname{cusp}}$ denotes the Jacobi form corresponding to $g$.
Using that the Petersson norm $\vert \vert g \vert \vert$ is equal to $2 N^{\frac{1}{4}} \vert \vert j_g \vert \vert$,
by Eichler-Zagier \cite[Theorem 5.3]{EZ} (cf.~\cite[Corollary 7.8]{BY}), we derive the coefficient formula 
\begin{align}\label{GKZ} c_g(\mu, m)^2 = c_{j_g}(\mu, m)^2 &= \frac{\vert \vert j_g \vert \vert^2}{ 2 \pi \vert \vert \phi \vert \vert^2} 
\cdot \vert D \vert^{\frac{1}{2}}\cdot L(1/2, \phi \otimes \eta_D) = \frac{\vert \vert g \vert \vert^2}{8 \pi N^{\frac{1}{2}} \vert \vert \phi \vert \vert^2} 
\cdot \vert D \vert^{\frac{1}{2}} \cdot L(1/2, \phi \otimes \eta_D) \end{align} to get 
\begin{align*} &\left[ y^{\phi}(\mu, m), y^{\phi}(\mu, m) \right]_{\operatorname{NT}} 
= \left[ c_g(\mu, m) y(f), c_g(\mu, m) y(f) \right]_{\operatorname{NT}} = c_g(\mu, m)^2 \cdot \left[  y(f), y(f) \right]_{\operatorname{NT}} \\
&= \frac{ \vert D \vert^{\frac{1}{2}} \vert \vert g \vert \vert^2}{8 \pi N^{\frac{1}{2}} \vert \vert \phi \vert \vert^2}  \cdot L(1/2, \phi \otimes \eta_D) 
\cdot  \frac{ N^{\frac{1}{2}}  }{ \pi \vert \vert g \vert \vert^2} \cdot L'(1/2, \phi) 
= \frac{ \vert D \vert^{\frac{1}{2}}  }{8 \pi^2 \vert \vert \phi \vert \vert^2} \cdot L(1/2, \phi \otimes \eta_D) \cdot L'(1/2, \phi). \end{align*} \end{proof}

\subsection{Class group twists}\label{twisted}

We now explain how to adapt Theorem \ref{GZF1} and Corollary \ref{GZF} to derive the full Gross-Zagier formula \cite[I Theorem (6.3)]{GZ}, 
which applies to twists by any character $\chi$ of the ideal class group $C(\mathcal{O}_k)$ for $k = {\bf{Q}}(\sqrt{D})$ as we consider above. 
Note that this general form of the Gross-Zagier formula is not derived in \cite{BY}. 
We take for granted all of the discussion above leading to Corollary \ref{GZF}, 
and fix a set of representatives as in $(\ref{representatives})$ above. 
Hence, we choose for each class $A$ a point $x_A \in \Omega_{A, \mu, m}({\bf{Q}})$ 
which gives rise to a negative definite space $U_A = V \cap x_A^{\perp}$.
Note that each space $U_A$ corresponds to an ideal representative of the class $A$ of $C(\mathcal{O}_k)$. 
We also obtain a negative definite lattice $\mathcal{N}_A = U_A \cap L_A$ 
and a positive definite lattice $\mathcal{P}_A = \mathcal{N}_A^{\perp} \subset L_A$. 
Recall that in $(\ref{representatives})$, we also fix for each class $A \in C(\mathcal{O}_k)$
a divisor $Z_A(\mu, m) = Z_A(\mu_r, - D/4N)$ on $X_{K_A} \cong Y_0(N)$. 
Fixing $f \in H_{1/2}(\omega_L)$ as in Lemma \ref{7.4}, and taking the restriction $f_A$ to the sublattice 
$L_A \subset L$ as in Lemma \ref{lattice} and $(\ref{IPrelations})$ we then define the corresponding divisor %
\begin{align*}Z_A(f_A) 
&= \sum\limits_{ \mu \in L_A^{\vee}/L_A } \sum\limits_{m \in {\bf{Q}} \atop m >0} c_{f_A}^+(\mu, -m) Z_A(\mu, m) \in \operatorname{Div}(Y_0(N)).\end{align*}
There exists for each $A \in C(\mathcal{O}_k)$ a divisor $C_A(f)$ supported on the cusps $X_0(N) \backslash Y_0(N)$ for which the divisor
$Z_A^c(f) = Z_A(f) + C_A(f)$ on $X_0(N)$ has degree zero. Again, the regularized theta lift $\Phi(f_A, \cdot)$ determines the automorphic
Green's function for this divisor $Z_A^c(f_A) \in \operatorname{Div}(X_0(N))$. Writing $\mathcal{Z}_A^c(f_A) = \mathcal{Z}_A(f_A) + \mathcal{C}_A(f_A)$
for the extension to the flat closure in $\mathcal{X}_0(N)$ of this divisor $Z_A^c(f_A)$, we obtain an arithmetic divisor 
\begin{align*}\widehat{\mathcal{Z}}^c_A(f_A) := (\mathcal{Z}_A^c(f_A), \Phi(f_A, \cdot)) = (\mathcal{Z}_A^c(f_A), G_{Z_A^c(f_A)}(\cdot)) \in \widehat{\operatorname{Ch}}^1(\mathcal{X}_0(N)).\end{align*}
We also consider for each class $A \in C(\mathcal{O}_k)$ the Fricke-invariant divisor of degree zero on $X_0(N)$ defined by
\begin{align*} y_A(\mu, m) = Z_A(\mu, m) - \frac{\deg Z_A(\mu, m))}{2} \left( (\infty) + (0) \right), \end{align*}
with $\mathcal{Y}_A(\mu, m)$ its flat closure in $\mathcal{X}_0(N)$. We also consider
\begin{align*} y_A(f_A) &= \sum\limits_{\mu \in L_A^{\vee}/L_A} \sum\limits_{m \in {\bf{Q}} \atop m >0} c_{f_A}^+(\mu, -m) 
y_A(\mu, m) \in \operatorname{Div}(X_0(N)),\end{align*}
with $\mathcal{Y}_A(f)$ its flat closure in $\mathcal{X}_0(N)$.

\subsubsection{Decompositions of basechange $L$-functions}

For each class, a variation of Lemma \ref{7.3} gives us
\begin{equation*}\begin{aligned} L(s, g, U_A) &= (4 \pi m)^{- \left( \frac{s+1}{2} \right)} \Gamma \left( \frac{s+1}{2} \right)  \sum\limits_{m \geq 1} 
\sum\limits_{\mu \in \mathcal{P}_A^{\vee}/ \mathcal{P}_A} r_{\mathcal{P}_A}(\mu, m) c_{g_A}(\mu, m ) m^{- \left( \frac{s+1}{2} \right)} \\
&= (4 \pi)^{- \left( \frac{s+1}{2} \right)} \Gamma \left( \frac{s+1}{2} \right) 
\sum\limits_{\lambda \in \mathcal{P}_A^{\vee} } c_{g_A}(\lambda, Q(\lambda)) Q(\lambda)^{- \left( \frac{s+1}{2} \right)} \\
&= (4 \pi m)^{- \left( \frac{s+1}{2} \right)} \Gamma \left( \frac{s+1}{2} \right) \frac{L_A(s+1/2, \phi)}{L(s+1, \eta_{D})}. \end{aligned}\end{equation*}
Here, we view $g = g_A \in S_{3/2}(\overline{\omega}_{L_A})$ as a form of weight $3/2$ and representation $\overline{\omega}_{ \mathcal{P}_A \oplus \mathcal{N}_A }$ 
via Lemma \ref{lattice} and $(\ref{IPrelations})$ (cf.~\cite[Lemma 3.1]{BY}), 
and we argue as in \cite[Lemma 7.3]{BY} that $c_{g_A}(\lambda, Q(\lambda)) = 0$ for all $\lambda \in \mathcal{P}_A^{\vee}$ 
unless $\lambda \in \mathcal{P}_A^{\vee} \cap L_A^{\vee} = {\bf{Z}} x_A$. We then define $L_A(s, \phi)$ by the corresponding relation
\begin{align*} L(s, \mathcal{S}_{\mu,m}(g_A)) = c_{g_A}(\mu,m) \cdot L_A(s, \phi). \end{align*}
Using the same argument as for $(\ref{HurwitzID})$, we compute 
\begin{align*} L'(0, g_A, U_A) &= (4 \pi m)^{-\frac{1}{2}} \Gamma \left(  \frac{1}{2} \right) \frac{L_A'(1/2, \phi)}{L(1, \eta_{D})} 
= \frac{ 2 N^{\frac{1}{2}}}{\pi} \cdot \frac{ c_{g_A}(\mu,m) L'_A(1/2, \phi)}{\deg Z_A(\mu, m)}, \end{align*}
which via Lemma \ref{7.4} (i) and (iii) is the same as 
\begin{align}\label{HurwitzA} L'(0, g_A, U_A) 
&= \frac{2 N^{\frac{1}{2}}}{ \pi \vert \vert g_A \vert \vert^2} \cdot \frac{ c_{g_A}(\mu,m) L'_A(1/2, \phi)}{\deg Z_A(\mu, m)}. \end{align}

We can interpret $L_A(s, \phi)$ as the partial/class basechange $L$-function of $\phi$ to $k$ with our unitary normalizations for 
the standard $L$-function (cf.~\cite{GZ}), so that in this setup we have the identifications of $L$-functions
\begin{align*} L(s, \phi \times \theta(\chi)) &= \sum\limits_{A \in C(\mathcal{O}_k)} \chi(A) L_A(s, \phi) L_A(s, \phi \otimes \eta_k)
=  \sum\limits_{A \in C(\mathcal{O}_k)} \chi(A) L(s, \phi \times \theta_A).\end{align*} 
To be clear, we define each $L_A(s, \phi)$ according to the partition of the lattice 
$ L = \bigoplus_{A \in C(O_k) \cong \mathcal{Q}_D} L_A$ so that 
\begin{align*} \sum\limits_{A \in C(\mathcal{O}_k)} L_A(s, \phi) &= L(s, \phi) \end{align*}
is identified with the finite part of the standard $L$-function $\Lambda(s, \phi) = L_{\infty}(s, \phi) L(s, \phi)$.
We can then define $L_A(s, \phi \otimes \eta)$ simply as the quadratic twist. 
Writing $\Pi = \operatorname{BC}_{k/{\bf{Q}}}(\pi(\phi))$ to denote the basechange lifting of automorphic representation 
$\pi(\phi)$ of $\operatorname{GL}_2({\bf{A}})$ associated to $\phi$ to $\operatorname{GL}_2({\bf{A}}_k)$, with standard $L$-function 
\begin{align*} L(s, \Pi) = L_{\infty}(s, \Pi)L(s, \Pi) = \Lambda(s, \phi) L(s, \phi \otimes \eta_k)
= L_{\infty}(s, \phi)L(s, \phi) L_{\infty}(s, \phi \otimes \eta_k) L(s, \phi \otimes \eta_k), \end{align*}
we have an identification of $L$-functions 
\begin{align*} \sum\limits_{A \in C(\mathcal{O}_k)} L_A(s, \phi) L_A(s, \phi \otimes \eta_k) &= L(s, \phi) L(s, \phi \otimes \eta_k) = L(s, \Pi).\end{align*} 
Moreover, we have for any character $\chi \in C(\mathcal{O}_k)^{\vee}$ the equivalence of $L$-functions 
\begin{align*} \sum\limits_{A \in C(\mathcal{O}_k)} \chi(A) L_A(s, \phi) L_A(s, \phi \otimes \eta_k) &= L(s, \Pi \otimes \chi) . \end{align*}
Note that we can justify this latter identification after comparing Dirichlet series expansions. Hence, we have 
\begin{equation*}\begin{aligned} \sum\limits_{A \in C(\mathcal{O}_k)} \chi(A) L_A(s, \phi) L_A(s, \phi \otimes \eta_k) 
&= L(s, \phi \times \theta(\chi)) := \sum\limits_{A \in C(\mathcal{O}_k)} \chi(A) L(s, \phi \times \theta_A). \end{aligned}\end{equation*}

%\begin{remark} Note that this realization of the Rankin-Selberg $L$-function 
%$L(s, \phi \times \theta(\chi))$ is distinct from that considered in 
%Gross-Zagier \cite{GZ}, where the corresponding $L_{\mathcal{A}}(s, f)$ for $\mathcal{A} \in C(\mathcal{O}_k)$ denotes to the 
%partial Rankin-Selberg $L$-function $L(s-1/2, \phi \times \theta_A)$ in our description above. In particular, the $L_A(s, \phi)$ here 
%forms a summand of the $\operatorname{GL}_2({\bf{A}})$-automorphic $L$-function $L(s, \phi)$ as described in Lemma \ref{7.3}. 
%In other words, we are working with some explicit form of the basechange
%equivalence $L(s, \Pi \otimes \chi) = L(s, \phi \times \theta(\chi))$ in this setup. \end{remark}

\subsubsection{Relation to arithmetic heights of Heegner divisors}

We argue as in the proof of Theorem \ref{GZF1} that when $D$ is prime to $2N$, 
we have for each class $A \in C(\mathcal{O}_k)$ the corresponding arithmetic height formulae 
\begin{align*} \left[ \widehat{\mathcal{Z}}_A^c(f_A), \mathcal{Z}_A(\mu, m) \right] 
&= - \frac{ N^{\frac{1}{2}} }{  \pi \vert \vert g_A \vert \vert^2} \cdot c_{A}(\mu, m) \cdot L'_A(1/2, \phi) \end{align*}
and 
\begin{align*} \left[  y_A(f_A), y_A(f_A) \right]_{\operatorname{NT}} 
&= \frac{N^{\frac{1}{2}}}{ \pi \vert \vert g_A \vert \vert^2} \cdot L_A'(1/2, \phi). \end{align*}
We then argue as in Corollary \ref{GZF1} that we can use the Waldspurger-like formula $(\ref{GKZ})$ to derive the formula
\begin{equation}\begin{aligned}\label{GZF1A} &\left[ y_A^{\phi}(\mu, m), y_A^{\phi}(\mu, m)\right]_{\operatorname{NT}} 
= \left[ c_{g_A}(\mu, m) y_A(f_A), c_{A}(\mu, m) y_A(f_A) \right]_{\operatorname{NT}} 
= c_{g_A}(\mu, m)^2 \cdot \left[ y_A(f_A), y_A(f_A) \right]_{\operatorname{NT}} \\
&= \frac{  \vert D \vert^{\frac{1}{2}} \vert \vert g_A \vert \vert^2 }{ 8 \pi N^{\frac{1}{2}} \vert \vert \phi \vert \vert^2  } \cdot L(1/2, \phi \otimes \eta_D) \cdot
\frac{ N^{\frac{1}{2}} }{  \pi \vert \vert g_A \vert \vert^2 } \cdot L_A'(1/2, \phi) 
= \frac{ \vert D \vert^{\frac{1}{2}} }{ 8 \pi^2 \vert \vert \phi \vert \vert^2 } \cdot L(1/2, \phi \otimes \eta_D) \cdot L_A'(1/2, \phi)\end{aligned}\end{equation}
for the $\phi$-isotypical components.
Taking the $\chi$-twisted linear combination for any $\chi \in C(\mathcal{O}_k)^{\vee}$, we obtain
\begin{equation}\begin{aligned}\label{chi} \left[ y_{\chi}^{\phi}, y_{\chi}^{\phi} \right]_{\operatorname{NT}} 
&:= \sum\limits_{A \in C(\mathcal{O}_k)} \chi(A) \left[ y_A^{\phi}(\mu, m), y_A^{\phi}(\mu, m)\right]_{\operatorname{NT}} \\
&= \frac{ \vert D \vert^{\frac{1}{2}} }{ 8 \pi^2 \vert \vert \phi \vert \vert^2 } 
\cdot L(1/2, \phi \otimes \eta_D) \sum\limits_{A \in C(\mathcal{O}_k)} \chi(A)  L_A'(1/2, \phi) \\
&= \frac{ \vert D \vert^{\frac{1}{2}} }{ 8 \pi^2 \vert \vert \phi \vert \vert^2 } \cdot L'(1/2, \Pi \otimes \chi) 
= \frac{ \vert D \vert^{\frac{1}{2}} }{ 8 \pi^2 \vert \vert \phi \vert \vert^2 } \cdot L'(1/2, \phi \times \theta(\chi)).\end{aligned}\end{equation}
    
\begin{theorem}[Gross-Zagier]\label{GZF}

Assume the fundamental discriminant $D$ is coprime to $2N$, and that $\vert D \vert >4$.
Let $f \in H_{1/2}(\omega_L)$ as described in Lemma \ref{7.4} with $g/\vert \vert g \vert \vert^2 = \xi_{1/2}(f) \in S_{3/2}(\overline{\omega}_L)$ 
and $\phi \in S_2^{\operatorname{new}, \star}(\Gamma_0(N))$ the corresponding Shimura lift.
For $\chi$ any character of the ideal class group $C(\mathcal{O}_k)$ of $k={\bf{Q}}(\sqrt{D})$, we have that
\begin{align*} \left[ y_{\chi}^{\phi}, y_{\chi}^{\phi} \right]_{\operatorname{NT}}
= \sum\limits_{A \in C(\mathcal{O}_k)} \chi(A) \left[ y_A^{\phi}(\mu, m), y_A^{\phi}(\mu, m)\right]_{\operatorname{NT}}
&= \frac{ \vert D \vert^{\frac{1}{2}} }{ 8 \pi^2 \vert \vert \phi \vert \vert^2 } \cdot L'(1/2, \phi \times \theta(\chi)). \end{align*}
%Expressing this terms of $\Lambda'(1/2, \phi \times \theta(\chi)) = L(1, \eta) L'(1/2, \phi \times \theta(\chi)) = \frac{2 \pi h_k}{w_k \vert D \vert^{\frac{1}{2}}} L'(1/2, \phi \times \theta(\chi))$, we obtain
%\begin{align*} \left[ y_{\chi}^{\phi}, y_{\chi}^{\phi} \right]_{\operatorname{NT}}
%&= \frac{ \vert D \vert^{\frac{1}{2}} }{ 8 \pi^2 \vert \vert \phi \vert \vert^2 } \cdot \left( \frac{w_k \vert D \vert^{\frac{1}{2}}}{  2 \pi h_k } \right)
%\Lambda'(1/2, \phi \times \theta(\chi)). \end{align*} %%There is a normalization: we must be proving the statement for Lambda, then dividing by L(1, \eta).
 \end{theorem}

\begin{corollary}\label{heightrelation} 

We have the arithmetic height formula
\begin{align*} 2  \cdot \left[ y_{\chi}^{\phi}, y_{\chi}^{\phi} \right]_{\operatorname{NT}}
&= \frac{\vert D \vert^{\frac{1}{2}}}{ 8 \pi^2 \vert \vert \phi \vert \vert^2 } \cdot \Lambda'(1/2, \phi \times \theta(\chi))
= -  2 \sum\limits_{A \in C(\mathcal{O}_k)} \chi(A) \left[ \widehat{\mathcal{Z}}^c(f_{0, A}) : \mathcal{Z}(V_{A,0}) \right],\end{align*}
where the twisted sum of special arithmetic divisors on the right-hand side represents the arithmetic Hirzebruch-Zagier 
divisors on $X_0(N) \times X_0(N)$ described for Theorem \ref{main} and Corollary \ref{ec} above. \end{corollary}

\section{Relation to metaplectic Fourier coefficients}

We now use the connection to the regularized theta lifts $\Phi(f_{1/2}, z) = G_{Z(f_{1/2})}(z) \in L^2(X_0(N))$ of Theorem \ref{GZF} 
and $\Phi(f_0, z) = G_{Z(f_0)}(z) \in L^2(X_0(N) \times X_0(N))$ of Theorem \ref{main} to relate the central derivative values 
$\Lambda'(1/2, \phi \times \theta(\chi))$ to Fourier coefficients of half-integral weight forms; cf.~the works \cite{BO}, \cite{KaSa}, and \cite{BFI}. 

\subsection{The setting of signature $(1,2)$ with $\Phi(f_{1/2}, z) \in L^{1 + \varepsilon}(X_0(N))$}

Consider the Maass form described in Lemma \ref{7.4}, adapted to class group representatives as in Theorem \ref{GZF} above.
Hence, we retain all of the setup of the previous section, with $(V,Q) = (\operatorname{Mat}_{2 \times 2}^{\operatorname{tr=0}}({\bf{Q}}), N \det(\cdot))$,
the lattice $L \subset V$ giving rise to the congruence subgroup $\Gamma_0(N) \subseteq \operatorname{SL}_2({\bf{Z}})$, 
and the sublattices $L_A \subset V$ corresponding to ideal classes $A \in C(\mathcal{O}_k)$ described in $(\ref{LAn=1})$. 
We fix a newform $\phi \in S_2^{ \operatorname{new}, \star} (\Gamma_0(N))$ which is invariant under the Fricke involution $w_N$.

\subsubsection{Quadratic sublattices}

If $k$ is an imaginary quadratic field of discriminant $d_k <0$, we fix a set of lattice representatives as in $(\ref{LAn=1})$ above, 
together with positive-norm vectors $x \in \Omega_{A, \mu, m}({\bf{Q}})$ and Heegner divisors $Z_A(\mu, m) \in Y_0(N)$ as in $(\ref{representatives})$. 
If $k$ is a real quadratic field of discriminant $d_k>0$, we fix for each class $A \in C(\mathcal{O}_k)$ 
an integer ideal representative $\mathfrak{a} \subset \mathcal{O}_k$. 
We then fix a Witt decomposition
\begin{align*} V_A &= \mathcal{K}_A \oplus {\bf{Q}} e_{A,1} \oplus {\bf{Q}} e_{A,2}, \end{align*}
so that the orthogonal complement $W_A = \mathcal{K}_A^{\perp} \subset V$ of the
subspace $\mathcal{K}_A$ of signature $(0, 1)$ is isomorphic to the subspace of signature $(1,1)$ determined by 
$(\mathfrak{a}_{\bf{Q}}, Q_{\mathfrak{a}}) = (\mathfrak{a}_{\bf{Q}}, {\bf{N}}_{k/{\bf{Q}}}(\cdot)/{\bf{N}} \mathfrak{a})$.
Here, we write $e_{A, j} \in V$ to denote the basis vectors with 
$(e_{A, j}, e_{A, j}) =0$ for $j=1,2$ and $(e_{A,1}, e_{A, 2}) =1$. We then write $L_{A, W} = L \cap W_A$ 
for the corresponding lattice as in the discussion above, leading to Theorem \ref{realquad} (for $n=1$).
Note that $(L_{A,W}, Q\vert_{W_A})$ corresponds to the quadratic lattice $(\mathfrak{a}, Q_{\mathfrak{a}}(\cdot))$.
Hence, for either case on the quadratic field $k$, we have for each $A \in C(\mathcal{O}_k)$ a  
sublattice $L_A \subset L$ corresponding to an ideal representative $\mathfrak{a} \subset \mathcal{O}_k$.

\subsubsection{Vector-valued Shimura lifts}

Fix a quadratic field $k$, real or imaginary, with discriminant $d_k$ and character $\eta_k(\cdot) = (\frac{d_k}{\cdot})$.
For each class $A \in C(\mathcal{O}_k)$, let $g_A \in S_{3/2}^{\operatorname{new}}(\overline{\omega}_{L_A})$ denote the 
holomorphic vector-valued cusp form of weight $3/2$ and representation $\overline{\omega}_{L_A}$ associated to $\phi$ 
by the Shimura correspondence $(\ref{Shimura})$. By Lemma \ref{7.4} (cf.~\cite[Lemma 7.4]{BY}, 
\cite[Lemma 7.3]{BO}), there exists for each class $A \in C(\mathcal{O}_k)$ a harmonic weak Maass form 
$f_{1/2, A} \in H_{1/2}(\omega_{L_A})$ such that  

\begin{itemize}

\item[(i)] We have the relation $\xi_{1/2}(f_{1/2, A}) = g_{A} / \vert \vert g_A \vert \vert^2$. 

\item[(ii)] The Fourier coefficients $c_{f_{1/2, A}}^+(\mu, m)$ of the holomorphic part $f^+_{1/2, A}$ lie in ${\bf{Q}}(\phi)$. 

\item[(iii)] The constant coefficient $c_{f_{1/2, A}}^+(0,0)$ vanishes. 

\end{itemize}

\begin{proposition}

For each class $A \in C(\mathcal{O}_k)$, consider the corresponding regularized theta lift 
\begin{align*} \Phi(f_{1/2, A}, z) = \Phi(f_{1/2, A}, z, 1) &= \int_{\mathcal{F}}^{\star} \langle \langle f_{1/2, A}(\tau), \theta_{L_A}(\tau, z, 1) \rangle \rangle d \mu(\tau) \\
&= \operatorname{CT}_{s=0} \left( \lim_{T \rightarrow \infty} \int\limits_{\mathcal{F}_T}
\langle \langle f_{1/2, A}(\tau), \theta_{L_A}(\tau, z, 1) \rangle \rangle v^{-s} d \mu(\tau) \right) \end{align*}
as a function of the variable $z \in D(V) = D^{\pm}(V) \cong \mathfrak{H}$. The following assertions are true. 

\begin{itemize}

\item[(i)] $\Phi(f_{1/2, A}, z)$ determines a weight-zero modular function on $X_0(N)$ with Laplacian eigenvalue $0$. 

\item[(ii)] $\Phi(f_{1/2, A}, z)$ is the automorphic Green's function for the special divisor $\mathcal{Z}^c(f_{1/2, A})$ on $\mathcal{X}_0(N)$. 

\item[(iii)] $\Phi(f_{1/2, A}, z) \in L^{1+\varepsilon}(X_0(N))$ for some $\varepsilon>0$. 

\end{itemize}

\end{proposition}

\begin{proof} See Theorem \ref{Bruinier} for $X_{K_A} \cong Y_0(N)$, extending to the cusps and using that $c_{f_{1/2, A}}^+(0,0)=0$. \end{proof}

\begin{remark} Note that for $A \in C(\mathcal{O}_k)$ with $k$ real quadratic, the Fourier series expansions of these modular functions $\Phi(f_{1/2, A}, z)$ 
can be calculated according to Bruinier-Ono \cite[Theorem 5.3, cf.~(5.10)]{BO}. More precisely, let $r \in {\bf{Z}}$ be any integer such that $d_k \equiv r^2 \bmod 4N$.
Let us for each lattice $L_A$ as defined in $(\ref{LAn=1})$ above fix a primitive isotropic vector $l_A \in L_A$ and $l_A' \in L_A^{\vee}$
so that $(l_A, l_A')=1$. We then have the corresponding negative definite subspace defined by $\mathcal{K}_A = L_A \cap l_A^{\perp} \cap l_A'$. 
We choose these vectors so that \begin{align*} \mathcal{K}_A = {\bf{Z}} \left(  \begin{array}{cc} 1 & 0 \\ 0 & - 1 \end{array}\right).\end{align*}
Given a vector $\lambda \in \mathcal{K}_A \otimes {\bf{R}}$, we write $\lambda >0$ if it is a positive multiple of 
$\left(\begin{array}{cc} 1 & 0 \\ 0 & -1\end{array} \right)$. We then have 
\begin{align*}\Phi(f_{1/2, A}, z) &= -4 \sum\limits_{ \lambda \in \mathcal{K}_A \atop \lambda >0 } \sum\limits_{b \bmod d_k} \eta_k(b) 
c_{f_{1/2, A}}^+ \left( \frac{ d_k  \lambda^2}{2}, r \lambda \right) \log \left\vert 1 - e \left( (\lambda, z) + \frac{b}{d_k} \right) \right\vert. \end{align*}

\end{remark}

\subsubsection{Relation to metaplectic Fourier coefficients}

If the functions $\Phi(f_{1/2, A}, z) \in L^{1+\varepsilon}(X_0(N))$ were square integrable, a minor generalization of the theorem of 
Katok-Sarnak \cite{KaSa} would relate the sums over CM and geodesic cycles of Theorems \ref{BY4.7} and \ref{realquad} for the case of $n=1$ to twisted sums over the half-integal weight 
forms $F_j$ related to $\Phi(f_{1/2, A}, z)$ by the Shimura correspondence $\operatorname{Shim}(F_j) = \Phi( f_{1/2, A}, \cdot)$ of the corresponding 
Fourier coefficients $c_{F_j}(d_k) \overline{c_{F_j}(1)}$. On the other hand, we have the following more precise result in this direction due to 
Bruinier-Funke-Imamoglu \cite{BFI}. To describe it, we first need to describe the following trace coefficients $\operatorname{tr}_{\mu, m}(\Phi(f_{1/2, A}))$ for each case on the quadratic field $k$.

When $d_k < 0$ so that $k$ is imaginary quadratic, define for each $m \in {\bf{Q}}$ and $\mu \in L_A^{\vee}/L_A$ the trace function 
\begin{align*} \operatorname{tr}_{\mu, m}\left( \Phi(f_{1/2, A}) \right) 
&= \sum\limits_{ x \in \Gamma_0(N) \backslash \Omega_{A, \mu, m}({\bf{Q}}) }
\frac{1}{\# \operatorname{Stab}_{\Gamma_0(N)}(x) } \cdot \Phi \left( f_{1/2, A}, \iota(D(V)_x) \right). \end{align*}
Here, we write $\iota$ to denote the identification $\iota: D(V) \cong \mathfrak{H}$. We also write 
$\iota(D(V)_x)$ to denote the image of $D(V)_x$ in the modular curve $X_0(N)$, so that for the special (Heegner) divisors $Z_A(\mu, m)$ we consider above,
\begin{align*} Z_A(\mu, m)({\bf{C}}) &= \sum\limits_{ x \in \Gamma_0(N) \backslash \Omega_{A, \mu, m}({\bf{Q}}) } \iota(D(V)_x). \end{align*}

When $d_k >0$ so that $k$ is real quadratic, we fix a vector $x \in V({\bf{Q}})$ of positive norm $m \in {\bf{Q}}$, and consider 
the corresponding geodesic in $D(V) \cong \mathfrak{H}$ defined by $\gamma_x = D(V)_x$. Here, we fix the following orientation:
\begin{align*} x = \left( \begin{array}{cc} 1 & 0 \\ 0 & -1\end{array} \right) \quad \implies \quad \gamma_x = \pm (0, i \infty) \quad
\text{is the imaginary axis with the orientation $\pm$}. \end{align*}
The orientation-preserving action of $\operatorname{SL}_2({\bf{R}})$ then induces an orientation on each geodesic $\gamma_x$. We also define
\begin{align*} d z_x 
&= \pm dz/\sqrt{m} z \quad \text{for} \quad x = \pm \sqrt{\frac{m}{N}}  \left( \begin{array}{cc} 1 & 0 \\ 0 & -1\end{array} \right). \end{align*}
Let $\alpha(x) = \operatorname{Stab}_{\Gamma_0(N)}(x) \backslash \gamma_x$, as well as its image in $X_0(N)$. Note that when
the stabilizer $\operatorname{Stab}_{\Gamma_0(N)}(x)$ is infinite, $\alpha(x)$ determines a closed geodesic in $X_0(N)$. 
We then define for each $m \in {\bf{Q}}_{>0}$ and $\mu \in L_A^{\vee}/L_A$ the corresponding trace function 
\begin{align*} \operatorname{tr}_{\mu, m}\left( \Phi(f_{1/2, A}) \right)  
&=  \frac{1}{2 \pi} \sum\limits_{ x \in \Gamma_0(N) \backslash \Omega_{A, \mu, m}({\bf{Q}}) } ~~ \int\limits_{\alpha(x)} \Phi(f_{1/2, A}, z) dz_x. \end{align*}
Here, each $\alpha(x)$ is a closed geodesic;\footnote{In the remaining case where $\alpha(x)$ is an infinite geodesic, equivalently 
when $\operatorname{Stab}_{\Gamma_0(N)}(x)$ is trivial and the complement $x^{\perp} \subset V$
determines an isotropic quadratic space -- which only happens if $Q(x) \in {\bf{N}}({\bf{Q}}^{\times})^2$ -- then the trace 
is defined according to the regularization procedure described in \cite[$\S$3.3]{BFI}, and the corresponding complementary trace in \cite[$\S$3.3.2]{BFI}.
However, as the orthogonal complement $x^{\perp} \subset V$ always determines an anisotropic subspace
$(W_A, Q\vert_{W_A}) \cong (\mathfrak{a}_{\bf{Q}}, Q_{\mathfrak{a}})$ in the setup we consider above, we do not need to consider these variations of the trace here.}
equivalently, $x^{\perp}$ is nonsplit over ${\bf{Q}}$. We again write
\begin{align*} \Omega_{A, \mu, m}({\bf{Q}}) = \left\lbrace x \in \mu + L_A: Q(x)=m \right\rbrace \end{align*}

\begin{theorem}[Bruinier-Funke-Imamoglu]\label{BFI} 

Fix any class $A \in C(\mathcal{O}_k)$. The generating series defined by 
\begin{equation*}\begin{aligned} &I_{1/2, \mu}(\Phi(f_{1/2, A}, \cdot), \tau) \\
&= -2 \sqrt{v} \operatorname{tr}_{\mu, 0}(\Phi(f_{1/2, A})) 
+ \sum\limits_{m<0} \operatorname{tr}_{\mu, m}(\Phi(f_{1/2, A})) 
\frac{ \operatorname{erf} (2 \sqrt{\pi \vert m \vert v}) }{2 \sqrt{\vert m \vert}} e(m \tau) 
+ \sum\limits_{m>0} \operatorname{tr}_{\mu, m}(\Phi(f_{1/2, A})) e(m \tau) \end{aligned}\end{equation*}
determines the $\mu$ part of a harmonic weak Maass form in $H_{1/2}(\omega_{L_A})$,
\begin{align*} I_{1/2}(\Phi(f_{1/2, A}), \tau) 
&= \sum\limits_{\mu \in L_A^{\vee}/L_A} I_{1/2, \mu}(\Phi(f_{1/2, A}), \tau) {\bf{1}}_{\mu}  \in H_{1/2}(\omega_{L_A}). \end{align*}

\end{theorem}

\begin{proof} This is special case of \cite[Theorem 4.1]{BFI}, detailed for the setup we consider above. \end{proof}

Now, we identify these traces with the CM cycles $Z(U_A) = Z(L_{A,0})$ and real geodesic cycles $\mathcal{G}(W_A)$ described for
Theorems \ref{BY4.7} and \ref{realquad} above (for the special case of $n=1$).

\begin{proposition}\label{traces} 

Let $k$ be any quadratic field of discriminant $d_k$ prime to the level $N$.
We have the following identifications of the traces defined above in terms of the sums $(\ref{Asums})$
of the regularized theta lift $\Phi(f_{1/2, A}, z)$ along the CM cycles $Z(U_A)$ defined in $(\ref{CMcycle})$ 
and the real geodesic cycles $\mathcal{G}(W_A)$ defined in $(\ref{geodesic})$. 

\begin{itemize}

\item[(i)] If $A \in C(\mathcal{O}_k)$ for $k$ an imaginary quadratic field of discriminant $d_k <0$, we have for each vector
$x \in \Omega_{A, \mu, m}({\bf{Q}})$ with orthogonal complement $U_A = U_A(x) := x^{\perp} \subset V$ the relation
\begin{align*} \operatorname{tr}_{\mu, m}(\Phi(f_{1/2, A})) &= \Phi(f_{1/2, A}, Z(U_A)). \end{align*}

\item[(ii)] If $A \in C(\mathcal{O}_k)$ for $k$ a real quadratic field of discriminant $d_k >0$, we have for each vector
$x \in \Omega_{A, \mu, m}({\bf{Q}}) \subset \mathcal{K}_A$ with orthogonal complement $W_A = W_A(x) := x^{\perp} \subset V$ the relation 
\begin{align*} \operatorname{tr}_{\mu, m}(\Phi(f_{1/2, A})) &= \Phi(f_{1/2, A}, \mathcal{G}(W_A)).\end{align*}

\end{itemize}

\end{proposition}

\begin{proof} Cf.~\cite[Proposition 7.2]{BY}. We use that $\Omega_{A, \mu, m}({\bf{A}}_f) = K_A x = K_{L_A} x$ in either case.  
Hence, for (i), we see from the corresponding definition $(\ref{special})$ of $Z_A(\mu, m)$ that $Z_A(\mu, m) = Z(U_A)$. 
The result then follows from the definitions. Similarly, for (ii), we see from the corresponding definition
$(\ref{geodesic})$ of the geodesic cycle $\mathcal{G}(W_A)$ with the natural identification $D(V)_x \cong D(W_A)$ that the identification follows from the definitions. \end{proof}

We can then deduce the following via Theorem \ref{BY4.7}, Theorem \ref{realquad}, and Lemma \ref{7.4} (cf.~\cite[Lemma 7.4]{BY}).

\begin{theorem}\label{metatrace} 

We have via Theorems \ref{BY4.7} and \ref{realquad} for the signature-$(1,2)$ spaces $(V_A,Q_A)$ the following identifications of central derivatives of Rankin-Selberg $L$-functions. 

\begin{itemize}

\item[(i)] Let $k$ be an imaginary quadratic field of discriminant $D = d_k <0$. 
Assume as in Lemma \ref{7.3} that $m = -D/4N$ for $D = -4Nm$ with $D \equiv r^2 \bmod 4N$, and take $\mu = \mu_r$.
Then, for $\chi$ any character of the ideal class group $C(\mathcal{O}_k)$, we have the relation 
\begin{align*} \sum\limits_{A \in C(\mathcal{O}_k)} \chi(A) \cdot c_{g_A}(\mu, m) \cdot \operatorname{tr}_{\mu, m} \left( \Phi(f_{A, 1/2}) \right) 
&= - \frac{ \vert D \vert^{\frac{1}{2}}}{ 8 \pi^2 \vert \vert \phi \vert \vert^2  } \cdot L'(1/2, \phi \times \theta(\chi)). \end{align*}
 
\item[(ii)] Let $k$ be a real quadratic field of discriminant $d_k >0$, and $x \in \Omega_{A, \mu, m}({\bf{Q}}) \subset \mathcal{K}_A$ a vector 
with orthogonal complement $W_A = W_A(x) := x^{\perp} \subset V$ as in Proposition \ref{traces} (ii). We have for each class $A \in C(\mathcal{O}_k)$ the relation 
\begin{equation*}\begin{aligned} &\operatorname{tr}_{\mu, m}(\Phi(f_{1/2, A})) \\ &= - \frac{2}{\operatorname{vol}(K_{W_A})} 
\left(  \operatorname{CT} \langle \langle f_{-1/2, A}^+(\tau), {\bf{1}}_{L_{W_A}^{\perp}} \otimes \mathcal{E}_{L_{W_A}}(\tau) \rangle \rangle
+ L'(0, g_A \times \theta_{L_{A, W}^{\perp}}) + I'(0, f_{1/2, A} \times \xi_{-\frac{1}{2}}(\theta_{L_{W_A}^{\perp}}) )\right). \end{aligned}\end{equation*}

\end{itemize}

\end{theorem}

\begin{proof}

For (i), we combine Proposition \ref{traces}, Theorem \ref{BY4.7}, Lemma \ref{7.3} (cf.~$(\ref{HurwitzA})$) and Lemma \ref{7.4} to find  
\begin{equation}\begin{aligned}\label{metaCM} \operatorname{tr}_{\mu, m}(\Phi(f_{1/2, A})) 
&= \Phi(f_{1/2, A}, Z(U_A)) = - \frac{\deg(Z(U_A))}{2} \cdot L'(0, \xi_{1/2}(f_{1/2, A}) \times \theta_{L_A^{\perp}}) \\
&= - \frac{\deg(Z(U_A))}{2} \cdot \frac{2 N^{\frac{1}{2}}}{  \pi \vert \vert g_A \vert \vert^2 } \cdot \frac{c_{g_A}(\mu,m) L'_A(1/2, \phi)}{\deg Z_A(\mu, m)} 
= -\frac{N^{\frac{1}{2}} \cdot c_{g_A}(\mu,m) }{  \pi \vert \vert g_A \vert \vert^2 } \cdot  L'_A(1/2, \phi). \end{aligned}\end{equation} 
A variation of the formula $(\ref{GKZ})$ implied by \cite[$\S$II.4, Corollary 1]{GKZ} gives us the corresponding relation
\begin{align}\label{GKZA} c_{g_A}(\mu, m)^2 &= \frac{\vert \vert g_A \vert \vert^2 \vert D \vert^{\frac{1}{2}}}{ 8 \pi N^{\frac{1}{2}} 
\vert \vert \phi \vert \vert^2} \cdot L_A(1/2, \phi \otimes \eta_D). \end{align}
Multiplying both sides of $(\ref{metaCM})$ by the Fourier coefficient $c_{g_A}(\mu, m)$ and applying $(\ref{GKZA})$, we find that 
\begin{equation*}\begin{aligned} c_{g_A}(\mu, m) \cdot \operatorname{tr}_{\mu, m} \left( \Phi(f_{A, 1/2}) \right) 
&=  -\frac{N^{\frac{1}{2}} \cdot c_{g_A}(\mu,m)^2 }{ \pi \vert \vert g_A \vert \vert^2 } \cdot  L'_A(1/2, \phi) 
= - \frac{ \vert D \vert^{\frac{1}{2}} \cdot L_A(1/2, \phi \otimes \eta_D) \cdot L_A'(1/2, \phi) }{ 8 \pi^2 \vert \vert \phi \vert \vert^2 }. \end{aligned}\end{equation*}
Taking the twisted linear combination for any class group character $\chi \in C(\mathcal{O}_k)^{\vee}$, we then find that 
\begin{equation*}\begin{aligned} \sum\limits_{A \in C(\mathcal{O}_k)} \chi(A) \cdot c_{g_A}(\mu, m) \cdot \operatorname{tr}_{\mu, m} \left( \Phi(f_{A, 1/2}) \right) 
&= - \frac{ \vert D \vert^{\frac{1}{2}}}{ 8 \pi^2 \vert \vert \phi \vert \vert^2  }  \sum\limits_{A \in C(\mathcal{O}_k)} \chi(A) L_A(1/2, \phi \otimes \eta_D) L_A'(1/2, \phi) \\
&= - \frac{ \vert D \vert^{\frac{1}{2}}}{ 8 \pi^2 \vert \vert \phi \vert \vert^2  } \cdot L'(1/2, \phi \times \theta(\chi)). \end{aligned}\end{equation*}

For (ii), we put together Proposition \ref{traces} and Theorem \ref{realquad} for this setting to get the claimed relation
\begin{equation*}\begin{aligned} &\operatorname{tr}_{\mu, m}(\Phi(f_{1/2, A})) 
= \Phi(f_{1/2, A}, \mathcal{G}(W_A)) \\ &= - \frac{2}{\operatorname{vol}(K_{W_A})} 
\left(  \operatorname{CT} \langle \langle f_{-1/2, A}^+(\tau), {\bf{1}}_{L_{W_A}^{\perp}} \otimes \mathcal{E}_{L_{W_A}}(\tau) \rangle \rangle
+ L'(0, \xi_{1/2}(f_{1/2, A}) \times \theta_{L_{A, W}^{\perp}}) + I'(0, f_{1/2, A} \times \xi_{-\frac{1}{2}}(\theta_{L_{W_A}^{\perp}}) )
\right). \end{aligned}\end{equation*} \end{proof}

\begin{remark} 

In Theorem \ref{metatrace} (ii), we expect there to be an analogue of the adjoint map of \cite[$\S$II.4]{GKZ}, 
to give a precise relation between $L(s, g_A)$ and $L(s, \phi)$. 
This would imply an analogue of \cite[$\S$II.4, Corollary 1]{GKZ}, giving a precise relation between the squares $c_{g_A}(\mu, m)^2$ and the central values $L(1/2, \phi \otimes \eta_D)$.
Some relation between these values is suggested by the works of Waldspurger \cite{Wa-de}, \cite{Wa-CdS}, and Kohnen-Zagier \cite{KZ}. 
A development of this approach would address the question posed in Bruinier-Ono \cite[Remark 4]{BO}. That is, the analogue of Theorem \ref{metatrace} (i)
in this setting would imply an explicit relation between the central derivative value $\Lambda'(1/2, \phi \times \theta(\chi))$ and the Fourier coefficients $f_{1/2, A}^+(\tau)$ 
of the holomorphic parts of $f_{1/2, A}(\tau) \in H_{1/2}(\omega_{L_A})$. \end{remark}

\subsection{The setting of signature $(2,2)$ with $\Phi(f_0, z) \in L^{1 + \varepsilon}(X_0(N) \times X_0(N))$}

We now return to the setup of Theorem \ref{main} and Corollary \ref{ec}.
Hence, for $k$ a real or imaginary quadratic field of discriminant $d_k$ prime to $N$ and 
quadratic Dirichlet character $\eta_k (\cdot) = (\frac{d_k}{\cdot})$, 
we fix for each ideal class $A \in C(\mathcal{O}_k)$
an integer ideal representative $\mathfrak{a}\subset \mathcal{O}_k$
with norm form $Q_{\mathfrak{a}}(\cdot) = {\bf{N}}_{k/{\bf{Q}}}(\cdot)/{\bf{N}} \mathfrak{a}$. We then
consider the quadratic space $(V_A, Q_A)$ of signature $(2,2)$ defined by 
$(V_A, Q_A) = (\mathfrak{a}_{\bf{Q}} \oplus \mathfrak{a}_{\bf{Q}}, Q_{\mathfrak{a}} - Q_{\mathfrak{a}})$.
We fix $L_A = N^{-1}\mathfrak{a} + N^{-1} \mathfrak{a} \subset V_A$ to be the lattice whose adelization
corresponds to the compact open subgroup $K_A = K_{L_A} = K_0(N) \oplus K_0(N)$, as described in Corollary \ref{lattices}.
We fix a cuspidal newform $\phi \in S_2^{\operatorname{new}}(\Gamma_0(N))$.
We then argue following Corollary \ref{existence} that there exists a vector-valued form $g_{\phi, A} \in S_2(\overline{\omega}_{L_A})$
lifting $\phi$. For each ideal class $A \in C(\mathcal{O}_k)$, we then take $f_{0, A} \in H_0(\omega_{L_A})$ to be
any cuspidal harmonic weak Maass form of weight zero and representation $\omega_{L_A}$ 
whose image under the differential operator $\xi_{1/2}$ equals $g_{\phi, A}$.

\begin{proposition}

For each class $A \in C(\mathcal{O}_k)$, consider the corresponding regularized theta lift 
\begin{align*} \Phi(f_{0, A}, z) = \Phi(f_{0, A}, z, 1) 
&= \int_{\mathcal{F}}^{\star} \langle \langle f_{0, A}(\tau), \theta_{L_A}(\tau, z, 1) \rangle \rangle d \mu(\tau) \\
&= \operatorname{CT}_{s=0} \left( \lim_{T \rightarrow \infty} \int\limits_{\mathcal{F}_T}
\langle \langle f_{0, A}(\tau), \theta_{L_A}(\tau, z, 1) \rangle \rangle v^{-s} d \mu(\tau) \right) \end{align*}
as a function of the variable $z \in D(V_A) = D^{\pm}(V_A) \cong \mathfrak{H}^2$. The following assertions are true. 

\begin{itemize}

\item[(i)] $\Phi(f_{0, A}, z)$ is a parallel weight-zero modular function on $X_0(N)^2$ with Laplacian eigenvalue $0$. 

\item[(ii)] $\Phi(f_{0, A}, z)$ is the automorphic Green's function for the special divisor $\mathcal{Z}^c(f_{0, A})$ on $\mathcal{X}_0(N)^2$. 

\item[(iii)] $\Phi(f_{0, A}, z) \in L^{1+\varepsilon}(X_0(N)^2)$ for some $\varepsilon >0$. 

\end{itemize}

\end{proposition}

\begin{proof} See Theorem \ref{Bruinier} for $X_{K_A} \cong Y_0(N)^2$, extending to the cusps and using that $c_{f_{0, A}}^+(0,0)=0$. \end{proof}

Again, we note that if these theta lifts $\Phi(f_{0,A}, z) \in L^{1 + \varepsilon}(X_0(N)^2)$ were square integrable, 
then a generalization of Katok-Sarnak \cite{KaSa} would relate the twisted linear combinations of 
Theorem \ref{RSIP} and Corollary \ref{simplified} to twisted linear combinations of metaplectic Hilbert modular forms (on $X_0(4N)^2$) of parallel 
weight $1/2$. Here, we expect there to be some version of Theorem \ref{BFI} that generalizes the theorems of Waldspurger \cite{Wa-de} and Gross-Kohnen-Zagier 
\cite[$\S$II.4, Corollary 1]{GKZ} relating central values of quadratic twists of $\operatorname{GL}_2({\bf{A}})$-automorphic $L$-functions to Fourier coefficients
of half-integral weight forms to central derivative values. 

\begin{conjecture} Retain the setups of Theorem \ref{BY4.7} and \ref{realquad}, with $k$ the corresponding quadratic field.
Through the connection to the sums $\Phi(f_{0, A}, V_{A, 0})$ and $\Phi(f_{0, A}, \mathcal{G}(W_A))$ of the automorphic Green's functions
$\Phi(f_{0, A})$ along CM cycles $V_{A, 0} \subset V_A$ or geodesic sets respectively via Theorem \ref{RSIP} and Corollary \ref{simplified}, 
we have for any class group character $\chi \in C(\mathcal{O}_k)^{\vee}$ a relation between the central derivative value $L'(1/2, \phi \times \theta(\chi))$ 
and a twisted linear combination of Fourier coefficients of some Hilbert modular form of parallel weight $3/2$ on $X_0(4N) \times X_0(4N)$,
as well as a relation to a twisted linear combination of Fourier coefficients of a harmonic weak Maass form of parallel weight $1/2$ and representations $\omega_{L_A}$.

\end{conjecture}\label{metatrace2}

That is, we expect to have an analogue of the main theorem of Bruinier-Funke-Imamoglu \cite{BFI} for this setup of type with
rational quadratic spaces $(V_A, Q_A)$ of signature $(2,2)$, linking to Fourier coefficients of harmonic weak Hilbert Maass forms of parallel weight $1/2$.
We also expect to have an analogue of the theorems of Waldspurger \cite{Wa-de}, \cite{Wa-CdS}, Kohnen-Zagier \cite{KZ}, 
and Gross-Kohnen-Zagier \cite[$\S$ II.4, Corollary 1]{GKZ} to express the central derivative values $L'(1/2, \phi \times \theta(\chi))$ -- at least for the principal
class group character $\chi = \chi_0$ where $L'(1/2, \phi \times \theta(\chi_0)) = L'(1/2, \phi) L(1/2, \phi \otimes \eta_k)$ -- relating to squares of Fourier 
coefficients of Hilbert modular forms of parallel weight $3/2$. Although we do not pursue the idea here, we use this method of proof of the Gross-Zagier 
formula via regularized theta lifts to illustrate the natural connection for some future work.

\end{document}